\documentclass[10pt]{amsbook}

\usepackage{amsmath, amsthm}
\usepackage{amssymb}
\usepackage{amscd}
\usepackage{latexsym}
\usepackage[latin1]{inputenc}
\usepackage[english]{babel}
\usepackage[shortlabels]{enumitem}
\usepackage{extpfeil}
\usepackage{bbm}

\usepackage{xypic}
\usepackage{graphicx}
\usepackage{changebar,color}

\usepackage{tikz-cd}

\DeclareMathOperator*{\bigast}{\raisebox{-0.6ex}{\scalebox{2.5}{$\ast$}}}
\newtheorem{theorem}{Theorem}[section]

\newtheorem{corollary}[theorem]{Corollary}
\newtheorem{lemma}[theorem]{Lemma}

\newtheorem{proposition}[theorem]{Proposition}

\newtheorem{scholium}[theorem]{Scholium}
\newtheorem{propositiondef}[theorem]{Proposition and Definition}

\theoremstyle{remark}

\theoremstyle{definition}
\newtheorem{example}[theorem]{Example}
\newtheorem{definition}[theorem]{Definition}

\theoremstyle{definition}
\newtheorem{definitions}[theorem]{Definitions}

\numberwithin{equation}{section}

\numberwithin{section}{chapter}

\newcommand{\Hom}{\operatorname{Hom}}

\usetikzlibrary{decorations.markings}
\makeatletter
\tikzcdset{
open/.code={\tikzcdset{hook, circled};},
closed/.code={\tikzcdset{hook, slashed};},
open'/.code={\tikzcdset{hook', circled};},
closed'/.code={\tikzcdset{hook', slashed};},
circled/.code={\tikzcdset{markwith={\draw (0,0) circle (.375ex);}};},
slashed/.code={\tikzcdset{markwith={\draw[-] (-.4ex,-.4ex) -- (.4ex,.4ex);}};},
markwith/.code={
\pgfutil@ifundefined{tikz@library@decorations.markings@loaded}%
{\pgfutil@packageerror{tikz-cd}{You need to say %
\string\usetikzlibrary{decorations.markings} to use arrow with markings}{}}{}%
\pgfkeysalso{/tikz/postaction={/tikz/decorate,
/tikz/decoration={
markings,
mark = at position 0.5 with
{#1}}}}},
}
\makeatother

\newcommand{\Vfa}{{\widetilde{\mathcal V}}}
\newcommand{\Vf}{{\widehat{\mathcal V}}}

\newcommand{\qaq}{\quad \mbox{and} \quad}

\renewcommand{\phi}{\varphi}

\def\A{\mathbb{A}}

\def\N{\mathbb{N}}
\def\Z{\mathbb{Z}}
\def\PP{\mathbb{P}}
\def\Q{\mathbb{Q}}
\def\R{\mathbb{R}}
\def\Rpa{\mathbb{R}_+^\ast}
\def\C{\mathbb{C}}

\def\F{\mathbb{F}}

\newcommand{\Qb}{{\overline {\mathbb Q}}}

\newcommand{\eg}{{\em e.g. }}

\newcommand{\cA}{{\mathcal  A}}
\newcommand{\cB}{{\mathcal  B}}

\newcommand{\cC}{{\mathcal  C}}

\newcommand{\cE}{{\mathcal  E}}
\newcommand{\cF}{{\mathcal  F}}

\newcommand{\cM}{{\mathcal M}}

\newcommand{\Laa}{{\mathcal L}}
\newcommand{\cO}{{\mathcal O}}

\newcommand{\cP}{{\mathcal P}}
\newcommand{\cT}{{\mathcal T}}

\newcommand{\cV}{{\mathcal  V}}
\newcommand{\cW}{{\mathcal  W}}

\newcommand{\cX}{{\mathcal X}}

\newcommand{\CTC}{\mathbf{CTC}}

\newcommand{\Eh}{{\widehat{E}}}
\newcommand{\Fh}{{\widehat{F}}}

\newcommand{\Vh}{{\widehat{V}}}

\newcommand{\Xh}{{\widehat{X}}}

\newcommand{\OK}{{{\mathcal O}_K}}
\newcommand{{\OL}}{{{\mathcal O}_L}}

\newcommand{\Lie}{{\rm Lie\,}}

\newcommand{\Spec}{{\rm Spec\, }}

\newcommand{\coker}{{\rm coker\,}}

\newcommand{\pr}{{\rm pr}}

\newcommand{\Pic}{{\rm Pic}}

\newcommand{\p}{\mathfrak p}

\newcommand{\Ebh}{{\widehat{\overline{E}}}}
\newcommand{\Fbh}{{\widehat{\overline{F}}}}

\newcommand{\Bb}{{\overline B}}
\newcommand{\Cb}{{\overline C}}
\newcommand{\Eb}{{\overline E}}

\newcommand{\Hb}{{\overline H}}
\newcommand{\Lb}{{\overline L}}
\newcommand{\Mb}{{\overline M}}
\newcommand{\Nb}{{\overline N}}
\newcommand{\Tb}{{\overline T}}
\newcommand{\Zb}{{\overline Z}}

\newcommand{\cOb}{\overline{\mathcal{O}}}
\newcommand{\cOK}{{\mathcal{O}_K}}

\newcommand{\dega}{\widehat{\rm deg }\,}

\newcommand{\im}{{\rm im\,}}

\newcommand{\fm}{{\mathfrak m}}

\newcommand{\an}{{\rm an}}

\newcommand{\Dir}{{\rm Dir}}

\newcommand{\ra}{\rightarrow}
\newcommand{\lrasim}{\stackrel{\sim}{\longrightarrow}}
\newcommand{\lra}{\longrightarrow}
\newcommand{\hra}{\hookrightarrow}

\newcommand{\hlra}{{\lhook\joinrel\longrightarrow}}

\newcommand{\hot}{{h^0_{\theta}}}

\newcommand{\Id}{{\mathrm{Id}}}

\renewcommand{\epsilon}{\varepsilon}

\newcommand{\hilb}{{\mathrm{Hilb}}}

\newcommand{\Db}{{\overline D}}

\newcommand{\et}{{\mathrm{\acute{e}t}}}

\newcommand{\supp}{{\mathrm{supp}\,}}

\newcommand{\Vcirc}{{\mathring{V}}}

\newcommand{\height}{{\mathrm{ht}}}

\renewcommand{\div}{{\mathrm{div}\,}}

\newcommand{\Ab}{{\overline A}}

\newcommand{\cS}{{\mathcal  S}}

\newcommand{\bD}{{\mathrm{b} \Delta}}
\newcommand{\cp}{{\mathrm{cp}}}
\newcommand{\bM}{{\mathbf{M}}}

\newcommand{\Ex}{{\mathrm{Ex}}}
\newcommand{\fa}{{f.-a.\!\! }}

\newcommand{\Spf}{{\rm Spf\, }}

\newcommand{\walpha}{{\widehat \alpha}}

\newcommand{\Ind}{{\mathrm{Ind}}}
\newcommand{\Cart}{{\mathrm{Cart}}}
\newcommand{\CbD}{{{\mathcal{C}^{\mathrm{b}\Delta}}}}

\newcommand{\Ld}{{L^2_1}}

\counterwithout{figure}{chapter}

\frontmatter

\date{today}

\title{Quasi-projective and formal-analytic arithmetic surfaces}
\author{Jean-Benoît Bost}
\address{Universit\'e Paris-Saclay, Laboratoire de Math\'ematiques d'Orsay, 91405 Orsay Cedex, France}
\email{jean-benoit.bost@math.u-psud.fr}
\author{Fran\c{c}ois Charles}
\address{\'Ecole Normale Sup\'erieure, DMA, UMR 8553 du CNRS, 75230 Paris Cedex, France, and Universit\'e Paris-Saclay, Laboratoire de Math\'ematiques d'Orsay, UMR 8628 du CNRS, 91405 Orsay Cedex, France}
\email{francois.charles@ens.fr}

\date{\today}

\begin{document}

\begin{abstract}

This memoir is devoted to the study of \emph{formal-analytic arithmetic surfaces}. These are arithmetic counterparts,  in the context of  Arakelov geometry, of germs of smooth complex-analytic surfaces along a projective complex curve,  or of smooth $2$-dimensional formal schemes fibered over a projective curve. Formal-analytic surfaces involve both an arithmetic and a complex-analytic aspect, and they  provide a natural framework for arithmetic algebraization theorems, old and new.

Formal-analytic arithmetic surfaces admit a rich  geometry --- whether considered intrinsically, or through their maps to arithmetic schemes, notably to arithmetic surfaces --- which parallels the geometry of complex analytic surfaces and its applications to the study of  complex algebraic varieties and algebraic surfaces. Notably the dichotomy between \emph{pseudoconvex} and \emph{pseudoconcave} formal-analytic arithmetic surfaces plays a central role in their geometry. 

Our study
 of formal-analytic arithmetic surfaces relies crucially on the use of real-valued numerical invariants. Some of these are intersection-theoretic, in the spirit of Arakelov intersection theory on projective arithmetic surfaces. Some other invariants involve infinite-dimensional geometry of numbers: they are defined by means of $\theta$-invariants attached to Euclidean lattices of infinite rank, and play the role of the dimension of spaces of sections of vector bundles in analytic or formal geometry. 
 
 Relating our new intersection-theoretic invariants to  more classical invariants, defined in terms of Arakelov intersection theory on projective arithmetic surfaces, leads us to investigate a new real invariant, the \emph{Archimedean overflow}, attached to an analytic map from a pointed compact Riemann surface with boundary to a Riemann surface. The Archimedean overflow may be expressed in terms of Green functions and harmonic measures, and thus may be related to the characteristic functions of Nevanlinna theory. 

Our results on the geometry of formal-analytic arithmetic surfaces admit diverse applications  to concrete problems of arithmetic geometry. 
Notably we generalize the arithmetic holonomicity theorem of Calegari-Dimitrov-Tang regarding the dimension of spaces of power series with integral coefficients satisfying some convergence conditions. 

We also establish an arithmetic counterpart of theorems of Lefschetz and Nori on fundamental groups of complex surfaces. This counterpart provides a bound on the index, in the \'etale fundamental group of a quasi-projective arithmetic surface, of the 
closed subgroup  generated by the \'etale fundamental groups of some arithmetic curve and of some compact Riemann surfaces with boundaries mapping 
to the arithmetic surface. 

In both these applications, the Archimedean overflow plays a central role. Actually these two applications derive from  arithmetic analogues, concerning pseudoconcave formal-analytic arithmetic surfaces,  of  results established by  Nori in the context of complex geometry, in his work on Zariski's conjecture. 

Transposing Nori's arguments in the arithmetic context requires the development of a more flexible version of Arakelov  intersection theory on arithmetic surfaces. This more flexible version, and the  construction of the Archimedean overflow as well, involve  the use of a class of Green functions on Riemann surfaces 
that satisfies both good functoriality and continuity properties, the Green functions 
with $\cC^{\bD}$ regularity.

\end{abstract}

\maketitle

\tableofcontents

\chapter*{Introduction}

\section[Introduction]{Formal-analytic surfaces as analogues of germs of analytic or formal surfaces along a projective curve}\label{Int01}

\subsection{} This memoir is devoted to the study of \emph{formal-analytic arithmetic surfaces} and to diverse applications of these to the geometry of projective and quasi-projective arithmetic surfaces.

Formal-analytic arithmetic surfaces have been introduced in \cite[Section 10.6]{Bost2020}, to provide a natural geometric framework to the algebraizations theorems  of the Chudnovskys in \cite{ChudnovskysGroth85} and \cite{ChudnovskysAcad85}, of  Andr\'e (see \cite{Andre04} for exposition and references), and of their developments in \cite{Bost01} and \cite{BostChambert-Loir07}. 
Special instances of formal-analytic arithmetic surfaces also occur implicitly in the recent paper \cite{CalegariDimitrovTang21} 
by Calegari, Dimitrov, and Tang, and their work has been an important inspiration for the authors of this memoir.\footnote{Understanding the relation between the arithmetic holonomicity theorem in \cite{CalegariDimitrovTang21} and their earlier results on formal-analytic arithmetic surfaces has 
been a major incentive for the authors to  establish the main result in Chapter 5 concerning the ``overflow" invariant, Theorem 5.4.1.}

At least implicitly,  formal-analytic arithmetic surfaces have been considered in various other contexts:  notably in the famous note 
\cite{Borel94} by E. Borel,   in the work of Harbater \cite{Harbater84, Harbater88}, and in the theory of Eisenstein series associated to loop groups over the integers, as developed by Garland and Patnaik in \cite{GarlandPatnaik08}.\footnote{See the recent work by Dutour and Patnaik \cite{DutourPatnaik22} for new developments and additional references.} 

Moreover formal-analytic arithmetic surfaces are closely related to 
Berkovich spaces over $\Z$, as studied by Poineau  \cite{Poineau2010}, \cite{Poineau2013} and Lamanissier and Poineau \cite{LemanissierPoineau2020}. They also constitute a natural ground for applying the new techniques of analytic geometry currently developed by Clausen and Scholze \cite{ClausenScholze22} in the framework of condensed mathematics, even though we will not attempt to do so here.   

\subsection{}\label{GraNor} The point of view developed in this memoir is firstly that formal-analytic arithmetic surfaces are interesting objects in themselves, which admit non-trivial global invariants, and secondly that  these invariants constitute a natural tool to investigate classical questions of arithmetic geometry.

Formal-analytic arithmetic surfaces constitute arithmetic counterparts of germs $\cV$ of complex analytic surfaces along a projective complex curve $C$, or in a more algebraic context, of two-dimensional smooth formal $k$-schemes $\widehat{\cV}$ with scheme of definition a projective curve $C$ over some field~$k$. 

As made clear by the classical work of Grauert on modifications \cite{Grauert62} and its application to singularities of surfaces \cite{Laufer71}, or by the work of Artin on contractions \cite{Artin70}, these germs of complex analytic surfaces and these formal surfaces naturally arise in the study of algebraic surfaces and two-dimensional algebraic spaces. Similarly, formal-analytic arithmetic surfaces are natural tools to investigate the geometry \emph{Ã  la Arakelov} of arithmetic surfaces, that is, of  two-dimensional integral quasi-projective flat schemes over $\Z$.

 In this memoir, we establish diverse results concerning morphisms between arithmetic surfaces and their fundamental groups which illustrate this philosophy. These arithmetic results may be seen as arithmetic counterparts of some of the results of Nori in \cite{Nori83} concerning quasi-projective complex surfaces. 
 
 The germs of analytic surfaces $\cV$ along a projective complex curve $C$ that appear in \cite{Nori83} satisfy a \emph{pseudoconcavity} condition --- basically a positivity condition on the normal bundle $N_C \cV$ or on the self-intersection $C\cdot C$ of $C$ in $\cV$ --- while in the work of Grauert and Artin on contractions, they satisfy a \emph{pseudoconvexity} condition, related to the negativity of $N_C \cV$ and $C\cdot C$. The dichotomy pseudoconvex/pseudoconcave turns out to govern also the arithmetic geometry of formal-analytic arithmetic surfaces. 
 
 
 \subsection{} To put it briefly, our aim in this memoir is to demonstrate that formal-analytic arithmetic surfaces admit a non-trivial geometry, relevant for the study of classical objects of arithmetic geometry --- including notably quasi-projective arithmetic surfaces and their \'etale fundamental groups. This geometry involves \emph{global real-valued invariants}, as in the  theory of heights  and Arakelov geometry. These invariants are suitably defined \emph{arithmetic intersection numbers}, and \emph{$\theta$-invariants} of possibly infinite dimensional Hermitian vector bundles, as defined in~\cite{Bost2020}.
 
 To achieve this aim without too lengthy foundational preliminaries, we have deliberately limited the generality of the class of formal-analytic arithmetic surfaces investigated in this memoir, by sticking to the class already introduced in~\cite[Chapter 10]{Bost2020}.\footnote{In \ref{regularfa}, we give  some  indications on a wider class of formal-analytic arithmetic surfaces to which most of our results extend.} 
 
 Our objectives have rather been (i) to clarify the geometric meaning of our constructions, notably by discussing in detail a series of results in complex analytic and algebraic geometry of which our main results results are arithmetic counterparts, (ii) to spell out a few ``concrete" consequences of our general finiteness results concerning pseudoconcave formal-analytic arithmetic surfaces that may be formulated in elementary terms, and (iii) to emphasize the role of a new Archimedean invariant --- the ``overflow" attached to a complex analytic map from a pointed compact Riemann with boundary to another Riemann surface --- which naturally arises when investigating the morphisms from formal-analytic to quasi-projective arithmetic surfaces.    
 
\section{Formal-analytic surfaces over $\Spec \Z$: Definition}\label{Int02} 

In this introduction, we present some of our main results concerning formal-analytic arithmetic surfaces, focusing on the  simple case of  formal-analytic arithmetic surfaces over $\Spec \Z$. 

\subsection{}\label{defBasicInt} We begin by introducing the main character of this memoir.
A \emph{smooth formal-analytic arithmetic surface 
 over $\Spec \Z$}  is defined as a triple:
$$\Vfa = (\widehat{\cV}, (V, O), \iota),$$
where:
\begin{itemize}
\item $\widehat{\cV}$ is a formal scheme, isomorphic to $\mathrm{Spf\, } \Z[[X]]$;
\item $V$ is a connected compact Riemann surface with non-empty boundary $V$,  equipped with a real structure,\footnote{that is, an antiholomorphic involution $c$. We shall call it ``complex conjugation."} and  $O$ is a real point\footnote{that is, a fixed point of the complex conjugation $c$.} in $\Vcirc$;
\item $\iota$ is an isomorphism of  complex formal curves, compatible with the real structures:
$$\iota: \widehat{\cV}_\C \lrasim \widehat{V}_O.$$
\end{itemize}
 
 The gluing data provided by the isomorphism $\iota$ may be described in more elementary terms as follows. 
 
 We may choose an analytic coordinate $z$ on some open neighborhood  of $O$ in 
 $V$ that is compatible with complex conjugation.\footnote{Namely it satisfies $z\circ c = \overline{z}.$} This coordinate establishes an isomorphism of complex formal curves:
 $$ \widehat{V}_O \lrasim \Spf \C[[z]].$$
 Moreover the isomorphism:
 $$\widehat{\cV}  \lrasim \mathrm{Spf\, } \Z[[X]]$$
induces an identification:
  $$\widehat{\cV}_\C  \lrasim \mathrm{Spf\, } \C[[X]].$$
  
  Therefore the data of the isomorphism $\iota$ is equivalent to the one of 
   a formal series $\psi \in \R[[X]]$ such that:
  \begin{equation}\label{CondpsiIntZero}
\psi(0) =0 \quad \mbox{and}  \quad \psi'(0) \neq 0;
\end{equation}
namely to the isomorphism:
$$\mathrm{Spf\, } \C[[X] \simeq  \widehat{\cV}_\C \stackrel{\iota}{\lrasim} \widehat{V}_O \simeq \mathrm{Spf\, } \C[[z]],$$
compatible with the real structures, is associated the series:
$$\psi := \iota^\ast z.$$

In particular, when the Riemann surface with boundary $V$ is a closed disk --- say when the pair $(V, O)$ is $(\overline{D}(0;1), 0)$ --- we may take the standard coordinate $z : \overline{D}(0;1) \hra \C$ as the analytic coordinate near $O$ in $V$, and we may associate a smooth formal-analytic arithmetic surface over $\Spec \Z$ to every formal series $\psi \in \R[[X]]$ satisfying  \eqref{CondpsiIntZero} (see \ref{VD01psiInt} below).

This discussion shows that formal-analytic arithmetic surfaces over $\Spec \Z$ are easily constructed ``soft" mathematical objects, and one may wonder whether they are worthy of interest. 

\subsection{} A first positive answer to this question is that, as already mentioned above,  formal-analytic arithmetic surfaces  arise naturally as the counterparts, in the dictionary between number fields and function fields, of the germs of analytic or formal surfaces fibered over a projective curve.

Let us explain this in more detail. 

\subsubsection{} Let $C$ be a smooth connected projective curve, say over the complex field $\C$, and let $\Sigma$ be a non-empty finite subset of $C$. The complement 
$$\mathring{C} := C \setminus \Sigma$$
is an affine curve, and its ring of regular functions $\cO(\mathring{C})$ is a Dedekind ring. 

 In the classical analogy between number fields and function fields, the ring of integers $\OK$ of a number field $K$ (resp. the scheme $\Spec \OK$) is seen  as the arithmetic counterpart of the ring $\cO(\mathring{C})$ (resp. of the smooth affine curve $\mathring{C}$), the set of Archimedean places of $K$ as the counterpart of $\Sigma,$ the Hermitian vector bundles over $\Spec \OK$ as the counterparts of the vector bundles over $C$, the (real valued) Arakelov degree of these Hermitian vector bundles  as the counterpart of the (integer valued) degree of vector bundles over $C,$ etc.

This analogy is pursued much further in Arakelov geometry, where a regular projective scheme $X$  over $\Spec \OK$, with the projective complex manifolds\footnote{We denote by $\sigma: K \hra \C$ the field embeddings of $K$ in $\C$. Their classes up to complex conjugation are in bijection with the Archimedean places of $K$. By $X_\sigma$, we denote the complex projective variety deduced from of the $\OK$-scheme $X$ by the base change $\sigma: \OK \ra \C$.} $(X_\sigma(\C))_{\sigma: K \hra \C}$  endowed with suitable K\"ahler structures, appears as the counterpart  of a smooth projective variety $X$ fibered over~$C$.   
  
\subsubsection{} The formal-analytic arithmetic surfaces investigated in this memoir constitute a new entry in the dictionary relating number fields and function fields. Their function field analogues are the following geometric objects. 

Consider a connected smooth complex analytic surface $\cV$, a surjective (necessarily flat) complex analytic map:
$$\pi_{\cV} : \cV \lra C$$
with connected fibers, and a complex analytic section of $\pi_\cV$:
$$\epsilon : C \lra \cV.$$

Assume moreover that, for every $x \in \Sigma,$ the connected curve $\pi_\cV^{-1}(x)$ is non-compact, and that we are given a reduced compact connected curve $F_x$ in $\pi_\cV^{-1}(x)$ containing $\epsilon (x)$. Then the divisor:
$$D:=\epsilon(C) + \sum_{x \in \Sigma} F_x$$
is compact and connected, and we may consider the germ $\cV^{\an}_D$ of complex analytic surface of $\cV$ along $D$. It is ``fibered over $C$", in the sense that it is equipped with the (germ of) analytic map:
\begin{equation}\label{piVanInt}
\pi_{\cV\mid \cV^{\an}_D} : \cV^{\an}_D \lra C.
\end{equation}

In the dictionary between number fields and function fields, the smooth formal-analytic surfaces over $\Spec \OK$ investigated in this memoir correspond to the germs of complex analytic surfaces $\cV^{\an}_D$ equipped with the map  \eqref{piVanInt} --- or to a formal variant of these, where germs of complex analytic surfaces along the compact divisor $D$ are replaced by formal surfaces admitting $D$ as scheme of definition.\footnote{This formal variant makes sense, not only over the complex field, but over an arbitrary base field.} 

\subsubsection{}\label{KeyExplainInt} The following remarks should clarify this correspondence. 

The germ of surface $\cV^{\an}_D$ along $D$ may be seen as the ``union" of the germ $\cV^{\an}_{\epsilon(\mathring{C})}$ of $\cV$ along the affine curve $\mathring{C}$, and of its germs $\cV^{\an}_{F_x}$   along the vertical divisors $F_x,$ $x \in \Sigma$, glued along the ``intersections":
\begin{equation}\label{InterIntx}
\cV^{\an}_{\epsilon(\mathring{C})} \cap  \cV^{\an}_{F_x}.
\end{equation}

In the above definition of smooth formal-analytic surfaces over $\Spec \Z,$ the ``algebraic" or ``formal" data:
$$\widehat{\cV} \simeq \Spf \Z[[X]]$$
and its structure map:
$$\widehat{\cV} \lra \Spec \Z$$
play the role of $\cV^{\an}_{\epsilon(\mathring{C})}$ and of the restriction:
$$\pi_{\cV\mid \epsilon(\mathring{C})} : \cV^{\an}_{\epsilon(\mathring{C})} \lra \mathring{C}.$$
The ``analytic data" $(V,O)$ play the role of $(\cV^{\an}_{F_x}, \epsilon_x)$, where $\epsilon_x$ denotes the germ of $\epsilon$ at $x$. Finally the isomorphism $\iota$ corresponds to the ``gluing isomorphism" along  \eqref{InterIntx}. 


\subsection{}\label{regularfa} In this monograph we work more generally with \emph{smooth formal-analytic surfaces over $\Spec \OK$}, where we denote by $K$ a number field, and by $\OK$ its ring of integers. A still more general --- and arguably more natural --- framework would have been the one of \emph{regular formal-analytic surfaces over $\Spec \Z$}. Those  are defined as triples $\Vfa:= (\widehat{\cV}, V, \iota),$ where:
\begin{itemize}
\item $\widehat{\cV}$ is a regular noetherian formal scheme of pure dimension $2$, whose scheme of definition $\vert \widehat{\cV} \vert$ is proper over $\Spec \Z$, of pure dimension 1;
\item $V$ is a compact Riemann with boundary, equipped with a real structure;
\item $\iota: \widehat{\cV}_\C \ra \mathring{V}$ is an embedding, compatible with complex conjugation, of the formal complex curve $\widehat{\cV}_\C$ into the interior  $\mathring{V}$ of $V$.
\end{itemize}

Many constructions and results in this memoir actually extend to this setting, provided  $\Vfa$ satisfies a natural connectedness condition, when moreover the following conditions are satisfied:
\begin{itemize}
\item
 every connected component of $V$ has an non-empty boundary; 
 \item for every prime $p$, every connected component of the formal scheme $\widehat{\cV}_{\F_p}$ over $\F_p$ is \emph{not} a scheme. 
 \end{itemize}
 
  However these extensions require  further foundational developments concerning formal schemes and their intersection theory, and we defer them to some future work.

\section{Formal-analytic surfaces over $\Spec \Z$: Further definitions and main results}

The above dictionary between number fields and function fields  extends naturally to various geometric objects involving  germs of analytic surfaces fibered over $C$, such as vector bundles or morphisms to algebraic varieties.

 The translation procedure for these diverse notions follows the same pattern as in \ref{KeyExplainInt}. Their arithmetic analogues are defined in terms of (i) an algebraic or formal part (that corresponds to their restriction over $\mathring{C}$ in the geometric context), (ii) an analytic part  (that corresponds to their restriction over the germs in $C$ of $x$ in the finite set $\Sigma$ that plays the role of Archimedean places), and (iii) some additional gluing data. In particular, as in ``classical" Arakelov geometry the arithmetic counterpart of vector bundles are Hermitian vector bundles over formal-analytic arithmetic surfaces.
 
 At this stage, the possibility of such a translation is hardly surprising. It is more remarkable that some basic notions of intersection theory on (possibly non-compact) complex analytic surfaces may be translated to formal-analytic arithmetic surfaces, in the spirit of Arakelov intersection numbers as defined in \cite{Deligne85} and \cite{Gillet-Soule90int}. 
 
 For instance, if $L$ is a (germ of) analytic line bundle over $\cV^{\an}_D$ and if $Z$ is a divisor supported on $D$, we may define the intersection number $L\cdot Z$ since $\vert Z \vert$ is compact.  Similarly, in the arithmetic setting, we may define the (real valued) Arakelov intersection number $\Lb \cdot Z$  of a Hermitian line bundle $\Lb$ over a formal-analytic arithmetic surface $\Vfa$  and of a suitably defined Arakelov divisor with compact support $Z$ on $\Vfa$.
 
 As  indicated in \ref{GraNor}  above, our main results in this memoir  concern formal-analytic surfaces $\Vfa$ that satisfy a suitable pseudoconcavity condition. 
  In this case, we shall firstly show that the ``spaces of global sections" of Hermitian vector bundles over $\Vfa$ satisfy some finiteness property, and that their ``absolute dimension" defined in terms of the $\theta$-invariants constructed in the monograph \cite{Bost2020}, satisfy remarkable estimates in terms of the Arakelov intersection numbers mentioned above.
 
 Secondly we will attach some intersection-theoretic invariants to morphisms from formal-analytic arithmetic surfaces to ``classical" quasi-projective arithmetic surfaces, and demonstrate their relevance to various questions involving quasi-projective arithmetic surfaces and their \'etale fundamental groups.

In this section, we give a sample of these results, which constitute a second positive answer to our previous question at the end of \ref{defBasicInt} concerning the significance of the notion of formal-analytic arithmetic surface. 

For simplicity, we  focus on the simple case of formal-analytic surfaces over $\Spec \Z$. We have tried to present precise and significant statements, without assuming any prior knowledge of Arakelov geometry, of potential theory, or of the $\theta$-invariants introduced in \cite{Bost2020}. Hopefully the self-contained character of this section will constitute an excuse for the terseness of its presentation. 

\medskip

We denote by $\Vfa:= (\widehat{\cV}, (V, O), \iota)$ a smooth formal-analytic arithmetic surface over $\Spec \Z$, as defined in \ref{defBasicInt}. 

The  structure map of $\widehat{\cV}$:
\begin{equation}\label{structurecVhat}
\widehat{\cV} \simeq \Spf \Z[[X]] \lra \Spec \Z
\end{equation}
defines an isomorphism:
$$\vert \widehat{\cV} \vert \lrasim \Spec \Z.$$
Its inverse defines a section (actually the unique section) of \eqref{structurecVhat}, which we shall denote by:
$$P: \Spec \Z \lra \widehat{\cV}.$$
We shall also denote by $P$ its image, that is the  scheme of definition $\vert \widehat{\cV} \vert$ of $\widehat{\cV}$.

\subsection{} If $X$ is an arithmetic scheme, namely a separated scheme of finite type over $\Spec \Z$, we define a morphism:
$$\alpha: \Vfa \lra X$$
as a pair:
$$\alpha:= (\widehat{\alpha}, \alpha^\an),$$
where:
$$\widehat{\alpha}: \widehat{\cV} \lra X$$
is a morphism of (formal) schemes, and where:
$$\alpha^{\an} : V \lra X(\C)$$
is a complex analytic map\footnote{analytic up to the boundary $\partial V$ of the Riemann surface with boundary $V$.} such that the following compatibility relation is satisfied:
\begin{equation}\label{alphaalphai}
\widehat{\alpha}_\C = \widehat{\alpha}^{\an} \circ \iota.
\end{equation}

In \eqref{alphaalphai}, we have denoted by:
$$\widehat{\alpha}_\C : \widehat{\cV}_\C \lra X_\C$$
the morphism of complex (formal) schemes deduced from $\widehat{\alpha}$ by the base change $\Z \hra \C,$ and by 
$$\widehat{\alpha}^{\an}: \widehat{V}_O \lra X_\C$$ the formal germ of ${\alpha}^{\an}$ at $O$. 

For instance, a morphism:
$f : \Vfa \ra \A^1_\Z$
is  a pair $(\widehat{f}, f^\an)$ where $\widehat{f}$ is an element of 
$\Gamma(\widehat{V}, \cO_{\widehat{V}}) \simeq \Z[[T]],$
and $f^{\an}$ an element of $\Gamma(V, \cO^{\an}_V)$ --- that is, a complex analytic function on $V$, analytic up to the boundary --- that satisfy the compatibility relation:
\begin{equation}\label{ffi}
\widehat{f}_\C = \widehat{f}^{\an} \circ \iota.
\end{equation}

These morphisms from $\Vfa$ to $\A^1_\Z$ define the $\Z$-algebra $\cO(\Vfa)$ of \emph{regular functions on} $\Vfa$.

We may similarly define the field $\cM(\Vfa)$ of \emph{meromorphic functions on} $\Vfa$, which is an extension field of $\Q$. Its elements are the pairs $f:= (\widehat{f}, f^\an)$ where  $\widehat{f}$ is a formal meromorphic function on $\widehat{V}$ --- or equivalently  an element of the fraction field $\mathrm{Frac\, } \Z[[T]]$ of $\Z[[T]]$ --- and $f^\an$ is a meromorphic function on $V$ (defined up to the boundary) such that the compatibility condition \eqref{ffi} is satisfied.

\emph{Vector bundles} and \emph{Hermitian vector bundles} over $\Vfa$ are defined by ``gluing" a vector bundle over $\widehat{\cV}$ and a vector bundle or a Hermitian vector bundle over the Riemann surface $V$. Namely a vector bundle  (resp. a Hermitian vector bundle) over $\Vfa$ is the data:
$$ E:=(\widehat E, E^\an, \phi) \quad  \quad (\mbox{resp. 
 ${\overline E}:=(\widehat E, E^\an, \phi, \Vert.\Vert)$}) $$
 of a vector bundle $\widehat E$ over $\widehat \cV$, of a complex analytic vector bundle $E^\an$ over $V$, and of an isomorphism of vector bundles over the complex formal curve $\widehat{\cV}_\C$:
 $$\phi: \widehat{E}_\C := \widehat{E} \otimes_\Z \C \lrasim \iota^\ast (E^\an_{\widehat{V}_O}),$$
 (resp. and of some $\cC^\infty$ Hermitian metric $\Vert. \Vert$ on the vector bundle $E^\an$ over $V$). These data are assumed to be compatible with complex conjugation.

\subsection{}\label{ArINtfaInt} It is possible to develop a version of Arakelov intersection theory on a formal-analytic arithmetic surface $\Vfa$ as above. In spite of its rudimentary character, this arithmetic intersection theory will allow us to associate 
some significant invariants to  formal-analytic arithmetic surfaces and to their morphisms with value in a ``classical" quasi-projective arithmetic surface --- namely in an integral  arithmetic scheme of dimension 2, quasi-projective and flat over $\Spec \Z$.

The main analytic tool for developing an Arakelov intersection theory on the formal-analytic arithmetic surface $\Vfa := (\widehat{\cV}, (V,O), \iota)$ are the \emph{Green functions} for the point $O$ in $V$ that satisfy the Dirichlet boundary condition. By definition, these are the real-valued $\cC^\infty$ functions $g$ on $V \setminus \{O\}$ such that:
\begin{equation}\label{Green1}
g_{\mid \partial V} =0,
\end{equation}
which admit a logarithmic singularity at $O$. This last condition means that, if $z$ is a local analytic coordinate on some open neigborhood $U$ of $O$ in $V,$ there exists $h \in \cC^\infty(U, \R)$ such that:
\begin{equation}\label{Green2}
g = \log \vert z - z(O) \vert^{-1} + h \quad \mbox{on $U \setminus \{O\}$.}
\end{equation} 

If $g$ is a Green function as above, invariant under the complex conjugation of $V$, then the pair $(P, g)$ may be seen as a compactly supported Arakelov divisor on $\Vfa$.

If moreover ${\Lb}:=(\widehat L, L^\an, \phi, \Vert.\Vert)$  is a Hermitian line bundle over $\Vfa$, we may define the \emph{height} of $P$ with respect to $\Lb$ as the Arakelov degree:
\begin{equation}\label{htpL}
\height_{\Lb}(P) := \dega P^\ast \Lb \in \R
\end{equation}
and the \emph{arithmetic intersection number of $\Lb$ and the Arakelov divisor $(P,g)$} as the sum:
\begin{equation}\label{LbPg}
\Lb \cdot (P, g) := \height_{\Lb}(P) + \int_V g\, c_1(\Lb_\C) \in \R.
\end{equation}

The definitions \eqref{htpL} and \eqref{LbPg} are similar to well known definitions concerning heights and Arakelov intersection numbers on ``classical" projective arithmetic surfaces. 

In the right-hand side of  \eqref{htpL}, $P^\ast \Lb$ is the Hermitian line bundle over $\Spec \Z$ defined by the free $\Z$-module of rank one $P^\ast \widehat{L}$ and by the norm $\Vert.\Vert_O$ on the complex line:
$$(P^\ast L)_\C \stackrel{\phi_{P_\C}}{\lrasim} L^\an_{\mid O}.$$
Moreover $ \dega P^\ast \Lb$ denotes the Arakelov degree of this Hermitian line bundle. If $s$ denotes a generator of the $\Z$-module $P^\ast \widehat{L}$, it is defined as:
$$\dega P^\ast \Lb := \log \Vert s \Vert_O^{-1}.$$

In the right-hand side of \eqref{LbPg}, we denote by $c_1(\Lb_\C)$ the first Chern form of the Hermitian line bundle $(L^\an, \Vert.\Vert)$ on $V$, defined by: 
$$c_1(\Lb)_{\mid U} := (2\pi i )^{-1} \partial \overline{\partial} \log \Vert s \Vert^2,$$
where $s$ is a non-vanishing complex analytic section of $L$ over some open subset $U$ of $M$. 

\subsection{}\label{equilibriumInt} Among the Green functions for the point $O$ in $V$, defined by conditions \eqref{Green1} and \eqref{Green2}, there is a distinguished one, namely the \emph{equilibrium potential} $g_{V, O}$, namely the unique Green function in the above sense that moreover is harmonic on $\mathring{V} \setminus \{O\}.$ 

If $V$ is embedded as a domain with $\cC^\infty$ boundary in some Riemann surface (without boundary) $V^+$, we may extend $g_{V, O}$ by $0$ on $V^+ \setminus V$. The extended function $g_{V, O}$ is continuous on $V^+ \setminus \{O\}$ and satisfies the following equality of currents on $V^+$:
$$\frac{i}{\pi} \partial \overline{\partial} g_{V, O} = \delta_O - \mu_{V, O},$$
where $\mu_{V, O}$ denotes a probability measure supported by the boundary $\partial V$ of $V$, the so-called \emph{harmonic measure} associated to the point $O$ in $V$, which is actually defined by a positive $\cC^\infty$  density on the smooth compact curve $\partial V$.

Moreover, by means of $g_{V,O}$ we may equip the tangent line $T_O V$ with a canonical norm, the \emph{capacitary norm} $\Vert.\Vert^{\mathrm{cap}}_{V,O}$, which may be defined as follows, in terms of a local analytic coordinate $z$ near $O$ and of the function $h$  in condition \eqref{Green2}:
\begin{equation}\label{capdefInt}
\Vert (\partial /\partial z)_{\mid P} \Vert^{\mathrm{cap}}_{V,O} = e^{- h(P)}.
\end{equation}

The normal bundle of $P$ in $\widehat{\cV}$, $N_P \widehat{\cV}$, is a line bundle over the section $P$. Its pull-back $P^\ast N_P \widehat{\cV}$ defines a line bundle over $\Spec \Z$. The complex line:
$$(P^\ast N_P \widehat{\cV})_\C \stackrel{D\phi_{\mid P_\C}}{\lrasim} T_O V$$
may be equipped with the capacitary norm $\Vert .  \Vert^{\mathrm{cap}}_{V,O}$, and we may attach to $\Vfa$ the following Hermitian line bundle over $\Spec \Z$:
$$\Nb_P\Vfa := (P^\ast  N_P \widehat{\cV}, \Vert .  \Vert^{\mathrm{cap}}_{V,O}).$$

Its Arakelov degree $\dega \Nb_P\Vfa$ turns out to be a fundamental invariant of $\Vfa$. It may also be interpreted as the self-intersection of the Arakelov divisor $(P, g_{V, O})$:
\begin{equation}\label{degaselfInt}\dega \Nb_P\Vfa = (P, g_{V, O}) \cdot (P, g_{V, O}).
\end{equation}

As in \cite[Chapter 10]{Bost2020}, we shall say that $\Vfa$ is \emph{pseudoconcave} when the following positivity condition is satisfied: 
\begin{equation}\label{pseudoconcavearith}
\dega \Nb_P \Vfa > 0.
\end{equation}
To a large extent, this memoir is an exploration of the consequences of this pseudoconcavity condition concerning the morphisms from $\Vfa$ to arithmetic schemes, and in particular to arithmetic surfaces.  

\subsection{}\label{VD01psiInt} Among the smooth formal-analytic surfaces $\Vfa := (\widehat{\cV}, (V, O), \iota)$ over $\Spec \Z,$ the ones such that the Riemann surface $V$ is simply connected admit a simple description. Up to isomorphism, these are the formal-analytic surfaces $\Vfa(\overline{D}(0;1), \psi)$ associated to some formal series $\psi$ in $\R[[X]]$ such that:
\begin{equation}\label{CondpsiInt}
\psi(0) =0 \quad \mbox{and}  \quad \psi'(0) \neq 0.
\end{equation}
by means of the following construction. 

By definition $\Vfa(\overline{D}(0;1), \psi)$ is the formal-analytic arithmetic surface $(\widehat{\cV}, (V, O), \iota)$ where\footnote{Recall that $\overline{D}(0;1)$ denotes the closed unit disk of center $0$ and radius  $1$ in $\C$.}: 
$$\widehat{\cV} := \mathrm{Spf\,} \Z[[X]], \quad V := \overline{D}(0;1), \quad O =0,$$
and:
$$\iota := \psi : \widehat{\cV}_\C \simeq \mathrm{Spf\,} \C[[X]] \lrasim \mathrm{Spf\,} \C[[z]] \simeq \widehat{D(0;1)}_0.$$

These formal-analytic arithmetic surfaces appear implicitly in \cite{CalegariDimitrovTang21}, through the associated algebra of regular functions $\cO(\Vfa(\overline{D}(0;1), \psi))$. This algebra admits an elementary description, as the ring of formal series $\widehat{\alpha}$ in $\Z[[T]]$ such that the complex formal series $\widehat{\alpha} \circ \psi^{-1}$ in $\C[[z]]$ --- where $\psi^{-1}$ denotes the compositional inverse of $\psi$ --- is the Taylor expansion at $0$ of some function $\alpha^\an$ holomorphic on some open neighborhood of $\overline{D}(0;1)$ in $\C$, or equivalently has a radius of convergence~$>1$.

The equilibrium potential $g_{V,O}$ attached to $(V, O) := (\overline{D}(0;1), 0)$ is the function $(z \mapsto \log^+ \vert z \vert^{-1})$, and the harmonic measure $\mu_{V,O}$ is the rotation invariant probability measure on the circle $\partial  \overline{D}(0;1)$.

The Hermitian line bundle $N_{\cP}\Vfa(\overline{D}(0;1), \psi)$ may be identified with $(\Z \, \partial/\partial X, \Vert.\Vert_\psi)$ where the metric $\Vert.\Vert_\psi$ satisfies:
$$ \Vert \psi'(0)^{-1} \partial/\partial X \Vert_\psi = 1.$$
Consequently:
\begin{equation*}
\dega N_{P}\Vfa(\overline{D}(0;1), \psi) = \log \vert \psi'(0)\vert^{-1},
\end{equation*}
and the pseudoconcavity condition \eqref{pseudoconcavearith} is satisfied if and only if:
$$\vert \psi'(0)\vert < 1.$$

\subsection{}   In this subsection, we present a first result concerning the geometry of pseudoconcave formal-analytic arithmetic surfaces, in which 
the 
notions of arithmetic intersection theory introduced in \ref{ArINtfaInt} and \ref{equilibriumInt} naturally enter.

\subsubsection{} Consider $\Vfa:= (\widehat{\cV}, (V, O), \iota) $ a smooth formal-analytic formal surface over $\Spec \Z$ as above, and assume that $V$ is equipped with a $\cC^\infty$ volume form invariant under complex conjugation.

Let moreover: 
$${\overline E}:=(\widehat E, E^\an, \phi, \Vert.\Vert)$$
be a Hermitian vector bundle over $\Vfa$.

To these data, we may attach the topological $\Z$-module $\Gamma(\Vf, \widehat E)$ of global sections of $\widehat E$ over $\widehat \cV$ --- it is a finitely generated projective $\Z[[X]]$-module --- and the space $\Gamma_{L^{2}}(V, \mu; E^\an, \Vert.\Vert)$
of complex analytic sections $s$ of $E$ over  $\mathring{V}$ such that:
$$\Vert s\Vert_{ L^{2}}^{2}:=\int_{\mathring{V}} \Vert s(x)\Vert^{2} d\mu(x)$$
is finite.

Endowed with the norm $\Vert.\Vert_{L^2}$, the space  $\Gamma_{L^{2}}(V, \mu; E^\an, \Vert.\Vert)$ is a complex Hilbert space, equipped  with a canonical real structure. Moreover the topological $\Z$-module 
 $\Gamma(\Vf, \widehat E)$ and the Hilbert space $(\Gamma_{L^{2}}(V, \mu; E^\an, \Vert.\Vert), \Vert.\Vert_{L^2})$ may be related by means of the ``gluing" isomorphisms $\iota$ and $\phi$ that define $\Vfa$ and $\widetilde{\overline E}$ respectively.
 
 Indeed the completed tensor product
 $\Gamma(\Vf, \widehat E) \hat{\otimes}_\Z \C$
 may be identified with the space $\Gamma (\Vf_\C, \widehat{E}_\C)$ of sections of the vector bundle $\widehat{E}_\C$ over the complex formal curve $\Vf_\C$. In turn, the isomorphisms $\iota$ and $\phi$ determine a canonical isomorphism:
 $$\Gamma (\Vf_\C, \widehat{E}_\C) \lrasim \Gamma( \widehat{V}_O, E^\an_{\widehat{V}_O}).$$
 
 Finally, by assigning its formal germ at $O$ to any $L^2$ holomorphic section of $E^\an$ over $\mathring{V}$, we define a ``jet map":
 $$\widehat{\eta}:  \Gamma_{L^{2}}(V, \mu; E^\an, \Vert.\Vert) \lra 
 \Gamma( \widehat{V}_O, E^\an_{\widehat{V}_O})
 \simeq \Gamma(\Vf, \widehat E) \hat{\otimes}_\Z \C.$$
 This map is easily seen to be injective, and to be continuous when $\Gamma( \widehat{V}_O, E^\an_{\widehat{V}_O})$ is equipped with its natural Fr\'echet topology. The  image of $\widehat{\eta}$ is actually dense  in $ \Gamma( \widehat{V}_O, E^\an_{\widehat{V}_O})$, and $\widehat{\eta}$ is compatible with complex conjugation.
 
 Accordingly the triple:
\begin{equation}\label{defpiastInt}
 \pi^{L^2}_{(\Vfa, \mu)\ast} {\overline E} := (\Gamma(\Vf, \widehat E), (\Gamma_{L^{2}}(V, \mu; E^\an, \Vert.\Vert), \Vert.\Vert_{L^2}), \widehat{\eta}),
 \end{equation}
 consisting in the topological $\Z$-module $\Gamma(\Vf, \widehat E)$, the Hilbert space $(\Gamma_{L^{2}}(V, \mu; E^\an, \Vert.\Vert), \Vert.\Vert_{L^2})$, and the jet map $\widehat{\eta}$ is an instance of a \emph{pro-Hermitian vector bundle over $\Spec \Z$} as defined in \cite[Chapter 5]{Bost2020}. 
 
 The pro-Hermitian vector bundles over $\Spec \Z$ constitute an infinite-dimensional generalization of the Hermitian vector bundles over $\Spec \Z$ --- that is, of Euclidean lattices --- and the monograph  \cite{Bost2020} develops a theory of the $\theta$-invariants $\hot$ attached to these objects. These invariants take  value in $[0, +\infty]$, and play the role, in an arithmetic setting, of the dimension over a base  field $k$ of the space of global sections $\Gamma(C, \widehat{E})$ of suitable ``pro-vector bundles" $\widehat{E}$ over a projective curve $C$ over $k$. In \cite[Chapter 7]{Bost2020} is constructed a natural class of pro-Hermitian vector bundles $\Ebh$ whose $\theta$-invariant $\hot(\Ebh)$ is well-defined and finite after any ``scaling" of their Hermitian structure, the \emph{$\theta$-finite pro-Hermitian vector bundles over $\Spec \Z$.}

\subsubsection{} Using the definition \eqref{defpiastInt}
 of $ \pi^{L^2}_{(\Vfa, \mu)\ast} {\overline E}$
and the notion of $\theta$-finite pro-Hermitian vector bundle over $\Spec \Z$ recalled above, 
we may formulate the following more precise version of a basic result  on pseudo-concave formal-analytic arithmetic surfaces established  in \cite[Chapter 10]{Bost2020}:

\begin{theorem} Let $\Vfa: = \big(\Vf, (V, P), \iota)$ be a pseudoconcave smooth formal-analytic  over $\Z$, and let  $\mu$ be a $\mathcal C^{\infty}$ positive volume form on $V$ invariant under complex conjugation.

(1)  For  every Hermitian vector bundle $\Eb$ over $\Vfa$, the pro-Hermitian vector bundle $\pi^{L^2}_{(\Vfa, \mu)\ast} {\overline E}$ is $\theta$-finite, and we may therefore define:
$$h^{0}_{\theta, L^{2}}(\Vfa, \mu; \Eb):= \hot \big( \pi^{L^2}_{(\Vfa, \mu)\ast}{\overline E} \big).$$ 

(2)  For every Hermitian line bundle $\Mb$ on $\Vfa$, when $D\in\mathbb N$ goes to infinity, we have:
\begin{equation}\label{equation:upper-boundInt}
h^{0}_{\theta, L^{2}}(\Vfa, \mu; \Mb^{\otimes D})=O(D^{2}).
\end{equation}
More precisely, when $\widehat\deg\,P^{*}\Mb<0$, we have:
\begin{equation}\label{equation:negative}
\lim_{D\ra+\infty}h^{0}_{\theta, L^{2}}(\Vfa, \mu; \Mb^{\otimes D})=0,
\end{equation}
and  in general:
\begin{equation}\label{equation:positiveInt}
\limsup_{D\ra+\infty}D^{-2} \, h^{0}_{\theta, L^{2}}(\Vfa, \mu; \Mb^{\otimes D})\leq \frac{1}{2}\frac{\big(\Mb \cdot (P, g_{\Vfa_\C})\big)^{2}}{\widehat\deg\, \overline N_{P}\Vfa}.
\end{equation}
\end{theorem}

\subsubsection{} As already indicated in \cite[Section 10.2]{Bost2020}, and explained in more details in Section \ref{subsubsection:alternative-geom} of this memoir, the $\theta$-finiteness of $ \pi^{L^2}_{(\Vfa, \mu)\ast} {\overline E}$ and the asymptotic estimate \eqref{equation:upper-boundInt} may be seen as arithmetic analogues of some classical results of Andreotti \cite{Andreotti63} concerning pseudoconcave complex analytic spaces.

As shown in \cite[Â§3-4]{Andreotti63} for   pseudoconcave complex analytic spaces, the asymptotic estimate \eqref{equation:upper-boundInt} implies an algebraicity result concerning the image of morphisms from pseudoconcave formal-analytic arithmetic surfaces to arithmetic schemes:

\begin{corollary}[Compare \protect{\cite[Theorem 10.8.1]{Bost2020}}]
For every pseudoconcave smooth for\-mal-analytic arithmetic surface $\Vfa$ over $\Spec \Z$ and for every morphism:
$$\alpha: \Vfa \lra X$$
from $\Vfa$ to some quasi-projective arithmetic scheme $X,$ there exists a quasi-projective arithmetic surface $S$ and a closed embedding $i: S \hra X$ such that $\alpha$ factors through $i$.
 \end{corollary}
 
 The more precise upper-bound \eqref{equation:positiveInt} will allow us to establish the degree bound \eqref{equation:degree-boundInt}  in Corollary \ref{CorInt} and consequently the bounds \eqref{MVfaInt} and \eqref{pi1Inteq} in our main  finiteness  results, Theorems \ref{degbd1Int}, \ref{meroInt},  and~\ref{pi1Int}.

\subsection{}\label{DefDalphaInt} Consider a morphism:
$$\alpha:= (\widehat{\alpha}, \alpha^{\an}): \Vfa \lra X,$$
from the smooth formal-analytic surface $\Vfa$ over $\Spec \Z$ with value in some  normal quasi-projective arithmetic surface $X$, and assume that the morphism:
$$\widehat{\alpha}_\Q: \Vfa_\Q \lra X_\Q,$$
from the germ of formal curve $\Vfa_\Q \simeq \Spf \Q[[T]]$ to the smooth curve $X_\Q$ over $\Q$, is not constant. 

We may introduce the Arakelov divisor with compact support on $X$ defined as the direct image by $\alpha$ of the Arakelov divisor $(P, g_{V,O})$ on $\Vfa$, namely:
\begin{equation}\label{alphaalphaintro}
\alpha_\ast (P, g_{V,O}) := (\widehat{\alpha}_\ast P, \alpha^{\an}_\ast g_{V,O}).
\end{equation}
Its self-intersection:
$$\alpha_\ast (P, g_{V,O}) \cdot\alpha_\ast (P, g_{V,O})$$
--- which makes sense in the formalism of arithmetic intersection on  quasi-projective arithmetic surfaces developed in Part 2 --- is a well defined real number.\footnote{Observe that the direct image $\alpha^{\an}_\ast g_{V,O}$ is not a Green function with $\cC^\infty$ regularity as used in classical arithmetic intersection theory, which is therefore not adequate to define this self-intersection.}

When moreover the formal-analytic arithmetic surface $\Vfa$ satisfies the pseudoconcavity condition \eqref{pseudoconcavearith}, the self-intersection \eqref{alphaalphaintro} is positive, and  we may attach to the morphism $\alpha$ the following positive invariant, which plays a central role in this monograph:
\begin{equation}\label{Dalphadefintro}
D(\alpha: \Vfa \ra X) := \frac{\alpha_\ast (P, g_{V,O}) \cdot\alpha_\ast (P, g_{V,O})}{\dega \Nb_P \Vfa}.
\end{equation}

According to \eqref{degaselfInt}, the invariant $D(\alpha: \Vfa \ra X)$ is the quotient by the self-intersection of the Arakelov divisor $(P, g_{V,O})$ of the self-intersection of its direct image by $\alpha$, and is therefore a natural invariant from a formal perspective. 
Remarkably enough, it is possible to express it in terms of classical quantities involving the ``finite" and ``Archimedean" components $\widehat{\alpha}$ and $\alpha^\an$ of the morphism $\alpha$.

For simplicity, we will only indicate here that this expression for  $D(\alpha: \Vfa \ra X)$ shows that it satisfies the  lower bound:
$$D(\alpha: \Vfa \ra X) \geq e(\alpha),$$
where $e(\alpha)$ denotes the ramification index of $\widehat{\alpha}_\Q$, and write down its special form 
when $\Vfa$ is the formal-arithmetic surface $\Vfa(\overline{D}(0;1), \psi)$ attached to a formal series $\psi \in \R[[X]]$ as in \ref{VD01psiInt} and when $X$ is the affine line $\A^1_\Z$:

\begin{proposition} For every $\psi \in \R[[X]]$ satisfying conditions \eqref{CondpsiInt} and every morphism:
$$\alpha : 
\Vfa(\Db(0,1), \psi) \lra \A^1_\Z$$
such that $\widehat{\alpha}_\Q$ is non-constant, 
the following equality holds:
\begin{equation*}\label{selfintpsiAInt}
\alpha_\ast (P, g_{\Vfa_\C}) \cdot \alpha_\ast (P,g_{\Vfa_\C}) = 2 \int_0^1 \int_0^1 \log \left\vert \alpha^\an (e^{2 \pi i t_1}) - \alpha^\an (e^{2 \pi i t_2}) \right\vert \, dt_1 \, dt_2.
\end{equation*}

In particular, when moreover $\Vfa(\Db(0,1), \psi)$ is pseudoconcave,\footnote{That is, when $\vert \psi'(0)\vert < 1$.}  we have:
\begin{equation}\label{DalphaA1Int}
D(\alpha: \Vfa \ra \A^1_\Z) = 2 \big( \log \vert \psi'(0)\vert^{-1} \big)^{-1}  \int_0^1 \int_0^1 \log \left\vert \alpha^\an (e^{2 \pi i t_1}) - \alpha^\an (e^{2 \pi i t_2}) \right\vert \, dt_1 \, dt_2.
\end{equation}
\end{proposition}

\subsection{} Having the invariant $D(\alpha: \Vfa \ra X)$ at our disposal, we may formulate, in a  simplified setting, some of the main results results of this memoir. 

\begin{theorem}\label{degbd1Int} Let $\Vfa$ be a pseudoconcave formal-analytic arithmetic surface over $\Spec \Z,$ and let $U$ and $V$ be two integral normal arithmetic surfaces. Consider a commutative diagram: 
\begin{equation}\label{VfaUVInt}
\xymatrix{
& V\ar[d]^{f}\\
\Vfa\ar[r]^{\alpha}\ar[ur]^{\beta} & U,
}
\end{equation}
where $\alpha$ and $\beta$ are morphisms from the formal-analytic arithmetic surface $\Vfa$ to the arithmetic surfaces $U$ and $V$, and where $f$ is a morphism of schemes. 

If $\alpha_\C$ is non-constant, and therefore so is  $\beta_\C$, then  $f$ is dominant and generically finite, and its degree $\deg f$  satisfies the upper bound:
\begin{equation}\label{equation:degree-boundInt}
\deg f \leq \frac{D(\alpha: \Vfa \ra U)}{D(\beta: \Vfa \ra V)}.
\end{equation}
\end{theorem}

The commutativity of the diagram \eqref{VfaUVInt} means, by definition, that the following two diagrams are commutative:
\begin{equation*}
\xymatrix{
& V\ar[d]^{f}\\
\widehat{\cV}\ar[r]^{\widehat\alpha}\ar[ur]^{\widehat\beta} & U,
}
\end{equation*}
and
\begin{equation*}
\xymatrix{
& V(\C) \ar[d]^{f_\C}\\
V \ar[r]^{\alpha^\an}\ar[ur]^{\beta^\an} & U(\C).
}
\end{equation*}
Actually, the commutativity of any of these two diagrams implies the commutativity of the other one.

Observe also that, combined with the estimates:
$$ D(\beta: \Vfa \ra V) \geq e(\beta) \geq 1,$$
 the degree bound \eqref{equation:degree-boundInt}  implies the following one:

\begin{corollary}\label{CorInt} With the notation of Theorem \ref{degbd1Int}, the following inequality holds:
\begin{equation}\label{equation:degree-boundIntBis}
\deg f \leq  D(\alpha: \Vfa \ra U).
\end{equation}
\end{corollary}
The fact that the right-hand side of \eqref{equation:degree-boundIntBis} does not depend of $V$ or $\beta$ plays a key role in the proof of diverse results concerning pseudoconcave arithmetic surfaces, for instance of  Theorems \ref{meroInt}, \ref{OVfafinInt}, and \ref{pi1Int}.

\begin{theorem}\label{meroInt}
For every pseudoconcave smooth formal-analytic arithmetic surface $\Vfa$ over $\Spec \Z,$ the field of meromorphic functions $\cM(\Vfa)$ is either $\Q$, or a finitely generated extension of $\Q$ of transcendence degree $1$.

Moreover, if $f$ is an element of $\cO(\Vfa)$ not in $\Q$, then $f$ seen as an element of $\cM(\Vfa)$ is transcendental over $\Q$, and the degree of $\cM(\Vfa)$  as a field extension of $\Q(f)$ satisfies the following upper bound:
\begin{equation}\label{MVfaInt}
[\cM(\Vfa) : \Q(f) ] \leq D(f: \Vfa \ra \A^1_\Z).
\end{equation}
 
\end{theorem}

In the main body of the text, we establish more general forms of Theorems \ref{degbd1Int}  and \ref{meroInt} where $\Vfa$ may be a smooth formal-analytic arithmetic surface over $\Spec \OK$, with $K$ an arbitrary number field and $\OK$ its ring of integers, and where the maps $\alpha$, $\beta$, or $f$ are allowed to be, not only morphisms, but more general \emph{meromorphic maps}. 

In particular, the second half of Theorem \ref{meroInt} still holds when $f$ is an arbitrary element of $\cM(\Vfa)$ not in $\Q$, with a suitable definition of the invariant $D(f)$ in the right-hand side of \eqref{MVfaInt}.

\subsection{} By elaborating on the degree bounds established in Theorems  \ref{degbd1Int}  and \ref{meroInt}, it is possible to establish further results concerning pseudoconcave formal-analytic arithmetic surfaces and their morphisms to quasi-projective arithmetic surfaces. We conclude this section by presenting two of these results.

\subsubsection{}  For every smooth formal-analytic surface $\Vfa$ over $\Spec \Z,$ the $\Z$-algebra $\cO(\Vfa)$ is a domain. Its fraction field $\mathrm{Frac} \, \cO(\Vfa)$ may be identified to a subfield of $\cM(\Vfa)$, but in general may be distinct from $\cO(\Vfa)$. 

According to Theorem \ref{meroInt}, this fraction field is a finitely generated extension of $\Q$, of transcendence degree at most 1. The following theorem establishes a stronger  finiteness result:

\begin{theorem}\label{OVfafinInt} For every pseudoconcave smooth formal-analytic arithmetic surface $\Vfa$ over $\Spec \Z$, the $\Z$-algebra $\cO(\Vfa)$ is finitely generated.
\end{theorem}

\subsubsection{} As in \ref{DefDalphaInt} above, let:
$$\alpha:= (\widehat{\alpha}, \alpha^{\an}): \Vfa \lra X$$
be  a morphism
from a smooth formal-analytic surface $\Vfa:= (\widehat{\cV}, (V, O), \iota) $ over $\Spec \Z$ with value in some  normal quasi-projective arithmetic surface $X$, and assume that the morphism 
$\widehat{\alpha}_\Q$ is non-constant. 

We may consider the \'etale fundamental groups $\pi^{\mathrm{\acute{e}t}}_1(V,O)$ and $\pi^{\mathrm{\acute{e}t}}_1(X, \alpha^\an(O))$. The former may be identified with the profinite completion of the topological fundamental group $\pi_1(V, O)$, and the map:
$$V \stackrel{\alpha^{\an}}{\lra} X_\C \lra X$$
defines a continuous morphism of profinite groups:
$$\alpha_\ast : \pi^{\mathrm{\acute{e}t}}_1(V,O) \lra  \pi^{\mathrm{\acute{e}t}}_1(X, \alpha^\an(O)).$$

The following theorem is an arithmetic avatar of a generalization due to Nori \cite{Nori83} of the classical theorem of Lefschetz concerning the fundamental groups of hyperplane sections of projective complex varieties.

\begin{theorem}\label{pi1Int}
With the above notation, if $\Vfa$ is pseudoconcave, then $\alpha_\ast (\pi^{\mathrm{\acute{e}t}}_1(V,O))$ is a subgroup of finite index in $\pi^{\mathrm{\acute{e}t}}_1(X, \alpha_\C(O))$. Moreover:
\begin{equation}\label{pi1Inteq}
[\pi^{\mathrm{\acute{e}t}}_1(X, \alpha^\an(O))) : \alpha_\ast (\pi^{\mathrm{\acute{e}t}}_1(V,O)) ] \leq D(\alpha: \Vfa \ra X).
\end{equation}
\end{theorem} 

These results on the \'etale fundamental groups of arithmetic surfaces may be applied to integral models of modular curves. 

For instance, for every integer $N \geq 3,$ we may consider the scheme $\mathcal{Y}(N)^{\mathrm{arith}}$ defined as in \cite[Section 2.5]{Katz76}\footnote{where it is denoted by$M(\Gamma(N)^{\mathrm{arith}})$.}  as representing the functor that maps a base scheme $S$ to the isomorphism classes of pairs $(\cE, \iota)$, where $\cE$ is an elliptic curve over $S$ and $\iota$ is an isomorphism of finite flat group schemes over $S$:
$$\iota: (\mu_N \times \Z/N\Z)_S \lrasim \cE[N].$$

It is a smooth affine curve over $\Spec \Z$, with geometrically irreducible fibers, and as a consequence of Theorem \ref{pi1Int}, we may prove:

\begin{corollary}
 For every integer $N \geq 3$ and every geometric point $\ast$ of  $\mathcal{Y}(N)^{\mathrm{arith}}$, the \'etale fundamental group $\pi^{\mathrm{\acute{e}t}}_1(\mathcal{Y}(N)^{\mathrm{arith}}, \ast)$ is finite.
\end{corollary}

\section{Contents of the memoir} In this subsection, we briefly describe the contents of the successive chapters of this memoir. The reader is referred to the introduction of each of these chapters for a more complete description.

In Part 1, consisting in Chapters 1 and 2, we present a series of results concerning algebraic and analytic complex varieties which constitute    ``geometric models" for the arithmetic results  established in Part 2 and 3 of this memoir.  Strictly speaking, the content of Part 1 is not used in Part 2 and 3. However it provides  motivation and  principles of proofs for our later arithmetic developments. 

In the results of complex algebraic geometry presented in  Chapters 1 and 2,  an important role is played by  auxiliary complex analytic surfaces satisfying a pseudoconcavity condition.  Besides Andreotti's foundational work on pseudoconcave complex spaces \cite{Andreotti63}, a major source of inspiration for these geometric results --- and consequently for the arithmetic results in this memoir --- has been Nori's famous article \cite{Nori83}  on the fundamental group of complex quasi-projective surfaces.

Most of the  geometric ideas underlying our work already appear in Chapter 1, in a simplified but significant setting.  Chapter 2, which develops  these geometric ideas in a more general framework, is more technical and could be skipped at first reading. 

Part 2 contains some foundational results, concerning Arakelov intersection theory on arithmetic surfaces and Green functions on Riemann surfaces. These results will be required in the final chapters of this work to transpose the geometric arguments in Chapters 1 and 2 to arithmetic surfaces. 

Technically, we shall need a formalism of intersection theory Ã  la Arakelov that allows us to handle direct images of Arakelov cycles by dominant morphisms between arithmetic surfaces, and also to consider arithmetic intersection theory on possibly non-projective arithmetic surfaces. Such a formalism is developed in Chapters 3 and 4. 

The content of Chapter 5 is of a purely analytic nature. It introduces an invariant attached to non-constant analytic maps from a pointed compact connected Riemann surface with boundary to another Riemann surface, its ``overflow". This invariant will  play a key role in the sequel, when computing self-intersections of Arakelov-cycles on arithmetic surfaces attached to morphisms from formal-analytic arithmetic surfaces. It enters in the expression for the invariant $D(\alpha: \Vfa \ra X)$ introduced in Subsection \ref{DefDalphaInt} above of which \eqref{DalphaA1Int} is a very special case. 

One of the most significant results of the memoir is established in Chapter 5, namely the alternative expression for the overflow in Theorem 5.4.1 and its application to the relation between the overflow and the Nevanlinna characteristic function.  

Part 3 is devoted to the main subject of the memoir: the formal analytic-arithmetic surfaces --- which, as explained in Sections \ref{Int01} and \ref{Int02} above, should be thought of as  arithmetic counterparts of the (germs of) analytic surfaces investigated in Part 1 --- and their applications to the study of arithmetic surfaces.

Chapter 6 introduces these formal analytic-arithmetic surfaces, and develops some rudimentary arithmetic intersection theory on them. Notably it introduces the dichotomy between pseudoconvex and pseudoconcave formal analytic-arithmetic surfaces. Section 6.3 investigates a simple but significant class of examples, the formal-analytic arithmetic surfaces $\Vfa(\overline{D}(0;1), \psi)$ mentioned in Subsection \ref{VD01psiInt} above, which are closely related to the constructions in \cite{CalegariDimitrovTang21}. 

Using some arguments of ``infinite dimensional geometry of numbers," we  also show in Chapter 6 that, in this class of examples, generic pseudoconcave formal analytic-arithmetic surfaces admit no non-constant regular function. Contrariwise, the pseudoconvex ones admit a large supply of regular functions, and are  therefore similar to  Stein manifolds and their modifications in complex analytic geometry. These results are not used in the following chapters, but put in  perspective our later results on pseudoconcave formal-analytic arithmetic surfaces.

Chapters 7 investigates the morphisms from formal-analytic arithmetic surfaces to quasi-pro\-jec\-tive arithmetic surfaces. It contains a few technical, but important results: the computations of intersection numbers attached to these morphisms, and the generalizations of these computations when morphisms are replaced by  meromorphic maps.

Chapter 8 establishes the main results of the memoir, concerning pseudoconcave formal-analytic arithmetic surfaces and their morphisms to arithmetic schemes. We notably establish some finiteness results concerning the field $\cM(\Vfa)$ of meromorphic functions on a pseudoconcave formal-analytic arithmetic surface $\Vfa$. These results are established by transposing in the arithmetic setting the arguments in Part 1, by using the new tools developed in Part 2. 

Chapter 9 contains some more advanced developments of these results: the finite generation of the algebra of regular functions on a pseudoconcave formal-analytic arithmetic surface, and some finiteness results concerning \'etale fundamental groups of arithmetic surfaces that are the analogues of the theorem of Lefschetz-Nori concerning \'etale fundamental groups of quasi-projective complex surfaces presented in Chapters 1 and 2. We conclude the chapter by some applications to finiteness results on  \'etale fundamental groups of integral models of modular curves.

\newpage

\vspace{3cm}

\begin{center}{\bf Leitfaden}
\end{center}

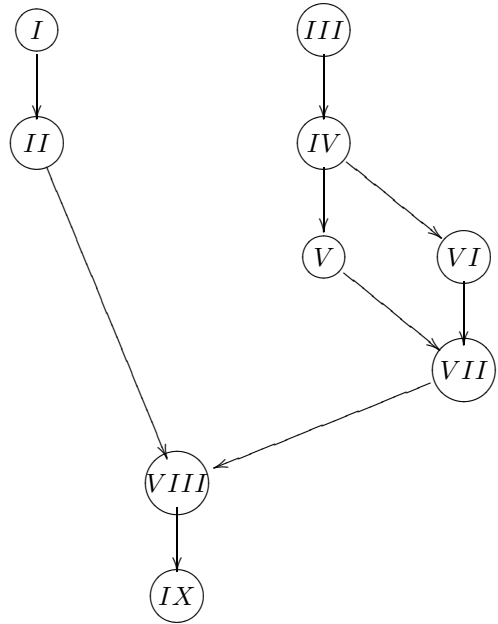
\vspace{1cm}

$$
\xymatrix{
&&*++[o][F-]{I}\ar[d]&  & *++[o][F-]{III} \ar[d] &  & \\
&&*++[o][F-]{II}\ar[dddr] & & *++[o][F-]{IV} \ar[d] \ar[dr] &  &  \\
&&&  &   *++[o][F-]{V}\ar[dr]&   *++[o][F-]{VI} \ar[d] \\
&&&  &   &  *++[o][F-]{VII} \ar[dll] \\
&&& *++[o][F-]{VIII} \ar[d]& & \\
&& &  *++[o][F-]{IX} & & }
$$

\vspace{2cm}

\newpage 

\section{Acknowledgments}

During the writing of this monograph, the authors benefited from the support of the ERC project AlgTateGro (Horizon 2020 Research and Innovation Programme, grant agreement No 715747).

The authors are very grateful to Frank Calegari, Vesselin Dimitrov, and Yunqing Tang for communicating them a first version of their paper \cite{CalegariDimitrovTang21} in August 2021,  to Vincent Pilloni for useful discussions of integral models of modular curves, and to Laurent Moret-Bailly and Madhav Nori for comments on a first version of this work. 

The notion of ``overflow" for a non-constant analytic morphism of a pointed connected Riemann surface with boundary to a Riemann surface was introduced, in the unramified case, during some discussion in 2009 between Antoine Chambert-Loir and one of the authors about possible developments of their joint paper \cite{BostChambert-Loir07}. Suitably extended, this notion plays a central role in the ``concrete" consequences of the results in this paper, and the authors would like to thank Antoine Chambert-Loir for this very helpful discussion.

\section{Conventions and notation}

\subsection{Notation}  
For every $x \in \R,$ we use the notation: $$x^+ := \max(0, x) \quad \mbox{and} 
\quad
\lfloor x \rfloor  := \max \{ n \in \Z \mid n \leq x \}.$$

For every $a \in \C$ and $r\in \R_+,$ we note:
$$\mathring{D}(a; r) := \{z \in \C \mid \vert z -a  \vert < r \} \quad \mbox{and} \quad 
\overline{D}(a;r) := \{z \in \C \mid \vert z -a \vert \leq r \}.$$

If $A$ is an integral domain, we denote  its field of fractions by $\mathrm{Frac}\, A$.

When $K$ is a number field, we will usually denote by $\{\sigma: K \hra\C\}$ the family of field  embedding of $K$ in $\C$, and by $\OK$ its ring of integers.

If $M$ is a module over some ring $A$, and if $B$ is some commutative $A$-algebra, we denote by $M_B$ the ``base changed" module $M\otimes_A B$. Similarly, if $\phi: M\lra N$ is a morphism of $A$-modules, we let: $$\phi_B := \phi \otimes_A {\Id}_B : M_B \lra N_B,$$
and if $\cX$ is some $A$-scheme, we let:
$$\cX_B := \cX \times_{\Spec A} \Spec B.$$

When $K$ is a number field, with ring of integers $\OK,$ and $\sigma: K \hlra \C$ is a field embedding, these base change constructions may be applied to $A:= \OK$ and $B:= \C$, considered as an $A$-algebra by means of $\sigma.$ Then we denote $M_B$ (resp. $\cX_B$) by $M_\sigma$ (resp. $\cX_\sigma$).

If $D$ is a Cartier divisor on a scheme or a complex analytic space $X,$ we denote by $\mathbbm 1_D$ the canonical  section of the line bundle $\cO_X(D)$ over $X$, which admits $D$ as divisor. When moreover $D$ is effective, the \emph{normal bundle} $N_D X$ is defined as the restriction to $D$ of the line bundle $\cO_X(D)$.  

A line bundle $L$ over a scheme $X$ projective over some field $k$ is \emph{nef} when, for every closed integral subscheme $C$ of dimension 1 in $X$, we have:
$$\deg_{C} L_{\vert C} := \deg_k c_1(L) \cap [C] \geq 0;$$ 
see for instance \cite[Section 2.5 and Definition 1.4]{FultonIT} for the definition of the operation $c_1(L) \cap \cdot$ and of the degree $\deg_k$ of $0$-cycles.

If $X$ is an integral projective surface over $k$, then a nef line bundle $L$ over $X$ is \emph{big} if and only if:
$$L\cdot L := \deg_k c_1(L)^2 \cap [X]$$
is positive; 
see \cite[section 2.2]{Lazarsfeld04}. The nef and big line bundles over $X$ are sometimes called \emph{numerically positive}; see for instance \cite{SemSzpiro81}.

A \emph{Hermitian vector bundle over a reduced analytic space $M$} is a pair $\Eb:= (E, \Vert.\Vert)$ where $E$ is a complex analytic vector bundle over $M$, and $\Vert.\Vert$ is a continuous Hermitian metric on $E$ over $M$.

A \emph{Hermitian vector bundle over a reduced, separated scheme $X$ of finite type over }$\Z$ is defined as a pair $\Eb := (E, \Vert.\Vert)$ where $E$ is a vector bundle over $X$ and $\Vert.\Vert$ is a continuous Hermitian metric on the complex analytic vector bundle $E_\C^{\an}$ over $X(\C)$; the metric $\Vert.\Vert$ is moreover required to be invariant under complex conjugation. We shall denote by $\Eb_\C$ the Hermitian vector bundle $(E_\C^{\an}, \Vert.\Vert)$ over the analytic space $X(\C).$

A \emph{Hermitian line bundle} is a Hermitian vector bundle of rank one.

If $\Lb := (L, \Vert.\Vert)$ is Hermitian line bundle with a $\cC^\infty$ metric $\Vert.\Vert$ over a complex manifold $M$, its \emph{first Chern form} $c_1(\Lb)$ is the $\cC^\infty$ real $(1,1)$-form on $M$ defined by:
$$c_1(\Lb)_{\mid U} := (2\pi i )^{-1} \partial \overline{\partial} \log \Vert s \Vert^2,$$
where $s$ is a non-vanishing complex analytic section of $L$ over some open subset $U$ of $M$. 

We shall denote by $\cM(\cV)$ the ring of  meromorphic functions (resp. of formal  meromorphic functions) on some complex analytic space (resp. on some formal scheme) $\cV$. 

We shall sometimes use the expression \emph{arithmetic scheme} to designate a separated scheme of finite type over $\Spec \Z$, not necessarily flat. 


\subsection{Riemann surfaces with boundary}\label{RSB}

As in \cite[10.5.1]{Bost2020}, we define a \emph{Riemann surface with boundary} as a pair $(V, V^{+})$ where $V^{+}$ is the germ of a Riemann surface along a closed $\mathcal C^{\infty}$ submanifold with boundary $V$ of $V^{+}$, of codimension $0$ in $V^{+}$. The \emph{interior} of $(V, V^{+})$ is defined as the interior $\mathring{V}$ of $V$ in $V^+$. It is a Riemann surface without boundary. The \emph{boundary} of $(V, V^{+})$ is defined as
$$\partial V:=V\setminus \mathring{V}.$$
We say that $(V, V^{+})$ is compact (resp. connected) when $V$ is. For simplicity's sake, we will often write $V$ for the Riemann surface with boundary $(V, V^{+}),$ but will write:
$$\alpha: V^+ \lra N$$
for a complex analytic map to emphasize that $\alpha$ is analytic up to the boundary.

A \emph{real structure} on the Riemann surface with boundary $(V, V^{+})$ is an antiholomorphic involution of $V^+$ that preserves $V$. 

A complex analytic vector bundle $E$ on the Riemann surface with boundary $(V, V^{+})$ is a complex analytic vector bundle on the germ of Riemann surfaces $V^+$. A Hermitian vector bundle $\Eb := (E, \Vert.\Vert)$ on the Riemann surface with boundary $(V, V^{+})$  is the data of a complex analytic vector bundle $E$ on $V^+$ and of a continuous Hermitian metric on $E$ over $V$.

\subsection{Arithmetic surfaces}\label{subsection:recollection-AS} We recall some basic definitions and some classical facts concerning arithmetic surfaces that will be used freely in this article. We refer to \cite{Shafarevich66, Lichtenbaum68, Lipman78, Artin86, Moret-Bailly89} for proofs and additional references.

\subsubsection{}\label{BasicAS} An \emph{arithmetic surface} $X$ is a scheme separated of finite type and flat over $\Spec \Z$, such that every component of $\vert X \vert$ has Krull dimension 2. An \emph{integral curve} in $X$ is a closed integral subscheme of $X$ of Krull dimension 1. An integral curve $C$ in $X$ is either \emph{horizontal}, when the morphism $C\ra \Spec \Z$ is flat, hence quasi-finite, or \emph{vertical}, when the morphism $C \ra \Spec \Z$ factorizes as 
$$C \lra \Spec \F_p \hlra \Spec \Z$$
for some prime $p$.


Let $X$ be a normal arithmetic surface. Its contains a maximal open regular subscheme $X_{\mathrm{reg}}$, and the complement $X \setminus X_{\mathrm{reg}}$ is a finite union of closed points of $X$. As a consequence of the resolution of singularities for arithmetic surfaces, the local Picard group at each of these points is finite (\cite{Moret-Bailly89}). This admits the following consequences:
\begin{enumerate}[(i)]
\item Any Weil divisor on $X$ is $\Q$-Cartier; in other words, the scheme $X$ is $\Q$--factorial. More precisely, there exists a positive integer $N(X)$ such that, for every Weil divisor $Z$ on $X$, the divisor $N(X) Z$ is a Cartier divisor.

\item If $X$ is an affine  normal arithmetic surface, an open subscheme $V$ of $X$ is affine if and only if the complement $\vert X \vert \setminus V$ has pure dimension 1, or equivalently is a finite union of integral curves. 

\end{enumerate}

By Nagata's compactification theorem, any arithmetic surface $X$ may be embedded as an open subscheme in some arithmetic surface $\overline{X}$ that is proper over $\Spec \Z$. When $X$ is normal, $\overline{X}$ may be chosen to be normal. Any such normal proper arithmetic surface is actually projective over $\Spec \Z$. Indeed a horizontal effective Cartier divisor\footnote{that is,   an effective Cartier divisor flat over $\Spec \Z$.} $D$ in a projective normal integral arithmetic surface $X$ is ample if and only if its support $\vert D \vert$ meets every vertical integral curve   (\cite{Lichtenbaum68}).

\subsubsection{} Let $X$ be an integral normal arithmetic surface. Then $X_\Q$ is a smooth integral $\Q$-scheme of dimension 1. Its field of constants $K$, defined as  the algebraic closure of $\Q$ in the field of rational functions $\kappa(X_\Q) = \kappa (X)$, is a number field. The scheme $X_\Q$ is canonically a $K$-scheme, and as such, is a smooth geometrically integral curve. If we denote by $\cO_K$ the ring of integers of $K$, the image of the inclusion $\cOK \hra \kappa(X)$ lies in the subring $\cO_X(X)$ of the field $\kappa(X)$ consisting of the regular functions on $X$ , since $X$ is integral and normal. The inclusion $\cO_K \hra \cO_X(X)$ defines a morphism of schemes:
$$\pi : X \lra \Spec \cO_K.$$ 
The morphism $\pi$ is smooth with geometrically connected fibers over some non-empty open subscheme of $\Spec \cO_K$.

Conversely, for any integral normal arithmetic surface $X$ and every number field $K$, if a morphism of schemes: 
$$\pi' : X \lra \Spec \cO_K$$
has an integral geometric generic fiber $X_{\overline{K}}$, then $K$ ``is" the field of constants of $X_\Q$, and $\pi'$ is the morphism $\pi$ described above. This is the case for instance when the morphism $\pi'$ admits a section.

When the  integral normal arithmetic surface $X$ is projective, the morphism $\pi$ coincides with the Stein factorization of the morphism $X \ra \Spec \Z$.

The Riemann surface $X(\C)$ whose points are the complex points of the $\Z$-scheme $X$ may be identified with the disjoint union: 
$$\coprod_{\sigma : K\hlra \C}X_{\sigma}(\C),$$
where $X_{\sigma}$ denotes the connected smooth complex  curve $X\times_{\OK, \sigma}\C.$

\subsubsection{} The results and techniques in \cite{Moret-Bailly89}  lead to the following descriptions of open or closed affine subschemes of an integral normal projective arithmetic surface $X$.


For every non-empty open subscheme $U$ of $X$, with complement the reduced closed subscheme  $F:= \vert X \setminus U \vert$, the following conditions are equivalent:
\begin{enumerate}[(i)]
\item the scheme $U$ is affine;
\item the scheme $F$ is purely one-dimensional; every vertical integral curve of $X$ meets $F$; moreover, if $F = H \cup V$ where $H$ (resp. $V$) is a finite union of horizontal (resp. vertical) integral curves, then every connected component of $V$ meets $H$;
\item there exists an ample effective Cartier divisor $D$ in $X$ whose support $\vert D \vert$ coincides with $F$. 
\end{enumerate}
When these conditions are satisfied, if $D_0$ is an effective divisor on $X_\Q$ whose support is $F_\Q$, the ample effective divisor $D$ in (iii) may be chosen such that $D_\Q$ is a positive multiple of $D_0$.

For every closed subscheme $Y$ of $X$, distinct of $X$, the following conditions are equivalent:
\begin{enumerate}[(i)]
\item the scheme $Y$ is affine;
\item every irreducible component of $\vert Y \vert$ is either a closed point or a horizontal integral curve in~$X$;
\item there exists an  affine open subscheme $U$ of $X$ containing $Y$.
\end{enumerate}
When these conditions hold, the affine open subscheme $U$ in (iii) may be chosen such that the complement $\vert X\setminus U \vert $ is a finite union of horizontal integral curves in $X$. 

\subsection{Green functions, Arakelov divisors, and $\CbD$ regularity} 
Starting from Chapter \ref{OverflowArchimede}, \emph{Green function} will mean Green function with $\CbD$ regularity, and \emph{Arakelov divisor} will mean Arakelov divisor with $\CbD$ Green function, etc.; see page \pageref{CbDeltanow}. 

\mainmatter

\part{Projective curves in analytic surfaces, Nori's finiteness theorems, and  pseudoconcavity}

\chapter[Projective surfaces and pseudoconcave analytic surfaces]{Projective surfaces, pseudoconcave analytic surfaces, and \'etale fundamental groups}

Part 1 of this memoir is devoted to a series of results concerning complex algebraic and analytic varieties that will constitute geometric models for the arithmetic results, concerning projective and formal analytic projective surfaces, that will be established in Parts 2 and 3.

Most of the geometric ideas underlying our work are presented in this first chapter. Notably we discuss diverse results concerning the geometry and topology of complex surfaces that appear (some of them implicitly) in Nori's famous paper \cite{Nori83} on the fundamental groups of quasi-projective complex surfaces. In this paper, Nori investigates the topology of these surfaces and of their embedded curves by considering some auxiliary complex analytic surfaces satisfying a suitable positivity condition, which turns out to be related to the pseudoconcavity of these surfaces, in the sense of Andreotti and Grauert (\cite{AndreottiGrauert62}, \cite{Andreotti63}). 

More specifically, in this chapter we elaborate on the ``easy" part of \cite{Nori83} --- which avoids the sophisticated arguments of deformation theory in complex analytic geometry that constitute the bulk of \cite{Nori83}\footnote{The ``transcendental" results in  \cite{Nori83}, concerning usual (and not only \'etale) fundamental groups, have been given an alternative proof by Napier and Ramachandran \cite{NapierRamachandran98}, relying on $L^2$ estimates for the $\overline{\partial}$ operator instead of deformation theory. Napier and Ramachandran actually establish a stronger version of Nori's results, notably a ``weak Lefschetz theorem" concerning suitable immersion of codimension greater than one. Degree estimates \emph{\`a la} Nori do not seem to have been established in the framework studied in \cite{NapierRamachandran98}, which however suggests interesting developments of the results of this monograph concerning \'etale fundamental groups of arithmetic schemes.} --- and we establish various extensions of Nori's results whose proofs require only basic techniques of algebraic and analytic geometry. These techniques also allow us to   establish some avatars of the finiteness and algebraicity results in \cite{Andreotti63} concerning pseudoconcave complex analytic surfaces. 

\medskip

The content of Part 1 is not explicitly used in Parts 2 and 3. From a strictly logical perspective, Part 1 could be skipped by a reader interested in the arithmetic results of this memoir only. However we believe that, independently of their own  geometrical interest, the results presented in Part 1 shed some light on the constructions and arguments developed in the arithmetic framework of Parts 2 and 3.  Let us briefly discuss this in more detail. 

In Part 1, a key role is played by complex analytic surfaces $\cV$ containing a (possibly singular) connected projective curve $C$. In most of our results, it is not the analytic surface $\cV$ itself but its germ $\cV^\an_C$ along $C$ that indeed matters. Actually our results admit variants involving the formal germ $\widehat{\cV}_C$ of $\cV$ along $C$, and may be generalized to the situation where $\widehat{\cV}_C$ is an arbitrary smooth formal complex surface containing $C$ as scheme of definition.\footnote{To stay closer to the  geometric intuition and to  the classical results in \cite{Andreotti63} and  \cite{Nori83}, we stick to the ``analytic" point of view and mention the ``formal" results only incidentally. The comparison of the constructions in  the formal and analytic contexts actually leads to delicate questions; see Subsection  \ref{functions}.}

The formal-analytic arithmetic surfaces investigated in this memoir are precisely the arithmetic counterparts of these formal complex surfaces $\widehat{\cV}_C$ admitting a projective algebraic curve $C$ as scheme of definition. The pseudoconcave  formal-analytic arithmetic surfaces, to which our main results in Part 3 are devoted,   are the counterparts of those formal complex surfaces $\widehat{\cV}_C$ such that the normal bundle of $C$ in $\widehat{\cV}_C$ satisfies a suitable positivity assumption,  whose complex analytic avatars are investigated in Part 1. 

Moreover the elementary character of our algebro-geometric arguments in Part 1 --- by opposition to the transcendental arguments in \cite{Andreotti63} and \cite{Nori83} --- explains why we are able to transpose 
 them in the arithmetic setting of Part 3.
 
 \medskip
 
 A superficial reading of Chapter 1 should be enough to grasp the geometric ideas underlying this memoir. Its simple framework --- with the above notation, the projective curve $C$ in the complex analytic surface $\cV$ is supposed to be smooth --- is however too restrictive when one wishes to pursue in detail the comparison with the arithmetic situation investigated in Parts~2 and~3. More general versions of the results of Chapter~1, which constitute a more complete analogue of our later arithmetic results, will be established in Chapter~2.

\section{Degree of morphisms between projective surfaces and  the Hodge index inequality}

The following proposition is a straightforward generalization of an estimate established by Nori in \cite[Proof of Lemma 5.1]{Nori83}. Its arithmetic conterpart in Arakelov geometry, stated in Theorem \ref{theorem:main-Arakelov} below,  will play a central role in this monograph.

\begin{proposition}\label{proposition:main-geom}
Let $X$ and $Y$ be two integral normal projective surfaces over an algebraically closed field $k$,  and let:
$$f : Y\lra X$$
be a dominant morphism of $k$-schemes.

Let $B$ be a divisor on $Y$, and let $A:=f_* B$ be its direct image in $X$.  If the self-intersection $B\cdot B$ is positive, then the degree $\deg f:=[k(Y):k(X)]$ of $f$
satisfies the following inequality: 
\begin{equation}\label{equation:inequality-degree}
\deg f \leq \frac{A \cdot A}{B \cdot B}.
\end{equation}
\end{proposition}

In \eqref{equation:inequality-degree}, $A.A$ and $B.B$ denote the respective self-intersections of the Weil divisors $A$ and $B$ on $X$ and $Y$ respectively, defined by means of intersection theory on projective normal surfaces as in \cite[II.(b)]{Mumford61}; see also \cite[Example 7.1.16]{FultonIT} and  \cite[2.4]{Bost99}.

\begin{proof}
Let us denote:
$$\delta:= \deg f.$$  For every divisor $D$ on $X$, we have:
\begin{equation}\label{equation:equal-zero}
(f^{*}A-\delta B)\cdot f^{*}D=f^{*}A\cdot f^{*}D-\delta \, B\cdot f^{*}D=\delta\, A\cdot D -\delta\,  f_{*}B\cdot D=0.
\end{equation}
Applied to $D=A$, \eqref{equation:equal-zero} implies:
\begin{equation}\label{equation:compare-delta}
(f^{*}A-\delta B)\cdot(f^{*}A-\delta B)=-\delta \, f^{*}A\cdot B + \delta^2 \,  B\cdot B= -\delta \, A\cdot f_\ast B + \delta^2 \,  B\cdot B= -\delta\, A\cdot A+\delta^{2}\, B\cdot B.
\end{equation}

If $H$ is an ample Cartier divisor on $X$, we have:
\begin{equation}\label{equation:ample-positive}
f^{*}H\cdot f^{*}H=\delta \, H\cdot H>0.
\end{equation}
Moreover, according to  \eqref{equation:equal-zero} applied to $D=H$, we have:
\begin{equation}\label{equation:orth}
(f^{*}A-\delta B)\cdot f^{*}H=0.
\end{equation}

From \eqref{equation:ample-positive}, \eqref{equation:orth} and the Hodge index inequality on the projective surface $Y$, we obtain:
$$(f^{*}A-\delta B)\cdot(f^{*}A-\delta B)\leq 0.$$
Together with \eqref{equation:compare-delta}, this proves \eqref{equation:inequality-degree} when $B\cdot B$ is positive.
\end{proof}

The following amplification of Proposition \ref{proposition:main-geom} will be useful in applications:
\begin{proposition}\label{proposition:main-geomAmplif}
Let $U$ and $V$ be two connected normal quasi-projective surfaces over an algebraically closed field $k$,  and let: 
$$f : V\lra U$$
be a dominant morphism of $k$-schemes. 
 Let $B$ be a divisor on $V$, whose support 
is proper over $k$, and let $A:=f_*(B)$ be its direct image.  

If the self-intersection $B\cdot B$ is positive, then the degree $\deg f:=[k(V):k(U)]$ of $f$
satisfies the following inequality: 
\begin{equation}\label{equation:inequality-degreebis}
\deg f \leq \frac{A \cdot A}{B \cdot B}.
\end{equation}
\end{proposition}

Observe that the self-intersections ${A \cdot A}$ and ${B \cdot B}$ are well defined since the support of $B$, and therefore the support of $A$ also, is proper over $k$.

\begin{proof} We may assume that $U$ and $V$ are open subschemes of some integral projective surfaces $X$ and $Y$. After possibly replacing $X$ by the closure in $X \times_k Y$ of the closure of the graph of $f$, we may also assume that the morphism $f$ extends to a $k$-morphism from $X$ to $Y$. Finally, after possibly replacing $X$ and $Y$ by their normalization, we may also assume that they are normal. Then \eqref{equation:inequality-degree} follows from Proposition \ref{proposition:main-geom} applied to the morphism $f:X \ra Y$.
\end{proof}

\section[Connectedness theorems and Nori's theorem on nodal curves]{Applications: connectedness theorems and Nori's theorem on \'etale fundamental groups of nodal curves in smooth surfaces}

\subsection{}\label{connectedness}
With the notation of Proposition \ref{proposition:main-geom}, let $C$ be a closed integral $1$-dimensional subscheme of $Y$ such that the restriction of $f$ to $C$ defines an isomorphism:
\begin{equation}\label{fCiso}
f_{\mid C}: C \lrasim f(C).
\end{equation}

Assume additionally that \emph{$f$ is \'etale at every point of $C$} or, equivalently, that $f$ defines an isomorphism:
$$\widehat Y_{C}\lrasim \widehat X_{f(C)}$$
between the formal completions of $Y$ and $X$ along $C$ and $f(C)$ respectively.

In this situation, the self-intersection numbers $C\cdot C$ and $f(C)\cdot f(C)$ coincide, and the inequality \eqref{equation:inequality-degree} applied to $B=C$ asserts that, \emph{when the self-intersection $C\cdot C$ is positive, then 
$\deg f =1$}, or equivalently \emph{the  morphism $f$ is birational}. When this holds, $f$ induces an isomorphism between some Zariski open neighborhoods of $C$ and $f(C)$ in $Y$ and $X$ respectively, and the following equality holds:
$$f^{-1}(f(C))=C.$$
In particular
 $f^{-1}(f(C))$ is connected. Conversely, since $f$ is \'etale at every point of $C$, the connectedness of $f^{-1}(f(C))$ implies that $f$ is birational.

This discussion shows that, in the situation  we consider, the inequality \eqref{equation:inequality-degree} applied to $B=C$ follows from the following connectedness statement Ã  la Fulton-Hansen\footnote{see for instance \cite{FultonLazarsfeld81} and \cite[3.3-4]{Lazarsfeld04} for presentations of these connectedness theorems and their applications and for references.}: \emph{for every dominant morphism $f: Y \ra X$ of connected normal projective surfaces, and every  closed integral subscheme $D$ of $X$ of dimension $1$ such that $D.D>0$, the inverse image $f^{-1}(D)$  is connected.}

We refer to \cite[Section 2]{Bost99} for a more complete discussion of related connectedness theorems, due notably to Zariski and C.P. Ramanujam. Let us also indicate that arithmetic counterparts of these theorems, established  in the context of Arakelov geometry,  play a key role in \cite{Bost99} and \cite{BostChambert-Loir07}.

\subsection{}\label{Bintegral}

Let us  keep the notation of Proposition \ref{proposition:main-geom} and, for the sake of simplicity, let us  assume that the surfaces $X$ and $Y$ are both smooth. Let $B$ be a closed integral subscheme of dimension $1$ of $Y$ such that \emph{the restriction:
$$f_{\mid B}:  B \lra f(B)$$ is a birational morphism}.

The image $A: =f(B)$ of $B$ defines a Cartier divisor on $X$. It coincides with the cycle theoretic image $f_\ast B$, and the inverse image $f^{*}A$ of $A$ is an effective Cartier divisor on $Y$ containing $B$ as a component. It may be written:
$$f^{*}A=eB+R,$$
where $e$ is a positive integer --- this is by definition the ramification index of $f$ along $B$ --- and $R$ is an effective Cartier divisor on $Y$ that intersects $B$ properly.

We have the following equalities of intersection numbers:
$$A \cdot A = A\cdot  f_\ast B= f^\ast A \cdot B = e\, B\cdot B + R\cdot B.$$
Together with Proposition \ref{proposition:main-geom}, this shows that, \emph{if the self-intersection}
$$B\cdot B := \deg_B N_B Y$$
\emph{is positive, then:}
\begin{equation}\label{equation:geom-case}
\deg f \leq  \frac{A\cdot A}{B\cdot B} = e \frac{A\cdot A}{A\cdot A - R\cdot B} = e+\frac{R\cdot B}{\deg_{B} N_{B}Y} .
\end{equation}

\subsection{}\label{Norinodalcurves} To highlight the geometric significance of the degree estimate \eqref{equation:geom-case}, let us present its application by Nori to the fundamental group of nodal curves embedded in smooth projective surfaces. 

In  \cite[section 5]{Nori83}, Nori considers the situation of \ref{Bintegral}, where moreover $A:= f(B)$ \emph{is a nodal curve}, with $r(A)$ singular points, and \emph{$f$ is \'etale at every point of $B$}. 

Then $e =1$, and $B$ also is a nodal curve, with $r(B) \leq r(A)$ singular points. Moreover the effective divisor $R := f^\ast A -B$ intersects $B$ transversally precisely at those smooth points $P$ of $B$ such that $A$ is singular at $f(P)$. 
Consequently:
$$R\cdot B = 2 (r(A) -r(B)),$$
and therefore:
$$B\cdot B = A \cdot A - R\cdot B =A\cdot A - 2 r(A) + 2 r(B) \geq A \cdot A - 2 r(A).$$

The conclusion of \ref{Bintegral} shows that, \emph{when moreover $A\cdot A > 2 r(A),$ then}:
\begin{equation}\label{ineqMori}
\deg f \leq \frac{A\cdot A}{A \cdot A - 2 r(A)}.
\end{equation}

This upper bound on $\deg f$ is the content of \cite[Lemma 5.1]{Nori83}. Remarkably the right-hand side of \eqref{ineqMori} depends only of the geometry of the curve $A$ embedded in the surface $X$, and not on the morphism $f$. As pointed out by Nori, when applied to \'etale morphisms $f$, it may be rephrased as the following result concerning \'etale fundamental groups:\footnote{The derivation of Proposition \ref{propNori} from the upper-bound \eqref{ineqMori} is similar to the derivation of Proposition \ref{easyNori} from the degree bound \eqref{equation:degree-bound-geomBis} below.}

\begin{proposition}[\cite{Nori83}, Proposition 5.4]\label{propNori} Let $A$ be a connected nodal projective curve embedded in a connected projective smooth surface $X$ over some algebraically closed field $k$. Let us denote by $r(A)$ its number of singular points, and by:
$$i \circ\nu: \widetilde{A} \stackrel{\nu}{\lra} A \stackrel{i}{\lra} X $$
the normalization of $A$ composed with the inclusion in $X$.

If the inequality:
$$A \cdot A > 2 r(A)$$
holds, then the image of the morphism of \'etale fundamental groups: 
\begin{equation}\label{nustar}
(i \circ \nu)_\ast: \pi_1^\et (\widetilde{A}, \widetilde{a}) \lra \pi_1^\et (X, x)
\end{equation}
induced by $i \circ\nu$ is a subgroup of finite index; moreover:
\begin{equation}\label{ineqMoripi1}
[\pi_1^\et (X, x): (i \circ \nu)_\ast(\pi_1^\et (\widetilde{A}, \widetilde{a}))] \leq \frac{A\cdot A}{A \cdot A - 2 r(A)}.
\end{equation}
\end{proposition}

In \eqref{ineqMoripi1}, $\widetilde{a}$ denotes a geometric point of $\widetilde{A}$ and $x$ its image by $i \circ \nu$.

To put Proposition \ref{propNori} in perspective, observe that, when $A$ is smooth, $r(A)$ vanishes, and Proposition \ref{propNori} asserts that, if the self-intersection:
 $$A\cdot A = \deg_A N_A X$$
is positive, then the inclusion morphism $i:A \hlra X$ induces a surjective morphism of \'etale fundamental groups:
$$i_\ast :  \pi_1^\et ({A}, {a}) \twoheadrightarrow \pi_1^\et (X, x).$$
This surjectivity is a consequence of the Lefschetz theorems in SGA2, and actually holds for any integral curve $A$ in $X$ with positive self-intersection. It also follows from the connectedness results in \ref{connectedness} above;  see \cite{GrothendieckSGA2} and \cite{Bost99}.

Observe also that, in striking contrast with
Nori's finiteness result \eqref{ineqMoripi1} and the surjectivity of $i_\ast$, the image of the morphism:
$$\nu_\ast: \pi_1^\et (\widetilde{A}, \widetilde{a}) \lra \pi_1^\et ({A},{a})$$
is a closed subgroup of \emph{infinite} index when $r(A)$ is positive.

\section[Maps from analytic thickening of projective curves to algebraic surfaces]{Analytic maps from analytic thickening of projective curves to algebraic surfaces}\label{subsubsection:analogue-main-geom} When $k= \C$, Nori establishes in \cite{Nori83} various generalized versions of the finiteness results on \'etale fundamental groups in Proposition \ref{propNori}. These generalizations concern topological fundamental groups, and rely on subtle arguments of deformation theory in complex analytic geometry. 

It turns out that diverse variants of Nori's results, concerning \'etale fundamental groups instead of topological fundamental groups, may be established by some variations on the simple algebraic arguments establishing the degree bound \eqref{ineqMori} and  Proposition \ref{propNori}.
These variants, like the more difficult transcendental results in \cite{Nori83}, involve complex analytic maps from some auxiliary (germ of) complex analytic surface containing a projective curve, with range the complex algebraic varieties under study (which will play the role of $X$ and $Y$ in \ref{Bintegral}, or of $X$ in Proposition \ref{propNori}).

In this subsection, we establish some simple but significant instances of these variants. They will constitute ``geometric models" for the upper bound on the degree of a morphism between arithmetic surfaces  and for the finiteness results concerning their \'etale fundamental group established in Theorems   \ref{theorem:main} and \ref{theorem:fundamental} below. 

The framework of this section and of the next one, where we consider an auxiliary analytic surface containing a \emph{smooth} projective curve, makes the derivation of our results especially straightforward. As discussed in \ref{PseudoAndreotti} below, this simple framework also highlights the relation between the geometric results in this first part --- and consequently of their arithmetic counterpart established in this article ---  with the classical  properties of pseudoconcave complex analytic spaces, introduced in the seminal works of Andreotti and Grauert (\cite{AndreottiGrauert62}, \cite{Andreotti63}).

\subsection{}\label{basicDiagr} We place ourself in a special instance of Proposition \ref{proposition:main-geomAmplif} with $k =\C$, and we consider a dominant morphism:
$$f: V \lra U$$
of connected smooth quasi-projective complex surfaces. 

We also consider a connected smooth complex analytic surface $\cV$ containing a smooth connected projective curve $C$, and we assume that we are given complex analytic  maps $\alpha$ and $\beta$ from $\cV$ to $U$ and $V$ respectively, that fit into a commutative diagram:
\begin{equation}\label{diagram:setup}
\begin{gathered}
\xymatrix{
& V\ar[d]^{f}\\
\cV\ar[r]^{\alpha}\ar[ur]^{\beta} & U.
}
\end{gathered}
\end{equation}

We finally assume that \emph{the image $\alpha(\cV)$ of $\alpha$ is Zariski dense in $U$} --- this holds if and only if $\alpha$ is \'etale at some point of $\cV$, or equivalently  outside a proper closed analytic subset of $\cV$ --- and  that \emph{the restriction of $\alpha$ to $C$:
$$\alpha_{\mid C} : C \lra \alpha(C)$$ is birational.} This implies that $f$ is dominant,  that the image $\beta(\cV)$ of $\beta$ is  Zariski dense in $V$,  and that the restriction of $\beta$ to~$C$:
$$\beta_{\mid C} : C \lra \beta(C)$$
 is birational. 

Then we 
 may define the ramification index $e(\alpha)$ of $\alpha$ along $C$ by the equality of analytic divisors on $\cV$: 
\begin{equation}\label{defeR}
\alpha^{*}(\alpha(C))=e(\alpha)\, C+R,
\end{equation}
where $R$ is an effective analytic divisor on $\cV$ that intersects $C$ properly. The integer $e(\alpha)$ is positive, the intersection number $R \cdot C$ is non-negative, and the following equalities of intersection numbers hold:
\begin{equation}\label{equation:self-int-alpha}
\alpha(C)\cdot \alpha(C)=C \cdot\alpha^{*}(\alpha(C))=e(\alpha)\, C\cdot C + R \cdot C.
\end{equation}

As a consequence, when the self-intersection: 
$$C \cdot C=\deg_{C} N_{C}\cV$$
is positive, we have:
\begin{equation}\label{equation:self-int-alphabis}
\alpha(C) \cdot \alpha(C)\geq e(\alpha)\, C\cdot  C>0,
\end{equation}
and, similarly:
\begin{equation}\label{equation:self-int-beta}
\beta(C) \cdot \beta(C)\geq e(\beta)\, C \cdot C>0.
\end{equation}

Finally, Proposition \ref{proposition:main-geom} applied to $B=\beta(C)$ and $A=\alpha(C)$, together with the relations \eqref{equation:self-int-alpha} and \eqref{equation:self-int-beta}, establishes the following:

\begin{proposition}\label{prop:degree-bound-geom} With the above notation, if the self-intersection $C \cdot C =\deg_{C}N_{C}\cV$ is positive, then:
\begin{equation}\label{equation:degree-bound-geom}
\deg f
\leq \frac{ \alpha(C)\cdot \alpha(C)}{e(\beta) \, C \cdot C} = \frac{e(\alpha)}{e(\beta)}+\frac{R \cdot C}{e(\beta)\deg_{C}N_{C}\cV},
\end{equation}
and consequently:
\begin{equation}\label{equation:degree-bound-geomBis}
\deg f\leq \frac{ \alpha(C)\cdot \alpha(C)}{C \cdot C} = e(\alpha)+\frac{R\cdot C}{\deg_{C}N_{C}\cV}.
\end{equation}
\end{proposition}

\subsection{}\label{pi1proof} In turn, the degree bound \eqref{equation:degree-bound-geomBis} implies the following result concerning \'etale fundamental groups:

\begin{proposition}[compare \cite{Nori83}, Weak Lefschetz Theorem, B and C]\label{easyNori}
Let $\cV$ be a smooth connected complex analytic surface containing a smooth connected projective curve $C$, and let: 
$$\alpha: \cV \lra {U}$$
be a  complex analytic map, with Zariski dense image, from $\cV$ to some smooth connected complex algebraic surface ${U}$.

If the morphism 
$\alpha_{\mid C}: C \ra \alpha(C)$
is birational and the self-intersection 
$C \cdot C = \deg_C N_C \cV$
is positive, then the image of the morphism between \'etale fundamental groups: 
\begin{equation}\label{alphaCast}
\alpha_{\mid C \ast}: \pi_1^\et (C,c) \lra \pi_1^\et({U}, x)
\end{equation}
induced by $\alpha_{\mid C} : C \ra {U}$
is a subgroup of finite index; moreover:
\begin{equation}\label{ineqMoripi1bis}
[\pi_1^\et ({U}, x): \alpha_{\mid C \ast}(\pi_1^\et (C,c))] \leq \frac{\alpha(C) \cdot \alpha(C)}{C\cdot C} = e(\alpha) + \frac{R \cdot C}{\deg_C N_C \cV}.
\end{equation}\end{proposition}

In \eqref{alphaCast}, $c$ denotes a complex point of $C$, and $x:= \alpha(C)$ its image in ${U}$. In the right-hand side of \eqref{ineqMoripi1bis}, the ramification index $e(\alpha)$ and the effective analytic divisor $R$ --- which meets $C$ properly --- are defined as above by  \eqref{defeR}.

The degree bound \eqref{equation:degree-bound-geomBis} and Proposition \ref{easyNori} contain the results in \ref{connectedness} and  \ref{Norinodalcurves} when $k = \C$ as special cases.\footnote{ Concerning \ref{connectedness}, at least when $X, Y$ and $C$ are assumed to be smooth.} 
Indeed, to recover the equality $\deg f =1$ in \ref{connectedness}, simply apply the bound  \eqref{equation:degree-bound-geomBis} to $\cV = Y,$ $\alpha = f$, and $\beta = \Id_Y$. 
To recover the results on nodal curves in \ref{Norinodalcurves}, apply \eqref{equation:degree-bound-geomBis} and Proposition \ref{easyNori} with $\cV$ a ``tubular neighborhood" of the morphism $i\circ \nu: \widetilde{A} \ra X$, and with $\alpha$ the canonical morphism from $\cV$ to $U:= X$, which is \'etale and coincides with $\nu$ on $C := \widetilde{A}$; see Figure~\ref{Tub}. Then $e(\alpha) = 1$ and the self-intersection $C\cdot C = \widetilde{A} \cdot \widetilde{A}$ is easily seen to be $A\cdot A - 2 r(A)$.

\begin{figure}

\begin{tikzpicture}
[samples=100]
\draw [semithick] [domain=-1.37:1.84] plot ({1.4+\x*\x -2}, {0.4+\x*(\x*\x-2)});
\draw [semithick] [domain=-1.86:-1.45] plot ({1.4+\x*\x -2}, {0.4+\x*(\x*\x-2)});
\draw [domain=-1.44:-1.385] [dotted] plot ({1.4+\x*\x -2}, {0.4+\x*(\x*\x-2)});
\draw [red] [domain=-1.35:1.84] [dashed] plot ({1.4+\x*\x -2 + 0.2*(3*\x*\x-2)/sqrt(9*\x*\x*\x*\x -8*\x*\x +4)}, {0.4+\x*(\x*\x-2) - 0.2*2*\x/sqrt(9*\x*\x*\x*\x -8*\x*\x +4)}); 
\draw [red]  [domain=-1.86:-1.44] [dashed] plot ({1.4+\x*\x -2 + 0.2*(3*\x*\x-2)/sqrt(9*\x*\x*\x*\x -8*\x*\x +4)}, {0.4+\x*(\x*\x-2) - 0.2*2*\x/sqrt(9*\x*\x*\x*\x -8*\x*\x +4)}); 
\draw [red]  [domain=-1.44:-1.35] [dotted] plot ({1.4+\x*\x -2 + 0.2*(3*\x*\x-2)/sqrt(9*\x*\x*\x*\x -8*\x*\x +4)}, {0.4+\x*(\x*\x-2) - 0.2*2*\x/sqrt(9*\x*\x*\x*\x -8*\x*\x +4)}); 
\draw [red]  [domain=-1.38:1.84] [dashed] plot ({1.4+\x*\x -2 - 0.2*(3*\x*\x-2)/sqrt(9*\x*\x*\x*\x -8*\x*\x +4)}, {0.4+\x*(\x*\x-2) +0.2*2*\x/sqrt(9*\x*\x*\x*\x -8*\x*\x +4)}); 
\draw [red] [domain=-1.86:-1.48] [dashed] plot ({1.4+\x*\x -2 - 0.2*(3*\x*\x-2)/sqrt(9*\x*\x*\x*\x -8*\x*\x +4)}, {0.4+\x*(\x*\x-2) + 0.2*2*\x/sqrt(9*\x*\x*\x*\x -8*\x*\x +4)}); 
\draw [red] [domain=-1.48:-1.38] [dotted] plot ({1.4+\x*\x -2 - 0.2*(3*\x*\x-2)/sqrt(9*\x*\x*\x*\x -8*\x*\x +4)}, {0.4+\x*(\x*\x-2) + 0.2*2*\x/sqrt(9*\x*\x*\x*\x -8*\x*\x +4)}); 
\draw [->] (3.8, 0.4) -- (5.5, 0.4)  node[midway, above] {$\alpha$};
\draw [semithick] [dotted] (13,4) to [bend right] (13,-1.7);
\draw [semithick] [dotted] (7,4) to [bend right] (7,-1.7);
\draw [semithick] [dotted] (7,4) to [bend right] (13,4);
\draw [semithick] [dotted] (7, -1.7) to [bend right] (13,-1.7);
\draw [semithick] [domain=-1.86:1.84] plot ({9.4+\x*\x -2}, {0.4+\x*(\x*\x-2)});
\draw (12.5, 3.5) node[below left]{$X$};
\draw (10, -1.3) node{$A$};
\draw (2, -1.3) node{$\widetilde{A}$};
\draw  [red] (3, 2.45) node{$\cV$};
\end{tikzpicture}

\caption{An analytic tubular neighborhood $\cV$ of the immersion $i\circ \nu: \widetilde{A} \ra X$.}\label{Tub}
\end{figure}

\begin{proof}[Proof of Proposition \ref{easyNori}] To establish the finiteness of the index: $$[\pi_1^\et ({U}, x): \alpha_{\mid C \ast}(\pi_1^\et (C,c))]$$
and the upper bound \eqref{ineqMoripi1bis}, we have to show that any open subgroup $H$ of the profinite group $G:= \pi_1^\et ({U}, x)$ that contains the closed subgroup $\alpha_{\mid C \ast}(\pi_1^\et (C,c))$ satisfies:
\begin{equation}\label{GHineq}
[G:H] \leq \frac{\alpha(C) \cdot \alpha(C)}{C\cdot C}.
\end{equation}

Any such subgroup $H$ gives rise to a commutative diagram:
\begin{equation*}
\begin{gathered}
\xymatrix{
& {V}\ar[d]^{{f}}\\
C \ar[r]^{\alpha_{\mid C}}\ar[ur]^{\beta_{\mid C}} & {U},
}
\end{gathered}
\end{equation*}
where ${V}$ is a smooth connected complex algebraic surface, ${f}$ is  a finite \'etale morphism of degree
\begin{equation}\label{fGH}
\deg {f} = [G:H],
\end{equation}
and where $\beta_{\mid C}$ is a morphism of complex algebraic varieties\footnote{Since the curve $C$ is projective,  by GAGA any complex analytic map from $C$ to a complex algebraic variety defines a morphism of complex algebraic varieties.}.

After possibly shrinking $\cV$, we may assume that the inclusion $C \hra \cV$ induces an isomorphism of fundamental groups. Since ${f}$ defines a complex analytic unramified covering, the map $\beta_{\mid C}$ uniquely extends to some complex analytic map:
$$\beta: \cV \lra  {V}$$
such that the following diagram is commutative:
\begin{equation*}
\begin{gathered}
\xymatrix{
& {V}\ar[d]^{{f}}\\
\cV \ar[r]^{\alpha}\ar[ur]^{\beta} & {U}.
}
\end{gathered}
\end{equation*}

Since the image of $\alpha$ is Zariski dense in ${U}$, when moreover $\alpha_{\mid C}$ is birational and the self-intersection $C\cdot C$ is positive, the upper bound 
\eqref{equation:degree-bound-geomBis} applies to $f$ and therefore:
$$\deg f \leq \frac{\alpha(C) \cdot \alpha(C)}{C\cdot C}.$$
Combined with \eqref{fGH}, this completes the proof of \eqref{GHineq}.
\end{proof}

\section{Germs of pseudoconcave analytic surfaces}\label{subsubsection:alternative-geom} 

\subsection{Sections of line bundles on germs of pseudoconcave analytic surfaces: dimension estimates and a new proof of the degree bound \eqref{equation:degree-bound-geomBis}}\label{pseudoconcaveEasy} Proposition \ref{prop:degree-bound-geom} admits an alternative proof which relies on the special properties of the analytic surface  $\cV$ --- more precisely of its   germ along the projective curve $C$ --- when the condition:
\begin{equation}\label{CVpseudoconcav}
\deg_C N_C \cV > 0.
\end{equation}
is satisfied. The positivity condition \eqref{CVpseudoconcav} turns out to be closely related to the pseudoconcavity properties of the germ of the analytic surface $\cV$ along the projective curve $C$; we shall discuss this in more details in \ref{PseudoAndreotti} below.

The central ingredient of this alternative proof will be the following finiteness result concerning sections of analytic line bundles on the analytic surface $\cV$.

\begin{proposition}\label{finitepseudoconcave} Let $\cV$ be a smooth connected complex analytic surface containing a smooth connected projective curve $C$ such that the positivity condition \eqref{CVpseudoconcav} is satisfied. For every complex analytic line bundle $M$ over $\cV$, the complex vector space $\Gamma(\cV, M)$ of holomorphic sections of $M$ over $\cV$ is finite dimensional. Moreover: 
\begin{equation}\label{equation:bound-dim}
\dim_{\C}\Gamma(\cV, M)\leq C(\deg_{C}M_{|C}),
\end{equation}
where, for every $n \in \Z$, we let:
\begin{equation}\label{defCn}
C(n):=\sum_{i\geq 0} (n+1-i\deg N_{C}\cV)^+.
\end{equation}

In particular, the following implication holds:
\begin{equation}\label{equation:bound-negative}
\deg_C M_{\mid C} < 0 \Longrightarrow \Gamma(\cV, M)=0,
\end{equation}
and when the integer $D$ goes to $+\infty$, we have:
\begin{equation}\label{equation:bound-gen}
\dim_{\C}\Gamma(\cV, M^{\otimes D})\leq \frac{(\deg_{C}M_{|C})^{2}}{2\deg_{C} N_{C}\cV} D^{2}+O(D).
\end{equation}
\end{proposition}

Proposition \ref{finitepseudoconcave} will be a consequence of the following elementary fact concerning spaces of sections of analytic lines bundles over a smooth projective complex curve:
\emph{for every complex analytic line bundle $L$ over $C$, the complex vector space $\Gamma (C, L)$ of its complex analytic sections is finite dimensional; moreover:} 
\begin{equation}\label{BasicDeg}
\dim_\C \Gamma (C, L) \leq \left(1 + \deg_C L\right)^+.
\end{equation}

The reader may compare the following proof to \cite[Lemma 4.2]{CharlesPoonen16} and \cite[10.2.4]{Bost2020}. 

\begin{proof}
 The complex vector space $E:=\Gamma(\cV, M)$ of analytic sections of $M$ over $\cV$
  admits a decreasing filtration $(E^{i})_{i\geq 0}$ defined by the order of vanishing along $C$, namely by setting:
 $$E^{i}:=\Gamma(\cV, M(-iC)), \quad \mbox{for every $i \in \N.$} $$
 
 For any nonnegative integer $i$, the space $E^{i+1}$ is the kernel of the restriction map:
 $$\Gamma(\cV, M(-iC))\lra \Gamma(C, M(-iC)_{|C}) \simeq \Gamma(C, M_{\mid C} \otimes N_C \cV^{\otimes (-i)}).$$
 As recalled above, the vector space $\Gamma(C, M(-iC)_{|C})$ is finite-dimensional, of  dimension bounded above by:
 $$\left(1+\deg_{C}M(-iC)_{|C}\right)^+= \left(1+\deg_{C}M_{|C}-i\deg N_{C}\cV\right)^+.$$
 This implies the estimate:
 \begin{equation}\label{GrE}
 \dim_\C E^i/E^{i+1} \leq \left(1+\deg_{C}M_{|C}-i\deg N_{C}\cV\right)^+
 \end{equation}
 
The positivity of $\deg N_{C}\cV$ implies that, when $i$ is large enough, the right-hand side of \eqref{GrE} vanishes. Moreover  $E^0 = E$ and the intersection of the $E^{i}$'s is reduced to $\{0\}$. This proves that $\Gamma(\cV, M)$ is finite-dimensional, as well as the upper bound \eqref{equation:bound-dim}.

From the definition \eqref{defCn} of $C(n)$, it is readily seen that  $C(n)$ vanishes when $n$ is negative, which implies \eqref{equation:bound-negative}. Moreover we have:
$$C(0)=1$$
and, when $n$ goes to $+\infty$: 
$$C(n) = \int_0^{\frac{n+1}{\deg N_C \cV}} (n +1 - x \deg N_C \cV) \, dx + O(n) 
= \frac{n^{2}}{2\deg_{C} N_{C}\cV}+O(n).$$
This implies the asymptotic estimate \eqref{equation:bound-gen}.
\end{proof}

\begin{corollary}\label{corollary:ineq-deg}
If $\gamma : \cV\ra Z$ is a complex analytic map, with Zariski dense image, from $\cV$ to an integral projective complex surface $Z$, and if $L$ is a big and nef  line bundle $L$ over $Z$, the following inequality holds:
\begin{equation}\label{equation:self-int-an}
L\cdot L\leq \frac{(\deg_{C}\gamma^{*}L)^{2}}{\deg_{C}N_{C}\cV}.
\end{equation}
\end{corollary}

\begin{proof}
For any integer $D$, we may consider the ``evaluation map":
$$\eta_{D} : \Gamma(Z, L^{\otimes D})\lra \Gamma(\cV, \gamma^{*}L_{|C}^{\otimes D})$$
defined by sending a section $s$ of the line bundle $L^{\otimes D}$ over $Z$ to  its pull-back $\gamma^\ast s$, which is an analytic section of the complex analytic line bundle $\gamma^\ast L^{\otimes D}$ over $\cV$. Since $\gamma$ has Zariski dense image, $\eta_{D}$ is injective, so that:
\begin{equation}\label{equation:restriction-injective}
\dim_{\C}\Gamma(Z, L^{\otimes D})\leq \dim_{\C}\Gamma(\cV, \gamma^{*}L^{\otimes D}).
\end{equation}

Moreover, since $L$ is big and nef, when $D$ goes to $+\infty$, we have:
\begin{equation}\label{equation:HS-big-and-nef}
\dim_{\C}\Gamma(Z, L^{\otimes D})\sim \frac{L \cdot L}{2} D^{2}.
\end{equation}

The inequality \eqref{equation:self-int-an} follows from \eqref{equation:restriction-injective} and the asymptotic relations  \eqref{equation:HS-big-and-nef} and \eqref{equation:bound-gen} by letting $D$ go to infinity.
\end{proof}

\begin{proof}[Alternative proof of the degree bound \eqref{equation:degree-bound-geomBis}]

We keep the notation of \ref{basicDiagr}. As observed in the proof of Proposition \ref{proposition:main-geomAmplif}, we may introduce projective compactifications $X$ and $Y$ of $U$ and $V$ such that $f$ extends to a morphism from $X$ to $Y$.  The diagram:
\begin{equation*}
\begin{gathered}
\xymatrix{
& Y\ar[d]^{f}\\
\cV\ar[r]^{\alpha}\ar[ur]^{\beta} & X.
}
\end{gathered}
\end{equation*}
is clearly commutative. 

As shown by \eqref{equation:self-int-alphabis}, the self-intersection $\alpha(C) \cdot \alpha(C)$ is positive, and therefore the line bundle $\cO_X(\alpha(C))$ on $X$ is big and nef. Since $f$ is dominant, its pull-back:
$$L:=f^{*}\mathcal \cO_X(\alpha(C))$$
is a big and nef line bundle over $Y$. Its self-intersection satisfies:
$$L \cdot L = f^{*}\mathcal \cO_X(\alpha(C)) \cdot f^{*}\mathcal \cO_X(\alpha(C)) = (\deg f)   \, \cO_X(\alpha(C))\cdot \cO_X(\alpha(C)) = (\deg f)  \, \alpha(C)\cdot \alpha(C).$$
Moreover the degree along $C$ of its pull-back by $\beta$ is:
$$\deg_C \beta^\ast L = \deg_C \beta^\ast f^\ast \cO_X(\alpha(C)) =\deg_C \alpha^\ast \cO_X(\alpha(C)) =\alpha^\ast (\alpha(C))\cdot C = \alpha(C)\cdot \alpha(C).$$

Therefore, applied to the map $\beta: \cV \ra Y$ and to the line bundle $f^\ast L$ over $Y$,  Corollary \ref{corollary:ineq-deg} establishes the inequality:
$$(\deg f)  \, \alpha(C)\cdot \alpha(C)  \leq \frac{(\alpha(C)\cdot\alpha(C))^2}{\deg_C N_C \cV}.$$
Since $\alpha(C) \cdot \alpha(C)$ is positive, this is equivalent to \eqref{equation:degree-bound-geomBis}. 
\end{proof}

\subsection{Sections of line bundles on germs of pseudoconcave analytic surfaces: application to algebraicity}\label{SectionsAlgebraicity} Remarkably enough, a simple variation on the above proof of Corollary \ref{corollary:ineq-deg} allows one to derive the following algebraicity result from Proposition \ref{finitepseudoconcave}:

\begin{proposition}[compare \cite{Bost2020}, Theorem 10.2.8]\label{algebraicityPseudoconcaveAnalytic} Let $\cV$ be a smooth connected complex analytic surface containing a smooth connected projective curve $C$ such that $\deg_C N_C \cV$ is positive.

If $\gamma: \cV \ra Z$ is a complex analytic map from $\cV$ to some complex quasi-projective variety $Z$, then either $f(\cV)$ coincide with $f(C)$ --- \emph{hence is a point or an irreducible projective curve} --- or is contained, as a Zariski dense subset, in some irreducible closed algebraic surface in $Z$.  
\end{proposition}

\begin{proof} We shall actually only use the following weak form the dimension estimates of Proposition~\ref{finitepseudoconcave}:
\begin{equation}\label{weakfinitepseudoconcave}
  \dim_\C \Gamma(\cV , M^{\otimes D}) = O(D^2)  \quad \mbox{when $D$ goes to $+\infty$}. 
  \end{equation}

   
   To establish Proposition \ref{algebraicityPseudoconcaveAnalytic}, we may clearly assume that $Z$ is projective and reduced, and that $f(\cV)$ is Zariski dense in $Z$, which is therefore irreducible since $\cV$ is connected. We shall prove the estimate:
   $$d := \dim Z \leq 2.$$
   
   This will complete the proof when $d=2$, and we shall leave to the reader the easy proof that $f(\cV) = f(C)$ when $d\leq 1$.
   
   To achieve this, let us choose an ample line bundle $L$ over $Z$, and  consider the evaluation maps:
   $$\eta_{D} : \Gamma(Z, L^{\otimes D})\lra \Gamma(\cV, \gamma^{*}L_{|C}^{\otimes D}),$$
   as in the proof of Corollary \ref{corollary:ineq-deg}. They are injective, and the estimates \eqref{equation:restriction-injective} still holds. Moreover, since $L$ is ample, we have:
   \begin{equation}\label{LampleD}
   \dim_\C \Gamma(Z, L^{\otimes D}) \sim \frac{\deg c_1(L)^d \cap [Z]}{d!} \, D^d \quad \mbox{when $D$ goes to $+\infty$}.
\end{equation}
The inequality $d\leq 2$ follows from \eqref{equation:restriction-injective} and from the asymptotic relations \eqref{weakfinitepseudoconcave} and \eqref{LampleD} by letting $D$ go to infinity.
\end{proof}

The derivation of the dimension estimates in Proposition~\ref{finitepseudoconcave} and its application to algebraicity  in Proposition \ref{algebraicityPseudoconcaveAnalytic} only involve the germ $\cV^{\an}_C$ of the analytic surface $\cV$ along the projective curve $C$. Actually they may be formulated in terms of the \emph{formal germ} $\widehat{\cV}_C$ of $\cV$ along $C$ and, in this form, may be  generalized to formal germs of surfaces along a projective curves, defined over an arbitrary algebraically closed field $k$. Indeed the above proofs  of  Proposition~\ref{finitepseudoconcave} and \ref{algebraicityPseudoconcaveAnalytic}, which only involve  a few basic results of algebraic geometry,  immediately extends to this framework; see \cite[section~5]{Hartshorne68}, \cite[section 3.3]{Bost01}, and \cite[section 10.2.4]{Bost2020} for related results in this context.

\subsection{Positivity of $\deg_C N_C \cV$ and pseudoconcavity}\label{PseudoAndreotti}

In the complex analytic setting, several of the results established in the previous subsections are actually consequences of  classical theorems of complex geometry, established by analytic techniques, concerning pseudoconcave complex analytic manifolds. This alternative approach sheds some light on the geometric meaning of the key positivity assumption \eqref{CVpseudoconcav}, and we want to discuss it briefly.\footnote{We assume some familiarity with complex analysis in several complex variables at the level of the last chapters in \cite{Gunning90I}, or of \cite[Chapters I to V]{FritzscheGrauert02}.}

\medskip
 
\subsubsection{} Consider 
 a smooth connected complex analytic surface $\cV$ containing a smooth connected projective curve $C$, as in Proposition  \ref{finitepseudoconcave}, and let us choose a $\cC^\infty$ metric $\Vert.\Vert$ on the line bundle $\cO_{\cV}(C)$ over $\cV$. Let $U$ be an open neighborhood of $C$ that is  relatively compact in $\cV$. For every $r$ in $\Rpa$, we may consider the ``tube of radius $r$ around $C$ in $\cV$":
 $$\overline{\cV}_r := \{ x \in U \mid \Vert \mathbf{1}_C (x) \Vert  \leq r \}.$$
 There exists $r_0$ in $\Rpa$ such that, for every $r \in (0, r_0],$ $\cV_r$ is a compact $\cC^\infty$ submanifold with boundary of $\cV$, with interior:
  $${\cV_r} := \{ x \in U \mid \Vert \mathbf{1}_C (x) \Vert  < r \}$$
  containing $C$ as a deformation retract, and with boundary:
   $$\partial\cV_r := \{ x \in U \mid \Vert \mathbf{1}_C (x) \Vert  = r \}.$$ 
   
   The following proposition is established by a straightforward computation that we leave as an exercise for the reader:
   
   \begin{proposition}\label{pseudoconcavetube} If the metric $\Vert.\Vert_{\mid C}$ on the line bundle $N_C \cV \simeq \cO_{\cV}(C)_{\vert C}$ over $C$, deduced by restriction from $\Vert.\Vert$, satisfies the pointwise positivity condition:
   \begin{equation}\label{pseudoconcavepointwise}
c_1( N_C \cV, \Vert.\Vert_{\vert C}) >0 \quad \mbox{ over $C$,}
\end{equation}
 then there exists $r_1 \in (0, r_0]$ such that, for every $r \in (0, r_1)$, the Levi form of the boundary $\partial \cV_r$ of $\overline{\cV}_r$ is everywhere negative.
   \end{proposition}
   
   The existence of a $\cC^\infty$ metric $\Vert.\Vert$ on $\cO_{\cV}(C)$ satisfying the positivity condition  \eqref{pseudoconcavepointwise} is easily seen to be equivalent to the positivity \eqref{CVpseudoconcav} of $\deg_C N_C \cV$. 
   Moreover the negativity of the Levi form of $\partial \cV_r$ is a \emph{pseudoconcavity} condition: it asserts that locally, around every boundary points of $\overline{\cV}_r$, there exists a holomorphic chart that maps 
   $\overline{\cV}_r$ to the complement of a strictly convex domain with $\cC^\infty$ boundary. 
   
   For these reasons, we  refer 
   to the positivity condition \eqref{CVpseudoconcav} on $\deg_C N_C \cV$
    as a pseudoconcavity condition. 

 \medskip
   
\subsubsection{}\label{discussionHartogsLevi}    From now on, we assume that the positivity condition \eqref{CVpseudoconcav} holds, 
that the Hermitian metric $\Vert. \Vert$ satisfies    \eqref{pseudoconcavepointwise}, and that $r_1$ satisfies the conclusion of Proposition \ref{pseudoconcavetube}. 

An application of Hartogs' extension theorem at boundary points of $\cV_r$ shows that, for every analytic line bundle $M$ over $\cV$ and for every $r\in (0, r_1)$, an analytic section of $M$ over ${\cV_r}$ extends analytically to some open neighborhood of $\overline{\cV}_r$ in $\cV$. This  implies that, for every $r \in (0, r_1),$ the restriction map:
$$\Gamma({\cV}_{r_1}, M) \lra \Gamma({\cV}_r, M) $$
is an isomorphism. 
   Consequently, if we introduce the space of germs of analytic sections of $M$ along~$C$:
   $$\Gamma(\cV^{\an}_C, M) := \varinjlim_{r > 0} \Gamma ({\cV}_r, M),$$
   then, for every open connected submanifold $\cV'$ of $\cV$ such that:
   $$C \subset \cV' \subseteq {\cV}_{r_1},$$
   the restriction map defines an isomorphism:
   $$\Gamma(\cV', M) \lrasim \Gamma(\cV^{\an}_C, M).$$
   
   Similarly, E. E. Levi's extension theorem\footnote{concerning the extension of  meromorphic functions across pseudoconcave boundaries; see for instance \cite[Chapter 1]{Siu74bis}.} implies that the restriction map defines an isomorphism:
   $$\cM(\cV') \lrasim \cM(\cV^{\an}_C)$$
   between the algebra of  meromorphic functions on $\cV'$ and the algebra of germs of  meromorphic functions along~$C$.
   
    
    The negativity of the Levi form of $\partial \cV_r$ also implies that, for every $r \in (0, r_1)$, the complex manifold ${\cV}_r$ is pseudoconcave in the sense of  Andreotti \cite{Andreotti63}.  As shown in \emph{loc.  cit.}, this implies by purely analytic arguments that, for every complex analytic line bundle $M$ over $\cV$, the space of analytic sections $\Gamma({\cV}_r, M)$ --- which, as observed above, is independent of $r \in (0, r_1)$ --- is finite dimensional. It also implies  the dimension estimates:
      $$\dim_C \Gamma(\cV_r, M^{\otimes D}) = O(D^2) \quad \mbox{when $D$ goes to $+\infty$}.$$ 
      Andreotti  also proves that the pseudoconcavity of   $\cV_r$ implies that the field $\cM(\cV_r) ( \simeq \cM(\cV_C^{\an})$) of  meromorphic functions over $\cV_r$ is an extension of finite type of $\C$ and satisfies:
   $$\deg {\rm tr}_\C \cM(\cV_r) \leq \dim \cV_r = 2.$$
   He deduces from this fact the ``algebraicity"  of the image of a complex analytic embedding of $\cV_r$ in a complex projective space.

   \medskip

\subsubsection{}\label{GrauertContract} For later comparison with the arithmetic situation (see \ref{Vfpsiconvex} below), let us emphasize the contrast between the properties of pseudoconcave germs of analytic surfaces that we have just discussed, and the ones of pseudoconvex germs, which one encounters   when one replaces the positivity condition  \eqref{CVpseudoconcav} by the opposite condition:   
\begin{equation}\label{CVpseudoconvex}
\deg_C N_C \cV < 0.
\end{equation}

When \eqref{CVpseudoconvex} holds, we may choose the $\cC^\infty$ metric $\Vert.\Vert$ on $\cO_\cV(C)$ such that, instead of  \eqref{pseudoconcavepointwise}, 
it satisfies:
 \begin{equation}\label{pseudoconvexpointwise}
c_1( N_C \cV, \Vert.\Vert_{\vert C}) < 0 \quad \mbox{ over $C$,}
\end{equation}
and the computation leading to Proposition \ref{pseudoconcavetube} shows that, for some $r_1 \in (0, r_0]$, the Levi form of $\partial \cV_r$ is everywhere positive for $r \in (0, r_1)$. In other words, the $\cV_r$ have strictly pseudoconvex boundary. 

As shown by Grauert in \cite{Grauert62}, this implies that the $\cV_r$ are holomorphically convex, and actually that there exists a ``contraction" of the curve $C$, namely a proper holomorphic map:
$$c: \cV \lra \cS$$
with range a normal complex analytic surface $\cS$, that maps $C$ to a point $O$ of $\cS$ and induces an isomorphism of complex manifolds:
$$c_{\mid \cV \setminus C} : \cV \setminus C \lrasim \cS \setminus \{O\}.$$
Moreover, for every $r \in (0, r_1),$ the image: 
$$\cS_r := c(\cV_r)$$
of $\cV_r$ is a Stein space, containing $\cS_{r'}$ as a Runge subdomain for every $r' \in (0, r)$.
The pull-back of functions by the contraction $c$ establishes isomorphisms of algebras:
$$c^\ast : \Gamma(\cS_r, \cO_\cS) \lrasim \Gamma(\cV_r, \cO_\cV),$$
and therefore the algebras $\Gamma(\cV_r, \cO_\cV)$ for  $r \in (0, r_1)$ are ``very large" --- notably every pair of points  in $\cV_r$, not both in $C$, may be separated by a function in $\Gamma(\cV_r, \cO_\cV)$ --- and increase strictly when $r$ decreases to zero. 

This is in striking contrast with the finite dimensionality and the independence of $r$ of the spaces $\Gamma(\cV_r, M)$  and $\cM(\cV_r)$ in the pseudoconcave case  discussed in \ref{discussionHartogsLevi} above.

\chapter{$\mathbf{CNB}$-divisors and fibered analytic surfaces}\label{CNBfibered}

In the  analogy between number fields and function fields, the counterpart of algebraic varieties over a number field $K$ are varieties over the function field $k(S)$ of some curve $S$ over a base field $k$, or equivalently the generic fibers of $k$-varieties that are fibered over $S$. From this perspective, the geometric analogues of the arithmetic results established in this memoir specifically concern  analytic and algebraic varieties fibered over a projective curve.  However, in this ``fibered" context, the simple framework of Subsections \ref{basicDiagr}, \ref{pi1proof}, and \ref{pseudoconcaveEasy} excludes significant examples. Indeed the self-intersection of sections of surfaces fibered over a projective curve is in general negative, as established by Arakelov \cite{Arakelov71} and Szpiro \cite[Chapitre III]{SemSzpiro81}.

It turns out that the results in Subsections \ref{basicDiagr}, \ref{pi1proof}, and \ref{pseudoconcaveEasy}, dealing with a smooth connected projective curve $C$ embedded in a smooth analytic surface $\cV$, may be extended to situations where $C$ is replaced by a general  connected projective curve, possibly non-reduced, satisfying a suitable positivity property  that generalizes the positivity \eqref{CVpseudoconcav} of the self-intersection of $C$.

These generalizations of our previous results will allow us to establish  proper  analogues of our later arithmetic results in Sections \ref{AnAlFibered} and \ref{ExampComp}. These are established by some amplifications of our earlier arguments in Subsections \ref{basicDiagr}, \ref{pi1proof}, and \ref{pseudoconcaveEasy}, which still only relies  on  basic properties of algebraic curves and surfaces.

The content of this chapter is more technical than the one of Chapter 1, and could be skipped by a reader mainly interested by the arithmetic results of this memoir.

\medskip 

The reader with an interest in complex analytic geometry will observe that the analytic discussion in Subsection \ref{PseudoAndreotti}, where the positivity of $\deg_C N_C \cV$ is related to the pseudoconcavity properties of the (germ of) analytic surface $\cV$,  is not generalized to the more general framework of this chapter, although it seems likely that it could be. 

For instance one expects that,  if an effective divisor $D$ in a connected smooth  complex analytic surface $\cV$ satisfies the condition $\mathbf{CNB}$ introduced in Subsection \ref{CNB} below, then
its  support $\vert D \vert$ admits a basis 
of open neighborhoods in $\cV$ that are pseudoconcave.  This would follow from a suitable generalization  of the simple argument on Levi forms in Proposition~\ref{pseudoconcavetube}. 

However, at this stage, we do not know whether such a basis of   pseudoconcave open neighborhoods exists for a general $\mathbf{CNB}$ divisor.  This difficulty with the analytic approach confers a special interest to the  algebraic approach in this chapter when dealing with general $\mathbf{CNB}$ divisors.

\section{Compact, connected, nef and big effective divisors in analytic surfaces}\label{nonreducedPseudoConcave}

\subsection{The condition $\mathbf{CNB}$ and its consequences}\label{CNB} In this section, we denote by $\cV$ a connected smooth complex analytic surface. For every effective  divisor $D$ in $\cV$, we may introduce the following conditions:

\noindent $\mathbf{CNB}$: \emph{The support $\vert D \vert$ of $D$ is compact and connected.} \emph{If we denote by 
$(D_i)_{i\in I}$ the  irreducible components of $D$, then we have:}
\begin{equation}\label{DiD}
\mbox{\emph{for every }$i \in I$, } D_i \cdot D \geq 0.
\end{equation}
\emph{Moreover:}
\begin{equation}\label{DD>0}
D \cdot D > 0.
\end{equation}

When $\vert D \vert$ is compact and connected, the non-negativity condition \eqref{DiD} is equivalent to the following one: 
\begin{equation}\label{Dnefbis}
  \mbox{\emph{for every effective divisor $\widetilde{D}$ in $\cV$, } }  \widetilde{D} \cdot D \geq 0.
\end{equation}
When the first two conditions in $\mathbf{CNB}$ are satisfied, the validity of \eqref{DD>0} is equivalent to the existence of $i \in I$ such that the inequality \eqref{DiD} is strict. 

Observe also that, when the surface $\cV$ is projective, an effective divisor $D$ satisfies the conditions $\mathbf{CNB}$ if and only if $D$, or equivalently the line bundle $\cO_\cV(D)$, is nef and big.\footnote{The connectedness of  $\vert D \vert$  when $\cO_\cV(D)$ is nef and big follows from the Hodge index inequality, by an argument of C.P. Ramanujam which actually establishes that $D$ is numerically connected. See for instance  \cite[Chapter II]{SemSzpiro81} or \cite[Proposition~2.2]{Bost99}.} The name  $\mathbf{CNB}$ for the above condition stands for {\bf c}ompact, connected, {\bf n}ef and {\bf b}ig, although the general formalism of big line bundles and divisors does not really makes sense on a general analytic surface. 

From now on, we suppose that $D$ is an effective divisor in $\cV$ that satisfies the conditions $\mathbf{CNB}$, we denote by $C$ an irreducible component of $D$ such that:
\begin{equation}\label{CDpositive}
C\cdot D >0,
\end{equation}
and we denote by $\mu$ the multiplicity of $C$ in $D$. We may decompose $D$ as a sum:
\begin{equation}\label{DCE}
D = \mu \, C + E,
\end{equation}
where $E$ is an effective divisor with compact support not containing $C$. 

\begin{proposition}\label{prop:boundpseudoconcaveD} For every complex analytic line bundle $M$ over $\cV$, the $\C$-vector space of  its analytic sections $\Gamma(\cV, M)$ is finite dimensional. Moreover:
\begin{equation}\label{boundMD}
\dim_\C \Gamma(\cV, M) \leq C( c_1(M)\cdot D)
\end{equation}
where, for every $n \in \Z$:
\begin{equation}\label{defCnD}
C(n) := \sum_{0 \leq i \leq \lfloor n/C\cdot D \rfloor} \left( 1 + \lfloor (n- i \, C\cdot D) /\mu \rfloor \right)^+.
\end{equation}
\end{proposition}

Observe that we still have $C(n) = 0$ if $n < 0,$ and $C(0) =1$, and that:
$$C(n) \sim \frac{n^2}{2 \mu \, C \cdot D},  \quad \mbox{when $n \lra + \infty$}.$$

\begin{proof}
Let us choose a point $P$ in $C_{\mathrm{reg}}.$

As in the proof of Proposition \ref{finitepseudoconcave}, we may consider the decreasing filtration $(E^i)_{i \in \N}$ of $\Gamma(\cV, M)$ defined by:
$$E^i := \Gamma(\cV, M \otimes \cO_\cV (-i C)).$$
 For every $i \in \N,$ we may also consider the restriction map:
 $$\eta^i: E^i = \Gamma(\cV, M \otimes \cO_\cV (-i C)) \lra   \Gamma(C, (M \otimes \cO_\cV (-i C))_{\vert C})= \Gamma(C, M_{\vert C} \otimes N_C\cV^{\otimes (- i)})$$ defined by: $$\eta^i(s) := s_{\vert C}.$$
 Its kernel is $E^{i+1}$.  
 
 For every $k \in \N$, we may introduce the 
 closed subscheme of $C$, supported by $P$, defined by the effective Cartier divisor $kP$ in $C$. We shall denote it by $kP$, and consider the composite map: 
 $${\eta}^i_k: E^i/E^{i +1}  \lra \Gamma(C, M_{\vert C} \otimes N_C\cV^{\otimes (- i)})  \lra (M_{\vert C} \otimes N_C\cV^{\otimes (- i)})_{\vert kP},$$
 defined by:
 $${\eta}^i_k([s]) := (\eta^i(s))_{\vert kP}.$$
 
 \begin{lemma}\label{CDmu} The map  $\eta^i_k$ is injective if the following inequality holds:
 $$i \, C \cdot D  + \mu k > c_1(M) \cdot D.$$
 \end{lemma}
 
 \begin{proof}[Proof of Lemma \ref{CDmu}]
  Let us consider a non-zero class $[s]$ in $E^i/E^{i+1}$ that lies in $\ker {\eta}^i_k$.
 The divisor of the analytic section $s$ of $M$ may be written:
 \begin{equation}\label{divsCDR}
\div s = i C + \widetilde{D} + R,
\end{equation}
where $\widetilde{D}$ and $R$ are effective divisor, where  $\widetilde{D}$ is supported on $\vert E \vert$ and $R$ meets $D$ properly. 
The scheme-theoretic intersection $R \cap C$ contains $kP$ as a subscheme, and consequently:
$$ R \cdot C \geq k.$$
Consequently:
$$R \cdot D = \mu \, R\cdot C + R\cdot E \geq \mu k.$$
Moreover, as already observed,  the non-negativity condition \eqref{DiD} on $D$ implies:
$$\widetilde{D} \cdot D \geq 0.$$
Together with \eqref{divsCDR}, this implies:
\begin{equation*}c_1(M) \cdot D = \div s \cdot D \geq i \, C\cdot D + \mu k.
\qedhere
\end{equation*}
\end{proof}

According to Lemma \ref{CDmu}, we have:
$$ i \geq  \lfloor (c_1(M) \cdot D) / (C \cdot D) \rfloor  +1 \Longrightarrow E^i/E^{i+1} = \{0\},$$
and:
$$ i \leq \lfloor (c_1(M) \cdot D) / (C \cdot D) \rfloor \Longrightarrow \dim_\C E^i/E^{i+1} \leq 1 + \lfloor (c_1(M) \cdot D - i \, C \cdot D) /\mu \rfloor.$$
Since the filtration $(E^i)_{i \in \N}$ satisfies
 $E^0 = \Gamma(\cV, M)$  and $\bigcap_{i \in \N} E^i =\{0\}$, 
 this implies the finite dimensionality of $\Gamma(\cV, M)$ and the upper bound \eqref{boundMD} on its dimension.
\end{proof}
 
%

Proposition \ref{prop:boundpseudoconcaveD} notably implies the validity of the ``weak dimension estimates"  \eqref{weakfinitepseudoconcave}. In turn, it leads to the following generalization of  the algebraicity result in Proposition \ref{algebraicityPseudoconcaveAnalytic}  by a straightforward modification of its proof: 

\begin{proposition}\label{algebraicityPseudoconcaveAnalyticD} Let $\cV$ be a smooth connected complex analytic surface containing an effective divisor $D$ satisfying condition $\mathbf{CNB}$. 
If $\gamma: \cV \ra Z$ is a complex analytic map from $\cV$ to some complex quasi-projective variety $Z$, then either $f$ is constant, or $f(\cV)$ coincides with $f(\vert D \vert)$ and is an irreducible projective curve,  or $f(\cV)$ is contained, as a Zariski dense subset, in some irreducible closed algebraic surface in $Z$.  
\end{proposition}

The computations of intersection numbers in Subsection \ref{basicDiagr} 
also may  be generalized to the present setting.

\begin{proposition}\label{selfintgamma}  Let $\cV$ be a smooth connected complex analytic surface containing an effective divisor $D$ satisfying condition $\mathbf{CNB}$, and let us keep the  notation introduced in \eqref{CDpositive} and \eqref{DCE} above. For every complex analytic map $\gamma: \cV \ra Z$ with Zariski dense image from $\cV$ to a connected smooth complex algebraic surface $Z$, such that $\gamma_{\mid C}: C \ra \gamma(C)$ is birational, the following estimates are satisfied:
\begin{equation}\label{eq:selintgamma}
\gamma_\ast D \cdot \gamma_\ast D \geq e(\gamma) \, \mu\,  C\cdot D \geq \mu\,  C \cdot D,
\end{equation}
where $e(\gamma)$ denotes the ramification index of $\gamma$ along $C$.
\end{proposition}

\begin{proof} According to the projection formula, the following equality of intersection numbers holds:
\begin{equation*}
\gamma_\ast D \cdot \gamma_\ast D = \gamma^\ast \gamma_\ast D \cdot  D.
\end{equation*}
Moreover we have:
$$\gamma_\ast C = \gamma(C),$$
and, according to the very definition of $e(\gamma),$ the irreducible curve $C$ occurs with multiplicity $e(\gamma)$ in the effective divisor  $\gamma^\ast \gamma(C).$
This implies  the following inequality of divisors on $\cV$:
\begin{equation*}
\gamma^\ast \gamma_\ast D = \mu\, \gamma^\ast \gamma_\ast C + \gamma^\ast \gamma_\ast E \geq \mu\,  \gamma^\ast \gamma_\ast C \geq e(\gamma) \, \mu\, C.
\end{equation*}
Together with the non-negativity condition \eqref{Dnefbis} on $D$, this implies:
$$\gamma^\ast \gamma_\ast D \cdot  D \geq e(\gamma) \, \mu\, C \cdot D,$$
and \eqref{eq:selintgamma} follows.
\end{proof}

Using Proposition \ref{selfintgamma} together with the general degree bound in Proposition \ref{proposition:main-geomAmplif}, we may derive a generalization of the degree bound  \eqref{equation:degree-bound-geomBis} in Proposition \ref{prop:degree-bound-geom}.

Indeed let us assume that $f:V \ra U$ is a  morphism of complex algebraic varieties, where $U$ and $V$  are two smooth connected quasi-projective complex surfaces, and that we are given complex analytic maps $\alpha$ and $\beta$ from $\cV$ to $X$ and $Y$ respectively, that fit into a commutative diagram:
\begin{equation}\label{diagram:setupenocre}
\begin{gathered}
\xymatrix{
& V\ar[d]^{f}\\
\cV\ar[r]^{\alpha}\ar[ur]^{\beta} & U.
}
\end{gathered}
\end{equation}
Let us assume that \emph{the image of $\alpha$ is Zariski dense in $U$} --- this implies that the image of $\beta$ is Zariski dense in $V$ and that $f$ is dominant --- and that \emph{the restriction of $\alpha$ to $C$}:
$$\alpha_{\mid C}: C \lra \alpha(C)$$
\emph{is birational}. Then $\beta_{\mid C}: C \lra \beta(C)$ also is birational, and Proposition \ref{selfintgamma} applied to $\gamma := \beta$ implies:
$$\beta_\ast D \cdot \beta_\ast D \geq C\cdot D >0.$$
Therefore Proposition \ref{proposition:main-geomAmplif} applies with $A := \alpha_\ast (D)$ and $B = \beta_\ast (D)$. This proves that $\deg f$ \emph{satisfies the following upper bound}:
\begin{equation}\label{degCNB}
\deg f \leq \frac{\alpha_\ast D \cdot \alpha_\ast D }{\beta_\ast D \cdot \beta_\ast D } \leq \frac{\alpha_\ast D \cdot \alpha_\ast D}{\mu \, C\cdot D}.
\end{equation}
 
This generalizes Proposition \ref{prop:degree-bound-geom}. In turn the upper bound \eqref{degCNB} implies a generalization of the results \emph{\`a la} Lefschetz-Nori concerning \'etale fundamental groups in Proposition  \ref{easyNori}, where now the conclusion reads:
\begin{equation}\label{ineqMoripi1ter}
[\pi_1^\et (U, x): \alpha_{\mid \vert D \vert \ast}(\pi_1^\et (\vert D \vert ,c))] \leq \frac{\alpha_\ast D \cdot \alpha_\ast D }{\mu \, C\cdot D}. 
\end{equation}

\subsection{Compatibility of  $\mathbf{CNB}$ divisors with modifications. Applications to meromorphic maps}\label{ModifCNB}

Condition $\mathbf{CNB}$  turns out to be preserved by modifications:
\begin{proposition}\label{CNBmodif}
Let $D$ be an effective divisor in some connected smooth analytic surface $\cV$, and let $\nu: \cV' \ra \cV$ be a modification\footnote{
 Recall that a modification $\nu:\cV' \ra \cV$ of $\cV$ is a proper complex analytic map of degree 1 from a connected smooth complex analytic surface $\cV'$ onto $\cV$. Then there exists a closed discrete subset $Z$ of $\cV$ such that: $c_{\mid \cV' \setminus \nu^{-1}(Z)}: \cV' \setminus \nu^{-1}(Z) \ra \cV \setminus Z$ is an isomorphism, and $\nu^{-1}(P)$  is a compact connected analytic curve in $\cV'$ for every $P \in \cV$.
} of $\cV$.  If $D$ satisfies $\mathbf{CNB}$, then the divisor $D' :=\nu^\ast (D)$ in $\cV'$ also satisfies $\mathbf{CNB}$. 

 \end{proposition}
 

\begin{proof} The compactness and the connectedness of $\vert D \vert$ imply the compactness and the connectedness of $\vert D' \vert= \nu^{-1}(\vert D \vert)$ since $\nu$ is proper with connected fibers. The validity of \eqref{DD>0} and \eqref{Dnefbis} for the divisor $D$, together with the projection formula, implies its validity for $D'$.
 \end{proof}
 
 Observe also that, with the notation of Proposition \ref{CNBmodif}, the projection formula also implies the equality:
\begin{equation}
D' \cdot D' = D \cdot D.
\end{equation}
Moreover, if $C$ is an irreducible component of $D$ of mutiplicity $\mu$, then its strict transform $C'$ in $\cV'$ is an irreducible component of multiplicity $\mu$ in $D'$,
and we have:
\begin{equation}
C'\cdot D' = C \cdot D.
\end{equation}

Proposition \ref{CNBmodif} and these observations allow one to generalize our previous results to situations, where instead of 
complex analytic maps from the analytic surface $\cV$ to a smooth algebraic surface, one considers meromorphic maps. 

To formulate this generalization, recall that, for every
 meromorphic map: $$\alpha: \cV \dashrightarrow X$$ of range a smooth projective complex surface $X$, there exists a modification: 
$$\nu': \cV' \lra \cV$$ of $\cV$ that is ``adapted to $\alpha$", namely  such that, composed with $\nu$, the meromorphic map $\alpha$ defines an actual complex analytic map:
$$\alpha' := \alpha \circ \nu: \cV \lra X.$$
The meromorphic map $\alpha$ defines a correspondence $\Gamma_\alpha$ from $\cV$ to $X$, proper over $\cV$, and for every divisor $D$ with compact support in $\cV$, we may consider the image $\Gamma_{\alpha \ast} D$ of $D$ by this correspondence. It is a divisor in $X$, effective when $D$ is effective, and in terms of a modification $\nu$ adapted to $\alpha$ as above, it may defined by the equality: 
$$\Gamma_{\alpha \ast} D = \alpha'_\ast \nu^\ast D.$$

\begin{proposition}\label{degboundmero} Let $\cV$ be a connected smooth  complex analytic surface, and let $D$ be an effective divisor in $\cV$ satisfying $\mathbf{CNB}$. Let $C$ be an irreducible component of $D$ such that:
$$C \cdot D >0,$$
and let $\mu$ be its multiplicity in $D$. Consider a commutative diagram::
\begin{equation}\label{diagram:setupprime}
\begin{gathered}
\xymatrix{
& Y\ar[d]^{f}\\
\cV\ar@{-->}[r]^{\alpha}\ar@{-->}[ur]^{\beta} & X,
}
\end{gathered}
\end{equation}
 where $X$ and $Y$ are connected smooth projective complex surfaces, $f$ is a dominant morphism, and $\alpha$ and $\beta$ are meromorphic maps.

If the image of $\alpha$ is Zariski dense in $X$, and if the restriction $\alpha_{\mid C}$ of $\alpha$ to $C$ is birational onto its image, then the degree of $f$ satisfies the following upper bound:
\begin{equation}\label{eq:degboundmero}
\deg f \leq \frac{\Gamma_{\alpha \ast} D  \cdot \Gamma_{\alpha \ast} D }{\Gamma_{\beta \ast} D  \cdot \Gamma_{\beta \ast} D } \leq \frac{\Gamma_{\alpha \ast} D  \cdot \Gamma_{\alpha \ast} D }{\mu \, C\cdot D}.
\end{equation}
\end{proposition}

\begin{proof} The degree bound  \eqref{eq:degboundmero} follows from the estimates \eqref{degCNB} --- which  generalizes the estimates in Proposition \ref{prop:degree-bound-geom} to general  $\mathbf{CNB}$ divisors --- applied to the commutative diagram:
\begin{equation*}
\begin{gathered}
\xymatrix{
& Y\ar[d]^{f}\\
\cV'\ar[r]^{\alpha'}\ar[ur]^{\beta'} & X
}
\end{gathered}
\end{equation*}
defined by some modification $\cV'$ of $\cV$ adapted both to $\alpha$ and $\beta$, and to the $\mathbf{CNB}$ divisor $D'$ in $\cV$.
\end{proof}

\subsection{Generalization to $\Q$-divisors}
Observe finally that, if $D$ is an effective divisor in a connected smooth analytic complex surface $\cV$ and $n$ a positive integer, $D$ satisfies $\mathbf{CNB}$ if an only if $nD$ does. Consequently the condition $\mathbf{CNB}$ makes sense, not only for effective divisors in $\cV$, but for any effective $\Q$-divisor. 

Moreover the results of \ref{CNB} and \ref{ModifCNB} extend to this more general setting, notably the degree bound  \eqref{degCNB}, the index bound \emph{Ã  la} Nori \eqref{ineqMoripi1ter}, the compatibility with modifications, and Proposition \ref{degboundmero}.  This is a formal consequence of the definitions and of the fact that the quotients of intersection numbers in the right-hand side of the estimates \eqref{degCNB}, \eqref{ineqMoripi1ter}, or \eqref{eq:degboundmero} are unchanged when the divisor $D$ is replaced by some positive multiple. 

\section[Analytic and algebraic varieties fibered over a projective curve]{Application to analytic and algebraic varieties fibered over a projective curve}\label{AnAlFibered}

\subsection{Complex analytic surfaces fibered over a projective curve: $\mathbf{CNB}$ divisors and equilibrium divisors}\label{cVfibered}
In this section we consider the following data: 
a connected smooth projective complex curve $C$,  a connected smooth complex analytic surface $\cV$,  a surjective (necessarily flat) complex analytic map, with connected fibers:
$$\pi_{\cV}: \cV \lra C,$$
a complex analytic section of $\pi_{\cV}$:
$$\epsilon: C \lra \cV,$$
and  its image:
$$C_\cV := \epsilon(C).$$

Let also $\Sigma$ be a non-empty finite subset of $C$, and let :
$$\mathring{C} := C \setminus \Sigma$$
be its complement. It defines a smooth connected affine complex algebraic curve. We shall assume that, \emph{for every $x \in \Sigma$, the (possibly non-reduced) connected curve $\pi_\cV^{-1}(x)$ is non-compact} and that we are given \emph{a compact connected curve $F_x \subset \vert \pi_\cV^{-1}(x) \vert$ such that:}
$$F_x \cap C_\cV\neq \emptyset.$$

\subsubsection{} From the data of the section $\epsilon_\cV$ and of the ``vertical curves" $(F_x)_{x \in \Sigma}$, we may carry out the following construction.

Recall that our  assumptions on the curves $\pi^{-1}(x)$ and $F_x$ imply that the  intersection form restricted to divisors in $\cV$ supported on 
$$F := \bigcup_{x \in \Sigma} F_x$$
is negative definite. In turns, this implies that a divisor $\Delta$ supported on $F$  is effective if $\Delta \cdot W \leq 0$ for every irreducible component of $F$.\footnote{See for instance \cite[Lemme 3.7]{Moret-Bailly89}, where  $(D.C) \geq 0$ should read $(D.C) \leq 0$.}

This implies that there exists a unique $\Q$-divisor $V$  supported on $F$ such that, for every irreducible component $W$ of $F$:
\begin{equation}\label{orthD}
(C_\cV+ V)\cdot W =0,
\end{equation}
and moreover that $V$ is effective. The support $\vert D \vert$  of the divisor 
$$D := C_\cV+ V$$
 is easily seen to be connected, again as a consequence of the negativity of the intersection form on divisors supported by $F$. In turn, this implies the equality:
 $$\vert V \vert = F.$$
 Moreover we have:
$$D\cdot D = C_\cV\cdot D,$$ 
and consequently $D$ satisfies condition $\mathbf{CNB}$ if and only if $D\cdot D$ is positive.

The construction of the divisor $D$ starting from the section $\epsilon$ and the vertical curves $(F_x)_{x \in \Sigma}$ is a geometric counterpart of the construction in Subsection \ref{PgVNb} 
below of the Arakelov divisor:
 $$(P, (g_{V_\sigma, P_\sigma})_{\sigma: K \hra \C})$$ attached to some formal-analytic surface:
  $$\Vfa:= (\Vf, (V_\sigma, P_\sigma, \iota_\sigma)_{\sigma \hra \C})$$ by means of the equilibrium potentials $g_{V_\sigma, P_\sigma}$ associated to the pointed compact Riemann surfaces with boundary $(V_\sigma, P_\sigma)$. Accordingly the divisor $D$ could be referred to as the \emph{equilibrium divisor} associated to $\epsilon$ and $(F_x)_{x \in \Sigma}$.
  
  \subsubsection{} Among the $\mathbf{CNB}$ $\Q$-divisors of the form $C_{\cV} + V$ with $V$ supported on $F$, when they exist, the equilibrium divisor $D$ satisfies the following extremal property:

\begin{proposition}\label{Denough} If there exists an effective $\Q$-divisor $\widetilde{V}$ supported on $F$ such that the divisor
\begin{equation}\label{Dtildedef}
\widetilde{D} := C_\cV + \widetilde{V}
\end{equation}
satisfies $\mathbf{CNB}$, then the divisor $D$ satisfies $\mathbf{CNB}$. Moreover, we have:
\begin{equation}\label{vgeqVt}
V \geq \widetilde{V},
\end{equation}
\begin{equation}\label{cDDt}
C_\cV\cdot D \geq C_\cV\cdot \widetilde{D},
\end{equation}
and:
\begin{equation}\label{DDDtDt}
D \cdot D \geq  \widetilde{D} \cdot \widetilde{D},
\end{equation}and equality holds in \eqref{DDDtDt} if and only if $\widetilde{D}= D$.
\end{proposition}

\begin{proof} Let  $\widetilde{V}$ an effective $\Q$-divisor supported by $F$ such that the divisor
$\widetilde{D}$ defined by \eqref{Dtildedef} satisfies $\mathbf{CNB}$, and let:
$$\Delta := V -\widetilde{V}.$$ 
 Then, for every irreducible component $W$ of $F$, we have:
 $$\Delta \cdot W = -(\widetilde{D} - D)\cdot W = - \widetilde{D}. W \leq 0.$$
 Since $\Delta$ is supported by $F$, it is therefore effective. This implies \eqref{vgeqVt}, and the inequality \eqref{cDDt} immediately follows. In particular, we have:
 \begin{equation}
 \label{CVD}
 C_\cV\cdot D \geq 0.
 \end{equation}
Moreover, according to \eqref{orthD}, we also have:
$$D \cdot \Delta = 0.$$
Consequently:
$$\widetilde{D}\cdot \widetilde{D} = (D -\Delta) \cdot (D-\Delta) = D \cdot D + \Delta \cdot \Delta.$$
Since $\Delta\cdot \Delta$ is non-positive and vanishes if and only if $\Delta =0,$ this establishes the inequality \eqref{DDDtDt} and the last assertion of Proposition \ref{Denough}. 

Finally the relations \eqref{orthD}, \eqref{CVD}, \eqref{DDDtDt}, and the positivity of $\widetilde{D} \cdot \widetilde{D}$ imply that $D$ satisfies~$\mathbf{CNB}$.\end{proof}

\begin{corollary}\label{cor:CNB} The following conditions are equivalent:
\begin{enumerate}[(i)]
\item $C_\cV \cup F$ is the support of a $\Q$-divisor satisfying $\mathbf{CNB}$;
\item $D \cdot D >0;$
\item $D$ satisfies $\mathbf{CNB}$.
\end{enumerate}
\end{corollary}

The analogue of Corollary \ref{cor:CNB} in the arithmetic setting will be established in Section \ref{pseudobis}.

\subsection{Compatibility with modifications}\label{CompModifFibered} In the present ``fibered" situation, the compatibility of condition $\mathbf{CNB}$ with modifications discussed in \ref{ModifCNB} may be complemented by the following remarks.

Let us keep the notation of \ref{cVfibered}, and let us consider   a modification $\nu: \cV' \ra \cV$ of $\cV$. Then the morphism:
$$\pi_{\cV'} := \pi_\cV \circ \nu: \cV' \lra C$$
is surjective with connected fibers, and the strict transform $C_{\cV'}$ of $C_{\cV}$ in $\cV'$ is the image of a section:
$$\epsilon': C \lra \cV'$$
of $\pi_{\cV'}$.

The indetermination locus of $\nu^{-1}$ meets $\mathring{C}_{\mid \cV}:= \epsilon(\mathring{C})$ in a finite subset $\epsilon(\widetilde{\Sigma})$. Let us define:
$$\Sigma' := \Sigma \cup \widetilde{\Sigma},$$
and for every $x$ in $\Sigma$ (resp. in $\widetilde{\Sigma}$): 
$$F'_x := \nu^{-1}(F_x) \quad \mbox{(resp. $F'_x := \nu^{-1}(\epsilon(x))$).} $$
The $F'_x,$ $x \in \Sigma',$ are compact connected curves in  $\vert \pi_{\cV'}^{-1}(x) \vert$ such that:
$$F'_x \cap C_{\cV'}\neq \emptyset.$$

The effective $\Q$-divisor $D' := \nu^\ast (D)$ in $\cV'$ may be written:
$$D' = C_{\cV'} + V'$$
where $V'$ is a divisor with support:
$$F' := \bigcup_{x \in \Sigma'} F'_x.$$
Moreover the projection formula together with \eqref{orthD} 
imply that, for every irreducible component $W'$ of $F'$, we have:
$$(C_{\cV'}+ V')\cdot W' =0.$$
This shows that the divisor $D'$ in $\cV'$ coincides with the equilibrium divisor associated to the section $\epsilon'$ and the family of vertical curves $(F'_x)_{x \in \Delta'}$ in $\cV'$.

In particular, we have:
$$D'\cdot D' = C_{\cV'}\cdot D'.$$
Moreover, as already observed:
$$D' \cdot D' = D\cdot D,$$
and therefore $D'$ satisfies $\mathbf{CNB}$ if and only if $D$ satisfies it.

\subsection{Applications: morphisms from fibered analytic surfaces to fibered algebraic surfaces}\label{BoundsFibered}
In this subsection and in the next one, we keep the notation of Subsection \ref{cVfibered} and we assume that \emph{the $\Q$-divisor $D$
has positive self-intersection}:
$$D\cdot D = C\cdot D >0,$$
 and therefore satisfies  $\mathbf{CNB}$.
 
 Observe that, if $X$ is a connected smooth projective complex surface and if
 \begin{equation*}
\begin{gathered}
\xymatrix{
\cV\ar@{-->}[r]^{\alpha}\ar[d]^{\pi_\cV} & X\ar[d]^{\pi_X}
\\
C \ar[r]^{=} &C
}
\end{gathered}
\end{equation*}
is a commutative diagram where $\alpha$ is a meromorphic map and $\pi_X$ is a morphism of complex algebraic varieties,\footnote{or equivalently, by GAGA, of complex analytic manifolds. The morphism $\pi_X$ is necessarily non-constant, hence surjective and flat.} then $\alpha_{\mid C_\cV}$ is trivially birational onto its image. Moreover the image of $\alpha$ is easily seen to be contained in some algebraic curve in $X$ if and only if $\alpha$ factorizes through $\pi_{\cV}$. In other words, the image of $\alpha$ is Zariski dense in $X$ if and only if $\alpha$ is not constant on the fibers of $\pi_\cV$.

 Applied to the present setting, Proposition \ref{degboundmero} for $\Q$-divisors takes the following form:

\begin{theorem}\label{degmerofibered} 
Consider a commutative diagram:
 \begin{equation}\label{diagram:setupprimefibered}
\begin{gathered}
\xymatrix{
& Y\ar[d]^{f}\\
\cV\ar@{-->}[r]^{\alpha}\ar@{-->}[ur]^{\beta}\ar[d]^{\pi_\cV} & X\ar[d]^{\pi_X}
\\
C \ar[r]^{\mathrm{Id}_C} &C,
}
\end{gathered}
\end{equation}
 where $X$ and $Y$ are two projective complex surfaces, with $X$ smooth and connected and $Y$ integral, where $\alpha$ and $\beta$ are meromorphic maps, and where $f$ and $\pi_X$ are morphisms of complex algebraic varieties. 
 
 If $\alpha$ is not constant on the fibers of $\pi_\cV$, 
 then $f$ is dominant and its degree satisfies:
\begin{equation}\label{degfgeomgen}
\deg f \leq 
\frac{\Gamma_{\alpha \ast} D  \cdot \Gamma_{\alpha \ast} D }{ C_\cV\cdot D} = \frac{\Gamma_{\alpha \ast} D  \cdot \Gamma_{\alpha \ast} D }{ D \cdot D}. 
\end{equation}
\end{theorem}

Actually Proposition \ref{degboundmero} immediately establishes Theorem \ref{degmerofibered}  when the surface $Y$ is smooth. By considering a resolution of $Y$, this implies Theorem \ref{degmerofibered} in general.

Similarly, the amplification of Proposition  \ref{easyNori} concerning $\mathbf{CNB}$ divisors discussed at the end of Section \ref{CNB}, implies the following result \emph{\`a la} Lefschetz-Nori concerning \'etale fundamental groups:

\begin{theorem}\label{LefschetzNoriCNB} Consider a commutative diagram:
\begin{equation}\label{pi1fibered}
\begin{gathered}
\xymatrix{
\cV\ar[r]^{\alpha}\ar[d]^{\pi_\cV} & U\ar[d]^{\pi_U}
\\
C \ar[r]^{\mathrm{Id}_C} &C,
}
\end{gathered}
\end{equation}
where $U$ is a connected smooth quasi-projective complex surface, $\alpha$ a complex analytic map, and $\pi_V$ a morphism of complex algebraic varieties.\footnote{The morphism $\pi_U$ is necessarily surjective and flat.}

If $\alpha$ is non constant on the fibers of $\pi_\cV,$ then the image of the morphism between \'etale fundamental groups: 
\begin{equation}\label{alphaCastbis}
\alpha_{\mid C_\cV \cup F \ast}: \pi_1^\et (C_\cV \cup F, c) \lra \pi_1^\et(U, x)
\end{equation}
induced by the map: $$\alpha_{\mid C_\cV \cup F} : C_\cV \cup F  \ra U$$
is a subgroup of finite index; moreover:
\begin{equation}\label{ineqMoripi1quart}
[\pi_1^\et (U, x): \alpha_{\mid C_\cV \cup F \ast}(\pi_1^\et (C_\cV \cup F, c))] \leq 
\frac{\alpha_\ast D \cdot \alpha_\ast D}{C_\cV \cdot D} =\frac{\alpha_\ast D  \cdot \alpha_\ast D}{D\cdot D}.
\end{equation}
 
\end{theorem}

In \eqref{alphaCastbis} and \eqref{ineqMoripi1quart},  $c$ denotes a complex point of $C$, and $x:= \alpha \circ \epsilon (c)$ its image in ${U}$. 

\subsection{The field of  meromorphic functions $\cM(\cV)$}\label{functions}

The map between fields of  meromorphic functions:
$$\pi_\cV^\ast : \cM(C)  \lra \cM(\cV), \quad f \longmapsto f \circ \pi_\cV$$
allows one to identify the field 
$\cM(C) = \C(C)$
of rational functions over $C$ with a subfield of the field $\cM(\cV)$ of  meromorphic functions over $\cV$. Then $\C(C)$ is easily seen to be algebraically closed in $\cM(\cV)$. Moreover any  meromorphic function $f \in \cM(\cV)$ defines a meromorphic map ``fibered over~$C$," namely:
$$f_C := (\pi_\cV, f) : \cV \dashrightarrow \PP^1_C:=  C \times \PP^1_C.$$

\begin{theorem}\label{McV} The following alternative holds: either $\cM(\cV) = \C(C),$ or $\cM(\cV)$ is an extension of finite type and of transcendence degree one of $\C(C)$. 

Moreover, for every $f$ in $\cM(\cV) \setminus \C(C),$ $\cM(\cV)$ is a finite extension of the purely transcendental extension $\C(C)(f)$ of $\C(C)$,  and its degree satisfies the following upper bound:
 \begin{equation}\label{degMcV}
[\cM(\cV): \C(C)(f)] \leq  \frac{\Gamma_{f_C \ast} D  \cdot \Gamma_{f_C \ast} D }{ C_\cV \cdot D} = \frac{\Gamma_{f_C \ast} D  \cdot \Gamma_{f_C \ast} D }{ D \cdot D}. 
\end{equation}
 \end{theorem}
 
 The right-hand side of \eqref{degMcV} is clearly unchanged when $\cV$ is replaced by some connected open neighborhood $\cV'$ of $C_\cV \cup F$. 
 This immediately implies that, in  Theorem  \ref{McV}, we may replace the complex analytic surface $\cV$  by its germ $\cV^{\an}_{C_\cV \cup F}$ along $C_{\cV} \cup F$:
 
 \begin{corollary}  Theorem  \ref{McV} still holds when the field $\cM(\cV)$ is replaced by the field $\cM(\cV^{\an}_{C_\cV \cup F})$ of germs of  meromorphic functions on $\cV$ along $C_\cV \cup F$. Moreover there exists a connected open neighborhood $\cV'$ of $C_\cV \cup F$ in $\cV$ such that any germ of  meromorphic functions in $\cM(\cV^{\an}_{C_\cV \cup F})$ extends to a  meromorphic function in $\cV'$.
 
\end{corollary}

As mentioned in \ref{SectionsAlgebraicity} above, the results of the first part of this monograph admit variants in formal geometry. We shall leave this to the interested reader, and will only mention that Theorem \ref{McV} still holds when the field $\cM(\cV)$ is replaced by the field $\cM(\widehat{\cV}_{C_\cV \cup F})$ of formal  meromorphic functions, in the sense of \cite{HironakaMatsumura68}, on the formal completion $\widehat{\cV}_{C_\cV \cup F}$ of $\cV$ along the projective curve $C_\cV \cup F$. The field $\cM(\cV^{\an}_{C_\cV \cup F})$ naturally embeds into the field $\cM(\widehat{\cV}_{C_\cV \cup F})$, and an intriguing question is whether these two fields always coincide.\footnote{These fields are easily seen to coincide when $\cM(\cV^{\an}_{C_\cV \cup F}) \neq \C(C).$ Accordingly the question is equivalent  to the following one: \emph{does the equality $\cM(\cV^{\an}_{C_\cV \cup F}) = \C(C)$  imply the equality} $\cM(\widehat{\cV}_{C_\cV \cup F}) = \C(C)$  ? We refer the reader to \cite{CommichauGrauert1981}, \cite{Hirschowitz1981}, \cite{Steinbiss86}, and \cite[Section VII.4]{GrauertPeternellRemmert94} for results and references concerning the validity of such a ``formal principle"  in similar situations.}

\begin{proof}[Proof of Theorem \ref{McV}] It is enough to show that, if $f \in \cM(\cV)$ does not belong to $\C(C)$ --- or equivalently is not constant on the fibers of $\pi_{\cV}$ --- then, for every $g \in \cM(\cV)$, the subfield $\C(C)(f,g)$ of $\cM(\cV)$ is a finite extension of $\C(C)(f)$, of degree at most:
 \begin{equation}\label{degMcVBis}
\frac{\Gamma_{f_C \ast} D  \cdot \Gamma_{f_C \ast} D }{ C_\cV \cdot D} = \frac{\Gamma_{f_C \ast} D  \cdot \Gamma_{f_C \ast} D }{ D \cdot D}. 
\end{equation}

Indeed this will implies that any finitely generated extension of $\C(C)(f)$ in $\cM(\cV)$ is a finite extension of $\C(C)(f)$ of degree at most \eqref{degMcVBis}, and consequently that $\cM(\cV)$ itself is a finite extension of $\C(C)(f)$ of degree at most \eqref{degMcVBis}.

Let us consider $f$ and $g$ in $\cM(\cV)$ as above. There exists a modification of $\cV,$
$$\nu: \cV' \lra \cV,$$
that is adapted both to $f$ and $g$, namely such that the  meromorphic functions $f' := f \circ \nu$ and $g' := g \circ \nu$ on $\cV'$ define actual analytic maps:
$$f': \cV' \lra \PP^1(\C) 
\quad \mbox{and} \quad
g':\cV' \lra \PP^1(\C).$$ 
We may consider the complex analytic map:
\begin{equation}\label{defgamma}
\gamma':= (\pi_{\cV'}, f', g') :  \cV' \lra C \times \PP^1(\C) \times \PP^1(\C),
\end{equation}
where $\pi_{\cV'} := \pi_{\cV} \circ \nu.$ Since the divisor $D' := \nu^\ast(D)$ in $\cV'$ satisfies  $\mathbf{CNB}$, we may apply Proposition \ref{algebraicityPseudoconcaveAnalyticD} to $\gamma'$. Moreover, since $f'$ is not constant on the fibers of $\pi_{\cV'}$, the image of 
$$f'_C := (\pi_{\cV'}, f') : \cV' \lra C \times \PP^1(\C) =: \PP^1_C$$
is Zariski dense. Consequently, we are in the last case in the conclusion of Proposition   \ref{algebraicityPseudoconcaveAnalyticD}, and there exists an irreducible closed surface $H$ in $C \times \PP^1(\C)^2$ which contains $\gamma'(\cV)$ as a Zariski dense subset.
 
If we denote by:
$$p_{12}: C \times \PP^1(\C)^2 \lra C \times \PP^1(\C) =: \PP^1_C$$
the projection on the first two factors, we may consider the commutative diagrams:
\begin{equation*}
\begin{gathered}
\xymatrix{
& H\ar[d]^{p_{12\mid H}}\\
\cV'\ar[r]^{f_C'}\ar[ur]^{\gamma'} & \PP^1_C,
}
\end{gathered}
\end{equation*}
and:
\begin{equation}\label{VHP1C}
\begin{gathered}
\xymatrix{
& H\ar[d]^{p_{12\mid H}}\\
\cV\ar@{-->}[r]^{f_C}\ar@{-->}[ur]^{\gamma} & \PP^1_C,
}
\end{gathered}
\end{equation}
where $\gamma := \gamma' \circ \nu^{-1}.$ Then, by construction, the maps $f_C$ and $\gamma$ induce isomorphisms of $\C$-algebras:
$$f_C^\ast : \C(\PP^1_C) \lrasim \C(C)(f)$$
and:
$$\gamma^\ast: \C(H) \lrasim \C(C)(f,g),$$ and we have:
$$[\C(C)(f,g): \C(C)(f)] = \deg p_{12\mid H}: H \lra \PP^1_C.$$
The degree bound \eqref{degfgeomgen} in Theorem \ref{degmerofibered} applied to the diagram \eqref{VHP1C} shows that this degree is bounded from above by \eqref{degMcVBis}. \end{proof}

%
%

\section{Examples and complements -- The universal meromorphic map $\phi: \cV \dashrightarrow \cV^{\mathrm{alg}}$}\label{ExampComp}

In this section, we complete the results of the previous section concerning analytic surfaces fibered over a projective curves endowed with a $\mathbf{CNB}$ divisor. Notably in Subsection \ref{deletinginfnear}, we discuss the construction of such analytic surfaces by blowing up points, starting from an arbitrary algebraic surface fibered over a projective curve, equipped  with a section. We also pursue the study of the field of  meromorphic functions on these analytic surfaces. 

\subsection{Deleting infinitely near points and $\mathbf{CNB}$ divisors}\label{deletinginfnear}

\subsubsection{} In this subsection, we consider the following analytic data: a connected smooth projective complex curve $C$, a connected smooth projective complex surface $X$, a surjective (necessarily flat) morphism of complex (algebraic or analytic) varieties:
$$\pi_X: X \lra C,$$
an (algebraic or analytic) section of $\pi_X$:
$$\epsilon : C \lra X,$$
and its image:
$$C_X := \epsilon(C).$$
The connectedness of the projective variety $X$ and the  existence of the section $\epsilon$ imply the connectedness of the fibers of $\pi_X$.
 
Moreover we choose a point $x$ of $C$ whose fiber $X_x := \pi_X^{-1}(x)$ is smooth.

\subsubsection{} We may construct a system of modifications of $X$:
\begin{equation*}
X =: X^0 \stackrel{\;\; \nu^1}{\longleftarrow} X^1 \stackrel{\;\;\nu^2}{\longleftarrow} X^2 \longleftarrow \cdots
{\longleftarrow} X^{n-1}\stackrel{\;\;\nu^n}{\longleftarrow} X^n \longleftarrow \cdots
\end{equation*}
 by defining:
 $$\nu^n: X^n \lra X^{n-1}$$ to be the blowing-up of $X^{n-1}$ at some point $x_{n-1}$, for every $n \geq 1,$ where the points $(x_n)_{n \geq 0}$ are chosen as follows:
\begin{itemize}
\item $x_0$ is a point in $X_x$ distinct of the intersection  $\epsilon(x)$ of $X_x$ with $C_X$;  

\item for every $n \geq 1,$ $x_n$ is a point of the exceptional divisor of $\nu^n$:
$$Z_n := (\nu^{n})^{-1}(x_{n-1})$$
 distinct from the intersection of $Z^n$ with the proper transform $\widetilde{Z}^{n-1}$ of $Z^{n-1}$ by $\nu^n$, where by convention $Z^0 := X_x$.
\end{itemize}

%


The section $\epsilon$ of $\pi_X$ lifts to a section:
$$\epsilon^n: C \lra X^n$$
of the morphism:
$$\pi_{X^n}: = \pi_X \circ \nu^1 \circ \cdots \circ \nu^n : X^n \lra C.$$

Let us denote by $\widetilde{X}^n_x$ the proper transform of $X_x$ by:
$$\nu^1 \circ \cdots \circ \nu^n : X^n \lra X.$$
The divisor $X^n_x := \pi_{X^n}^{-1} (x)$ may be written:
$$X^n_x = \widetilde{X}^n_x + \sum_{1 \leq i \leq n} E^n_i,$$
where
$E^n_n = Z^n$ is the exceptional divisor of $\nu^n$,  and where $E^n_i$ is the proper transform of $E^{n-1}_i$ by $\nu^n$ when $i < n$. (Observe that $E^{n}_{n-1} = \widetilde{Z}^{n-1}$, and that when $i < n-1$, $\nu^n$ is actually an isomorphism over $E^{n-1}_i$.)

By construction, the $E^n_i$ are smooth rational curves in $X^n_x$. Moreover the possibly non-vanishing intersection numbers between the divisors $C_{X^n} := \epsilon_{X^n}(C)$, $\widetilde{X}^n_x$, and $E^n_i$, $1 \leq i \leq n$, are as follows:
$$C_{X^n} \cdot C_{X^n} = C_X \cdot C_X, \quad C_{X^n} \cdot  \widetilde{X}^n_x = 1, 
\quad  \widetilde{X}^n_x  \cdot \widetilde{X}^n_x  = -1, \quad \widetilde{X}^n_x \cdot E^n_1 = 1, $$
$$E^n_i \cdot E^n_{i+1} = 1 \quad \mbox{and} \quad 
E^n_i \cdot E^n_{i} = -2 \quad \mbox{if $1\leq i \leq n-1,$}$$
and:
$$E^n_n \cdot E^n_{n} = -1.$$
See Figure \ref{IntersectionNumbers}. This is easily established by induction on $n$.

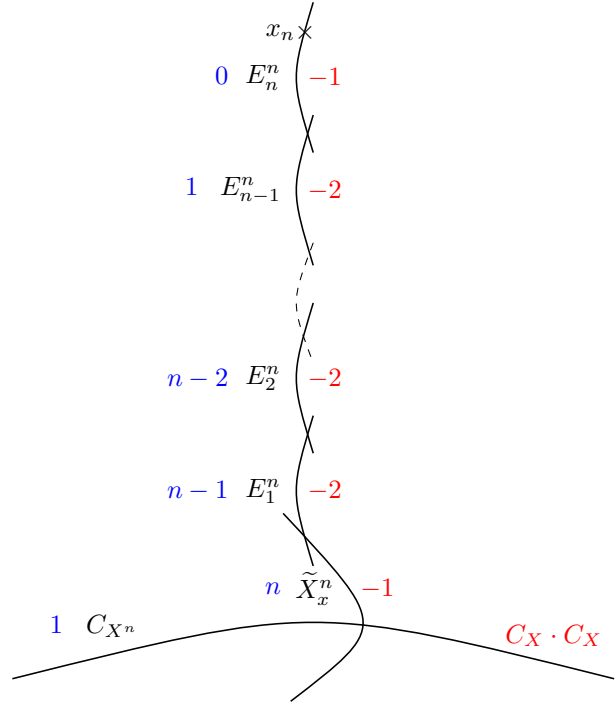
\begin{figure}
\begin{tikzpicture}
[semithick]
\draw (-0.11,4.6) node {$\times$}; 
\draw (-0.11,4.6) node[left] {$x_n$}; 
\draw (0,5) .. controls (-0.3,4)  .. (0,3);
\draw (-0.3, 4) node[left]{$E^n_n$};
\draw [blue] (-1, 4.03) node[left]{$0$};
\draw [red]  (-0.2, 4) node[right]{$-1$};
\draw (0,3.5) .. controls (-0.3,2.5) .. (0,1.5);
\draw (-0.3, 2.5) node[left]{$E^n_{n-1}$};
\draw [blue] (-1.4, 2.55) node[left]{$1$};
\draw [red] (-0.2, 2.5) node[right]{$-2$};
\draw [thin] [dashed] (0,1.8) .. controls (-0.3,1)  .. (0,0.2);
\draw (0,1) .. controls (-0.3,0) and (-0.3,0) .. (0,-1);
\draw (-0.3, 0) node[left]{$E^n_2$};
\draw [blue] (-1, 0) node[left]{$n-2$};
\draw [red] (-0.2, 0) node[right]{$-2$};
\draw (0,-0.5) .. controls (-0.3,-1.5) .. (0,-2.5);
\draw (-0.3, -1.5) node[left]{$E^n_1$};
\draw [blue] (-1, -1.5) node[left]{$n-1$};
\draw [red] (-0.2, -1.5) node[right]{$-2$};
\draw (-0.4,-1.8) .. controls (1,-3.3) .. (-0.3,-4.3);
\draw (0.4, -2.8) node[left]{$\widetilde{X}^n_x$};
\draw [blue] (-0.3, -2.8) node[left]{$n$};
\draw [red] (0.5, -2.8) node[right]{$-1$};
\draw (-2.2, -3.3) node[left]{$C_{X^n}$};
\draw [blue] (-3.2, -3.28) node[left]{$1$};
\draw (-4,-4) .. controls (0,-3) .. (4,-4);
\draw [red] (3.2 , -3.7) node[above]{$C_X\cdot C_X$};
\end{tikzpicture}

\caption{The divisors $C_{X^n}$ and $X^n_x = \tilde{X}^n_x + \sum_{i=1}^n E_i^n$ inside $X^n$. In red, the self-intersections of their components; in blue, their multiplicity in the divisor $D$.}\label{IntersectionNumbers}
\end{figure}

\subsubsection{} For every $n \geq 1$, we may consider the compact connected curve in $X_x^n$:
$$F_x := \widetilde{X}_x^n \cup \bigcup_{1 \leq i \leq n-1} E^n_i.$$
If $\cV_n$ is a connected open neighborhood, in the analytic topology, of $C_{X^n} \cup F_x$ in $X^n$, we may apply the construction in  \ref{cVfibered} to $\cV := \cV_n,$ $\pi_\cV := \pi_{X \mid \cV_n}$, $\Sigma := \{ x \}$, and $\epsilon := \epsilon^n$.

A straightforward computation shows that the ``equilibrium divisor" $D =: D^n$ defined by this construction is:
$$D^n = C_{X^n} + n \widetilde{X}^n_x + \sum_{1 \leq i \leq n-1} (n-i) E^n_i.$$
Its self-intersection satisfies:
$$D^n \cdot D^n = C_{X^n} \cdot D^n = C_X \cdot C_X + n.$$ 

Consequently, if $n$ satisfies the inequality:
$$n \geq 1 + (-C_X \cdot C_X)^+,$$
then the self-intersection $D^n \cdot D^n$ is positive, and $D^n$ is a  $\mathbf{CNB}$ divisor in $\cV_n$. This contrasts with the fact that  the self-intersection $C_X \cdot C_X$ is in general negative. Indeed it follows from the work of Arakelov \cite{Arakelov71} and Szpiro \cite[Chapitre III]{SemSzpiro81} that, if $C_X \cdot C_X$ is positive (respectively vanishes), then the smooth fibers of $\pi_X : X \ra C$ are rational curves (resp. $\pi_X$ is isotrivial).

\subsection{Non-constant morphisms to fibered algebraic surfaces}\label{Nonconstfibered} In this subsection, we use the notation of \ref{BoundsFibered} and \ref{functions}. So we denote by $\pi_\cV: \cV \ra C$ a connected smooth complex analytic surface fibered over $C,$ and by $D:= C_\cV + V$ its ``equilibrium divisor", and we assume that $D$ satisfies $\mathbf{CNB}$.

A variant of the proof of Theorem \ref{McV} allows one to derive the following result from the degree bound in Theorem \ref{degmerofibered}.

\begin{theorem}\label{cvXfin} Consider a commutative diagram:
\begin{equation*}\label{diagram:complement}
\begin{gathered}
\xymatrix{
\cV\ar@{-->}[r]^{\alpha}\ar[d]^{\pi_\cV} & X\ar[d]^{\pi_X}
\\
C \ar[r]^{\mathrm{Id}_C} &C,
}
\end{gathered}
\end{equation*}
where $X$ is a smooth connected projective complex surface, $\alpha$ is a meromorphic map, and $\pi_X$ is a morphism of complex algebraic varieties.

If $\alpha$ is not constant on the fibers of $\pi_\cV$, then the map between fields of  meromorphic functions:
$$\alpha^\ast : \C(X) \lra \cM(\cV)$$
is a finite extension of fields, and its degree satisfies the following upper bound:
\begin{equation}\label{cVX}
[\cM(\cV) : \alpha^\ast \C(X)] \leq 
\frac{\Gamma_{\alpha \ast} D  \cdot \Gamma_{\alpha \ast} D }{ C_\cV\cdot D} = \frac{\Gamma_{\alpha \ast} D  \cdot \Gamma_{\alpha \ast} D }{ D \cdot D}. 
\end{equation}
\end{theorem}

\begin{proof} The field $\C(X)$ is clearly an extension of transcendence degree one of $\C(C)$. According to Theorem \ref{McV}, the field $\cM(\cV)$, which contains $\alpha^\ast \C(X)$,  is therefore an extension of finite type and of transcendence degree one of $\C(C)$. This implies that $\cM(\cV)$ is a finite extension of $\alpha^\ast \C(X)$. To complete the proof of \eqref{cVX},
we  shall  show that, for every $\phi \in \cM(\cV),$ the degree of the extension $\alpha^\ast \C(X)(\phi) = \alpha^\ast \C(X)[\phi]$ of $\alpha^\ast \C(X)$ satisfies the following inequality:
\begin{equation}\label{cVXphi}
[ \alpha^\ast \C(X)(\phi) : \alpha^\ast \C(X)] \leq 
\frac{\Gamma_{\alpha \ast} D  \cdot \Gamma_{\alpha \ast} D }{ C_\cV\cdot D} = \frac{\Gamma_{\alpha \ast} D  \cdot \Gamma_{\alpha \ast} D }{ D \cdot D}. 
\end{equation}

To achieve this, we consider an integral complex projective surface $Y$ whose field of rational functions $\C(Y)$, as a $\C$-algebra, is isomorphic to $\alpha^\ast \C(X)(\phi)$. We may assume that the composite map:
$$\C(X) \stackrel{\alpha^\ast}{\lra} \alpha^\ast \C(X) (\phi) \lrasim \C(Y)$$
is induced by a morphism $f: Y \ra X$ of complex algebraic varieties. Moreover the map:
$$\C(Y) \lrasim \alpha^\ast \C(X) (\phi) \,\hlra \, \cM(\cV)$$
is induced by some meromorphic map:
$$\beta: \cV \dashrightarrow Y.$$
The maps $\alpha,$ $\beta,$ and $f$ fit into a commutative diagram \eqref{diagram:setupprimefibered}, and \eqref{cVXphi}
follows from the conclusion \eqref{degfgeomgen} of Theorem \ref{degmerofibered}, since:
\begin{equation*} \deg f = [ \alpha^\ast \C(X)(\phi) : \alpha^\ast \C(X)].
\qedhere
\end{equation*}
\end{proof}

For later reference, we state the following straightforward consequence of Theorem \ref{cvXfin}:

\begin{corollary}\label{immersionuniv} Consider a commutative diagram:
\begin{equation*}\label{diagram:complementbis}
\begin{gathered}
\xymatrix{
\cV\ar[r]^{\alpha}\ar[d]^{\pi_\cV} & U\ar[d]^{\pi_U}
\\
C \ar[r]^{\mathrm{Id}_C} &C,
}
\end{gathered}
\end{equation*}
where $U$ is a smooth connected complex surface, $\alpha$ is a complex analytic map, and $\pi_U$ is a morphism of complex algebraic varieties.

If $\alpha$ is an open immersion,\footnote{or equivalently, an injective complex analytic local diffeomorphism.} then the pull-back by $\alpha$ induces an isomorphism:
$$\alpha^\ast : \C(U) \lrasim \cM(\cV).$$
\end{corollary}

\subsection{The universal meromorphic map $\phi: \cV \dashrightarrow \cV^{\mathrm{alg}}$}\label{universalcV} 

In this subsection, as in \ref{Nonconstfibered}, we use the notation of 2.2.3-4. Moreover we suppose that the field $\cM(\cV)$ is not reduced to $\C(C)$.

\subsubsection{} Under the above assumption, according to Theorem  \ref{McV}, the field $\cM(\cV)$ is an extension of finite type and of transcendence degree one of $\C(C)$. Moreover the field $\C(C)$ is easily seen to be algebraically closed in $\cM(\cV)$, and the restriction of  meromorphic functions on $\cV$ to the section $\epsilon_\cV$ defines a place:
$$\epsilon_\cV^\ast : \cM(\cV) \dashrightarrow \C(C)$$
of the field $\cM(\cV)$ with values in $\C(C)$ that restricts to the identity on $\C(C)$.

Consequently there exists a smooth projective geometrically irreducible curve $\cC$ over $\C(C)$, a $\C(C)$-rational point $P$ of $\cC$, and an isomorphism of field extensions of $\C(C)$:
$$i:  \kappa(\cC) = \C(C)(\cC) \lrasim \cM(\cV)$$
such that the place $\epsilon_\cV^\ast$ of $\cM(\cV)$ coincides, via the isomorphism $i$, with the place of  $\kappa(\cC)$ defined by the rational point $P$; in other words:
$$\epsilon^\ast \circ i = P^\ast.$$
Moreover the triple $(\cC, P, i)$ is uniquely determined,  up to a unique isomorphism, by the above conditions. 

We may choose a smooth projective model over $C$ of the curve $\cC$ over $\C(C)$ --- namely a smooth connected projective complex surface $S$, equipped with a surjective morphism of complex varieties:
$$\pi_S : S \lra C,$$
such that the generic fiber of $\pi_X$ is isomorphic to $\cC$. Then the $\C(C)$-rational point $P$ of $\cC \simeq S_{\C(C)}$ defines a section of $\pi_S$:
$$\epsilon_S: C \lra X.$$
Moreover the isomorphism of field extensions of $\C(C)$:
$$i: \C(S) \simeq \kappa(\cC) \lrasim \cM(\cV)$$
is induced by a meromorphic map
$\phi: \cV \dashrightarrow S$ which fits into a commutative diagram:
\begin{equation*}
\begin{gathered}
\xymatrix{
\cV\ar@{-->}[r]^{\phi}\ar[d]^{\pi_\cV} & S\ar[d]^{\pi_S}
\\
C \ar[r]^{\mathrm{Id}_C} &C,
}
\end{gathered}
\end{equation*}
The map $\phi$ is non-constant on the fibers of $\pi_\cV$, and satisfies:
$$\epsilon_S = \phi \circ \epsilon_\cV.$$

The family $(S, \pi_S, \epsilon_S, \phi)$ so constructed is unique up to a birational isomorphism of $S$ fibered over $C$.

We shall denote by:
$$\cV^{\mathrm{alg}} := S$$
for this projective variety over $C$ --- which is well defined up to fibered birational isomorphism
-- and the meromorphic map:
$$\phi: \cV \dashrightarrow \cV^{\mathrm{alg}}$$
will be called the \emph{universal meromorphic map} from $\cV$ to a projective complex variety fibered over~$C$. It is readily seen to satisfy the universal property indicated by this terminology.

\subsubsection{} Let us consider a commutative diagram:
\begin{equation*}\label{diagram:complementter}
\begin{gathered}
\xymatrix{
\cV\ar[r]^{\alpha}\ar[d]^{\pi_\cV} & X\ar[d]^{\pi_X}
\\
C \ar[r]^{\mathrm{Id}_C} &C,
}
\end{gathered}
\end{equation*}
where $X$ is a smooth connected projective complex surface, $\alpha$ is a complex analytic map, and $\pi_X$ is a morphism of complex algebraic varieties.

According to Corollary \ref{immersionuniv}, if $\alpha$ is an open immersion, we may choose $X$ for $\cV^{\mathrm{alg}}$, and the map:
$$\alpha: \cV \lra X =: \cV^{\mathrm{alg}}$$
is the universal meromorphic map. However we will show in Subsection \ref{univram}  that in general the universal meromorphic map $\phi: \cV \dashrightarrow \cV^{\mathrm{alg}}$ may be ramified along the section $\epsilon_\cV$.

\subsubsection{}
Once a model $\cV^{\mathrm{alg}}$ has been chosen, we may introduce a modification:
$$\nu: \cV' \lra \cV$$
that is adapted to the meromorphic map $\phi: \cV \dashrightarrow \cV^{\mathrm{alg}}$, and denote by:
$$\phi': = \phi \circ \nu : \cV' \lra \cV^{\mathrm{alg}}$$
the associated complex analytic map. Then we may consider the following $\Q$-divisor in $\cV^{\mathrm{alg}}$:
$$\Gamma_{\phi\ast} D := \phi'_\ast \nu^\ast D.$$

It is an effective $\Q$-divisor which may be written:
$$\Gamma_{\phi \ast} D = C_{\cV^{\mathrm{alg}}}+ W,$$
where $C_{\cV^{\mathrm{alg}}}$ is the image of the canonical section $\epsilon_{\cV^{\mathrm{alg}}}$ of $\pi_{\cV^{\mathrm{alg}}}$, and where $W$ is a vertical $\Q$-divisor in ${\cV^{\mathrm{alg}}}$. Using that $D$ and therefore $\nu'^\ast D$ satisfy the condition $\mathbf{CNB}$ and the projection formula, one easily sees that the $\Q$-divisor $\Gamma_{\phi \ast} D$ is actually big and nef.

This implies that the set of closed integral  curves in $\cV^{\mathrm{alg}}$ disjoint from its support $\vert \Gamma_{\phi \ast} D \vert$ is finite. We shall denote by $E$ the union of these curves; it is the largest closed reduced subscheme of $\cV^{\mathrm{alg}}$ of pure dimension one that is disjoint from $\vert \Gamma_{\phi \ast} D \vert$. 

\subsubsection{} The previous discussion
admits a straightforward variant where the complex manifold $\cV$ (resp. the field $\cM(\cV)$) is replaced by the germ of analytic manifold $\cV^{\mathrm{an}}_{C \cup F}$ of $\cV$ along $C \cup F$ (resp. by the field $\cM(\cV^{\mathrm{an}}_{C \cup F})$ of germs of meromorphic functions along $C \cup F$).

In this way, we define a universal (germ of) meromorphic map:
\begin{equation}\label{unifGerm}
\phi: \cV^{\mathrm{an}}_{C \cup F} \dashrightarrow [\cV^{\mathrm{an}}_{C \cup F}]^{\mathrm{alg}}
\end{equation}
from $\cV^{\mathrm{an}}_{C \cup F}$ to a smooth projective complex surface fibered over $C$, which induces an isomorphism of field extensions of $\C(C)$:
$$\phi^\ast: \C([\cV^{\mathrm{an}}_{C \cup F}]^{\mathrm{an}}) \lrasim \cM(\cV^{\mathrm{alg}}_{C \cup F}).$$

The existence of this universal meromorphic map and the fact that the divisor $\Gamma_{\phi \ast} D$ in $[\cV^{\mathrm{an}}_{C \cup F}]^{\mathrm{alg}}$ is big and nef allow one to establish some finiteness result concerning the subalgebra $\cA$ of  $\cM(\cV^{\mathrm{an}}_{C \cup F})$ 
consisting in the germs of meromorphic function on $\cV^{\mathrm{an}}_{C \cup F}$ that are holomorphic on $\cV^{\mathrm{an}}_{C \cup F} \setminus \pi_{\cV}^{-1}(\Sigma).$ 
The analogue of this result in the arithmetic setting will be the content of Theorem \ref{finOVfa}.

From now on, for simplicity we shall denote:
$$\Vfa := \cV^{\mathrm{an}}_{C \cup F}.$$
Moreover we shall  denote by $E$ the largest closed reduced subscheme of $\Vfa^{\mathrm{alg}}$ of pure dimension one that is disjoint from $\vert \Gamma_{\phi \ast} D \vert$. 

Recall that the complement of the non-empty finite set $\Sigma$ in the projective curve $C$ defines an affine complex curve: $$\mathring{C} := C \setminus \Sigma.$$ We shall denote its ring of regular function by:
$$\C[\mathring{C}] := \Gamma (\mathring{C}, \cO^{\rm{alg}}_{\mathring{C}}) = \varinjlim_n \Gamma(C, \cO_C(n \Sigma)).$$
Clearly $\cA$ is a $\C[\mathring{C}]$-algebra.

Let us denote by $E$ the (finite) union of the closed integral curves in ${\Vfa^{\mathrm{alg}}}$ disjoint from 
$\vert \Gamma_{\phi \ast} D \vert.$
The following proposition is established by  analyzing when the divisor of the pull-back $\phi^\ast f$ in $\cM(\Vfa)$ of a rational function $f \in \C(\Vfa^{\mathrm{alg}})$ is contained in $\pi^{-1}(\Sigma)$. The details of its proof will be left to the reader.

\begin{proposition} The isomorphism of field extensions of $\C(C)$:
$$\phi^\ast : \C({\Vfa^{\mathrm{alg}}}) \lrasim \cM(\Vfa)$$
restricts to an isomorphism of $\C[\mathring{C}]$-algebras:
$$\phi^\ast: \Gamma\big({\Vfa^{\mathrm{alg}}}\setminus{(E \cup \pi_{\Vfa^{\mathrm{alg}}}}^{-1}(\Sigma)\big), \cO_{\Vfa^{\mathrm{alg}}} ) \lrasim \cA.$$
\end{proposition}

Recall that Zariski \cite{Zariski54} has shown that, for any smooth projective complex surface $X$ and any effective divisor $C$ in $X$, the $\C$-algebra:
$$\Gamma (X \setminus \vert C \vert, \cO_X) \simeq \varinjlim_n \Gamma(X, \cO_X(nC))$$
is finitely generated.\footnote{The paper \cite{Zariski54} is reproduced in \cite[p. 261-274]{Zariski73}. The proof of this finiteness generation result, incomplete in the original publication \cite{Zariski54}, is completed in \cite[p. 275]{Zariski73}.}

Applied to $X := \Vfa^{\mathrm{alg}}$ and $C := E \cup \pi_{\Vfa^{\mathrm{alg}}}^{-1}(\Sigma)$, this implies:

\begin{corollary} $\cA$ is a $\C[\mathring{C}]$-algebra of finite type.
\end{corollary}

\subsection{The universal meromorphic map $\phi: \cV \dashrightarrow \cV^{\mathrm{alg}}$ may be ramified}\label{univram} In this subsection, we construct some examples that show that the  ``universal meromorphic map"
$$\phi_\cV : \cV \dashrightarrow \cV^{\mathrm{alg}}$$
introduced in Subsection \ref{universalcV} may be ramified along the section $C_\cV$ of $\pi_\cV$.

\subsubsection{Algebraic data} As in \ref{deletinginfnear}, let us consider a connected smooth projective complex curve $C$, a connected smooth projective complex surface $X$, a surjective (necessarily flat) morphism of complex (algebraic or analytic) varieties:
$$\pi_X: X \lra C,$$
an (algebraic or analytic) section of $\pi_X$:
$$\epsilon : C \lra X,$$
and its image:
$$C_X := \epsilon(C).$$
 
 Let moreover $\Sigma$ be a non-empty finite subset of $C$, and for every $x \in \Sigma$, let $F_x$ be a compact connected reduced curve $F_x$ strictly contained in $\vert \pi_X^{-1}(x)\vert$ such that:
 $$F_x \cap C_X \neq \emptyset,$$ and let:
 $$F := \bigcup_{x \in \Sigma} F_x.$$ 
 
 To these data, we may associate an equilibrium $\Q$-divisor $D$ supported by $C_X \cup F$ by the construction of Subsection \ref{cVfibered}\footnote{applied to the complex surface $\cV$ defined as an open neighborhood in $X$ of $C_X \cup F$ that does not contain any of the fibers $\pi^{-1}(x)$ for $x\in \Sigma.$}, namely the unique $\Q$-divisor of the form:
 $$D = C_X + V$$
 with $V$  supported on $F$, such that, for every component $W$ of $F$, we have:
 \begin{equation}\label{DW}
 D \cdot W =0.
 \end{equation}
 
 We shall make the following further assumptions:\begin{itemize}
\item the $\Q$-divisor $D$ has integral coefficients; in other words, $D$ is a divisor in $X$; 
\item the divisor $D$ satisfies $\mathbf{CNB}$; equivalently, the intersection number $D\cdot D = C_X \cdot D$ is positive. 
\end{itemize}

The construction in \ref{deletinginfnear} allows one to construct instances of such objects from any projective surface fibered over $C$ endowed with a section.

\subsubsection{Constructing analytic surfaces fibered over $C$}\label{ConstrAn}

 We consider the following additional data:
 \begin{itemize}
\item 
 a positive integer $e$; 
 \item a connected open neighborhood $\cV$, in the analytic topology, of $C_X \cup F$ in $X$, such that $\cV$ contains no fiber $\pi^{-1}(x)$, $x \in C$; 
 \item a smooth effective analytic divisor\footnote{or equivalently, a closed complex submanifold of dimension $1$ in $\cV$.} $R$ in $\cV$, meeting $C_X$ transversally,  such that the following two conditions are satisfied:
\begin{equation}\label{RF}
R \subset X \setminus F,
\end{equation}
and:
\begin{equation}\label{mode}
C_X\cdot R \equiv - \, C_X\cdot D \quad \mod e.
\end{equation}
\end{itemize}

For every positive integer $e$, a neighborhood $\cV$ of $C_X \cup F$ and a divisor $R$ in $\cV$ satisfying the conditions above clearly exist.

We shall denote by:
$$\iota: \cV \lra X$$
the inclusion morphism.

Consider the complex analytic line bundle $\cO_\cV(R + C_X +V)$ over $\cV$. According to \eqref{DW}, \eqref{RF} and \eqref{mode}, the degree of its restriction to every component of $C_X \cup F$ is divisible by $e$. Consequently, after possibly replacing $\cV$ by a smaller neighborhood of $C \cup F$ in $X$, we may assume that there exists a complex  analytic line bundle $\Laa$ over $\cV$ and an isomorphism of complex analytic line bundles over~$\cV$:
\begin{equation}
s: \cO_\cV(R + C_X +V) \lrasim \Laa^{\otimes e}.
\end{equation}

Using $\Laa$ and $s$, we may construct the cyclic cover:
$$c: \cW \lra \cV$$
obtained by taking the $e$-th root out of $s$;  see for instance \cite[Sections 3.3-14]{EsnaultViehweg92}. 
By construction,  $\cW$ is a normal complex analytic surface, equipped with an action of the cyclic group $G := \Z/e\Z,$ and $c$ defines an isomorphism of complex analytic spaces:
\begin{equation}\label{WGV}
\cW/ G \lrasim \cV.
\end{equation}
We shall denote by $\sigma: \cW \lrasim \cW$ the automorphism of $\cW$ defined by the action of the generator $[1]$ of~$\Z/e \Z$.

The cover $c$ is totally ramified over $C_X$, and therefore establishes an isomorphism:
$$c_{\mid C_\cW}: C_\cW \lrasim C_X$$
between $C_X$ and the image $C_\cW$ of some complex analytic section $\epsilon_\cW$ of the complex analytic morphism:
$$\pi_\cW := \pi_X \circ \iota \circ c: \cW \lra C.$$

We may finally  consider a resolution\footnote{The normal analytic surface $\cW$ has a finite set of singular points, included in $c^{-1}((R \cup C \cup F)_{\mathrm{sing}})$, where $(R \cup C \cup F)_{\mathrm{sing}}$ denotes the finite set of singular points of $R \cup C \cup F$.}:
$$\nu: \widetilde{\cW} \lra \cW$$
of $\cW$ and the morphism:
$$\pi_{\widetilde{\cW}}: = \pi_W \circ r : \widetilde{\cW} \lra C.$$
The proper transform $C_{\widetilde{W}}$ of $C_W$ by $\nu$ is the image of an analytic section of  $\pi_{\widetilde{\cW}}:$
$$\epsilon_{\widetilde{W}}: C \lra \widetilde{W}.$$

Let $\Sigma'$ be the finite subset of $C \setminus \Sigma$ such that:
$$R \cap C_X = \epsilon_X (\Sigma'),$$
and let:
$$\widetilde{\Sigma} := \Sigma \sqcup{\Sigma'}.$$
For every $x \in \widetilde{\Sigma},$ we define:
\begin{align*}
\widetilde{F}_x &:= (c\circ \nu)^{-1} (F_x)  \quad \quad \, \mbox{if $x \in \Sigma$} \\
                        &:= \nu^{-1} (\epsilon_X(x)) \quad  \quad  \;\;\; \mbox{if $x \in \Sigma'$.}
\end{align*} 

Then, for every $x \in \widetilde{\Sigma},$ $\widetilde{F}_x$ is a compact connected curve in $\pi_{\cW}^{-1}(x)$, which itself is non-compact, and contains $\epsilon_{\widetilde{\cW}}(x)$. Moreover the $\Q$-divisor in  $\widetilde{\cW}$:
$$\widetilde{D} := \frac{1}{e'} \, (c \circ \nu)^\ast D$$
is easily seen to be the equilibrium divisor associated to the the section $\epsilon_{\widetilde{\cW}}$ and to the family $(\widetilde{F}_x)_{x \in  \widetilde{\Sigma}},$
and to satisfy $\mathbf{CNB}$.

\subsubsection{} We have constructed the following morphisms of complex analytic or algebraic varieties fibered over $C$:
\begin{equation}
\widetilde{\cW} \stackrel{r}{\lra} \cW 
\stackrel{c}{\lra} \cV
\stackrel{\iota}{\lra} X,
\end{equation} 
and we may consider the induced maps between fields of  meromorphic functions:
\begin{equation}\C(X) \stackrel{\iota^\ast}{\lrasim}
\cM(\cV) 
\stackrel{c^\ast}{\lra}
\cM(\cW)
\stackrel{\nu^\ast}{\lrasim}
\cM(\widetilde{\cW}).
\end{equation}
The map $\iota^\ast$ is indeed an isomorphism as a consequence of Corollary \ref{immersionuniv} applied to $\alpha := \iota.$


The main result of this subsection is the following proposition. 

\begin{proposition} If the analytic divisor $R$ is Zariski dense in $X$, then the cover $c$ induces an isomorphism of fields of  meromorphic functions:
\begin{equation}
c^\ast: \cM(\cV) \lrasim \cM(\cW).
\end{equation}
\end{proposition}

\begin{proof} The action of $G$ on the complex analytic surface $\cW$ induces an action of $G$ on the field $\cM(\cW)$. Moreover, according to the isomorphism \eqref{WGV}, its fixed field coincides with the image of $\cM(\cV)$ by $c^\ast$:
\begin{equation*}
\cM(\cW)^G = c^\ast \cM(\cV).
\end{equation*}

This shows that $\cM(\cW)$ is a cyclic extension of $c^\ast \cM(\cV)$ of degree $e'$ dividing $e$, of Galois group:
$$\big\{ \sigma^{\ast k}; k \in \Z/e' \Z \big\}.$$
In particular, there exists $\phi \in \cM(\cW)$ such that:
$$\cM(\cW) = c^\ast \cM(\cV) [\phi]$$
and:
\begin{equation}\label{sigmaphi}
\sigma^\ast \phi = e^{2 \pi i/e'} \phi.
\end{equation}
 
 If $e' >1,$ then $\phi \neq 0,$ and $\phi^{e'}$ belongs to $c^\ast \cM(\cV) = c^\ast \iota^\ast \C(X)$, and may therefore be written:
 $$\phi^{e'} = c^\ast \iota^\ast \psi$$
 for some $\psi$ in $\C(X)^\times$. In particular, the following equality of analytic Weil divisors in $\cW$ holds:
 \begin{equation}\label{eprimediv}
 e' \div \phi = c^\ast \iota^\ast \div \psi.
 \end{equation}
The map $c$ is totally ramified over every component of $R$. Consequently the points in $c^{-1}(R)$ are fixed under the action of $G$ on $\cW$. Together with \eqref{sigmaphi}, this implies that the support $\vert \div \phi \vert$ of $\div \phi$ contains $c^{-1}(R)$. Using \eqref{eprimediv}, we conclude that the support $\vert \div \psi \vert$ of $\div \psi$ contains $R$. 

This contradicts the Zariski density of $R$ in $X$, and establishes the equality $e' =1$.
\end{proof}

It is clear that we may find $\cV$ and $R$ as in \ref{ConstrAn} such that $R$ is Zariski dense in $X$. It is enough to construct $\cV$ and $R$ satisfying \eqref{RF} and \eqref{mode} such that $R \cap C_X$ is non-empty, and such that, for some $x \in R \cap C_X,$ the germ of $R$ at $x$ is not the branch of an algebraic curve through $x$. Then $R \cap \cV'$ would actually be Zarishi dense in $X$ for every open analytic neighborhood $\cV'$ of $C_X \cup F$ in~$\cV$.

\begin{corollary}  If the analytic divisor $R$ is Zariski dense in $X$, then the map:
\begin{equation}\label{cWX}
\iota \circ c \circ r: \widetilde{\cW} \lra X
\end{equation}
is ``the" universal meromorphic map from $\widetilde{\cW}$ to a projective complex variety fibered over $C$.
\end{corollary}

Observe that the universal meromorphic map \eqref{cWX} is ramified, with ramification index $e$, along~$\epsilon_{\widetilde{\cW}}.$

\part{$\CbD$ Green functions on Riemann surfaces  and  intersection theory on quasi-projective arithmetic surfaces}

\chapter[Green functions  with  $\mathcal C^{\infty}$  and $L^2_1$ regularity and arithmetic intersection]{Green functions  with  $\mathcal C^{\infty}$  and $L^2_1$ regularity and arithmetic intersection numbers}\label{chapterL21}

Part 2 of this memoir is devoted to various results concerning Arakelov intersection pairings on arithmetic surfaces and Green functions on Riemann surfaces. These results will be used in Part 3 to transpose in an arithmetic setting the results in complex algebraic, analytic, and formal geometry presented in Part 1. 

This chapter and the next one are devoted to some extensions of  Arakelov intersection theory on arithmetic surfaces adapted to our needs. In this chapter, we 
develop the version of  Arakelov intersection theory in \cite{Bost99}, that relies on the use of $ L^2_1$ Green functions, in a possibly non-projective situation. Moreover we show that it satisfies some nice functoriality properties. These tools allow us to establish 
in Theorem  \ref{theorem:main-Arakelov} an arithmetic version of the degree bound \emph{Ã  la} Nori  in Proposition \ref{proposition:main-geom}
that has been our geometric starting point in Chapter 1. 

The necessity of adapting Arakelov intersection theory on arithmetic surfaces to a possible non-projective setting has led us to give a self-contained presentation of the theory, and this chapter is in principle accessible with no previous familiarity with Arakelov geometry. 

\section{The Arakelov degree of $0$-cycles and the height of $1$-cycles}

\subsection{The Arakelov degree of $0$-cycles}

Let $X$ be an arithmetic scheme, that is,  a separated scheme of finite type over $\Z$. A \emph{$0$-cycle} on $X$ is a finite formal sum:
\begin{equation*}
Z = \sum_{i \in I} n_i P_i
\end{equation*}
where the $P_i$ are closed points of $X$ and the $n_i$ are integers. The residue fields:
$$\kappa (P_i) = \cO_{X,P_i}/\mathfrak{m}_{X, P_i}$$
are finite, and the \emph{arithmetic} or \emph{Arakelov degree} of $Z$ is defined as the real number:
\begin{equation*}
\dega Z := \sum_{i \in I} n_i \log \vert \kappa(P_i)\vert.
\end{equation*}

If $Y$ is another arithmetic scheme and $f:X \ra Y$  a morphism of  schemes, then the images $f(P_i)$ are closed points of $Y$, and the direct image $f_\ast Z$ is the $0$-cycle defined by the equality:
\begin{equation*}
f_\ast Z := \sum_{i \in I} [\kappa(P_i): \kappa(f(P_i))] \, f(P_i),
\end{equation*}
Then we have, as a straightforward consequence of the definition of the Arakelov degree:
\begin{equation}\label{compdegadir}
\dega f_\ast Z = \dega Z.
\end{equation}

When $X$ is $\Spec \Z,$ every $0$-cycle on $X$ is of the form:
\begin{equation*}
\div q := \sum_{p} v_p (q) \, [p]
\end{equation*}
for some $q \in \Q^\ast$ --- where the sum runs over the prime numbers $p$ and where $v_p$ denotes the $p$-adic valuation and $[p]$ the closed point $p\Z$ of $\Spec \Z$ --- and we have:
\begin{equation}\label{degaQ}
\dega \div q = \sum_p v_p(q) \log p = \log \vert q \vert.
\end{equation}

The Arakelov degree of $0$-cycles is clearly uniquely determined by its compatibility with direct images \eqref{compdegadir} and the relation \eqref{degaQ}. Observe also that, if $X$ is a $\F_p$-scheme for some prime $p$, then the $0$-cycle $Z$ on $X$ considered as a $\F_p$-scheme admits a ``geometric degree", defined as in \cite[Definition 1.4]{FultonIT}:
$$\deg_{\F_p} Z := \sum_{i \in I} n_i [\kappa(P_i): \F_p] \in \Z,$$
which is related to its arithmetic degree by the equality:
\begin{equation}\label{degaFp}
\dega Z = \deg_{\F_p} Z  . \log p.
\end{equation}

\subsection{The Arakelov degree of $0$-cycles rationally equivalent to zero on proper one-dimensional schemes}\label{degarat}

In this subsection, we denote by $C$ be a scheme \emph{proper over $\Spec \Z$, integral and of dimension one}. 

The (support of the) image the morphism $f: C \ra \Spec \Z$ is a closed integral subscheme of $\Spec \Z$, and therefore either $\Spec \Z$, or a closed point $[p]$ of $\Spec \Z$. Accordingly the following alternative holds: either $C$ is finite and flat over $\Spec \Z$, or $C$ is a proper $\F_p$-scheme for a uniquely determined prime~$p$.

\subsubsection{}  When $C$ is finite and flat over $\Spec \Z$, then $C$ is affine, the ring $\Gamma(C, \cO_C)$ is an order in some number field $K$, and if $\cO_K$ denotes the ring of integers of $K$, then $\Spec \cO_K$ may be identified with the normalization of $C$.

In this case, the relation \eqref{degaQ} admits the following generalization. For every $r \in \kappa(C)^\times = K^\times,$ the Arakelov degree of the $0$-cycle on $C$ defined by the divisor\footnote{It is defined, as in \cite[1.2-3]{FultonIT}, by the equality: $\div r := \sum_{x \in C_0} \mathrm{ord}_x r  . x,$ where $C_0$ denotes the set of closed points of $C$.} $\div r$ satisfies:
 \begin{equation}\label{degaK}
\dega \div r = \sum_{\sigma \in C(\C)} \log \vert \sigma(r)\vert,
\end{equation}
where $C(\C)$ denotes the set of complex points of the scheme $C$; this set may be identified with the set, of cardinality $[K:\Q]$, of the field embeddings $\sigma: K \ra \C$. Indeed if $q := N_{K/\Q} r$ denotes the image of $r$ by the norm map from $K$ to $\Q$, and if $f$ denotes the morphism of schemes from  $C$ to $\Spec \Z$, then we have:
$$f_\ast \div r = \div q$$
and:
$$\vert q \vert = \prod_{\sigma \in C(\C)} \vert \sigma (r) \vert,$$
and therefore \eqref{degaK} follows from \eqref{compdegadir} and \eqref{degaQ}.

\subsubsection{}\label{vertrat}  When $C$ is a proper $\F_p$-scheme, then the ring $\Gamma(C, \cO_C)$ is a finite extension $\F_q$ of $\F_p$ and $C$ is a geometrically integral projective curve over $\F_q$. Moreover the relation \eqref{degaFp} and the basic properties of the degree of $0$-cycles on proper schemes over a field (see \cite[1.4]{FultonIT}) imply that, for any $r \in \kappa(C)^\times$, we have:
\begin{equation}\label{degdivzero}
\dega \div r =0.
\end{equation}

\subsection{The height of $1$-cycles with  support proper over $\Spec \Z$}
Let us keep the notation of \ref{degarat} and assume that $C$ is a (necessarily closed) subscheme of some reduced, separated scheme $X$ of finite type over $\Spec \Z$. Assume moreover that $X$ is endowed with a Hermitian line bundle $\Lb := (L, \Vert. \Vert)$. 

Let us consider a non-zero rational section $s$ of the line bundle $L_{\mid C}$ over $C$, and its divisor $\div s$. 

In case (1), the set $C(\C)$ of complex points of $C$ may be identified with a subset of $X(\C)$, and accordingly the Hermitian metric $\Vert.\Vert$ on $L_\C^{\an}$ defines a Hermitian metric $\Vert.\Vert_x$ on the fiber $L_x$ for every $x \in C(\C)$. We may therefore consider the positive real numbers $(\Vert s(x) \Vert_x)_{x \in C(\C)}$, and the relation \eqref{degaK} shows that the real number:
\begin{equation}\label{htLdef}
\height_{\Lb} (C) := \dega \div s - \sum_{x \in C(\C)} \Vert s(x) \Vert_x
\end{equation}
does not depend on the choice of $s$. It defines the \emph{height} of $C$ with respect to $\Lb$.

In case (2), the relation \eqref{degdivzero} shows that $\dega \div s $ does not depend on the choice of $s$. We shall denote it by $\dega L_{\mid C}.$ According to the relation \eqref{degaFp}, we have:
\begin{equation*}
\dega L_{\mid C} = \deg_{C/\F_p} L . \log p = \deg_{C/\F_q} L . \log q, 
\end{equation*}
where $\deg_{C/\F_p} L$ (resp. $\deg_{C/\F_q} L$) denotes the usual ``geometric" degree of the line bundle $L$ over $X$ seen as a proper $\F_p$-scheme (resp. a geometrically integral projective curve over $\F_q$). 

It is convenient to use some uniform notation, covering both cases (1) and (2), and to let in case (2):
$$\height_{\Lb}(C) := \dega L_{\mid C}.$$
 We will also use the alternative notation:
$$\dega (\Lb \vert C) := \height_{\Lb}(C).$$
In the framework of \cite{BostGilletSoule94}, this real number would actually be denoted by $\dega (\hat{c}_1(\Lb) \vert C)$.

Observe that the height $\dega (\Lb \vert C)$ depends ``additively" on $\Lb$. Namely, if $\Lb_1$ and $\Lb_2$ are two Hermitian line bundles over $X$, we have:
\begin{equation}\label{degaAdd}
\dega (\Lb_1 \otimes \Lb_2 \vert C) = \dega (\Lb_1 \vert C) +\dega (\Lb_2 \vert C).
\end{equation}

The definition of $\dega (\Lb \vert C)$ extends by linearity to an arbitrary $1$-cycle $C$ in $X$ whose support $\vert C \vert$ is proper over $\Spec \Z$. Namely, any such $1$-cycle is a finite formal sum:
$$C = \sum_{i \in I} n_i C_i$$
where the $C_i$ are closed integral subschemes of $X$ of dimension one, proper over $\Spec \Z$, and the $n_i$ are integers, and  we  define:
\begin{equation}\label{heightdefcycle}
\dega (\Lb \vert C) = \sum_{i \in I} n_i \, \dega (\Lb \vert C_i).
\end{equation}

\begin{proposition} Let $f:X \lra Y$ be a morphism of reduced separated schemes of finite type over $\Spec \Z$. For every $1$-cycle $D$ over $X$ such that $\vert D \vert$ is proper over $\Spec \Z$ and every Hermitian line bundle $\Lb$ over $Y$, the following equality holds:
\begin{equation}\label{adjonction}
\dega (f^\ast \Lb \vert D) = \dega (\Lb \vert f_\ast D).
\end{equation}
\end{proposition}

In the right-hand side of \eqref{adjonction}, we denote by $f_\ast D$ the direct image of $D$ by $f$, defined as in \cite[1.4]{FultonIT}. It is a $1$-cycle on $Y$ whose support is easily seen to be proper over $\Spec \Z$.

\begin{proof} We may assume that $D$ is an integral scheme $C$ as in \ref{degarat} above. When $C$ is an horizontal curve, namely in case (1), the equality \eqref{adjonction} is a straightforward consequence of the definitions. When $C$ is vertical, namely in case (2), it follows from the projection formula in \cite[Proposition 2.5 (c)]{FultonIT} applied to morphisms between one-dimensional projective schemes over a finite field.
\end{proof}

\section{Green functions with $\mathcal C^{\infty}$ regularity and $\ast$-products on Riemann surfaces}\label{subsection:smooth}

\subsection{Green functions with $\mathcal C^{\infty}$ regularity  on a Riemann surface}\label{Greensmoothdef} Let $M$ be a Riemann surface, and let $D$ be a divisor on $M$. Recall that we denote by $\mathbbm 1_{D}$ be the canonical  section  with divisor $D$ of the analytic line bundle $\cO_M (D)$. 

\subsubsection{Definitions}\label{DefGrenCinfty} A \emph{Green function with $\mathcal C^{\infty}$ regularity} for the divisor $D$ on $M$ is a $\mathcal C^{\infty}$ function: 
$$g : M\setminus |D|\lra \R$$
such that the equality:
$$\Vert \mathbbm 1_{D}\Vert_g=e^{-g}$$
defines a metric  on the restriction of $\mathcal O_M(D)$ to $M\setminus |D|$ that extends to a $\mathcal C^{\infty}$ metric $\Vert.\Vert_g$ on $\mathcal O_M(D)$ over $M$. 

This condition may be equivalently formulated by requiring that $g$ has  logarithmic singularities at the points of $|D|$. Namely, for every open subset $U$ of $M$ and every invertible  function $f$ on $U$ such that: 
$$\mathrm{div} f=D_{|U},$$
there exists $h$ in  $\mathcal C^{\infty}(U, \R)$ such that:
\begin{equation}\label{equation:condition-green}
g=\log |f|^{-1} + h \quad \mbox{over  $U\setminus|D|.$}
\end{equation}

Observe that this implies that, like the function $\log |f|^{-1}$, the function $g$ is locally $L^{1}$ on $U$, and that the distribution on $U$ it defines satisfies the equality of currents on $U$:
$$\frac{i}{\pi} \, \partial\overline\partial g = -\delta_{D} + \frac{i}{\pi}\, \partial\overline\partial h.$$
This shows that a Green function $g$ as defined above is locally $L^{1}$ on $M$ and that, considered as a distribution on $M$, it satisfies the equality:
\begin{equation}\label{equation:ddg}
\frac{i}{\pi} \, \partial\overline\partial g +\delta_{D} = \omega(g)
\end{equation}
for some $\mathcal C^{\infty}$  real $2$-form $\omega(g)$ on $M$.

Conversely, according to the ellipticity of the operator $\partial \overline{\partial}$ acting on distributions on $M$, any real distribution $g$ on $M$ that satisfies the equality \eqref{equation:ddg} for some   $\mathcal C^{\infty}$ real $2$-form $\omega(g)$ on $M$ defines a Green function with $\mathcal C^{\infty}$ regularity for the divisor $D$ on~$M$.

Recall that the first Chern form of a  Hermitian line bundle $\Lb= (L, \Vert.\Vert)$ with $\cC^\infty$ metric over $M$ is the $2$-form over $M$ defined locally by the equality:
\begin{equation}\label{Chernformdef}
c_1(\Lb_\C) = \frac{1}{2\pi i} \partial \overline{\partial} \log \Vert s \Vert^2,
\end{equation}
where $s$ denotes a local non-vanishing complex analytic section of $L$. With this notation, if $g$ is a Green function with $\cC^\infty$ regularity for a divisor $D$ on $M$, we also have:
$$\omega(g) = c_1 \big(\cO_M(D), \Vert.\Vert_g\big).$$

\subsubsection{Capacitary metrics}\label{CapMetricDef}
Let $P$ be a point of the Riemann surface $M$, and let $g$ be a Green function for the divisor $P$ in $M$. The attached Hermitian metric $\Vert.\Vert_{g}$ on $\mathcal O_M(P)$ 
 defines, by restriction to the fiber 
$$\mathcal O(P)_{|P}\simeq T_{P}M,$$
a Hermitian metric on the one-dimensional complex vector space $T_{P}M$. By definition, this metric is the \emph{capacitary metric} $\Vert. \Vert_{g}^{\mathrm{cap}}.$

If $z$ is a local holomorphic coordinate on some neighborhood $U$ of $P$ in $M$, we have the following equality of distributions on $U$:
$$g=\log|z-z(P)|^{-1}+f$$
for some function $f$ in $\mathcal C^{\infty}(U)$, and the definition of the capacitary metric $\Vert. \Vert_{g}^{\mathrm{cap}}$ may be expressed by the following formula:
$$\Vert (\partial/\partial z)_{|P} \Vert_{g}^{\mathrm{cap}}=e^{-f(P)}.$$

\subsection{The $\ast$-product of two Green functions with $\mathcal C^{\infty}$ regularity}

\subsubsection{}\label{SPbasic} Let  $D_{1}$ and $D_{2}$ be two divisors on a Riemann surface $M$ such that:
$$|D_{1}|\cap |D_{2}|=\emptyset,$$
and let $g_{1}$ and $g_{2}$ be  Green functions with $\mathcal C^{\infty}$ regularity for $D_{1}$ and $D_{2}$ respectively. 

The\emph{ $\ast$-product} of $g_{1}$ and $g_{2}$  is defined as the real current of degree 2 on $M$:
$$g_{1}\ast g_{2}:=g_{2}\, \delta_{D_{1}}+g_{1}\, \omega(g_{2}),$$
where, as above, $\omega (g_2)$ is defined by:
 $$\omega(g_{2})=\frac{i}{\pi} \partial\overline\partial g_{2} +\delta_{D_{2}}.$$

When additionally the intersection $\mathrm{supp} g_{1} \cap \mathrm{supp} g_{2}$
of the supports of $g_{1}$ and $g_{2}$ is compact, the $\ast$-product $g_{1}\ast g_{2}$ is supported on a compact subset of $M$, and we may consider the integral:
$$\int_{M}g_{1}\ast g_{2} \in \R.$$

It is actually symmetric in   $g_{1}$ and $g_{2}$; namely:
\begin{equation}\label{starprodsym}
\int_{M}g_{1}\ast g_{2}=\int_{M}g_{2}\ast g_{1}.
\end{equation}
This equality is a direct consequence of Green's formula --- itself a consequence of Stokes formula --- which asserts that, for any two distributions $\phi$ and $\psi$ on $M$ with disjoint singular supports such that  $\mathrm{supp}\, \phi \cap \mathrm{supp}\, \psi$ is compact, the following equalities 
holds:
\begin{equation}\label{GreenFormula}
\int_{M}\phi\, i\partial\overline\partial\psi=-i\int_{M}\partial\phi\wedge\overline\partial\psi=-i\int_{M}\partial\psi\wedge\overline\partial\phi=\int_{M}\psi\,i\partial\overline\partial\phi.
\end{equation}

The integral of the $\ast$-product also satisfies the following two properties as a straightforward consequence of its definition and of Stokes formula:

$\mathbf{SP}_1$: \emph{If the supports of $g_{1}$ and $g_{2}$ are disjoint, then:} $$\int_{M}g_{1}\ast g_{2}=0 ;$$

$\mathbf{SP}_2$:  \emph{For any two functions $f_{1}$ and $f_{2}$ in $\cC^\infty(M, \R)$ such that the intersections $\mathrm{supp}\, f_{1}\cap \mathrm{supp}\, g_{2},$ $\mathrm{supp}\, f_{2}\cap \mathrm{supp}\, g_{1}$ and $\mathrm{supp}\, f_{1}\cap \mathrm{supp}\,  f_{2}$ are compact, the following equality holds}:
\begin{equation}\label{equation:star}
\int_{M}(g_{1}+f_{1})\ast (g_{2}+f_{2})=\int_{M}g_{1}\ast g_{2}+\int_{M}f_{1}\, \omega(g_{2})+\int_{M}f_{2}\, \omega(g_{1})-\frac{i}{\pi}\int_{M}\partial f_{1}\wedge\overline\partial f_{2}.
\end{equation}

For any pair of divisors $D_1$ and $D_2$ on $M$ with disjoint supports,  the function that attaches the integral $\int_{M}g_{1}\ast g_{2}$ to a pair $(g_{1}, g_{2})$ of Green functions for  the divisors $D_{1}$ and $D_{2}$ such that $\mathrm{supp}\, g_{1} \cap \mathrm{supp}\, g_{2} $ is compact is easily seen to be characterized by the  properties $\mathbf{SP}_1$ and~$\mathbf{SP}_2$.

\subsubsection{} The last term in the right-hand side of \eqref{equation:star} is, up to a factor $\pi$, the opposite of the \emph{Dirichlet scalar product of $f_1$ and $f_2$}:
$$\langle f_{1}, f_{2}\rangle_{\mathrm{Dir}}:=i\int_{M}\partial f_{1}\wedge\overline\partial f_{2},$$
which is defined for any pair of functions $(f_1, f_2)$ in $\cC^\infty(M, \R)$ such that $\mathrm{supp} f_1 \cap \mathrm{supp} f_2$ is compact.

Actually, for every such pair of functions, we have:
$$i\int_{M}\partial f_{1}\wedge\overline\partial f_{2}=i\int_{M}\Big(\partial f_{1}\wedge\overline\partial f_{2}-\frac{1}{2}df_{1}\wedge df_{2}\Big)=\frac{i}{2}\int_{M}(\partial f_{1}\wedge\overline\partial f_{2}+\partial f_{2}\wedge\overline\partial f_{1}),$$
and, if $z=x+iy$ is a local holomorphic coordinate on $M$, we may write locally:
$$i\Big(\partial f_{1}\wedge\overline\partial f_{2}-\frac{1}{2}df_{1}\wedge df_{2}\Big)=\frac{i}{2}\left(\partial f_{1}\wedge\overline\partial f_{2}+\partial f_{2}\wedge\overline\partial f_{1}\right)=\frac{1}{2}\left(\frac{\partial f_{1}}{\partial x}\frac{\partial f_{2}}{\partial x}+\frac{\partial f_{1}}{\partial y}\frac{\partial f_{2}}{\partial y}\right)dx\wedge dy.$$

The bilinear form $\langle.,.\rangle_{\mathrm{Dir}}$ on $\mathcal C_c^{\infty}(M, \R)$ is clearly symmetric and nonnegative. When $M$ is compact and connected, it induces a structure  of real pre-Hilbert space on the quotient $\cC^{\infty}(M, \R)/\R$.

\subsection{Functoriality of Green functions with $\cC^\infty$ regularity on Riemann surfaces}\label{functCinftyGreen}

The functoriality properties of Green functions with $L^2_1$ and  $\mathcal C^{\bD}$ regularity will play an important role in the next chapters. For comparison, in this subsection we present a short description of the functoriality properties of Green functions with $\cC^\infty$ regularity.\footnote{We refer the reader to \ref{FunctCurrents}  and \ref{FunctbisMeasRS} below for some background concerning the pull-back and push-forward operations on forms and currents by analytic maps between Riemann surfaces.}

\begin{proposition}\label{functCinfty} Let $M$ and $N$ be two Riemann surfaces, with $N$ connected, and let $f: N \ra M$ be  a non-constant complex analytic map.

1) If $D$ is divisor on $M$ and $g$ a Green function with  $\cC^\infty$ regularity for $D$ on $M$, then the $\cC^\infty$ function $f^\ast g $ on $N \setminus f^{-1}(\vert D \vert) = N \setminus \vert f^\ast D \vert$ is a Green function with $\cC^\infty$ regularity for the divisor $f^\ast D$ on~$N$. 

Moreover the following equality of $\cC^\infty$ 2-forms on $N$ is satisfied:
\begin{equation*}
\omega(f^\ast g) = f^\ast \omega(g).
\end{equation*}

2) If $E$ is a divisor on $N$ and $h$ a Green function with $\cC^\infty$ regularity for $E$ on $N,$ and if $f_{\mid \supp h}$ is proper\footnote{This implies that $f_{\mid \vert E \vert}$ is proper, and therefore that the divisor $f_\ast E$ on $M$ is well-defined.} and if $\omega(h)$ vanishes on some neighborhood of the ramification locus $\mathrm{Ram}(f)$ of $f$,  then the distribution $f_\ast h$ on $M$ is a Green function with $\cC^\infty$ regularity for $f_\ast E$ on $M$. 

Moreover the following equality of $\cC^\infty$ 2-forms on $M$ is satisfied:
\begin{equation*}
\omega( f_\ast h) = f_\ast \omega(h).
\end{equation*}

3) Let $(D,g)$ and $(E,h)$ be as in 1) and 2). If the following additional conditions are satisfied:
\begin{equation}\label{DE2bis}
f^{-1}(\vert D \vert) \cap \vert E\vert   = \emptyset
\end{equation}
and:
\begin{equation}\label{DE1bis}
\mbox{$\supp g \cap f(\supp h)$ is compact},  
\end{equation}
then the following equality holds:
\begin{equation}\label{fghbis}
\int_N (f^\ast g) \ast h = \int_M g \ast (f_\ast h).
\end{equation}
\end{proposition}

Observe that, in 3), $f^\ast g$ and $h$ are Green functions for the divisor $f^\ast D$ and $E$, which have disjoint support according to \eqref{DE2bis}; moreover, according to \eqref{DE1bis} and to the properness of  $f_{\mid \supp h}$, the intersection:
 $$\supp (f^\ast g) \cap \supp h = f^{-1} (\supp g) \cap \supp h$$ is compact. This shows that the left-hand side of \eqref{fghbis} is well defined.
 
 Moreover, $g$ and $f_\ast h$ are Green functions for the divisors $D$ and $f_\ast E$; these divisors   have disjoint supports as a consequence of \eqref{DE2bis}; moreover $\supp g \cap \supp (f_\ast h)$ is contained in $\supp g \cap f(\supp h)$ and is therefore compact by \eqref{DE1bis}. This shows that the right-hand side of \eqref{fghbis} is well defined.
 
 The proof of Proposition \ref{functCinfty} will be left as an exercise for the interested reader.
 
 Concerning Part 3) of Proposition \ref{functCinfty}, observe that the operation of push-forward does \emph{not} preserve in general the class of Green functions with $\cC^\infty$ regularity. Actually this operation does not preserve $\cC^\infty$ functions when ramification occurs. 
 
 For instance, if $e$ is a positive integer, the push-forward of the function:
 $$\vert. \vert^2 : \mathring{D}(0;1) \lra \R, \quad z \longmapsto \vert z \vert^2$$
 by the complex analytic map:
 $$f_e : \mathring{D}(0;1) \lra \mathring{D}(0;1), \quad z \longmapsto z^e$$
 is the function:
 $$f_{e \ast} \vert.\vert^2 = e \vert.\vert^{2/e}.$$
 It is not $\cC^\infty$ if $e > 1$.

\section{Arakelov divisors on arithmetic surfaces and arithmetic intersection numbers}

In this section, we denote by $X$ an integral normal arithmetic surface.

\subsection{Arakelov divisor and Hermitian line bundles}\label{DefArDiv} 
 An \emph{Arakelov divisor} (resp. an \emph{Arakelov-Cartier divisor}) on $X$ is a pair $(D,g)$, where $D$ is a Weil (resp. Cartier) divisor on $X$ and $g$ a Green function with $\cC^\infty$ regularity, invariant under complex conjugation, for the divisor $D_\C$ on the Riemann surface $X(\C)$.

Recall that, to the Green function $g$ is associated a $\cC^\infty$ metric $\Vert . \Vert_g$ on the analytic line bundle $\cO_{X(\C)}(D_\C)$ over the Riemann surface $X(\C)$. Accordingly, when moreover $D$ is a Cartier divisor, we may associate to  $(D,g)$ the Hermitian line bundle with $\cC^\infty$ metric:
\begin{equation}
\cOb_X(D,g) := (\cO_X(D), \Vert.\Vert_g).
\end{equation}

We shall denote by $Z^1(X)$ and $Z^1_{\mathrm{Cart}}(X)$  the additive groups of Weil and Cartier divisors on $X$, and by
$\overline{Z}^1(X)$ and $\overline{Z}^1_{\mathrm{Cart}}(X)$  the additive groups of Arakelov and Arakelov-Cartier divisors on $X$. The quotient group $Z^1(X)/Z^1_{\mathrm{Cart}}(X)$, and consequently $\overline{Z}^1(X)/\overline{Z}^1_{\mathrm{Cart}}(X)$, is finite. 

To any invertible rational function $f$ in $\kappa(X)^\times,$ we may associate the following Arakelov-Cartier divisor:
$$\overline{\div} f := \big(\div f, \log \vert f_\C \vert^{-1}\big) \in \overline{Z}_\Cart^1(X).$$
The map:
$$\overline{\div} :  \kappa(X)^\times \lra  \Zb^1(X)$$
is a morphism of abelian groups. Its image $\overline{\div} \kappa(X)^\times$ is the group of Arakelov divisors \emph{rationally equivalent to zero}, and its cokernel defines the \emph{Arakelov-Chow group} of $X$:
$$\overline{CH}^1(X):= \overline{Z}^1(X) / \overline{\div} \kappa(X)^\times .$$
We shall also consider its subgroup of finite index:
$$\overline{CH}_\Cart^1(X):= \overline{Z}_\Cart^1(X) / \overline{\div}\kappa(X)^\times.$$

If we denote by 
$\overline{\Pic}(X)$ the group of isomorphism classes of Hermitian line bundles with $\cC^\infty$ metrics on $X$, we define an isomorphism of abelian group:
\begin{equation}\label{c1hat}
\widehat{c}_1: \overline{\Pic}(X) \lrasim \overline{CH}_\Cart^1(X)
\end{equation}
by mapping the isomorphism class of a Hermitian line bundle with $\cC^\infty$-metric $\Lb:= (L, \Vert .\Vert)$ to the class $\widehat{c}_1(\Lb)$ of the Arakelov-Cartier divisor:
$$\overline{\div}_{\Lb} \, s := (\div s, \log \Vert s_\C \Vert^{-1})$$
of an arbitrary non-zero rational section $s$ of $L$. The inverse isomorphism $\widehat{c}_1^{-1}$ sends the class of some Cartier-Arakelov divisor $(Z,g) \in \Zb_\Cart^1(X)$ to the isomorphism class of 
$\cOb_X (Z,g).$ 


An Arakelov divisor $(D,g)$ will be said to be \emph{compactly supported} when the support $\vert D \vert$ of $D$ is proper over $\Spec \Z$ and when the support $\supp g$ of the Green function $g$ is a compact subset of $X(\C)$. The subgroup of $\Zb^1(X)$ (resp. of $\Zb^1_\Cart (X)$) defined by the compactly supported Arakelov (resp. Cartier-Arakelov) divisors on $X$ will be denoted by $\Zb^1_c(X)$ (resp. by $\Zb^1_{\Cart,c}(X)$).

\subsection{The Arakelov intersection pairing: definitions and basic properties}\label{ArIntCinfty}

\subsubsection{Definitions} If $\Lb := (L, \Vert.\Vert)$ is a Hermitian line bundle over $X,$ with a $\cC^\infty$ metric $\Vert.\Vert$,  and if $(D, g)$ is an Arakelov divisor in $\Zb^1_c(X),$ we define the \emph{Arakelov intersection pairing} of $\Lb$ and $(D,g)$ as the real number:
\begin{equation}\label{defArInt}
 \Lb \cdot  (D,g) := \dega \big(\Lb\vert D\big) + \int_{X(\C)} g \, c_1(\Lb_\C).
\end{equation}

If $(D',g')$ (resp. $(D,g)$) is an Arakelov divisor in $\Zb_\Cart^1(X)$ (resp. in $\Zb^1_c(X)$), we define their \emph{Arakelov} or \emph{arithmetic intersection number} as the real number: 
\begin{equation}\label{defArIntBis}
(D',g') \cdot (D,g)    :=  \cOb_X(D',g') \cdot (D,g) = \dega \big(\cOb_X(D',g') \vert D\big) + \int_{X(\C)} g\,  \omega(g').
\end{equation}

The arithmetic intersection pairing:
\begin{equation}\label{defArIntBisbis}
. \cdot . : \overline{Z}^1_{\mathrm{Cart}}(X) \times \overline{Z}^1_c(X) \lra \R
\end{equation}
defined by \eqref{defArIntBisbis} is clearly bilinear. Since $\overline{Z}^1_{\mathrm{Cart}}(X)$ has finite index in $\overline{Z}^1(X)$, the pairing \eqref{defArIntBisbis} uniquely extends to a bilinear pairing:
\begin{equation}\label{defArTer}
 . \cdot . : \overline{Z}^1(X) \times \overline{Z}^1_c (X) \lra \R.
\end{equation}

In concrete terms, the pairing \eqref{defArTer} may be described as follows. If $(D',g')$ and $(D,g)$ are Arakelov divisors in $\overline{Z}^1(X)$ and $\overline{Z}^1(X)_c$ respectively, we may choose a positive integer $N$ such that the divisor $ND'$ is Cartier, and then, by definition, the intersection pairing of $(D',g')$ and $(D,g)$ is the real number:
\begin{multline}\label{defArIntTerter}
(D',g') \cdot (D,g)    =  N^{-1}  \, (ND',N g') \cdot (D,g) = N^{-1} \cOb_X(ND',N g') \cdot (D,g) \\ = N^{-1} \dega (\cOb_X(ND', Ng') \vert D) + \int_{X(\C)} g\, \omega(g').
\end{multline}

\subsubsection{Properties of the Arakelov intersection pairing}\label{PropArInt} 
By construction, the intersection pairing is compatible with linear equivalence in the first variable. Namely:
\begin{proposition}\label{ArEqLin} For every $f \in \kappa(X)^\times$ and every 
$(D,g) \in \overline{Z}^1(X)_c$, we have:
\begin{equation}
\overline{\div} f  \cdot  (D,g) =0.
\end{equation}
\end{proposition}

The following proposition describes how the intersection pairing of two Arakelov divisors depends of their Green functions.
\begin{proposition}\label{IntCinftyDep} For any two Arakelov divisors $(D,g) \in \overline{Z}^1(X)_c$ and $(D',g') \in \overline{Z}^1(X)$, and any two functions invariant under complex conjugation $f\in \cC_c^\infty(X(\C), \R)$ and $f'\in \cC^\infty(X(\C), \R)$, the following equality holds:
\begin{equation}\label{intArmodGr}
(D', g' + f') \cdot (D, g+f) = (D', g') \cdot (D,g) + \int_{X(\C)} f' \, \omega(g) + \int_{X(\C)} f \, \omega(g') - \pi^{-1} \langle f', f\rangle_{\mathrm{Dir}}.
\end{equation}
\end{proposition}

\begin{proof} According to the bilinearity of the Arakelov intersection pairing, the validity of \eqref{intArmodGr} follows from the following three equalities:
\begin{equation}\label{intArmodGr:eq1}
(D',g')\cdot (0, f) = \int_{X(\C)} f \,\omega(g'),
\end{equation}
\begin{equation}\label{intArmodGr:eq2}
(0,f') \cdot (D,g) =\int_{X(\C)} f'\, \omega (g), 
\end{equation}
and:
\begin{equation}\label{intArmodGr:eq3}
(0, f')\cdot (0,f) = - \pi^{-1}  \langle f', f\rangle_{\mathrm{Dir}}.
\end{equation}

The equality \eqref{intArmodGr:eq1} follows from the definition \eqref{defArInt} of the intersection pairing and from the relation:
$$c_1( \cOb_X(D', g')_\C) = c_1( \cO_{X(\C)}(D'_\C), \Vert.\Vert_{g'}) = \omega (g').$$

Formula \eqref{intArmodGr:eq2} is established by the following chain of equalities:
\begin{align}
(0,f') \cdot (D,g) & = \dega \big(\cOb_X (0,f') \vert D\big) + \int_{X(\C)} g \, \frac{i}{\pi} \partial \overline{\partial} f'  \label{f'Dg1}\\
&  = \int_{X(\C)} f' \, \delta_{D_\C} + \int_{X(\C)} g \, \frac{i}{\pi} \partial \overline{\partial} f'  \label{f'Dg2}\\
& = \int_{X(\C)} \Big( f' \, \delta_{D_\C} +  f' \, \frac{i}{\pi} \partial \overline{\partial} g \Big) \label{f'Dg3}\\
& =   \int_{X(\C)} f' \, \omega (g). \label{f'Dg4}
\end{align}
Indeed \eqref{f'Dg1} follows from the definition \eqref{defArInt} of the intersection pairing and from the relation:
\begin{equation}\label{c1f'}
c_1(\cOb_X(0, f')_\C) = c_1(\cO_{X_\C}, \Vert.\Vert_{f'}) = \frac{i}{\pi} \partial \overline{\partial} f'.
\end{equation}
The equality \eqref{f'Dg2}
follows from the definition \eqref{htLdef}, applied with $\Lb$ the restriction of   $\cOb_X (0,f')$ to the components of $D$ and with $s=1$. The equality \eqref{f'Dg3} follows from Green's formula \eqref{GreenFormula}, and \eqref{f'Dg4} from the definition \eqref{equation:ddg} of $\omega(g)$.

Finally, using again \eqref{defArInt} and \eqref{c1f'}, we obtain:
\begin{equation*}
(0, f')\cdot (0,f) = \int_{X(\C)} f \, \frac{i}{\pi} \partial \overline{\partial} f' =  - \pi^{-1}  \langle f', f\rangle_{\mathrm{Dir}}. \qedhere
\end{equation*}
 \end{proof}

Recall that if $D$ and $D'$ are two Cartier divisors on $X$, the intersection $D\cdot D'$ is a $0$-cycle on $\vert D \vert \cap \vert D' \vert$ which is well defined and depends symmetrically of $D$ and $D'$, up to rational equivalence supported by $\vert D \vert \cap \vert D' \vert$. 

Using again that $Z^1_{\mathrm{Cart}}(X)$ has finite index in $Z^1(X)$, this allows one to define the intersection $D\cdot D'$ of any two Weil divisor $D$ and $D'$ on $X$: it is  a $0$-cycle with $\Q$-coefficients supported by $\vert D \vert \cap \vert D' \vert$, which is well defined and depends symmetrically of $D$ and $D'$, up to rational equivalence supported by $\vert D \vert \cap \vert D' \vert$.

When moreover $\vert D \vert \cap \vert D' \vert$ is proper over $\Spec \Z$ and its generic fiber $\vert D_\Q \vert \cap \vert D'_\Q\vert$ is empty,\footnote{This last condition holds precisely when the divisors $D_\C$ and $D'_\C$ on the Riemann surface $X(\C)$ have disjoint supports.} then $\vert D \vert \cap \vert D' \vert$ is a finite union closed points and of vertical curves in $X$ proper over $\Spec \Z$, and therefore as observed in \ref{vertrat} above, rational equivalences  supported by $\vert D \vert \cap \vert D' \vert$ preserve the Arakelov degree of $0$-cycles. Accordingly, in this situation, the Arakelov degree $\dega D \cdot D'$ is well defined and satisfies:
$$\dega D \cdot D' = \dega D' \cdot D.$$

\begin{proposition}\label{arintstarprod} For any two Arakelov divisors  $(D',g') \in \overline{Z}^1(X)$  and $(D,g) \in \overline{Z}^1(X)_c$  such that $D'_\Q$ and $D_\Q$ have disjoint supports, 
the following equality holds:
\begin{equation}\label{defarintstar}
(D', g') \cdot (D,g) = \dega D' \cdot D + \int_{X(\C)} g \ast g'.
\end{equation}
\end{proposition}

\begin{proof} The bilinearity of the intersection pairing easily implies that, to establish the equality \eqref{defarintstar}, we may assume that one of the following additional conditions is satisfied: (i) the divisor $D'$ is a Cartier divisor; moreover  $D'$ and $D$ meet properly, namely the intersection $\vert D' \vert \cap \vert D \vert$ of their supports is a finite set of closed points; or (ii) the Arakelov divisors $(D',g')$ and $(D,g)$ are of the form $(V', 0)$ and $(V,0)$, where $V'$ is a vertical Cartier divisor in $X$, and $V$ a vertical divisor in $X$, proper over $\Spec \Z$.

In case (ii), \eqref{defarintstar} is a straightforward consequence of the definitions. In case (i), we may argue as follows. 

Applied to  the restriction $\Lb$ of   $\cOb_X (D',g')$ to the components of $D$ and to $s=\mathbbm{1}_{D'}$, the definition  \eqref{defArInt} gives:
\begin{equation}\label{htDg}
\dega (\cOb(D',g') \vert D) =  \dega D' \cdot D  - \int_{X(\C)} \log \Vert {\mathbbm{1}}_{D'_\C} \Vert_{g'} \,\delta_{D_\C} 
=  \dega D' \cdot D  + \int_{X(\C)} {g'}\,  \delta_{D_\C}.
\end{equation}
Consequently, using successively the definition of the arithmetic intersection pairing, the relation \eqref{htDg}, and the definition of $g \ast g'$, we obtain:
\begin{align*}
(D', g') \cdot (D,g) : &= \dega (\cOb(D',g') \vert D) + \int_{X(\C)} g\, c_1\big(\cO_{X(\C)}(D'_\C), \Vert.\Vert_g'\big) \\
& =  \dega D' \cdot D  + \int_{X(\C)} \big( {g'}\,  \delta_{D_\C}  + g \, \omega(g')) \\
& =   \dega D' \cdot D  + \int_{X(\C)} g \ast g'. \qedhere
\end{align*} 
\end{proof}

\begin{corollary}\label{ArIntSym} For any two Arakelov divisors $(D,g)$ and $(D',g')$ in $\overline{Z}^1(X)_c$,  we have:
\begin{equation}\label{intarsym}
(D', g') \cdot (D,g) = (D, g) \cdot (D',g').
\end{equation}
\end{corollary}

\begin{proof} Using the bilinearity of the arithmetic intersection pairing, one readily sees that the symmetry property \eqref{intarsym} follows from its validity in the following two special cases: (i) when $\vert D \vert_\Q \cap \vert D' \vert_\Q$ is empty; and (ii) when $(D', g')  =(D,g)$. In case (i), \eqref{intarsym} follows from Proposition \ref{arintstarprod} and from the symmetry \eqref{starprodsym} of the integral of the $\ast$-product. In case (ii), it is clear.
\end{proof}

\subsubsection{The Arakelov intersection pairing on projective arithmetic surfaces} In the special case where the integral normal arithmetic surface $X$ is  projective, and therefore $\Zb^1(X)$ and $\Zb^1_c(X)$ coincide, one recovers  the ``classical" arithmetic intersection theory on integral normal projective arithmetic surfaces from the above constructions. 

Indeed from the properties of the Arakelov intersection pairing established in  \ref{PropArInt}, one immediately derives:  

\begin{scholium}\label{scholArInt} If $X$ is an integral normal projective arithmetic surface, then one defines a symmetric bilinear pairing:
$$. \cdot . : \overline{CH}^1(X) \times \overline{CH}^1(X) \lra \R$$
by setting:
$$[(D, g)]\cdot [(D', g')] := (D,g) \cdot (D',g')$$
for any two Arakelov divisors $(D,g)$ and $(D',g')$ on $X$.

This pairing is the unique bilinear pairing such that the following equality holds:
$$[(D, g)]\cdot [(D', g')]  = \dega D \cdot D' + \int_{X(\C)} g \ast g',$$
for any two Arakelov divisors $(D,g)$ and $(D',g')$ on $X$ such that $D_\Q$ and $D'_\Q$ have disjoint supports.
\end{scholium}

The  starting point of Arakelov geometry has been the definition of the arithmetic intersection number attached to a pair of Arakelov divisors on some regular projective arithmetic surface, as in Scholium \ref{scholArInt}. 

The original contributions  of Arakelov \cite{Arakelov74} and Faltings \cite{Faltings84}  focused on  the arithmetic intersection pairing attached to Arakelov divisors $(Z, g)$ defined by Green functions $g$ satisfying a suitable normalization condition.\footnote{Namely, the $2$-forme $\omega(g)$ was required to be a multiple of the  $2$-form $\beta_{\mathrm{Ar}}$, defined by \eqref{betaAr} below, on each of the connected components $X_{\sigma}(\C)$ of $X(\C)$, which were assumed to have positive genus.} The general definition, allowing arbitrary Green functions with $\cC^\infty$ regularity as defined in \ref{Greensmoothdef} above, appears in the work of Deligne \cite{Deligne85} and Gillet-Soul\'e~\cite{Gillet-Soule90int}.

\subsection{Functoriality properties}

Let $f: X' \ra X$  be a morphism between two integral normal arithmetic surfaces.

The inverse image of Hermitian line bundles with $\cC^\infty$ metrics defines a morphism of groups:
$$f^\ast: \overline{\Pic}(X) \lra \overline{\Pic}(X'), $$
and therefore, by means of the isomorphisms:
$$\widehat{c}_1: \overline{\Pic}(X) \lrasim \overline{CH}_\Cart^1(X)  \quad \mbox{and} \quad
\widehat{c}_1: \overline{\Pic}(X') \lrasim \overline{CH}_\Cart^1(X'),$$
a pull-back map between Arakelov-Cartier Chow groups:
\begin{equation}\label{fastArCart}
f^\ast: \overline{CH}_\Cart^1(X) \lra \overline{CH}_\Cart^1(X').
\end{equation}

When $X$ and $X'$ are proper over $\Spec \Z$ (that is, projective arithmetic surfaces), this map is compatible with the Arakelov intersection product. Namely, for every pair $(\alpha, \beta)$ of elements of $\overline{CH}_\Cart^1(X),$ the following equality holds:
\begin{equation*}
f^\ast \alpha \cdot f^\ast \beta = \deg f \; \alpha\cdot\beta,
\end{equation*}
where $\deg f$ denotes the degree\footnote{Recall that  $\deg f$ vanishes when $f$ is not dominant, and that $\deg f := [\kappa(X'): f^\ast \kappa(X)]$ when $f$ is dominant, hence generically finite.} of $f$; see for instance  \cite[2.1.1, 2.2, and Proposition 2.3.1, ~(iv)]{BostGilletSoule94}. 

When $f$ is dominant, the pull-back map \eqref{fastArCart} ``lifts" to Arakelov-Cartier divisors. Namely, using the functoriality of Green functions with $\cC^\infty$-regularity by pull-back stated in Proposition \ref{functCinfty}, we may define a map:
\begin{equation*}
f^\ast: \overline{Z}_\Cart^1(X) \lra \overline{Z}_\Cart^1(X'), \quad (D,g) \longmapsto (f^\ast D, f_\C^\ast g),
\end{equation*}
which is easily seen to be compatible with rational equivalence and to induce the map \eqref{fastArCart}.

One may wonder whether, as with usual Chow groups, the Arakelov-Chow groups also admit some covariant functoriality, namely whether to a map $f:X' \ra X$ as above is naturally attached a map $f_\ast$ from  $\overline{CH}^1(X')$ to $\overline{CH}^1(X)$. At the level of Arakelov cycles, it should be defined by a formula of the  type:
\begin{equation}\label{f_asttent}
f_\ast(D',g') := (f_\ast D', f_{\C \ast} g'),
\end{equation}
say when $f$ is proper, or when the Arakelov cycle $(D',g')$ on $X'$ is compactly supported.

The  right-hand side of \eqref{f_asttent} does \emph{not} make sense in general, even if $f$ is assumed to be dominant and proper (and therefore $f_\C$ to be a finite morphism) because of the limited functoriality of Green functions with $\cC^\infty$ regularity discussed in Subsection \ref{functCinftyGreen}.\footnote{However it defines an Arakelov cycle on $X$ when for instance the $2$-form $\omega(g')$ vanishes 
on some neighborhood of the ramification locus of $f_\C$.}
 This lack of covariant functoriality of the ``classical" Arakelov-Chow groups defined by using Green functions with $\cC^\infty$ regularity will be addressed in the next sections of this memoir, by introducing 
Green functions with weaker regularity.

\section{Green functions with $L^{2}_{1}$ regularity on Riemann surfaces}\label{subsubsection:L21}

In many applications of arithmetic intersection theory and heights on arithmetic surfaces, it is crucial to extend the original scope of the theory and make it possible to use Green functions -- or, equivalently, Hermitian metrics on line bundles -- with a weaker regularity than the $\cC^\infty$ regularity we required in the discussion above. Diverse  more flexible formalisms involving possibly non $\cC^\infty$ Green functions and Hermitian metrics have been notably developed in  \cite{Bost99}, \cite{Kuhn01} and \cite{BKK07}. 

In the next two sections, we discuss the formalism introduced in \cite{Bost99}, based on the use of Green functions with $\Ld$ regularity. We extend the results in \emph{loc. cit.}, by developing this formalism on possibly non-projective arithmetic surfaces, and by investigating its functoriality properties.

\subsection{Green functions with $L^{2}_{1}$ regularity and $\ast$-products}\label{Ldstarprod}

 \subsubsection{}
  When one is interested in defining the Arakelov intersection numbers \eqref{defArIntBis} and \eqref{defArIntTerter} in the largest possible generality, a natural regularity class for the Green functions  --- namely the $L^2_1$ regularity --- is suggested by the characterization of the map $$(g_{1}, g_{2})\longmapsto \int_M g_{1}\ast g_{2}$$ by properties $\mathbf{SP}_1$ and $\mathbf{SP}_2$ in \ref{SPbasic}.

Recall that a distribution $\phi$ over a differentiable manifold $M$ is said to be locally $L^{2}_{1}$ when the current $d\phi$ is locally $L^{2}$ (and therefore $\phi$ itself if locally $L^2$, and actually locally $L^p$ for every $p\in [1, +\infty)$ when $M$ has dimension 2). A \emph{Green function with $L^{2}_{1}$ regularity} (or shortly a $L^{2}_{1}$ Green function) for a divisor $D$ on a Riemann surface $M$ is defined as in \ref{Greensmoothdef} by allowing the function $h$ appearing in \eqref{equation:condition-green} to be locally $L^{2}_{1}$ over $U$, instead of being $\mathcal C^{\infty}$.

If $g_{0}$ is a Green function for $D$ with $\mathcal C^{\infty}$ regularity, then the $L^{2}_{1}$ Green functions for $D$ are precisely the distributions on $M$ of the form:
$$g=g_{0}+f,$$
where $f$ is a locally $L^{2}_{1}$ function on $M$.

By the elliptic regularity of the operator $\partial\overline\partial$, the $L^{2}_{1}$ Green functions are exactly those real distributions $g$ on $M$ such that the current $\omega(g)$ defined as in \eqref{equation:ddg} by:
$$\omega(g) := \frac{i}{\pi} \partial\overline\partial g +\delta_{D}$$
belongs to the space of locally $L^{2}_{-1}$ currents of degree $2$ over $M$.

\subsubsection{} In this paragraph, for simplicity, let us  assume that the Riemann surface $M$ is compact and connected. 

Observe that in property  $\mathbf{SP}_2$,  if $f_1$ and $f_2$ are assumed to be of class $L^2_1$ instead of $\mathcal C^\infty$, while $g_1$ and $g_2$ are still Green functions with $\mathcal C^\infty$ regularity, then  in the equality \eqref{equation:star} --- namely:
$$\int_{M}(g_{1}+f_{1})\ast (g_{2}+f_{2})=\int_{M}g_{1}\ast g_{2}+\int_{M}f_{1}\, \omega(g_{2})+\int_{M}f_{2}\, \omega(g_{1})-\frac{i}{\pi}\int_{M}\partial f_{1}\wedge\overline\partial f_{2}$$
---  the right-hand side  is still well-defined, since both $\int_{M}f_{1}\, \omega(g_{2})$ and $\int_{M}f_{2}\, \omega(g_{1})$ clearly are, and the Dirichlet scalar product: 
$$\langle f_{1}, f_{2}\rangle_{\mathrm{Dir}}=i\int_{M}\partial f_{1}\wedge\overline\partial f_{2}$$
also is, because both $df_1$ and $df_2$ are $L^2$. 

As a consequence, the equality \eqref{equation:star} may be used to define the integral of the $\ast$-product of any two $L^2_1$ Green functions for $D_1$ and $D_2$, written respectively as $g_1+f_1$ and $g_2+f_2$ with $f_1$ and $f_2$ some $L^2_1$ functions on $M$. This definition is readily seen to be independent of the choice of the $\mathcal C^\infty$ Green functions $g_1$ and $g_2$ for $D_1$ and $D_2$. Moreover, the equality \eqref{equation:star} holds for any two $L^2_1$ Green functions $g_1$ and $g_2$ for $D_1$ and $D_2$, and any two functions $f_1$ and $f_2$ in the  Dirichlet space $L^2_1(M, \R)$ of real $L^2_1$ functions on $M$.\footnote{Note that, for any such $f_1, f_2, g_1$ and $g_2$, the last three integrals in \eqref{equation:star} are well-defined, as the integrals of the product of currents in $L^2_1$ and $L^2_{-1}$, or of two $L^2$ 1-forms.} 

Observe also that the validity of \eqref{equation:star} implies that, when it is defined, the integral of the $\ast$-product of $\Ld$-Green functions still satisfy the symmetry \eqref{starprodsym}.

Since the space $L^2_1(M,\R)/\R$ may be identified with the completion of the space $\mathcal C^\infty(M, \R)/\R$ with respect to the Dirichlet scalar product\footnote{See for instance \cite[3.1.2]{Bost99}} $\langle ., \rangle_{\mathrm{Dir}}$, this discussion shows that $L^2_1$-regularity is the weakest possible regularity condition on Green functions that makes it possible to define the intersection number of arbitrary Arakelov divisors $(Z, g)$ with $g$ of this regularity.

\subsubsection{}\label{SPLd} More generally, when $M$ is possibly non-compact, a straightforward extension of  the previous arguments allows one to define the integral of the $\ast$-product:
$$\int_M g_1 \ast g_2$$
when $g_1$ and $g_2$ are $L^2_1$ Green functions attached to some divisors $D_1$ and $D_2$ with disjoint supports in $M$ such that the intersection $\supp g_1 \cap \supp g_2$ is compact. 

Actually properties   $\mathbf{SP}_1$ and $\mathbf{SP}_2$ remain valid in this generality, where $g_1$ and $g_2$ denote some  Green functions with $L^2_1$ regularity, and where in $\mathbf{SP}_2$, $f_1$ and $f_2$ are real valued locally $L^2_1$ functions. 
Moreover, for any pair of divisors $D_1$ and $D_2$ on $M$ with disjoint support,  the function that attaches the integral $\int_{M}g_{1}\ast g_{2}$ to a pair $(g_{1}, g_{2})$ of  $L^2_1$ Green functions for  the divisors $D_{1}$ and $D_{2}$ such that $\mathrm{supp}\, g_{1} \cap \mathrm{supp}\, g_{2} $ is compact is still  characterized by the validity of this extension of  properties $\mathbf{SP}_1$ and~$\mathbf{SP}_2$.  


\subsection{Functoriality of Green functions with $L^2_1$ regularity}

It turns out that Green functions with $L^2_1$ regularity on Riemann surfaces enjoy remarkable functoriality properties: contrary to Green functions with $\cC^\infty$ regularity, they constitute a class of Green functions that is preserved by  the operation of direct image by non-constant complex analytic maps. 

\subsubsection{Functoriality of forms and currents on complex analytic manifolds}\label{FunctCurrents} 

Recall that to any complex analytic map $f : X\ra Y$ of complex analytic manifolds are attached a pull-back map $f^{*}$,  that sends $\cC^\infty$ differential forms on $Y$ to $\cC^\infty$ differential forms on $X$, and a push-forward map $f_{*}$, deduced from $f^{*}$ by duality. The map $f_\ast$  sends a  current $T$ of degree $d$ on $X$ such that the restriction $f_{|\mathrm{supp}(T)}$ of $f$ to the support of $T$ is a proper map, to the current $f_{*}T$ of degree $d-2\dim_{\C} X+2\dim_{\C}Y$ on $Y$ defined by:
$$\int_{Y}\alpha\, f_{*}T=\int_{X}f^{*}\alpha\; T$$
for any $\cC^\infty$ differential  form $\alpha$ with compact support on $Y$ of degree $2\dim_{\C} X-d$.

The maps $f^{*}$ and $f_{*}$ are compatible with the decomposition of differential forms and currents into forms and currents of type $(p, q)$, $p, q \in \N$,  and with the operators $d,$ $ \partial$ and $\overline\partial.$ The constructions of $f^{*}$ and $f_{*}$ are functorial in $f$.

\subsubsection{} In this paragraph, we denote by $M$ and $N$ be two Riemann surfaces, with $N$ connected, and by $f: N \ra M$  a non-constant complex analytic map.

The following propositions describe the functoriality of currents on $M$ and $N$ of degree 1 (resp. 0, resp. 2) with regularity $L^2$ (resp. $L^2_1$, resp. $L^2_{-1}$) with respect to the map $f$.  They follow from basic properties of currents on manifolds, combined with the fact that the map $f$ is locally of the form $(z \mapsto w = z^e)$ for some positive integer $e$ in suitable local analytic charts $z$ and $w$ on $N$ and $M$. The details of the
proofs are left to the reader.\footnote{The proof of Corollary \ref{ddbarpullback} may be simplified by the following observation: \emph{ a locally $L^2_{-1}$ current of degree 2 on a Riemann surface vanishes if its support is discrete.} This observation also implies that a $L^2_{1}$ Green current for a divisor $D$ uniquely determines $D$.}

\begin{proposition}\label{Funct1} 
1) If $\alpha$ is a locally $L^2$ 1-form on $M,$ then its pull-back\footnote{defined pointwise, almost everywhere on $N$.} $f^\ast \alpha$ is a locally $L^2$ 1-form on $N$.

2) If $\beta$ is a locally $L^2$ 1-form on $N$ and if the restriction $f_{\mid \supp \beta}$ is proper, then the current $f_\ast \beta$ is a locally $L^2$ 1-form on $M$.

3) If $\alpha$ and $\beta$ are as in 1) and 2), and if $\supp \alpha \cap f(\supp \beta)$ is compact, then the following equality holds:
\begin{equation*}
\int_N f^\ast \alpha \wedge \beta = \int_M \alpha \wedge f_\ast \beta.
\end{equation*}
\end{proposition}

\begin{proposition}\label{Funct0} 1) If $\phi$ is a locally $L^2_1$ function on $M$, then its pull-back\footnote{defined pointwise, almost everywhere  on $N$.} $f^\ast \phi$ is a locally $L^2_1$ function on $N$.
 
 2) If $\psi$ is a locally $L^2_1$ function on $N$ and if $f_{\mid \supp \psi}$ is proper, then the distribution $f_\ast \psi$ is a locally $L^2_1$ function on $M$.
 
\end{proposition}


\begin{proposition}\label{Funct2} 
 1) If $\chi$ is a locally $L^2_{-1}$ current of degree 2 on $M$, then its pull-back 
 $f^\ast \chi$ may be defined as the unique locally $L^2_{-1}$ current of degree 2 on $N$ such that, for every $L^2_1$ function $\psi$ on $N$ with compact support, the following equality holds:
 \begin{equation}\label{psichi}
\int_N \psi \, f^\ast \chi = \int_M f_\ast \psi \, \chi.
\end{equation}

The equality \eqref{psichi} is actually satisfied for every locally $L^2_1$ function $\psi$ on $N$ such that $f_{\mid \supp \psi}$ is proper and $\supp \chi \cap f(\supp \psi)$ is compact.
 
 2) If $\omega$ is a locally $L^2_{-1}$ current of degree 2 on  $N$ and if $f_{\mid \supp \omega}$ is proper, then the current $f_\ast \omega$ is a locally $L^2_{-1}$ current of degree 2 on $M$. Moreover, for every locally $L^2_1$ function $\phi$ on $M$ such that $\supp \phi \cap f(\supp \omega)$ is compact, the following equality holds:
  \begin{equation}\label{duality L21_1}
\int _N f^\ast \psi \, \omega =\int_M \psi \, f_\ast \omega.
\end{equation}
\end{proposition}

\begin{corollary}\label{ddbarpullback} For every locally $L^2_1$ function $\phi$ on $M$, the following equality holds:
\begin{equation}\label{partialpartialfast}
\partial \overline{\partial} (f^\ast \phi) = f^\ast (\partial \overline{\partial} \phi).
\end{equation}
\end{corollary}

In the left-hand side of \eqref{partialpartialfast}, the function $f^\ast \phi$  is defined  as in Proposition \ref{Funct0}, 1). In its right-hand side, the current $ f^\ast (\partial \overline{\partial} \phi)$ is defined  by Proposition \ref{Funct2}, 1) applied to $\chi =  
\partial \overline{\partial} \phi$.

\subsubsection{} From the functoriality properties of $L^2_1$ functions and of $L^2_{-1}$ currents of degree 2  in Propositions  \ref{Funct0} and \ref{Funct2}, one easily deduces the following functoriality properties of Green functions with $L^2_1$ regularity on Riemann surfaces: 
\begin{proposition}\label{FunctGreenLd} Let $M$ and $N$ be two Riemann surfaces, with $N$ connected, and $f: N \ra M$  a non-constant complex analytic map.

1) If $D$ is a divisor on $M$ and $g$ a Green function with  $L^2_1$ regularity for $D$ on $M$, then the function $f^\ast g $ on $N$, pointwise defined almost everywhere on $N$, is a Green function with $L^2_1$ regularity for the divisor $f^\ast D$ on~$N$. 

Moreover the following equality of $L^2_{-1}$ currents of degree 2 on $N$ is satisfied:
\begin{equation*}
\omega(f^\ast g) = f^\ast \omega(g).
\end{equation*}

2) If $E$ is a divisor on $N$ and $h$ a Green function with $L^2_1$ regularity for $E$ on $N,$ and if $f_{\mid \supp h}$ is proper\footnote{This implies that $f_{\mid \vert E \vert}$ is proper, and therefore that the divisor $f_\ast E$ on $M$ is well defined.}, then the distribution $f_\ast h$ on $M$ is a Green function with $L^2_1$ regularity for $f_\ast E$ on $M$. 

Moreover the following equality of $L^2_{-1}$ currents of degree 2  on $M$ is satisfied:
\begin{equation*}
\omega( f_\ast h) = f_\ast \omega(h).
\end{equation*}

3) Let $(D,g)$ and $(E,h)$ be as in 1) and 2). If the following additional conditions are satisfied:
\begin{equation}\label{DE2}
f^{-1}(\vert D \vert) \cap \vert E\vert  = \emptyset
\end{equation}
and:
\begin{equation}\label{DE1}
\mbox{$\supp g \cap f(\supp h)$ is compact},  
\end{equation}
then the following equality holds:
\begin{equation}\label{fgh}
\int_N (f^\ast g) \ast h = \int_M g \ast (f_\ast h).
\end{equation}
\end{proposition}

\subsubsection{} The following proposition relates direct images of currents of the form $\log \vert \phi \vert$, where $\phi$ denotes a meromorphic function, and norm maps associated to algebraic morphisms of complex curves.

\begin{proposition}\label{fastNorm} Let $f: N \ra M$ be a finite morphism between some smooth irreducible complex algebraic curves, and let:
$$N_{\C(N)/\C(M)} : \C(N) \lra \C(M)$$
be the norm map associated to the finite extension of fields of rational functions:
$$f^\ast: \C(M) \lra \C(N).$$ For every $\phi$ in $\C(M)^\times,$ the following equality of locally $L^1$ distributions holds:
\begin{equation}\label{eq:fastNorm}
f_\ast \log \vert \phi \vert = \log \vert N_{\C(N)/\C(M)}(\phi) \vert.
\end{equation}
\end{proposition} 

This easily follows from the fact that both sides of \eqref{eq:fastNorm} are locally $L^1$ and from the expression \eqref{fastcont} recalled in Subsection \ref{FunctbisMeasRS} below for the direct image of a continuous function.

\section[Arakelov intersection theory with $L^2_1$ Green functions]{Arakelov intersection theory with $L^2_1$ Green functions on integral normal arithmetic surfaces}

\subsection{Construction and properties of the Arakelov intersection pairing}

In this subsection, we denote by $X$ an integral normal arithmetic surface. 

\subsubsection{Arakelov divisors with $\Ld$ Green functions}
Using Green functions with $\Ld$ regularity instead of Green functions with $\cC^\infty$ regularity in the definitions in \ref{DefArDiv}, we  define the group $\Zb^1(X)^\Ld$ of \emph{Arakelov divisors} on $X$ \emph{with $\Ld$ Green functions}.

We may also consider its subgroups $\Zb_\Cart^1(X)^{\Ld}$  of 
Cartier-Arakelov divisors with $\Ld$ Green functions 
and  $\Zb_c^1(X)^{\Ld}$ of 
compactly supported Arakelov divisors with $\Ld$ Green functions, and their intersection:
$$\Zb_{\Cart,c}^1(X)^{\Ld} := \Zb_\Cart^1(X)^{\Ld} \cap \Zb_c^1(X)^{\Ld}.$$

To this extended notion of Arakelov divisors are as associated the following $\Ld$ variants of the Arakelov-Chow groups of $X$:
$$\overline{CH}^1(X)^\Ld:= \overline{Z}^1(X)^\Ld / \overline{\div} \kappa(X)^\times$$
and:
$$\overline{CH}_\Cart^1(X)^\Ld:= \overline{Z}_\Cart^1(X)^\Ld / \overline{\div}\kappa(X)^\times.$$

\subsubsection{}\label{LdIntAr}

We may  extend the Arakelov intersection pairing defined in \ref{ArIntCinfty} above to Arakelov divisors with $\Ld$ Green functions. Indeed, from Proposition \ref{IntCinftyDep} and the validity of property $\mathbf{SP}_2$ in the $\Ld$ setting discussed in \ref{SPLd}, we immediately derive:

\begin{propositiondef}\label{defArIntLd}
The Arakelov intersection pairing \ref{defArTer} admits a unique extension:
\begin{equation}\label{defArTerLd}
 . \cdot . : \overline{Z}^1(X)^\Ld \times \overline{Z}^1(X)^\Ld_c \lra \R
\end{equation}
such that the equality \eqref{intArmodGr}, namely:
$$
(D', g' + f') \cdot (D, g+f) = (D', g') \cdot (D,g) + \int_{X(\C)} f' \, \omega(g) + \int_{X(\C)} f \, \omega(g') - \pi^{-1} \langle f', f\rangle_{\mathrm{Dir}},
$$
remains valid for any two Arakelov divisors $(D,g) \in \overline{Z}_c^1(X)^\Ld$ and $(D',g') \in \overline{Z}^1(X)^\Ld$ and any two functions invariant under complex conjugation $f \in \Ld(X(\C), \R)_c$ and $f'\in \Ld(X(\C),\R)_{\mathrm{loc}}$.
\end{propositiondef}

In turn, Proposition \ref{defArIntLd} readily implies:

\begin{corollary}\label{Ldstilltrue} Propositions \ref{ArEqLin} and \ref{arintstarprod} and Corollary  \ref{ArIntSym} remain valid when $\Zb^1(X)$ and $\Zb_c^1(X)$ are replaced by $\Zb^1(X)^\Ld$ and $\Zb_c^1(X)^\Ld$.
\end{corollary}

\subsubsection{The Arakelov-Chow group $\overline{CH}^1(X)^\Ld$ when $X$ is projective}\label{LdXproj}

When $X$ is an integral normal projective arithmetic surface, the constructions in \ref{LdIntAr} specialize to the generalization of the classical arithmetic intersection theory on projective arithmetic surfaces introduced in \cite{Bost99}. Indeed they immediately imply the following variant of Scholium \ref{scholArInt}:

\begin{proposition}[see \protect{\cite[Section 5.3]{Bost99}}] \label{scholArIntLd} If $X$ is an integral normal projective arithmetic surface, then one defines a symmetric bilinear pairing:
$$. \cdot . : \overline{CH}^1(X)^\Ld \times \overline{CH}^1(X)^\Ld \lra \R$$
by setting:
$$[(D, g)]\cdot [(D', g')] := (D,g) \cdot (D',g')$$
for any two Arakelov divisors $(D,g)$ and $(D',g')$ in $\Zb^1(X)^\Ld$.

This pairing is the unique bilinear pairing such that the following equality holds:
$$[(D, g)]\cdot [(D', g')]  = \dega D \cdot D' + \int_{X(\C)} g \ast g',$$
for any two Arakelov divisors $(D,g)$ and $(D',g')$ in $\Zb^1(X)^\Ld$ such that $D_\Q$ and $D'_\Q$ have disjoint supports.
\end{proposition}

The intersection pairing in Proposition \ref{scholArIntLd} satisfies an analogue of the Hodge Index Theorem concerning the  intersection pairing on the N\'eron-Severi group of a smooth projective surface over field, that has been  proved in the  setting of the original arithmetic intersection theory of Arakelov \cite{Arakelov74} by Faltings \cite{Faltings84} and Hriljac \cite{Hriljac85}.  We refer the reader to \cite[Section 5.5]{Bost99} for details and references. 

In this memoir, we will use the following simple form of this arithmetic Hodge Index Theorem:

\begin{proposition}\label{HodgeIndAr} Let $X$ be an integral normal projective arithmetic surface. If $\Hb$ is an element of $\Zb^1(X)^\Ld$ such that:
\begin{equation}
\label{Hbselfpos}
\Hb \cdot \Hb > 0,
\end{equation}
then, for every $\Cb$ in $\Zb^1(X)^\Ld$ such that:
$$\Hb \cdot \Cb =0,$$
the following inequality holds:
$$\Cb \cdot \Cb \leq 0.$$
\end{proposition}

Recall that Arakelov divisors $\Hb$ in $\Zb^1(X)^\Ld$ that satisfy the condition \eqref{Hbselfpos} of positive self-intersection are easily constructed. Indeed, if $\Hb_0:= (H, g)$ is an Arakelov divisor on $X$ such that the degree $\deg H_\Q$ of the divisor $H_\Q$ on the projective curve $X_\Q$ is positive, then for every $\lambda \in\R_+$ large enough, the Arakelov divisor $\Hb:= (H, g + \lambda)$ satisfies \eqref{Hbselfpos}, since:
 $$\Hb \cdot \Hb = \Hb_0 \cdot \Hb_0 + 2 \lambda \deg H_\Q.$$

\subsection{Functoriality properties}\label{FunctLd}

In this subsection, we denote by $f: X' \ra X$ a \emph{dominant} morphism between two integral normal arithmetic surfaces, and we denote by $\deg f$ its degree, namely the positive integer defined as the degree of the field extension:
\begin{equation}\label{fastext}
f^\ast : \kappa(X) \lra \kappa(X').
\end{equation}

To the map $f$ are associated a pull-back and a push-forward map between suitable spaces of Arakelov divisors with $\Ld$ regularity on $X$ and $X'$, namely the following morphisms of $\Z$-modules:
\begin{equation}\label{pbf1}
f^\ast : \overline{Z}_\Cart^1(X)^\Ld \lra \overline{Z}_\Cart^1(X')^\Ld, \quad (Z,g) \longmapsto (f^\ast Z, f_\C^\ast g)
\end{equation}
and:
\begin{equation}\label{pff1}
f_\ast : \overline{Z}_c^1(X')^\Ld \lra \overline{Z}_c^1(X')^\Ld, \quad (Z,g) \longmapsto (f_\ast Z, f_{\C \ast} g).
\end{equation}

When moreover the morphism $f: X' \ra X$ is proper, we may also define the following maps: 
\begin{equation}\label{pbf2}
f^\ast : \overline{Z}_{\Cart,c}^1(X)^\Ld \lra \overline{Z}_{\Cart,c}^1(X')^\Ld, \quad (Z,g) \longmapsto (f^\ast Z, f_\C^\ast g)\end{equation}
and:
\begin{equation}\label{pff2}
f_\ast : \overline{Z}^1(X')^\Ld \lra \overline{Z}^1(X)^\Ld, \quad (Z,g) \longmapsto (f_\ast Z, f_{\C \ast} g).
\end{equation}

\begin{proposition}\label{pushpullAr} If $f$ is proper, then, for every $\Zb$ in $\overline{Z}_\Cart^1(X)^\Ld,$ the following equality holds:
\begin{equation}\label{eq:pushpullAr}
f_\ast f^\ast \Zb = \deg \! f  \; \Zb.
\end{equation}
\end{proposition}
\begin{proof} This follows from the validity of  similar formulas,  concerning Cartier divisors in ``classical" algebraic geometry, as established in \cite[Proof of Proposition 2.3 (c)]{FultonIT}, and concerning  $\Ld$ functions and Green functions.
\end{proof}

The maps \eqref{pbf1} and \eqref{pff2} are compatible with linear equivalence. To formulate this compatibility, recall that 
to the field extension \eqref{fastext}  is associated a norm map:
$$N_{\kappa(X')/\kappa(X)}: \kappa(X') \lra \kappa(X),$$
which defines a morphism of multiplicative groups from $\kappa(X')^\times$ to $\kappa(X)^\times.$

\begin{proposition}\label{FunctLinearEqu} 1) For every rational function $\phi \in \kappa(X)^\times,$ the following equality holds in $ \overline{Z}_{\Cart}^1(X')$:
\begin{equation}\label{pull-backdivAr}
f^\ast \overline{\div} \phi = \overline{\div} (f^\ast \phi).
\end{equation}
2) When the morphism $f$ is proper, for every rational function  $\phi '\in \kappa(X')^\times,$ the following equality holds in $ \overline{Z}^1(X)$:
\begin{equation}\label{pushfddivAr}
f_\ast \overline{\div} \phi' = \overline{\div} N_{\kappa(X')/\kappa(X)}(\phi').
\end{equation}
\end{proposition}

\begin{proof} The relations \eqref{pull-backdivAr} and \eqref{pushfddivAr} are straightforward consequences of the definitions, and for \eqref{pushfddivAr}, of the similar relation in algebraic geometry --- established in \cite[Proposition 1.4 (b) and  Example 20.1.3]{FultonIT} --- and from Proposition \ref{fastNorm}.
\end{proof}

\begin{corollary} The map \eqref{pbf1} defines a pull-back morphism between Arakelov-Chow groups:
\begin{equation}\label{pbf1CA}
f^\ast : \overline{CH}_\Cart^1(X)^\Ld \lra \overline{CH}_\Cart^1(X')^\Ld, \quad [(Z,g)] \longmapsto [(f^\ast Z, f_\C^\ast g)].
\end{equation}

When moreover $f$ is proper, the map \eqref{pff2} defines a push-forward morphism between Arakelov-Chow groups:
\begin{equation}\label{pff2CA}
f_\ast : \overline{CH}^1(X')^\Ld \lra \overline{CH}^1(X)^\Ld, \quad [(Z,g)] \longmapsto [(f_\ast Z, f_{\C \ast} g)].
\end{equation}
\end{corollary}

The pull-back and push-forward maps defined in \eqref{pbf1} and \eqref{pff2} and the Arakelov intersection pairing satisfy the usual projection formula:

\begin{proposition}\label{ProjForAr} For every $\Zb$ in $\Zb^1_\Cart(X)^\Ld$ and every $\Zb'$ in $\Zb^1_c(X')^\Ld,$ the following equality holds:
\begin{equation}\label{AdjArInt}
f^\ast \Zb \cdot \Zb' = \Zb \cdot f_\ast \Zb'.
\end{equation}

When moreover $f$ is proper, the equality \eqref{AdjArInt}  holds for every $\Zb$ in $\Zb^1_{\Cart,c}(X)^\Ld$ and every $\Zb'$ in $\Zb^1(X')^\Ld.$
\end{proposition}

\begin{proof} Consider some Arakelov cycles with $L^2_1$ regularity $\Zb:= (Z,g)$ in $\Zb^1_\Cart(X)^\Ld$ and $\Zb':= (Z', g')$ in $\Zb^1_c(X')^\Ld.$ 

Both sides of \eqref{AdjArInt} are unchanged when $\Zb$ is replaced by a linearly equivalent Arakelov-Cartier divisor; this follows from Proposition \ref{ArEqLin}, which remains valid when the Green current $g$ has $\Ld$ regularity as observed in Corollary \ref{Ldstilltrue}, and from Proposition \ref{FunctLinearEqu}, 1).

Accordingly, to establish \eqref{AdjArInt}, we may assume that $\vert Z \vert$ and $f(\vert Z' \vert)$ do not meet in $X_\Q$. Then $\vert f^\ast Z \vert$ and $\vert Z' \vert$ do not meet in $X'_\Q$, and according to the extension of Proposition \ref{arintstarprod}
mentioned  in Corollary \ref{Ldstilltrue}, we have:
\begin{equation}\label{LHSProj}
f^\ast \Zb \cdot \Zb'  = (f^\ast Z, f_\C^\ast g) \cdot (Z', g') = \dega f^\ast Z \cdot Z' +  \int_{X'(\C)} f^\ast_\C g \ast g'
\end{equation}
and:
\begin{equation}\label{RHSProj}
\Zb \cdot f_\ast \Zb' = (Z, g) \cdot (f_\ast Z', f_{\C \ast} g') = \dega Z \cdot f_\ast Z' + \int_{X(\C)} g \ast f_{\C \ast} g'. 
\end{equation}

The equality \eqref{AdjArInt} will follow from the following two equalities:
\begin{equation}\label{AdjArInt1}
\dega f^\ast Z \cdot Z' = \dega Z \cdot f_\ast Z'
\end{equation}
and:
\begin{equation}\label{AdjArInt2}
\int_{X'(\C)} f^\ast_\C g \ast g' = \int_{X(\C)} g \ast f_{\C \ast} g'. 
\end{equation}

The equality \eqref{AdjArInt1} follows from the equality of $0$-cycles: 
$$f_\ast (f^\ast Z \cdot Z') = Z \cdot f_\ast Z',$$
that holds up to 
linear equivalence supported by the scheme $\vert Z \vert \cap f(\vert Z' \vert)$, which is proper over $\Spec \Z$ with empty generic fiber. 
This equality is an instance of the projection formula in algebraic geometry, as established in \cite[Proposition 2.3 (c)]{FultonIT}.

Finally \eqref{AdjArInt2} follows from Proposition \ref{FunctGreenLd}. 

The second part of the proposition, valid when $f$ is proper, is established by a similar argument, which we  leave to the reader.
\end{proof}

\begin{corollary}\label{degintAr} When $f$ is proper, for every $\Zb_1$ in $\overline{Z}_\Cart^1(X)^\Ld$
and every $\Zb_2$ in $\overline{Z}_{\Cart,c}(X)^\Ld$, the following equality holds:
\begin{equation}\label{pull-backint}
f^\ast \Zb_1 \cdot f^\ast \Zb_2 = \deg\!f  \;\Zb_1 \cdot \Zb_2.
\end{equation}

\end{corollary}

\begin{proof}
This follows from the first part of Proposition \ref{ProjForAr} applied to $\Zb = \Zb_1$ and to $\Zb' = f^\ast \Zb_2$, and from Proposition \ref{pushpullAr} applied to $\Zb = \Zb_2$.
\end{proof}

Since a suitable positive multiple of an Arakelov divisor is an Arakelov-Cartier divisor, it is possible to define $f^\ast \Zb$ as a ``$\Q$-Arakelov divisor" on $X'$ for any Arakelov divisor $\Zb$ on $X$. The intersection pairing and the direct image maps also extends to  $\Q$-Arakelov divisors, and Propositions \ref{pushpullAr} and \ref{ProjForAr} and Corollary \ref{degintAr} remain valid in this more general setting.

\subsection{Application: bounding the degree of morphisms between arithmetic surfaces}

Using the formalism of arithmetic intersection theory developed in the previous sections, the statements and the proofs  of Propositions \ref{proposition:main-geom} and \ref{proposition:main-geomAmplif} may be transposed to the framework of arithmetic surfaces and  Arakelov intersection theory.  
 
 \begin{theorem}\label{theorem:main-Arakelov}Let $U$ and $V$ be two integral normal arithmetic surfaces, and let:
 $$f: V \lra U$$ be a dominant morphism of schemes. Let $\Bb:= (B, g_B)$ be an  Arakelov divisor in $\Zb^1_c(V)^\Ld$, and let: 
$$\Ab = f_\ast \Bb := (f_\ast B, f_{\C \ast} g_B) \in \Zb^1_c(U)^\Ld$$
be its direct image by $f$.

If the self intersection $\overline{B}\cdot \overline{B}$ is positive, then the degree $\deg f := [\kappa(V): f^\ast\kappa(U)]$ of $f$ satisfies the following inequality:
\begin{equation}\label{equation:inequality-degreeAr}
\deg f \leq \frac{\Ab \cdot \Ab}{\Bb \cdot \Bb}.
\end{equation}
  \end{theorem}

\begin{proof} As indicated above, the proof of Theorem~\ref{theorem:main-Arakelov} follows the same lines as the proofs of Propositions \ref{proposition:main-geom} and \ref{proposition:main-geomAmplif}. Since Theorem~\ref{theorem:main-Arakelov}  plays a central role in this memoir, we provide some details.

Firstly we may assume that $U$ and $V$ are respectively open subschemes of some integral normal \emph{projective} arithmetic surfaces, and that $f$ extends to a morphism from $Y$ to $X$, which we will still denote by $f$. \emph{Mutatis mutandis}, this follows from the construction in the proof of Proposition  \ref{proposition:main-geomAmplif}.

After possibly multiplying $\Bb$ and therefore $\Ab$ by some positive integer --- which leaves the right-hand side of  \eqref{equation:inequality-degreeAr} unchanged --- we may also assume that $\Ab$ is an Arakelov-Cartier divisor. 

Let us denote:
$$\delta:= \deg f.$$  For every Arakelov-Cartier divisor $\Db$ on $X$, we have:
$$f^\ast \Ab \cdot f^\ast \Db = \delta \Ab \cdot \Db \quad \mbox{and} \quad \Bb \cdot f^\ast \Db = f_\ast \Bb \cdot \Db,$$ 
as a consequence of Corollary  \ref{degintAr} and of the projection formula in Proposition \ref{ProjForAr}. Consequently we have:
\begin{equation}\label{equation:equal-zeroBis}
(f^{*}\Ab-\delta \Bb)\cdot f^{*}\Db=f^{*}\Ab\cdot f^{*}\Db-\delta \, \Bb\cdot f^{*}\Db=\delta\, \Ab\cdot \Db -\delta\,  f_{*}\Bb\cdot \Db=0.
\end{equation}

The projection formula also implies the equality:
$$f^\ast \Ab \cdot \Bb =\Ab \cdot f_\ast \Bb.$$
Therefore, applied to $\Db=\Ab$, \eqref{equation:equal-zeroBis} implies:
\begin{equation}\label{equation:compare-deltaBis}
(f^{*}\Ab-\delta \Bb)\cdot(f^{*}\Ab-\delta \Bb)=-\delta \, f^{*}\Ab\cdot \Bb + \delta^2 \,  \Bb\cdot \Bb= -\delta \, \Ab\cdot f_\ast \Bb + \delta^2 \,  \Bb\cdot \Bb= -\delta\, \Ab\cdot \Ab+\delta^{2}\, \Bb\cdot \Bb.
\end{equation}

Let us choose an Arakelov-Cartier divisor $\Hb$ over $X$ with positive self-intersection. Again as a consequence of Corollary \ref{degintAr}, we have:
\begin{equation}\label{equation:ample-positiveBis}
f^{*}\Hb\cdot f^{*}\Hb=\delta \, \Hb\cdot \Hb>0.
\end{equation}
Moreover, according to  \eqref{equation:equal-zeroBis} applied to $\Db=\Hb$, we have:
\begin{equation}\label{equation:orthBis}
(f^{*}\Ab-\delta \Bb)\cdot f^{*}\Hb=0.
\end{equation}

From \eqref{equation:ample-positiveBis}, \eqref{equation:orthBis} and the Hodge index inequality on the arithmetic surface $Y$ (cf. Proposition \ref{HodgeIndAr}), we obtain:
$$(f^{*}A-\delta B)\cdot(f^{*}A-\delta B)\leq 0.$$
Together with \eqref{equation:compare-deltaBis}, this proves \eqref{equation:inequality-degreeAr} when $\Bb\cdot \Bb$ is positive.
\end{proof}

\chapter{Green functions with $\mathcal C^{\bD}$ regularity}\label{chapterCbD}

A Green function with $L^2_1$ regularity associated to a divisor $D$ on a Riemann surface $M$ is a locally $L^p$ function on $M\setminus|D|$ for every $p \in (1, +\infty)$, but is possibly not essentially bounded on every non-empty open subset of $M$. 
Accordingly its value at a specific point of $M\setminus|D|$ is not defined in general. 

For this reason our initial definition of the Arakelov intersection number by  the equality \eqref{defArIntBis}, namely:
$$
(D,g) \cdot (D',g')    : = \dega \big(\cOb_X(D,g) \vert D'\big) + \int_{X(\C)} g'\,  \omega(g),
$$
does not make sense in general when the Green function $g$ is only assumed to have $L^2_1$ regularity. Indeed, according to the above observation, it is already the case of the height $\dega \big(\cOb_X(D,g) \vert D'\big),$ and the integral $\int_{X(\C)} g'\,  \omega(g)$ also does not make sense for a general Green function $g$ with $L^2_1$ regularity.

In this chapter, we introduce a class of regularity for Green functions and Hermitian metrics, the  $\mathcal C^{\bD}$ regularity, that is intermediate between the $\cC^\infty$ and the $L^2_1$ regularity. The above expression for the intersection pairing of Arakelov divisors with  $\cC^\infty$  Green functions  will still be valid when working with Arakelov divisors with $\mathcal C^{\bD}$ Green functions. At the same time,  $\mathcal C^{\bD}$ regularity will be much more flexible than $\cC^\infty$ regularity. Notably, like $L^2_1$ regularity, it  will be compatible with direct images by finite analytic morphisms between Riemann surfaces. Moreover the Green functions associated to compact Riemann surfaces with boundary, which will play a key role in the next chapters, turn out to have $\mathcal C^{\bD}$ regularity.

\section{The spaces $\cC^{\bD}(M)$ and $\bM^{\cp}(M)$}
In this section, we denote by $M$ a Riemann surface.

\subsection{Definitions} 
\begin{definitions} We denote by $\mathcal C^{\bD}(M)$ the space of continuous functions $f:M \ra \R$ such that the current  $i\partial\overline\partial f$ is a Radon measure on $M$.

We denote by $\bM^{\cp}(M)$ the space of real Radon measures $\mu$ on $M$ that may be written locally on $M$ as $i\partial\overline\partial f$
for some continuous function $f$. 
\end{definitions}

The space $\mathcal C^{\bD}(M)$  may be thought of as a complex analogue of the space $\mathcal C^{\mathrm{bv}}(U)$ of continuous functions with bounded variation on an open subset $U$ of $\R$. Indeed a  continuous function on $U$ has bounded variation if and only if the current $df$ is a measure.

The superscript $\cp$ in $\bM^{\cp}(M)$ stands for {\bf c}ontinuous {\bf p}otential.

\subsection{Basic properties}

In this subsection, we state some basic properties of the spaces $\mathcal C^{\bD}(M)$ and $\bM^{\cp}(M)$. These properties are consequences of the basic properties of the operator $\partial\overline\partial$ on a Riemann surface, and their proofs will be left to the reader. 

\begin{proposition}
A continuous real-valued function $f$ on $M$ belongs to $\mathcal C^{\bD}(M)$ if and only if the current $i\partial\overline\partial f$ belongs to $\bM^{\cp}(M)$. 

If $M$ is compact and connected, we have an exact sequence of $\R$-vector spaces:
$$0\lra \R\, \hlra \, \mathcal C^{\bD}(M)\stackrel{i\partial\overline\partial}{\lra} \bM^{\cp}(M)\stackrel{\int_M}{\lra} \R\lra 0.$$

If $M$ is connected and not compact, we have exact sequences of $\R$-vector spaces:
$$0\lra \mathcal H(M)\, \hlra \, \mathcal C^{\bD}(M)\stackrel{i\partial\overline\partial}{\lra} \bM^{\cp}(M)\lra 0$$
and 
$$0\lra \mathcal C_c^{\bD}(M)\stackrel{i\partial\overline\partial}{\lra} \bM_c^{cp}(M)\stackrel{\int_M}{\lra} \R\lra 0,$$
where $\mathcal H(M)$  denotes the space of real-valued harmonic functions on $M$, and $\mathcal C_c^{\bD}(M)$ (resp. $\bM_c^{cp}(M)$) the subspace of $\mathcal C^{\bD}(M)$ (resp. of $\bM^{\cp}(M)$) defined by its elements with compact support.
\end{proposition}

\begin{proposition}\label{CbDLd} A function $f$ in $\mathcal C^{\bD}(M)$ is locally $L^{2}_{1}$ on $M$. A measure $\mu$ in $\bM^{\cp}(M)$ is a locally $L^{2}_{-1}$ current on $M$. Moreover, when $\supp f \cap \supp \mu$ is compact,  the pairing: $$\int_{M}f\, \mu$$ of $f$ and $\mu$ in which $f$ is considered as a function on $M$ and $\mu$ as a measure coincides with the pairing of $f$ and $\mu$ considered as $L^{2}_{1}$ and $L^{2}_{-1}$ currents respectively.
\end{proposition}

In particular, when $M$ is compact, the Dirichlet scalar product of two elements 
$f_{1}$ and $f_{2}$ in $\mathcal C^{\bD}(M)$ 
satisfies:
$$\langle f_{1}, f_{2}\rangle_{\mathrm{Dir}} =i\int_{M}\partial f_{1}\wedge \overline\partial f_{2} =\int_{M}f_{1}\, i\partial \overline\partial f_{2}  = \int_{M}f_{2}\, i\partial \overline\partial f_{1}.
$$
where the last two integrals are to be understood as the pairing of the continuous function $f_{\alpha}$ and the real measure $i\partial\overline\partial f_{\beta}$, with $\{\alpha, \beta\}=\{1, 2\}$.

The following additional properties of the spaces $\mathcal C^{\bD}(M)$ and $\bM^{\cp}(M)$ will be useful when constructing $\mathcal C^{\bD}(M)$ Green functions.

\begin{proposition}\label{Addbd}
(1) The $\R$-vector space $\mathcal C^{\bD}(M)$ is a subalgebra of the $\R$-algebra $\mathcal C^{0}(M)$ of continuous functions on $M$.

(2) Let $I$ be an interval of $\R$ and let $f$ be an element of $\mathcal C^{\bD}(M)$ with values in $I$. If $\chi : I\ra\R$ is a function of class\footnote{namely a $\cC^1$-function with locally Lipschitz derivative.} $\mathcal C^{1+\mathrm{Lip}}$, then the composition $\chi\circ f$ belongs to $\mathcal C^{\bD}(M)$.

(3) If $\mu$ is a measure in $\bM^{\cp}(M)$ and $\rho$ is a function in $\mathcal C^{\bD}(M)$ such that the measure $i\partial\overline\partial \rho$ is absolutely continuous with respect to the Lebesgue measure\footnote{or, equivalently, if $i\partial\overline\partial\rho$ is defined by a $2$-form that is locally $L^{1}$.}, then the measure $\rho\mu$ lies in $\bM^{\cp}(M)$.

\end{proposition}

\section{Green functions  with $\mathcal C^{\bD}$  regularity and $\ast$-products}

In this section, we denote by $M$ a Riemann surface.

\subsection{Green functions  with $\mathcal C^{\bD}$  regularity and Hermitian line bundles with $\mathcal C^{\bD}$ metric}

\subsubsection{}We  define a \emph{Green function with $\mathcal C^{\bD}$ regularity} or, for short, a \emph{$\mathcal C^{\bD}$ Green function} for a divisor $D$ on $M$ is defined
as in \ref{Greensmoothdef} by allowing the function $h$ appearing in \eqref{equation:condition-green} to be locally $\mathcal C^{\bD}$ over $U$, instead of being $\mathcal C^{\infty}$.

According to the ellipticity of operator $\partial \overline{\partial}$ on $M$, a real distribution $g$ on $M$ is a $\mathcal C^{\bD}$ Green function for $D$ on $M$ if and only if the current: 
$$\omega(g):=\frac{i}{\pi}\partial\overline\partial g+\delta_{D}$$
is a measure in $\bM^{\cp}(M)$.

According to Proposition \ref{CbDLd}, a Green function with  $\mathcal C^{\bD}$ regularity is  a Green function with $\Ld$ regularity.

\subsubsection{}\label{CbDvaria} Let $\Lb = (L, \Vert.\Vert)$ be a Hermitian line bundle over $M$. We will say that the Hermitian metric is $\mathcal C^{\bD}$ when it satisfies the following equivalent conditions, where $\Vert.\Vert_0$ denotes a $\cC^\infty$ metric on $L$:
\begin{enumerate}[(i)]
\item there exists a function $f \in \cC^\bD (M)$ such that $\Vert .\Vert = e^{- f} \Vert. \Vert_0;$
\item for every non-vanishing $\cC^\infty$ section $s$ of $L$ over some open subset $U$ of $M,$ the continuous function $\Vert s \Vert$ on $U$ belongs to $\cC^\bD(U)$.
\end{enumerate}

When this holds, the ``first Chern form" $c_1(\Lb)$ of $\Lb$, defined locally by the equality \eqref{Chernformdef}, is a measure in $\bM^{\cp}(M)$. With the notation of (i), we have:
$$c_1(\Lb) = c_1(L, \Vert.\Vert_0) + \frac{i}{\pi} \partial \overline{\partial} f. $$

The datum of a Green function with $\mathcal C^{\bD}$ regularity for a divisor $D$ in $M$ is equivalent to the datum of a $\mathcal C^{\bD}$ Hermitian metric on $\cO_M(D)$. As in \ref{DefGrenCinfty}, the Green function $g$ is associated to the metric $\Vert.\Vert_g$ such that:
 $$\Vert \mathbbm 1_{D}\Vert_g=e^{-g} \quad \mbox{on $M \setminus \vert D \vert.$}$$
 
 Moreover  the definition of the capacitary metrics in \ref{CapMetricDef} still makes sense for Green function with $\mathcal C^{\bD}$ regularity.

\subsection{Approximating Green functions  with $\mathcal C^{\bD}$  regularity by $\mathcal C^{\bD}$ functions}\label{Approx}

\subsubsection{} The following approximation results will play a key role when dealing with  the $\ast$-products of Green functions with $\mathcal C^{\bD}$ regularity.
\begin{proposition}\label{proposition:define-int}
Let $D=\sum_{P}n_{P}P$ be a divisor in $M$ with finite support, let $g$ be a $\mathcal C^{\bD}$ Green function with compact support for $D$ in $M$, and let
$\mu$ be a measure in $\bM^{\cp}(M).$ 

(1) There exists a sequence $(g_{n})_{n\geq 1}$ in $\mathcal C^{\bD}_{c}(M)$ such that:
\begin{enumerate}[a.]
\item \label{item:coincide}for any neighborhood $V$ of $|D|$ in $M$, we have: 
$$g_n = g \quad \mbox{on $M\setminus V$}$$ for $n$ large enough; 
\item  \label{item:weakly-converges} the sequence of measures $\big(i\partial\overline\partial (g_{n}-g)\big)$ converges weakly to $0$;
\item  for every $P$ in $|D|$ such that $n_{P}>0$ (resp. $n_{P}<0$), the sequence $(g_{n})$ is increasing (resp. decreasing) on some neighborhood of $P$ in $M$.
\end{enumerate}

(2)  For every sequence $(g_{n})$ in $\mathcal C_c^{\bD}(M)$ satisfying conditions \ref{item:coincide} and \ref{item:weakly-converges}, the limit: 
\begin{equation}\label{equation:limit-mu}
\lim_{n\ra+\infty}\int_{M}g_{n}\, \mu
\end{equation}
exists in $\R$. It is independent of the choice of the sequence $(g_{n})$ in $\mathcal C_c^{\bD}(M)$ satisfying conditions \ref{item:coincide} and \ref{item:weakly-converges}.
\end{proposition}

Regarding condition \ref{item:weakly-converges}, observe that the current:
\begin{equation}\label{signmeasures}
i\partial\overline\partial (g_{n}-g)=i\partial\overline\partial g_{n}-\pi\omega(g)+\pi\delta_{D}
\end{equation}
is indeed a real measure on $M$. The sequence of measures $\big(i\partial\overline\partial (g_{n}-g)\big)$ is said to converge weakly to $0$ when, for  every $\phi$ in $\mathcal C^{0}_{c}(M)$, we have:
\begin{equation}\label{equation:limit-mubis}
\lim_{n\ra+\infty}\int_{M}\phi \, i\partial\overline\partial (g_{n}-g)=0. 
\end{equation}
 This holds precisely when the sequence of total masses $\big(||| i\partial\overline\partial (g_{n}-g)|||\big)$ of the signed measures \eqref{signmeasures} is bounded, and when \eqref{equation:limit-mubis} holds for every $\phi$ in $\mathcal C^{\infty}_{c}(M)$.

\begin{proof}
(1)  The sequence $(g_n)$ may be constructed by a standard truncation procedure. Namely,  for every point $P$ in  $|D|$,  choose a  complex analytic chart: 
$$z_P : U_P\lrasim D(0, 1)$$
defined on some open neighborhood $U_P$ of $P$ in $X$, such that $z_P(P)=0.$ We may assume that  the open subsets $(U_P)_{P \in \vert D\vert}$ of $X$ are pairwise disjoint. For every $P$ in $\vert D\vert$, the restriction of the Green function $g$ to $U_P$ may be written:
$$g_{|U_P}=n_P \log |z_P|^{-1} + h_P$$
for some function $h_P$ in $\mathcal C^{\bD}(U_P).$ 

For any real number $r$ with $0<r<1$, we  define a function $\widetilde g_r$ in $\mathcal C^{\bD}(M)$ by letting:
\begin{align*} 
\widetilde g_r(x)& :=n_P\min(\log r^{-1}, \log |z_P|^{-1})+h_P(x)   \quad  \quad \mbox{if $x\in U_P$}, \\
& :=g(x) \quad  \quad \quad \quad \mbox{if $x\in M\setminus\coprod_{P\in|D|}U_P.$}
\end{align*}
Indeed, over $U_P$, the following equality of currents holds:
$$\frac{i}{\pi}\partial\overline\partial \widetilde g_r = n_P\frac{i}{\pi}\partial\overline\partial \log^{-}\frac{r}{|z_P|}+\frac{i}{\pi}\partial\overline\partial h_P$$
and, consequently, over $M$, we have:
$$\frac{i}{\pi}\partial\overline\partial(\widetilde g_r-g)=\sum_{P\in|D|} n_P(\delta_P-z_P^*\mu_r),$$
where $\mu_r$ is the measure on the open unit disk $D(0, 1)$ defined by:
$$\mu_r(\phi):=\int_0^1\phi(e^{2i\pi t} \, r)\,  dt.$$

Clearly, the measure $i\partial\overline\partial(\widetilde g_r-g)$ converges weakly to $0$ as $r$ goes to $0$, and, for every decreasing sequence $(r_n)$ in $(0,1)$ converging to $0$, the sequence $(g_n):=(\widetilde g_{r_n})$ satisfies Conditions 
a., b., and c..

(2) Consider a sequence $(g_n)_{n\geq 1}$ in $\mathcal C^{\bD}_c(M)$ satisfying conditions \ref{item:coincide} and \ref{item:weakly-converges}. Let us choose some complex analytic charts $(z_P)_{P\in |D|}$ as in (1) and, for every $P\in|D|$, choose $\phi_P$ in $\mathcal C^{\bD}(U_P)$ such that: 
$$\mu_{|U_P}=i\partial\overline\partial\phi_P.$$

According to \ref{item:coincide}, there exists a positive integer $N$ and a compact subset $K$ of $\bigcup_{P\in|D|} U_P$ such that, for any $n\geq N$, $g$ coincides with $g_n$ on $M\setminus K$. In particular, if $n$ and $n'$ are two integers with $n, n'\geq N$, the function $(g_{n'}-g_n)_{|U_P}$ is compactly supported, and therefore, according to Green's formula:
\begin{equation}\label{gnnprime}
\int_{U_P}(g_{n'}-g_n)\, \mu=\int_{U_P}(g_{n'}-g_n)\, i\partial\overline\partial\phi_P=\int_{U_P}\phi_P\,i\partial\overline\partial(g_{n'}-g_n).
\end{equation}

According to condition \ref{item:weakly-converges}, the compactly supported measure $i\partial\overline\partial(g_{n'}-g_n)$ on $U_P$ converges weakly to $0$ when $\min(n, n')$  goes to infinity.  Together with \eqref{gnnprime}, this proves that $\int_{U_P}(g_{n'}-g_n)\mu$ converges to $0$ when $\min(n, n')$  goes to infinity, and therefore that  the sequence $\big(\int_{U_P} g_n\mu\big)$ converges in $\R$. 

Since, for every $n\geq N$, we have:
$$\int_M g_n\, \mu=\sum_{P\in |D|}\int_{U_P} g_n\, \mu + \int_{M\setminus \bigcup_{P\in|D|} U_P} g\, \mu,$$
this proves the existence of the limit \eqref{equation:limit-mu}. Its independence of the choice of a sequence $(g_n)_{n\geq 1}$ satisfying \ref{item:coincide} and \ref{item:weakly-converges} is clear.
\end{proof}

\subsubsection{}\label{CbDIntegral} With the notation of Proposition \ref{proposition:define-int}, we may \emph{define} the quantity $\int_{M}g\, \mu$ as follows:
\begin{equation}\label{equation:define-int}
\int_{M}g\,\mu:=\lim_{n\ra+\infty}\int_{M}g_{n}\, \mu.
\end{equation}

One should beware that, in spite of the existence of the limit in the right-hand side of \eqref{equation:define-int}, the Green function $g$ might not be integrable with respect to the positive measure $|\mu|$. However, this kind of pathology is seldom encountered 
in practice, thanks to the following  consequence of Proposition \ref{proposition:define-int}:

\begin{corollary}\label{corollary:sign-ok}
With the notation of Proposition \ref{proposition:define-int}, assume that every element $P$ of $|D|$ admits a neighborhood $U$ in $M$ such that either $\mu_{\mid U}$, the restriction  of $\mu$ to $U$, or $-\mu_{\mid U}$ is  a positive measure. Then the Green function $g$ is integrable with respect to $|\mu|$, and the integral $\int_{M}g\, \mu$ of $g$ with respect to $\mu$ coincides with the limit of the sequence $\left(\int_{M}g_{n}\, \mu\right).$
\end{corollary}

\begin{proof}
It is enough to show that, if $P$ is any point of $|D|$, the Green function $g$ is integrable with respect to $|\mu|$  on some neighborhood $U$ of $P$ in $M$. To prove this, we may assume that the divisor $D$ is actually $P$ itself. In this situation, the Corollary follows from Proposition \ref{proposition:define-int} (2) applied to a sequence $(g_n)$ satisfying conditions a., b., and c. and from Lebesgue's Monotone Convergence Theorem.
\end{proof}

\subsection{The $\ast$-product of Green functions with $\mathcal C^{\bD}$ regularity}
\begin{proposition}\label{proposition:compute-lim}
Let $D$ be a divisor in $M$ with finite support, let $g$ be a $\mathcal C^{\bD}$ Green function for $D$ with compact support, and let $\phi$ be a function in $\mathcal C^{\bD}(M)$. Then the following equality holds:
$$\int_M g\,\frac{i}{\pi}\partial\overline\partial\phi=\int_M\phi\, \omega(g)-\int_M\phi\, \delta_D,$$
where the left-hand side is defined by the limit procedure \eqref{equation:define-int} with $\mu=\frac{i}{\pi}\partial\overline\partial\phi.$
\end{proposition}

\begin{proof}
Choose a sequence $(g_n)_{n\geq 1}$ satisfying conditions \ref{item:coincide} and \ref{item:weakly-converges} in Proposition \ref{proposition:define-int}. For any positive integer $n$, we have:
\begin{equation}\label{equation:compute-lim}
\int_M g_n\,\frac{i}{\pi}\partial\overline\partial\phi=\int_M \phi \, \frac{i}{\pi}\partial\overline\partial g_n=\int_M \phi\, \frac{i}{\pi}\partial\overline\partial(g_n-g)+\int_M\phi\, (\omega(g)-\delta_D).
\end{equation}
According to  b., the quantity $\int_M \phi\,\frac{i}{\pi}\partial\overline\partial(g_n-g)$ in \eqref{equation:compute-lim} converges to $0$ when $n$ goes to infinity.
\end{proof}

\begin{proposition}\label{proposition:star-cbd} Let $D_1$ and $D_2$ be two divisors on $M$ with disjoint supports, and let $g_1$ and $g_2$ be two Green functions with $\mathcal C^{\bD}$ regularity for $D_1$ and $D_2$. If $\supp g_1 \cap \supp g_2$ is compact, then the following equality holds:
\begin{equation}\label{equation:star-explicit}
\int_{M}g_{1}\ast g_{2}=\int_{M}g_{2}\, \delta_{D_{1}}+\int_{M}g_{1}\, \omega(g_{2}).
\end{equation}
\end{proposition}

The left-hand side of \eqref{equation:star-explicit} is defined by means of the construction in Subsection \ref{Ldstarprod}, since $g_1$ and $g_2$ are Green functions with  $\Ld$ regularity.

In the right-hand side of \eqref{equation:star-explicit}, the first integral is well-defined since $g_{2}$ is continuous on $M\setminus|D_{2}|$, which contains $|D_{1}|$. The second integral is defined by the equality \eqref{equation:define-int}. Corollary \ref{corollary:sign-ok} shows that it is the actual integral of an integrable function when the measure $\omega(g_{2})$ has a well-defined sign on some neighborhood of any point of $|D_{1}|.$

\begin{proof} Using the validity of Property  $\mathbf{SP}_1$ in the $\Ld$ framework discussed in \ref{SPLd} above and Proposition \ref{Addbd}, one readily sees that, to prove \eqref{equation:star-explicit}, one may assume that $g_1$ and $g_2$ are compactly supported. Then the validity of \eqref{equation:star-explicit} easily follows from its validity when $g_1$ and $g_2$ are Green functions with $\cC^\infty$ regularity, and from Property $\mathbf{SP}_2$ and Propositions \ref{CbDLd} and \ref{proposition:compute-lim}. We  leave the details to the interested reader.
\end{proof}

\section{Examples}

\subsection{Equilibrium potentials on compact Riemann surfaces with boundary}\label{CompactBoundEqu}

\subsubsection{} As in \cite[10.5.1]{Bost2020}, we define a \emph{Riemann surface with boundary} as a pair $(V, V^{+})$ where $V^{+}$ is a (germ of a) Riemann surface (without boundary) along a closed $\mathcal C^{\infty}$ submanifold with boundary $V$ of $V^{+}$, of codimension $0$ in $V^{+}$.  

The \emph{interior} of $(V, V^{+})$ is defined as the interior $\mathring{V}$ of $V$, and its \emph{boundary} is:
$$\partial V:=V\setminus V^{+}.$$
We say that $(V, V^{+})$ is compact (resp. connected) when $V$ is. 

For simplicity's sake, we shall often write $V$ for the Riemann surface with boundary $(V, V^{+}),$ but shall denote by: 
$$\alpha: V^+ \lra N$$
a  map from $(V, V^+)$ to some manifold $N$ to emphasize that $\alpha$ is a germ of map along $V$. For instance, when $N$ is a complex analytic manifold, an analytic map $\alpha$ as above will be a map analytic up to the boundary.

\subsubsection{} Let $V$ be a \emph{connected compact} Riemann surface with boundary such that $\partial V$ \emph{is non\-emp\-ty}, and let $P$ be a point of the interior $\mathring V$ of $V$. 

One defines the \emph{Green function} $g_{V, P}$ of $P$ in $V$ as the unique function:
$$g_{V, P} : V^{+}\setminus\{P\}\lra \R$$
satisfying the following conditions:
\begin{enumerate}[(i)]
\item $g_{V, P}$ is continuous on $V^{+}\setminus\{P\}$ and vanishes on $V^{+}\setminus \mathring V;$
\item $g_{V, P}$ is harmonic on $\mathring V\setminus\{P\};$
\item $g_{V, P}$ admits a logarithmic singularity at $P$;  namely, if $z$ denotes a uniformizing parameter on $V$ at $P$, the difference $g_{V, P}-\log|z|^{-1}$ is bounded on some (pointed) neighborhood of $P$ in $\mathring V$.
\end{enumerate}

See for instance instance  \cite[5.1-2]{Taylor11}, or \cite[Appendix]{Bost99} for a construction in a more general setting. 

The function $g_{V, P}$ is a Green function with $\cC^\bD$ regularity for $P$ in $V^{+}$, supported on $V$ and  the measure: 
$$\mu_{V, P}:=\omega(g, P)=\frac{i}{\pi}\partial\overline\partial g_{V, P}+\delta_{P}$$
is supported on the boundary $\partial V$. 
In particular, the restriction of $g_{V, P}$ to the Riemann surface $\mathring V$ is a Green function with $\mathcal C^{\infty}$ regularity for $P$ in $\mathring V$. A straightforward application of the maximum principle for harmonic functions shows that $g_{V, P}$ is positive on $\mathring V$.

The measure $\mu_{V, P}$ is the so-called \emph{harmonic measure} attached to the point $P$ in $V$. It is a probability measure defined by a positive $\mathcal C^{\infty}$ density on the smooth compact curve $\partial V$. 

The capacitary metric on $T_{P}V$ associated to $g_{V, P}$ will be denoted by: 
$$\Vert .\Vert_{V, P}^{\mathrm{cap}}:=\Vert .\Vert_{g_{V, P}}^{\mathrm{cap}}.$$

\begin{example}\label{example:disk}
When $(V, P)=(\overline D(0, 1), 0)$, we have:
$$g_{V, P}=\log^{+} |z|^{-1}$$
and therefore:
$$\Vert (\partial/\partial z)_{|P}\Vert^{\mathrm{cap}}_{\overline D(0, 1), 0}=1.$$
\end{example}


\subsection{Some integrability results} 

\subsubsection{} Let $\Omega$ be an open neighborhood of $\Db(0,1)$ in $\C$ and let $\mu$ be a measure in $\bM^\cp(\Omega)$. 

We may apply Proposition \ref{proposition:define-int} to the Green function  with $\mathcal C^{\bD}$ regularity for the point $0$ in $\C$:
$$g := \log^+ \vert z \vert^{-1},$$
introduced in Example \ref{example:disk},  and to its truncations:
$$g_n := \min(g, \log r_n^{-1}),$$
where $(r_n)$ denotes a decreasing sequence in $(0,1)$, of limit $0$.
We obtain the existence of the limit:
$$\int_{\Db(0,1)} \log \vert z \vert^{-1} \, d\mu(z) := \lim_{n \ra +\infty} \Big(\int_{\Db(0, 1) \setminus \Db(0, r_n)} \log \vert  z \vert^{-1} \, d\mu(z) + \log r_n^{-1} \mu(\Db(0,r_n)) \Big).$$ 
Moreover, if $\mu$ is positive, then $\log \vert z \vert^{-1}$ is $\mu$-integrable near $0$, as already observed in Corollary~\ref{corollary:sign-ok}.

\subsubsection{} Any continuous subharmonic function on a Riemann surface $M$, and consequently the difference $f$ of two such functions, is an element of $\mathcal C^{\bD}(M)$, by the very definition of $\mathcal C^{\bD}(M)$.  

Moreover, if $z$ is a local coordinate on $M$, the function $\log |z|^{-1}$ is locally integrable with respect to the measure $|\mu|$ where: 
$$\mu:=\frac{i}{\pi}\partial\overline\partial f.$$
This  follows directly from Corollary \ref{corollary:sign-ok}.

\subsection{Some wild $\cC^\bD$ functions on $\mathring D(0,1)$} Let us emphasize that there exist instances of $\mathcal C^{\bD}$ functions $f$ for which the above property of integrability of $\log \vert z \vert^{-1}$ with respect to the measure $\vert \partial\overline\partial f \vert$ does not hold. Such functions may not be written locally as the difference of two continuous subharmonic functions. 

For instance, let $\alpha$ be a real number with $1<\alpha<2$ and consider the function $\phi_{\alpha} : \mathring D(0,1)\ra\R$ defined by:
$$\phi_{\alpha}(z)=(\log |z|^{-1})^{-\alpha}\cos(\log |z|^{-1}) \quad \mbox{if $z\in \mathring D(0,1)\setminus\{0\}$},$$
and: 
$$\phi_{\alpha}(0)=0.$$
Then $\phi_\alpha$ belongs to $\CbD(\mathring D(0,1)),$ but $\log |z|^{-1}$ is not integrable with respect to $|\mu_{\alpha}|$ near $0$, where: 
$$\mu_{\alpha}:=\frac{i}{\pi}\partial\overline\partial\phi_{\alpha},$$

Relatedly, for every $\alpha$ with $1<\alpha<2$, the measure $\widetilde\mu_{\alpha}$ on $\mathring D(0,1)$ defined by:
$$\widetilde\mu_{\alpha}:=|z|^{-2}(\log |z|^{-1})^{-\alpha}\cos(\log |z|^{-1}) \, dx \, dy$$
is an element of $\bM(\mathring D(0,1))^{\cp}$ such that $\log |z|^{-1}$ is not integrable with respect to $|\widetilde\mu_{\alpha}|$ near $0$, and cannot be written, on any open neighborhood $\Omega$ of $0$ in $\mathring D(0,1)$, as the difference of two positive measures in $\bM(\Omega)^{\cp}$.

We leave the proofs of these  assertions to the interested reader.

\section{Functoriality  of $\mathcal C^{\bD}$ Green functions}\label{subsection:direct-inverse}

\subsection{Functoriality of functions and measures on Riemann surfaces}\label{FunctbisMeasRS}
 In paragraph \ref{FunctCurrents}, we recalled the functorial constructions --- the pull-back $f^\ast$ of $\cC^\infty$-forms and the push-forward $f_\ast$ of currents --- associated to a morphism  $f: X \ra Y$ of (complex) manifolds.

The pull-back map $f^{*}$ extends to differential forms with continuous coefficients. Dually, the push-forward $f_{*}T$ of a current with measure coefficients on $X$ is well-defined as a current with measure coefficients on $Y$.

We are interested in the properties of $f^{*}$ and $f_{*}$ in the situation where $X$ and $Y$ are Riemann surfaces and the holomorphic map $f$ is ``nowhere locally constant," namely, it is non-constant when restricted to any connected component of $X$. In this case, the maps $f^{*}$ and $f_{*}$ satisfy the following additional properties.

Consider a continuous function $\phi$ on $X$ such that $f_{\mid \supp \phi}$ is proper. 
Then the push-forward $f_{*}\phi$ is (the current defined by) a continuous function on $Y$, namely the function defined by: 
\begin{equation}\label{fastcont}
f_{*}\phi(y)=\int_{X}\phi\, \delta_{f^{*}(y)}.
\end{equation}
Here $f^{*}(y)$ denotes the divisor on $X$ inverse image by $f$  of the divisor $y$ in $Y$, namely:
$$f^{*}(y):=\sum_{x\in f^{-1}(y)}e_{x} \, x,$$
where $e_{x}$ is the ramification index of $f$ at $x$. If $\phi$ has compact support, then $f_{*}\phi$ has compact support. The expression \eqref{fastcont} for $f_\ast \phi$ also shows that, if $f$ is proper of degree $\delta $, then for any continuous function $\psi$ on $Y$, we have:
$$f_\ast f^\ast \psi = \delta \psi.$$

Dually, if $\mu$ denotes a measure on $Y$, we may define its inverse image $f^{*}\mu$ on $X$ as the measure such that, for every continuous function $\phi$ on $X$ with compact support, the following equality holds:
\begin{equation}\label{fastmu}
\int_{X}\phi \, f^{*}\mu=\int_{Y}f_{*}\phi\, \mu.
\end{equation}

The construction of $f^{*}\mu$ is compatible with the weak topology on the space of measures, which is defined by duality with the space of continuous functions with compact support. Moreover, when $\mu$ is a $2$-form with continuous coefficients on $Y$, the pull-back measure $f^{*}\mu$ coincides with the pull-back of $\mu$ as a differential form. Finally, when the measure $\mu$ defines a $L^2_{-1}$ current, the measure $f^\ast \mu$ defined by \eqref{fastmu} coincides with the $L^2_{-1}$ current $f^\ast \mu$ defined by the construction in Proposition \ref{Funct2}  1) applied to the 
$L^2_{-1}$ current $\mu$. 
\subsection{Functoriality for $\CbD$ functions and  $\bM^{\cp}$ measures}

The  construction of $\CbD$ functions and $\bM^{\cp}$ measures is compatible with the pull-back and push-forward operations on continuous functions and measures on Riemann surfaces described above: 

\begin{proposition}\label{proposition:stability-f}
Let $f : X\ra Y$ be a holomorphic map between two Riemann surfaces. Assume that $f$ is nowhere locally constant.
\begin{enumerate}[(i)]
\item The pull-back maps $f^{*}$ defined above satisfy:
$$f^{*}(\mathcal C^{\bD}(Y))\subset \mathcal C^{\bD}(X)$$
and
$$f^{*}(\bM^{\cp}(Y))\subset \bM^{\cp}(X).$$
Moreover, for any $\psi\in\mathcal C^{\bD}(Y),$ the following equality holds in $\bM^{\cp}(X):$
\begin{equation}\label{alpha}
f^{*}(i\partial\overline\partial\psi)=i \partial\overline\partial f^{*}\psi.
\end{equation}
\item For any $\phi$ in $\mathcal C^{\bD}(X),$ (resp. any $\mu$ in $\bM^{\cp}(X)$) such that the restriction of $f$ to the support of $\phi$ (resp. to the support of $\mu$) is proper, the direct image $f_{*}\phi$ (resp. $f_{*}\mu$) belongs to $\mathcal C^{\bD}(Y)$ (resp. $\bM^{\cp}(Y)$).
\end{enumerate}
\end{proposition}

\begin{proof} To prove (i), it is enough to show that, for every $\psi$ in $\CbD(Y)$, the equality \eqref{alpha} is satisfied, where $f^{*}(i\partial\overline\partial\psi)$ is defined by the relation \eqref{fastmu} and where $i \partial\overline\partial f^{*}\psi$ is defined as a current on $X$. This is a consequence of Corollary \ref{ddbarpullback} and of the compatibility between the pull-back maps $f^\ast$ on continuous functions and measures and on $L^2_1$ functions and $L^2_{-1}$ currents.

Assertion (ii) is a straightforward consequence of the compatibility of $\partial\overline\partial$ with the direct images of currents. This implies the equality of currents:
$$f_{*}(i\partial\overline\partial \phi)=i\partial\overline\partial(f_{*}\phi),$$
and proves that $i\partial\overline\partial(f_{*}\phi)$ is a measure.
\end{proof}

Together with the definition of $\mathcal C^{\bD}$ Green functions, Proposition \ref{proposition:stability-f} implies immediately:
\begin{corollary}\label{corollary:stability-f}
Let $f : X\ra Y$ be a holomorphic map between two Riemann surfaces. Assume that $f$ is nowhere locally constant.
\begin{enumerate}[(i)]
\item If $E$ is a divisor on $Y$ and $g$ is a $\mathcal C^{\bD}$ Green function on $Y$ for $E$, then $f^{*}g$ is a $\mathcal C^{\bD}$ Green function on $X$ for the divisor $f^{*}E$. Moreover, the following equality holds in $\bM^{\cp}(X)$:
$$f^{*}\omega(g)=\omega(f^{*}g).$$
\item If $D$ is a divisor on $X$ and $g$ is a $\mathcal C^{\bD}$ Green function on $X$ for $D$ such that the restriction of $f$ to the support of $g$ is proper, then $f_{*}g$ is a $\mathcal C^{\bD}$ Green function on $Y$ for the divisor $f_{*}D$ on $Y$. Moreover, the following equality holds in $\bM^{\cp}(Y)$:
$$f_{*}\omega(g)=\omega(f_{*}g).$$
\end{enumerate}
\end{corollary}

In $(i)$, the pull-back of $f^{*}g$ is defined a priori as a $\mathcal C^{\bD}$ function on the complement $X\setminus f^{-1}(|E|)$ of the support of $f^{*}E$.

In $(ii)$, the properness of $f$ on the support of $g$ implies that the restriction of $f$ is proper on a neighborhood of $|D|$, which implies that the divisor $f_{*}D$ is well-defined.

Observe also that these constructions of direct and inverse images of Green current with $\CbD$ regularity are compatible with the constructions of direct and inverse images of Green current with $\Ld$ regularity in Proposition \ref{FunctGreenLd}. 

Observe also that part $(i)$ of Corollary \ref{corollary:stability-f} holds for Green functions with $\mathcal C^{\infty}$ regularity instead of $\mathcal C^{\bD}$ Green functions. However, part $(ii)$ does not hold when $f$ is not \'etale, since the direct image of a $\mathcal C^{\infty}$ function by a ramified map is not $\mathcal C^{\infty}$ in general, as discussed in  Subsection \ref{functCinftyGreen}.

Finally the construction in paragraph \ref{CbDIntegral} of the integral $\int_M g \, \mu$ where $g$ is a $\CbD$ Green function with compact support and $\mu$ is a measure in $\bM^{\cp}(M)$, satisfies the following projection formulas:

\begin{proposition}\label{projGbDelta} 
(i) With the  notation of Corollary \ref{corollary:stability-f}, (i), if $\supp f^\ast g$ is compact, then for every measure $\mu$ in $\bM^{\cp}(X),$ the following equality holds:
$$\int_X f^\ast g \, \mu = \int_Y g \, f_\ast \mu.$$

(ii) With the notation of  Corollary \ref{corollary:stability-f}, (ii), if $\supp g$ is compact, then for every measure $\mu$ in $\bM^{\cp}(Y),$ the following equality holds:
$$\int_X g \, f^\ast \mu = \int_Y f_\ast g \, \mu.$$
\end{proposition}

This is a simple consequence of the definitions, and the details of the proof will be left to the reader.

\section{Application to intersection theory on arithmetic surfaces}\label{section:intersection-CbD}

Let $X$ be an integral normal arithmetic surface.

\subsection{Arakelov divisors with $\cC^\bD$ Green functions}

By using Green functions with  $\cC^\bD$ regularity instead of Green functions with $\cC^\infty$ regularity in the definitions in \ref{DefArDiv}, we  define the group $\Zb^1(X)^{\CbD}$ of {Arakelov divisors} on $X$  with $\CbD$ Green functions.

This group lies in between the groups of Arakelov divisors with $\cC^\infty$ and $\Ld$ Green functions:
$$\Zb^1(X) \subset \Zb^1(X)^\CbD \subset \Zb^1(X)^{\Ld}.$$
Similarly we may define its subgroups $\overline{Z}_{\Cart}^1(X)^\Ld ,$ $\overline{Z}_{c}^1(X)^\Ld,$ and $\overline{Z}_{\Cart,c}^1(X)^\Ld,$ and the associated Arakelov-Chow groups:
$$\overline{CH}^1(X)^\CbD:= \overline{Z}^1(X)^\CbD / \overline{\div} \kappa(X)^\times$$
and
$$\overline{CH}_\Cart^1(X)^\CbD= \overline{Z}_\Cart^1(X)^\CbD / \overline{\div}\kappa(X)^\times.$$
These groups 
 lies between the previously defined Arakelov-Chow groups $\overline{CH}^1(X)$ and $\overline{CH}^1(X)^\Ld$, and $\overline{CH}_\Cart^1(X)$ and $\overline{CH}_\Cart^1(X)^\Ld$ respectively.

The isomorphism \eqref{c1hat} extends to an isomorphism:
\begin{equation}\label{c1hatCbD}
\widehat{c}_1: \overline{\Pic}(X)^\CbD \lrasim \overline{CH}_\Cart^1(X)^\CbD,
\end{equation}
where $\overline{\Pic}(X)^\CbD$ denotes the group of isomorphism classes of Hermitian line bundles with $\CbD$ metric on $X$. As for the isomorphism \eqref{c1hat}, the inverse of \eqref{c1hatCbD} sends the class of some Cartier-Arakelov divisor $(Z,g) \in \Zb_\Cart^1(X)^\CbD$ to the isomorphism class of the Hermitian line bundle:
$$\cOb(Z,g) := (\cO_X(D), \Vert.\Vert_g).$$

\subsection{Heights, arithmetic intersection, and $\CbD$ regularity}\label{intCbD}

The following proposition asserts that the relation \eqref{defArInt} --- which was our starting point when developing the arithmetic intersection theory for Arakelov divisors defined by Green functions with $\cC^\infty$ regularity --- still holds in the $\CbD$ framework.

\begin{proposition}\label{ArIntCbD} For every Arakelov divisor $(D',g')$  (resp. $(D,g)$) in $\Zb_\Cart^1(X)^\CbD$ (resp. in $\Zb^1_c(X)^\CbD$), the following equality holds:
\begin{equation}\label{eq:ArIntCbD}
(D',g') \cdot (D,g)    = \dega \big(\cOb_X(D',g') \vert D\big) + \int_{X(\C)} g\,  \omega(g').
\end{equation}
\end{proposition}

The left-hand side of \eqref{eq:ArIntCbD} is defined by the construction in \ref{LdIntAr}. Indeed $(D',g')$  and $(D,g)$ belong to  $\Zb_\Cart^1(X)^\Ld$ and  $\Zb^1_c(X)^\Ld$ respectively. The integral in the right-hand side of \eqref{eq:ArIntCbD} is defined by the construction in \ref{CbDIntegral}. According to Corollary \ref{corollary:sign-ok}, it is an ``actual" integral when the measure $\omega(g')$ has a well-defined sign near each point of $\vert D_\C \vert.$

\begin{proof} The relation \eqref{eq:ArIntCbD} is valid when $g'$ and $g$ have $\cC^\infty$ regularity. Using Propositions \ref{defArTerLd} and \ref{proposition:compute-lim}, one readily sees that it extends to the general case. 
\end{proof}

\subsection{Functoriality properties}

The functoriality properties of $\CbD$ functions  and of Green functions established in Section \ref{subsection:direct-inverse}
immediately imply that the pull-back and push-forward maps constructed in \ref{FunctLd} between groups of Arakelov divisors (resp. Arakelov-Chow groups) with $\Ld$ regularity define by restriction some pull-back and push-forward maps between groups of Arakelov divisors (resp. Arakelov-Chow groups) with $\CbD$-regularity.

\bigskip 

\bigskip


\emph{In the next chapters of this memoir, unless otherwise specified, by ``Green function," we shall mean ``Green function with $\CbD$ regularity,"  by ``Arakelov divisor,"  we shall mean ``Arakelov divisor with $\CbD$ Green function."} \label{CbDeltanow}

\emph{For simplicity, we will write 
$\Zb^1(X),$ $\overline{Z}_{\Cart}^1(X)$, $\overline{Z}_{c}^1(X)$, instead of $\overline{Z}^1(X)^\CbD,$
$\overline{Z}_{\Cart}^1(X)^\CbD,$ $\overline{Z}_{c}^1(X)^\Ld,$ and $\overline{CH}^1(X)$ and $\overline{CH}_\Cart^1(X)$ instead of 
$\overline{CH}^1(X)^\CbD$ and $\overline{CH}_\Cart^1(X)^\CbD$.
}

\chapter{The Archimedean Overflow  $\mathrm{Ex}(\alpha : (V, P)\ra N)$}\label{OverflowArchimede}

In this chapter, we introduce an invariant, the ``overflow'' $\mathrm{Ex}(\alpha : (V, P)\ra N)$, attached to
a pointed compact connected  Riemann surface with (non-empty) boundary  $(V, P)$, a Riemann surface $N$, and a non-constant map, analytic up to the boundary:
$$\alpha: V \lra N.$$

This invariant 
 will play a key role when computing self-intersections of Arakelov-cycles on arithmetic surfaces attached to morphisms from formal-analytic arithmetic surfaces. In this chapter, we investigate it from a purely analytic perspective. The central result of the chapter --- and arguably one of the key computations in this memoir --- is the alternative expression for $\mathrm{Ex}(\alpha : (V, P)\ra N)$ in Theorem  \ref{proposition:explicit} in terms of a Green function on the Riemann surface $N$. 
 
 In the special case when $\alpha$ is \'etale at $P$ and $N$ is compact, the invariant $\mathrm{Ex}(\alpha, g)$ has been introduced in 2009 in some unpublished work of A. Chambert-Loir and the first named author of this monograph. Propositions \ref{proposition:comparison-Deligne} and  \ref{proposition:excess-nonnegative} were also established in this special case.

\section{The invariant $\mathrm{Ex}(\alpha, g)$} 

\subsection{Definition of $\mathrm{Ex}(\alpha, g)$}

In this section, we consider the following data:
\begin{enumerate}[(i)]
\item two  Riemann surfaces $M$ and $N$, with $M$ connected, and a non-constant complex analytic map:
$$\alpha: M \lra N;$$
\item a point $P$ in $M$, and a Green function $g$ (with $\cC^{\bD}$ regularity) with compact support for the divisor $P$ in $M$.  
\end{enumerate}

We denote by $e$ the ramification index of $\alpha$ at $P$: if we let $Q := \alpha(P),$ $e$ is the multiplicity of $P$  in the effective divisor $\alpha^{*}(Q)$ on $M$.

\begin{definition} With the notation above, 
we let: 
\begin{equation}\label{equation:definition-ex}
\mathrm{Ex}(\alpha, g):=\int_{M} g\, \delta_{\alpha^{*}(Q)-e P}+\int_{N}\alpha_{*}g\,\alpha_{*}\omega(g).
\end{equation}
\end{definition}

Observe that the first integral in the right-hand-side of \eqref{equation:definition-ex} is a well-defined real number since the support of the divisor  $$ R := \alpha^{*}(Q)-e P$$
does not contain $P$ by definition of the ramification index $e$. The second integral also is a well-defined real number; indeed $\alpha_{*}g$ is a Green function with compact support and $\mathcal C^{\bD}$ regularity for the divisor $Q$ in $N$, and 
$\alpha_{*}\omega(g)=\omega(\alpha_{*}g)$
is a measure in $\bM^{\cp}(N)$, and the construction in \ref{CbDIntegral} applies. 

Using the projection formula in Proposition \ref{projGbDelta}, (ii), the definition \eqref{equation:definition-ex} of the invariant $\mathrm{Ex}(\alpha, g)$ may be rewritten as follows:
\begin{equation}\label{Exdefbis}
\mathrm{Ex}(\alpha, g)=\int_{M}g  \left(\delta_{\alpha^{*}(Q)-e P}+\alpha^{*}\alpha_{*}\omega(g)\right).
\end{equation}

\subsection{The invariant $\mathrm{Ex}(\alpha, g)$ and Deligne pairings}

When the Riemann surface $N$ is compact and connected, and therefore defines a connected smooth projective complex curve, the invariant $\mathrm{Ex}(\alpha, g)$ occurs naturally when one investigates the metric properties of certain canonical isomorphisms that are attached canonically to the pointed Riemann surface $(M, P)$ and the map $\alpha : M\ra N$. 

We keep the notation of the previous subsection and we assume that $N$ is compact and connected.

 Consider the canonical isomorphism:
$$\iota_{1} : T_{Q} N\lrasim T_{P}M^{\otimes e}$$
defined as the composition of the ``obvious'' isomorphisms:
$$T_{Q} N\lrasim \mathcal O(Q)_{|Q}=\mathcal O(Q)_{|\alpha(P)}\lrasim \alpha^{*}\mathcal O(Q)_{|P}\lrasim\mathcal O(eP+R)_{|P}\lrasim\mathcal O(P)_{|P}^{\otimes e}\lrasim T_{P}M^{\otimes e}.$$
If $z$ (resp. $w$) is a uniformizing parameter on $M$ at $P$ (resp. on $N$ at $Q$) such that $\alpha^{*}(w)=z^{e}$, then:
$$\iota_1\left((\partial /\partial w)_{\mid Q}\right)= (\partial /\partial z)_{\mid P}^{\otimes e}.$$

We may also consider the complex line $\langle \mathcal O(Q), \mathcal O(Q)\rangle$,  defined as the Deligne pairing  with itself of the line bundle $\mathcal O(Q)$ on the projective curve $N$. By the very definition of the Deligne pairing, we have a canonical isomorphism:
$$\iota_{2} : \langle \mathcal O(Q), \mathcal O(Q)\rangle\lrasim\mathcal O(Q)_{|Q}\lrasim T_{Q} N.$$

Since $g$ is a Green function for $P$ on $M$, we may equip the line $T_{P}M$ with the capacitary metric $\Vert.\Vert_{g}^{\mathrm{cap}}$, and the line $T_{P}M^{\otimes e}$ with its $e$-th tensor power $\Vert.\Vert_{g}^{\mathrm{cap}, \otimes e}$.

Since $\supp g$ is compact, we may consider the Green function $\alpha_{*}g$ for $Q$ in $N$, and equip $\mathcal O(Q)$ with the associated Hermitian metric $\Vert.\Vert_{\alpha_{*}g}$, defined by the equality:
$$\Vert\mathbbm 1_{Q}(y)\Vert_{\alpha_{*}g}=e^{-\alpha_{*}g(y)}=\exp\left(-\int_{M}g\delta_{\alpha^{*}(y)}\right)$$
for every $y\in N\setminus \{Q\}$. 
 In turn, to the Hermitian metric $\Vert.\Vert_{\alpha_{*}g}$ on the line  bundle $\cO(Q)$, is associated a Hermitian metric $\Vert.\Vert^{\mathrm{Del}}_{\alpha_{*}g}$ on   the complex line $\langle \mathcal O(Q), \mathcal O(Q)\rangle$ by the construction in 
 \cite{Deligne85}, which immediately extends to the present setting,  where the Hermitian metric on line bundles over complex projective curves are not necessarily $\cC^\infty$ but  have $\cC^{\bD}$ regularity.

The following proposition describes the compatibility between 
the isomorphism of complex lines:
$$\iota_{1}\circ\iota_{2} : \langle \mathcal O(Q), \mathcal O(Q)\rangle\lrasim T_{P}M^{\otimes e}$$
and the Hermitian norms  $\Vert.\Vert^{\mathrm{Del}}_{\alpha_{*}g}$ and $\Vert.\Vert_{g}^{\mathrm{cap}, \otimes e}$ on these lines.

\begin{proposition}\label{proposition:comparison-Deligne}
For any 
$v\in \langle \mathcal O(Q), \mathcal O(Q)\rangle$, the following equality holds:
\begin{equation}\label{equation:comparison}
\Vert \iota_{1}\circ\iota_{2}(v)\Vert_{g}^{\mathrm{cap}, \otimes e}=\exp\left({-\mathrm{Ex}(\alpha, g)}\right) \Vert v\Vert^{\mathrm{Del}}_{\alpha_{*}g}.
\end{equation}
\end{proposition}

The equality \eqref{equation:comparison} is a straightforward consequence of the definitions of the Hermitian norms 
$\Vert.\Vert_{g}^{\mathrm{cap}}$ and $\Vert.\Vert^{\mathrm{Del}}_{\alpha_{*}g}$. Proposition \ref{proposition:comparison-Deligne} provides a  motivation for the introduction of the invariant $\mathrm{Ex}(\alpha, g)$, and  explains why it arises in the computations of arithmetic intersection numbers in the following sections.  However we will not use Proposition \ref{proposition:comparison-Deligne} explicitly in the sequel, and the details of its proof will be left to the reader.

\section{The invariant $\mathrm{Ex}(\alpha : (V, P)\ra N)$}

In this section, we consider a connected compact Riemann surface with boundary $V$ with non-empty boundary $\partial V$, and $P$ a point in its interior $\mathring{V}$. 

We denote by $g_{V, P}$ be the  Green function associated  to the point $P$ of the Riemann surface with boundary $V$ and by $\mu_{V, P}$ its harmonic measure, as defined in \ref{CompactBoundEqu} above.

\begin{definition}\label{definition:overflow}
For every non-constant complex analytic map: 
$$\alpha : V^{+}\lra N$$
from $V^+$ to a Riemann surface $N$, we define the \emph{overflow} of the morphism $\alpha$ from the pointed compact Riemann surface with boundary $(V,P)$ to $N$ as the real number:
\begin{equation}\label{equation:definition-overflow}
\mathrm{Ex}(\alpha : (V, P)\ra N):=\mathrm{Ex}(\alpha, g_{V, P}).
\end{equation}
\end{definition}

If we denote by $e$ the ramification index of $\alpha$ at the point $P$, according to \eqref{Exdefbis}, we have:
\begin{equation}\label{equation:definition-overflowBis}
\mathrm{Ex}(\alpha : (V, P)\ra N) =\int_{V}g_{V, P}\left(\delta_{\alpha^{*}(\alpha(P))-eP}+\alpha^{*}\alpha_{*}\mu_{V, P}\right).
\end{equation}
The expression \eqref{equation:definition-overflowBis}  shows that, like the Green function $g_{V,P}$ and the measure  $\mu_{V,P}$, the overflow $\mathrm{Ex}(\alpha : (V, P)\ra N)$ is nonnegative. 

Using that the measure:
$$g_{V, P} \, \omega(g_{V,P}) = g_{V,P} \, \mu_{V,P}$$
is zero, one readily checks that $\mathrm{Ex}(\alpha : (V, P)\ra N)$ also admits the following expression in terms of $\ast$-product:
\begin{equation}\label{equation:definition-overflowstar}
\mathrm{Ex}(\alpha : (V, P)\ra N) =\int_{V} (\alpha^\ast \alpha_\ast g_{V,P} - e\, g_{V,P}) \ast g_{V,P}.
\end{equation}

Recall that a morphism of complex analytic spaces is called \emph{finite} when it is proper with finite fibers. 
Finite morphisms between Riemann surfaces are classically known as ``ramified coverings." 

\begin{proposition}\label{proposition:excess-nonnegative}
With the notation of Definition \ref{definition:overflow}, the overflow $\mathrm{Ex}(\alpha : (V, P)\ra N)$ vanishes if and only if the complex analytic map 
$\alpha_{|\mathring V} : \mathring V\ra \alpha(\mathring V)$
is a finite  morphism 
totally ramified\footnote{namely, a finite analytic morphism of degree the ramification index $e$ of $\alpha$ at $P$.} at $P$.
\end{proposition}

In particular, when $e=1$, that is when $\alpha$ is \'etale at $P$, the overflow $\mathrm{Ex}(\alpha : (V, P)\ra N)$ vanishes if and only if $\alpha_{\mathring{V}}$ is an open embedding of Riemann surfaces. 

The non-negativity of  the invariant  $\mathrm{Ex}(\alpha : (V, P)\ra N)$ and the above characterization of its vanishing when $e =1$ 
motivates the terminology \emph{overflow} (or \emph{excess}) for this invariant.
The notation $\mathrm{Ex}$ stands for {\bf{ex}}cess or {\bf ex}undatio.

The proof of Proposition \ref{proposition:excess-nonnegative} will rely on the following lemma.

\begin{lemma}\label{lemma:criterion-finite}
The complex analytic map $\alpha_{|\mathring V}:\mathring V\ra \alpha(\mathring V)$ is a finite morphism if and only if the following equality holds:
\begin{equation}\label{equation:criterion-finite}
\alpha(\partial V)\cap \alpha(\mathring V)=\emptyset. 
\end{equation}
\end{lemma}

\begin{proof}
Since $\alpha$ is nonconstant and $V$ is compact, the fibers of $\alpha_{\mid V}$ are finite.  Therefore $\alpha_{|\mathring V}:\mathring{V} \ra \alpha(\mathring{V})$ is a finite  morphism if and only if it is a proper holomorphic map, namely if and only if, for any compact subset $K$ of $\alpha(\mathring V)$, $\alpha^{-1}(K)\cap\mathring V$ is compact. 

The equality \eqref{equation:criterion-finite} is readily seen to be equivalent to:
\begin{equation}\label{equation:criterion-finite-bis}
\alpha^{-1}(\alpha(\mathring V))\cap \partial V=\emptyset.
\end{equation}

Assume that \eqref{equation:criterion-finite} holds, and let $K$ be a compact subset of $N$ contained in $\alpha(\mathring V)$. Then $\alpha^{-1}(K)\cap V$ is closed in the compact manifold with boundary $V$, so that it is compact. The equality \eqref{equation:criterion-finite-bis} implies:
$$\alpha^{-1}(K)\cap \partial V=\emptyset$$
so that $\alpha^{-1}(K)\cap V$ is actually contained in $\mathring V$. This proves that $\alpha^{-1}(K)\cap\mathring V$ is compact.

Conversely, assume that the intersection $\alpha(\partial V)\cap \alpha(\mathring V)$ is nonempty, and let $x$ (resp. $y$) be a point of $\mathring V$ (resp. $\partial V$) such that $\alpha(x)=\alpha(y)$. Since $\alpha(\partial V)$ is nowhere dense in $N$, we may find a compact subset $K$ of the open set $\alpha(\mathring V)$ of $N$ such that $K$ contains $\alpha(x)$ and $K\setminus \alpha(\partial V)$ is dense in $K$; for instance, a small disk containing $\alpha(x)$ in its interior will do.

Since $K\setminus \alpha(\partial V)$ is dense in $K$ and the map $\alpha$ is open, the set $\alpha^{-1}(K)\setminus \partial V$ is dense in $\alpha^{-1}(K).$ In particular, 
$$\alpha^{-1}(K)\cap \mathring V=\alpha^{-1}(K)\setminus \partial V$$
is not compact, as its closure in $V$ contains $y\in \partial V.$
\end{proof}

\begin{proof}[Proof of Proposition \ref{proposition:excess-nonnegative}]
If we denote by $e$ the ramification index of $\alpha$ at $P$, we may write:
\begin{equation}\label{equation:expression-overflow}
\mathrm{Ex}(\alpha : (V, P)\ra N)=\int_V g_{V, P}\, \delta_{\alpha^{*}(\alpha(P))-eP}+\int_N \alpha_*g_{V, P}\,\alpha_*\mu_{V, P}.
\end{equation}

The function $g_{V, P}$  is positive on $\mathring V$ and vanishes on $V^+ \setminus \mathring{V}$. Consequently the function $\alpha_*g_{V, P}$ is positive on $\alpha(\mathring V)$, and vanishes on $N\setminus \alpha(\mathring V)$. Moreover the support of the positive measure $\mu_{V, P}$ is $\partial V$, and therefore the support of the positive measure $\alpha_*\mu_{V, P}$ is $\alpha(\partial V).$ 

This implies that the integral:
$$\int_V g_{V, P}\, \delta_{\alpha^{*}(\alpha(P))-eP}$$
vanishes if and only if the support  of the divisor $\alpha^{*}(\alpha(P))-eP$ is disjoint of $\Vcirc$, and that 
the integral: 
$$\int_N \alpha_*g_{V, P}\,\alpha_*\mu_{V, P}$$
vanishes if and only if the intersection: 
$$\alpha(\partial V)\cap \alpha(\mathring V)$$
is empty, and therefore, according to  Lemma \ref{lemma:criterion-finite}, if and only if $\alpha_{|\mathring V}:\mathring V\ra \alpha(\mathring V)$ is a finite morphism.

Finally the expression \eqref{equation:expression-overflow} for the overflow $\mathrm{Ex}(\alpha : (V, P)\ra N)$ shows that it vanishes if and only if $\alpha_{|\mathring V}:\mathring V\ra \alpha(\mathring V)$ is a finite morphism and:
$$\alpha_{|\mathring V}^\ast (\alpha(P)) = e P.$$
This last condition precisely means that the finite morphism $\alpha_{|\mathring V}$ is totally ramified at $P$.
\end{proof}

\section{Green functions for the diagonal of a Riemann surface}

In addition to its definition by  formulas \eqref{equation:definition-overflowBis} or \eqref{equation:expression-overflow}, the overflow  $\mathrm{Ex}(\alpha : (V, P)\ra N)$  admits an alternative expression that  will be established in
the next section, 
and turns out to  be   useful in applications, as demonstrated by its role in the derivation of the ``arithmetic holononomicity theorem" of \cite{CalegariDimitrovTang21} in Section \ref{AppCDT}.

Besides the Green function $g_{V,P}$ and the harmonic measure $\mu_{V,P}$, this alternative expression involves a suitable notion of Green function for the diagonal of  a Riemann surface $N$, possibly non-compact, which we introduce in this section.

\subsection{}\label{Greendiagdef}  Let $N$ be a Riemann surface equipped with a real $\cC^\infty$ $2$-form $\beta$.

We define \emph{a Green function for the diagonal of $N$ associated to the $2$-form $\beta$} 
as a function:
$$g : N\times N\lra (-\infty, +\infty]$$
that satisfies the following three  conditions:
\begin{enumerate}[(i)]
\item \emph{The function $g$ is symmetric;} namely, for every $(Q_{1}, Q_{2})\in N^{2}$, we have:
$$g(Q_{1}, Q_{2})=g(Q_{2}, Q_{1}).$$
\item \emph{On the complement $N \times N \setminus \Delta_N$ of the diagonal $\Delta_N$ of $N$, $g_N$ takes values in $\R$ and is $\cC^\infty$. Along $\Delta_{N}$, $g$ admits a logarithmic singularity};  namely,  for every local complex analytic coordinate $z$ defined on some  open subset $U$ of $N$,  there exists a function $h$ in $\mathcal C^{\infty}(U\times U, \R)$ such that, for every $(Q_{1}, Q_{2})$ in $U^{2}$:
\begin{equation}\label{equation:log-sing}
g(Q_{1}, Q_{2})=\log|z(Q_{1}-z(Q_{2})|^{-1}+h(Q_{1}, Q_{2}).
\end{equation}

This implies that $g$ is locally $L^{1}$ on $N\times N$ and defines a distribution on $N\times N$  such that the current:
$$\omega(g):=\frac{i}{\pi}\partial\overline\partial g+\delta_{\Delta_{N}}$$
is a $\cC^\infty$ $2$-form on $N\times N.$
\item \emph{If $p_{1}$ and $p_{2}$ denote the two projections from $N\times N$ onto $N$, then the $2$-form $\omega(g)-p_{1}^{*}\beta-p_{2}^{*}\beta$ vanishes on the fibers of $p_{1}$ and $p_{2}$.}
\end{enumerate}

When these conditions are satisfied, then for every point $Q$ of $N$, the function $g(Q, .)$ is a Green function with $\mathcal C^{\infty}$ regularity for the divisor $Q$ in $N$, and we have:
$$\omega(g(Q,.))=\frac{i}{\pi}\partial\overline\partial g(Q, .)+\delta_{Q}=(Q, \mathrm{Id}_N)^*\omega(g)=\beta.$$
The capacitary metric $\Vert.\Vert^{\mathrm{cap}}_{g(Q, .)}$ on $T_{Q}N$ will be denoted by $\Vert.\Vert^{\mathrm{cap}}_{g}.$ With the notation of \eqref{equation:log-sing}, we have:
$$\Vert (\partial /\partial z)_{Q}\Vert^{\mathrm{cap}}_{g}=e^{-h(Q, Q)}$$  
for every $Q$ in $U$. In particular, this construction defines a $\cC^\infty$ metric $\Vert.\Vert^{\mathrm{cap}}_{g}$ on the tangent bundle $T_{N}$ of $N$. Alternatively, this metric may be defined as follows: the line bundle $\mathcal O(\Delta_{N})$ on $N\times N$ may be endowed with the $\cC^\infty$  metric $\Vert.\Vert_{g}$ defined by the equality:
$$\Vert \mathbbm 1_{\Delta_{N}}\Vert_{g}=e^{-g},$$
and the adjunction isomorphism:
$$\mathcal O(\Delta_{N})_{|\Delta_{N}}\simeq N_{\Delta_{N}}(N\times N)\simeq T_{N}$$
becomes an isometry when $\mathcal O(\Delta_{N})_{|\Delta_{N}}$ is endowed with the restriction of the metric $\Vert.\Vert_{g}$ and $T_{N}$ with the capacitary metric $\Vert.\Vert^{\mathrm{cap}}_{g}.$

\subsection{Examples}\label{example:Green}

\subsubsection{}\label{GreenA1} The function:
$$g_{\C} : \C^{2}\lra (-\infty, +\infty],\,\, (z, z')\mapsto \log|z-z'|^{-1}$$
is a Green function for the diagonal of $\C$ associated to the $2$-form $\beta=0.$ The capacitary metric $\Vert.\Vert^{\mathrm{cap}}_{g_\C}$ on $T_\C$ is the ``standard" metric, which satisfies:
$$\Vert \partial/ \partial z\Vert^{\mathrm{cap}}_{g_\C} =1.$$

\subsubsection{}\label{GreenP1} The function: 
$$g_{\PP^1(\C)}: \PP^1(\C)^2 \lra  (-\infty, +\infty]$$
defined by:
\begin{equation}\label{gP1def}
g_{\PP^1(\C)}((z_0: z_1), (z'_0:z'_1) ) := \log \bigg| \frac{z_0 z'_1 - z_1 z'_0}{(\vert z_0\vert^2 + \vert z_1 \vert^2)^{1/2} (\vert z'_0\vert^2 + \vert z'_1 \vert^2)^{1/2}} \bigg|^{-1}
\end{equation}
is a Green function for the diagonal in $\PP^1(\C)$ associated to the Fubini-Study 2-form $\beta = \omega_{\rm{FS}}$, defined by:
$$\omega_{\rm{FS}\mid \C} :=  \frac{i}{2\pi} \frac{dz \wedge d\overline{z}}{(1 + \vert z \vert^2)^2}.$$

The Green function $g_{\PP^1(\C)}$ is nonnegative everywhere, and it is invariant under the diagonal action of $U(2)$ on $\PP^1(\C)^{2}$.
The capacitary metric $\Vert.\Vert^{\mathrm{cap}}_{g_{\PP^1(\C)}}$ on $T_{\PP^1(\C)}$ is the $U(2)$-invariant metric such that:
$$\Vert (\partial/\partial z)_{\mid 0}\Vert^{\mathrm{cap}}_{g_{\PP^1(\C)}} =1.$$

\subsubsection{}\label{Greengen}  When the Riemann surface $N$ is compact and connected, and  when the real $\cC^\infty$ $2$-form $\beta$ satisfies: 
\begin{equation}\label{beta1}
\int_{N}\beta=1,
\end{equation}
then there exists a Green function $g$ for the diagonal of $N$ associated to $\beta$ in the sense above. It is  unique up to an additive constant. 

Indeed, if $(\alpha_{k})_{1\leq k\leq q}$ is an orthogonal  basis of $\Omega^1(N)$ endowed with its canonical Hermitian scalar product $\langle .,.\rangle$ defined by:
$$\langle\alpha, \alpha'\rangle :=\frac{i}{2}\int_{N}\alpha\wedge\overline\alpha',$$
then the closed real $\cC^\infty$ $2$-form on $N \times N$:
$$\gamma:=p_{1}^{*}\beta+p_{2}^{*}\beta+\frac{i}{2}\sum_{k=1}^{q}(p_{1}^{*}\alpha_{k}\wedge p_{2}^{*}\overline\alpha_{k}-p_{1}^{*}\overline\alpha_{k}\wedge p_{2}^{*}\alpha_{k})$$ defines the same cohomology class as the current $\delta_{\Delta_N}$. Therefore, according to the $\partial \overline{\partial}$-lemma, which applies on the compact K\"ahler manifold $N\times N,$ there exists a real distribution  $g$ on $N \times N$, unique up to some additive constant, such that:
$$\frac{i}{\pi}\partial\overline\partial g = \gamma - \delta_{\Delta_N}.$$
Using the ellipticity properties of the operator $\partial\overline\partial,$ one easily see that the distribution $g$ is defined by a Green function for the diagonal of $N$ associated  to $\beta$.\footnote{A similar but more elementary argument shows that, conversely, the normalization condition \eqref{beta1} is implied  by the existence of a Green function for the diagonal of $N$ associated to $\beta$.}

Moreover the integral:
$$\int_{N}g(P, .)\beta$$
is independent of $P$ in $N$. Indeed this integral is easily seen to define a $\cC^\infty$ function of $P$ and the associated distribution on $N$ to coincide with
$p_{1\ast} (g\,  p_2^\ast \beta).$
This distribution is harmonic, and therefore coincides with a constant function, since we have:
\begin{align*}
i \partial\overline\partial \, p_{1\ast} (g\,  p_2^\ast \beta) & = p_{1\ast} ( i \partial\overline\partial g  \wedge p_2^\ast \beta) ) \\
&=  \pi \, p_{1\ast} \left((\gamma - \delta_{\Delta_N}) \wedge p^\ast_2 \beta\right) \\
& = \pi \,  p_{1\ast} (\gamma \wedge p^\ast_2 \beta) - \pi \beta \\
& = \pi \,  p_{1\ast} (p^\ast_1 \beta \wedge p^\ast_2 \beta) - \pi \beta \\
& = \pi \,  \left(\int_N \beta \right) \,  \beta - \pi \beta = 0.
\end{align*}

This shows that there exists a Green function $g$ for the diagonal of $N$ associated to $\beta$ that moreover satisfies the condition:
$$\int_{N}g(P, .)\, \beta = 0 \quad \mbox{for every $Q \in N$.}$$
This normalized Green function is uniquely determined by $\beta$. 

When $\beta$ is the normalized volume form associated to some conformal Riemannian metric $\mathbf{g}$ on $N$, the normalized  Green function $g$ coincides with the usual Green function associated to the Laplace operator on the Riemannian manifold $(N, \mathbf{g})$.

When the genus $q$ is positive and when $\beta$ is the $2$-form:
\begin{equation}\label{betaAr}
\beta_{\mathrm{Ar}} := \frac{i}{2q} \sum_{k=1}^q \alpha_k \wedge \overline{\alpha}_k,
\end{equation}
the normalized Green function $g$ is the Arakelov-Green function of $N$, introduced by Arakelov in~\cite{Arakelov74}.

\section{Overflow and Green functions for the diagonal}
\label{subsection:explicit}

In this section, we consider a connected compact Riemann surface $V$ with nonempty boundary $\partial V$, and $P$ a point of its interior $\mathring V$.  We also consider a Riemann surface $N$ (without boundary) endowed with a smooth $2$-form $\beta$, and a Green function $g_{N}$ for the diagonal of $N$ associated to $\beta$.

\subsection{An alternative expression for $\mathrm{Ex}(\alpha : (V, P)\ra N)$}\label{alternative}

To formulate the alternative expression for the overflow that constitutes the object of this section, we need to introduce some notation. 

Let $\alpha : V^{+}\ra N$ be a non-constant holomorphic map, and let $e$ be the ramification index of $\alpha$ at~$P$. Then the $e$-the jet $\alpha^{[e]}(P)$ of $\alpha$ at $P$ may be identified with a non-zero element of the complex line: 
$$\Hom_{\C}(T_{P}V^{\otimes e}, T_{\alpha(P)}N)\simeq T_{P}V^{\vee, \otimes e}\otimes T_{\alpha(P)}N.$$
If $z$ (resp. $w$) is a uniformizing parameter on $M$ at $P$ (resp. on $N$ at $Q$) such that $\alpha^{*}(w)=z^{e}$, then:
$$\alpha^{[e]}(P) = (dz_{\mid P})^{\otimes e} \otimes (\partial /\partial w)_{\mid Q}.$$
The complex line $T_{P}V^{\vee, \otimes e}\otimes T_{\alpha(P)}N$ may be endowed with the Hermitian norm  $\Vert.\Vert^{\mathrm{cap}}_e$ deduced  by duality and tensor products from the capacitary metric $\Vert.\Vert^{\mathrm{cap}}_{V, P}$ on $T_{P}V$ and the capacitary metric $\Vert.\Vert^{\rm cap}_{g_{N}}$ on $T_{\alpha(P)}N$.

\begin{theorem}\label{proposition:explicit}
With the previous notation, for any nonconstant holomorphic map
$\alpha : V^{+}\ra N$ of ramification index $e$ at $P$, 
the following equality holds:
\begin{multline}\label{equation:explicit}
 \mathrm{Ex}\big(\alpha : (V, P)\ra N)\big) 
 =2\int_{V}g_{V, P} \, \alpha^{*}\beta-\int_{(\partial V)^{2}}g_{N}(\alpha(z_{1}),\alpha(z_{2}))\, d\mu_{V, P}(z_{1})\, d\mu_{V, P}(z_{2}) \\  - \log\Vert \alpha^{[e]}(P)\Vert^{\mathrm{cap}}_e .
\end{multline}
\end{theorem}

The second integral in the right-hand side of \eqref{equation:explicit} is well-defined. Indeed the function:
$$(\partial V)^2 \lra (-\infty, +\infty], \quad (z_1, z_2) \longmapsto g_N(\alpha(z_1), \alpha(z_2))$$
is integrable with respect to the measure $\mu_{V,P} \otimes \mu_{V,P}$ on $(\partial V)^2$. Actually, as a consequence of the logarithmic behavior of $g_N$ along the diagonal, this function is locally $L^1$ on the compact $\cC^\infty$ surface $(\partial V)^2$,\footnote{This follows for instance from the following elementary result: if $\gamma_1$ and $\gamma_2$ are two $\cC^\infty$ immersions of $(0,1)$ into $\C$ and $n_1$ and $n_2,$ two positive integers, then
$\log \vert \gamma_1(t_1)^{n_1} -  \gamma_2(t_2)^{n_2}\vert^{-1}$ is a locally $L^1$ function of $(t_1, t_2) \in (0,1)^2$.} and the measure  $\mu_{V,P} \otimes \mu_{V,P}$ is defined by a continuous density on $(\partial V)^2$. This integrability will also be a consequence of the proof below.

\begin{proof}
For simplicity, we write $g$ (resp. $\mu$) for $g_{V, P}$ (resp. $\mu_{V, P}$). They satisfy the following equality of currents on $V^{+}:$
\begin{equation}\label{equation:Lelong}
\frac{i}{\pi}\partial\overline\partial g=\mu-\delta_{P}.
\end{equation}
Moreover we have:
$$\frac{i}{\pi}\partial\overline\partial g_{N}=\gamma-\delta_{\Delta_{N}},$$
where $\gamma$ is a smooth $2$-form such that, for any $Q\in N$, 
$$(\mathrm{Id}_N, Q)^{*}\gamma=\beta.$$

For every $x$ in $V$, we have the following equality of currents on $V^+$:
\begin{align*}
\alpha^{*}\alpha_{*}\delta_{x}=\alpha^{*}\delta_{\alpha(x)} & =(\alpha, \alpha(x))^{*}\delta_{\Delta_{N}}\\
& =(\alpha, \alpha(x))^{*}(\gamma-\frac{i}{\pi}\partial\overline\partial g_{N})\\
& =\alpha^{*}\beta-\frac{i}{\pi}\partial\overline\partial(\alpha, \alpha(x))^{*}g_{N}.
\end{align*}
Consequently, for every $x$ in $V\setminus\alpha^{-1}(\alpha(P)),$ we have:
\begin{align}
\nonumber\int_{V} g\, \alpha^{*}\alpha_{*}\delta_{x}& =\int_{V}g\, \alpha^{*}\beta-\int_{V^+} g \, \frac{i}{\pi}\partial\overline\partial(\alpha, \alpha(x))^{*}g_{N}\\
& =\int_{V}g\, \alpha^{*}\beta-\int_{V^+} (\alpha, \alpha(x))^{*}g_{N} \, \frac{i}{\pi}\partial\overline\partial g\label{equation:Stokes}\\
& =\int_{V}g\, \alpha^{*}\beta-\int_{V} (\alpha, \alpha(x))^{*}g_{N}\, (\mu -\delta_{P})\label{equation:Lelong-cons}\\
&\nonumber =\int_{V}g\, \alpha^{*}\beta-\int_{x'\in \partial V} g_{N}(\alpha(x'), \alpha(x))\, d\mu(x')+g_{N}(\alpha(P), \alpha(x)).
\end{align}
The equality \eqref{equation:Stokes} follows from Green's formula, and \eqref{equation:Lelong-cons} follows from \eqref{equation:Lelong}.

The intersection of $\alpha^{-1}(\alpha(P))$ with $\mathrm{supp}\, \mu =\partial V$ is finite, and therefore: $$\mu(\alpha^{-1}(\alpha(P)))=0,$$ since $\mu$ is defined by a continuous density on the smooth curve $\partial V$. As a consequence:
\begin{align}
\int_{V}g\, \alpha^{*}\alpha_{*}\mu 
& =\int_{\partial V}\Big(\int_{V}g\, \alpha^{*}\alpha_{*}\delta_{x}\Big)d\mu(x)\nonumber\\
&=\int_{V}g\, \alpha^{*}\beta-\int_{(\partial V)^{2}} g_{N}(\alpha(x'), \alpha(x))\, d\mu(x')\, d\mu(x)+\int_{\partial V}g_{N}(\alpha(P), \alpha(x))\, d\mu(x).
\label{equation:measure-term}
\end{align}

Moreover we have:
\begin{align*}
\delta_{\alpha^{-1}(\alpha(P))-eP} & =\alpha^{*}\alpha_{*}\delta_{P}-e\,\delta_{P}\\
& =(\alpha, \alpha(P))^{*}\delta_{\Delta_{N}}-e\,\delta_{P}\\
& =(\alpha, \alpha(P))^{*}\left(-\frac{i}{\pi}\partial\overline\partial g_{N} +\gamma\right) -e\left(-\frac{i}{\pi}\partial\overline\partial g+ \mu\right)\\
&=\frac{i}{\pi}\partial\overline\partial\phi + \alpha^{*}\!\beta-e\mu,
\end{align*}
where:
$$\phi:=eg-(\alpha, \alpha(P))^{*}g_{N}$$
is defined a priori as a  locally $L^{1}$ function on $V^+$. It is actually continuous on a neighborhood of $P$, and its value at $P$ is:
\begin{equation}\label{equation:define-phi}
\phi(P)=\log\Vert \alpha^{[e]}(P)\Vert^{\mathrm{cap}}_e.
\end{equation}

Finally $g$ vanishes on $V^+ \setminus \Vcirc$, in particular on $\partial V = \supp \mu$, and therefore:
\begin{align}
\int_{V}g\, \delta_{\alpha^*(\alpha(P))-eP} &=\int_{V^+}g\, \frac{i}{\pi}\partial\overline\partial \phi + \int_{V} g\, \alpha^{*}\!\beta\nonumber\\
& =\int_{V^+}\phi\frac{i}{\pi}\partial\overline\partial g+ \int_{V} g\,\alpha^{*}\beta\nonumber\\
& =\int_{V^+}\phi\, (\mu - \delta_P) + \int_{V} g\, \alpha^{*}\beta\nonumber\\
& =-\phi(P) -\int_{\partial V}(\alpha, \alpha(P))^{*}\!g_{N}\, \mu+ \int_{V} g\, \alpha^{*}\!\beta
\label{equation:delta-term}
\end{align}
The equality \eqref{equation:explicit} follows from adding \eqref{equation:measure-term} and \eqref{equation:delta-term}, together with \eqref{equation:define-phi}.
\end{proof}

\subsection{Examples}\label{ExamplesOverflow} We may specialize Theorem \ref{proposition:explicit} to the situation  where $N$ is the Riemann surface $\C$ or $\PP^1(\C)$, and $g_N$ is the Green function introduced in   \ref{GreenA1} and 
\ref{GreenP1}.

When $N=\C$, Theorem \ref{proposition:explicit} takes the following form:

\begin{proposition}\label{proposition:explicit-C}
For any nonconstant holomorphic map
$\alpha : V^{+}\ra \C$ of ramification index $e$ at $P$,
the following equality holds:
\begin{equation}\label{equation:explicit-C}
\mathrm{Ex}(\alpha : (V, P)\ra\C)=\int_{(\partial V)^{2}}\log|\alpha(z_{1})-\alpha(z_{2})| \, d\mu_{V, P}(z_{1})\, d\mu_{V, P}(z_{2}) - \log\Vert \alpha^{[e]}(P)\Vert_{V, P}^{\mathrm{cap}, \otimes(-e)}.
\end{equation}
\end{proposition}

In the last term of the right-hand side of \eqref{equation:explicit-C}, we have used that, as the range $N$ of $\alpha$ is $\C,$ the $e$-th jet $\alpha^{[e]}(P)$ of $\alpha$ at $P$ is an element of the line:
$$\Hom_\C(T_P V^{\otimes e}, T_{\alpha(P)} \C) \simeq (T_P V^\vee)^{\otimes e},$$
and that the Hermitian norm $\Vert.\Vert^{\mathrm{cap}}_{V,P}$ on this line is the norm $\Vert.\Vert_{V, P}^{\mathrm{cap}, \otimes(-e)}$ on $(T_P V^\vee)^{\otimes e}$ deduced from the capacitary norm $\Vert.\Vert_{V, P}^{\mathrm{cap}}$ on $T_P V$ by duality and tensor power. In terms of some local analytic coordinate $z$ on $M$ defined on  an open neighborhood of $P,$ we have:
$$\alpha^{[e]}(P) =\frac{1}{e!}\frac{d^e \alpha(P)}{dz^e} \, dz_{\mid P}^{\otimes e}.$$

Consequently, as a special case of Proposition \ref{proposition:explicit-C}, we obtain:

\begin{proposition}\label{prop:D-to-C}
For any nonconstant holomorphic function
$\alpha : \overline D(0,1)^{+}\ra \C$ of ramification index $e$ at $0$,
we have:
$$\mathrm{Ex}(\alpha : (\overline D(0,1), 0) \ra \C)=\int_{[0,1]^{2}}\log|\alpha(e^{2i\pi t_{1}})-\alpha(e^{2i\pi t_{2}})| \, dt_{1}\, dt_{2}-\log|\alpha^{(e)}(0)/e!|.$$
\end{proposition}

Applied to $N=\PP^1(\C)$ and to the non-negative Green function $g_{\PP^1(\C)}$ defined by  \eqref{gP1def}, Proposition \ref{proposition:explicit} immediately implies:

\begin{proposition}\label{exP1ineq} For every non-constant  function $\alpha: V^+ \ra \PP^1(\C)$ of ramification index $e$ at $P$,  we have:
\begin{equation}
\mathrm{Ex}(\alpha : (V, P)\ra\PP^1(\C))  \leq 2 \int_V g_{V,P} \, \alpha^\ast \omega_{\rm{FS}} -  \log\Vert \alpha^{[e]}(P)\Vert^{\mathrm{cap}}_e. 
\end{equation}
\end{proposition}

When $\alpha$ is a  meromorphic function  defined on some open disk $\mathring{D}(0, R)$, 
Proposition \ref{exP1ineq} applied with $V$ a closed disc $\overline{D}(0, r)$ contained in $\mathring{D}(0, R)$ provides an upper-bound of the overflow  $\mathrm{Ex}(\alpha : (\overline D(0,r), 0) \ra \PP^1(\C))$ in terms of the characteristic function $T_\alpha$ 
of Nevanlinna-Ahlfors-Shimizu\footnote{See for instance \cite{Ahlfors1930}, \cite{Ahlfors1935}, and \cite[Chapter 1]{Griffiths76}.}  attached to~$\alpha$.

\begin{proposition}\label{exboundchar} Let $\alpha: \mathring{D}(0, R) \ra \PP^1(\C)$ be a non-constant  meromorphic function for some $R$ in $(0, +\infty]$. Assume that $\alpha(0)$ is not $\infty$, and let $e$ be the ramification index of $\alpha$ at $0$. For every $r$ in $(0, R),$ the following estimate holds:
$$\mathrm{Ex}(\alpha : (\overline D(0,r), 0) \ra \PP^1(\C)) \leq 2\,  T_\alpha (r)  - e \log r - \log \frac{\vert \alpha^{(e)}(0)/e!\vert}{1+ \vert \alpha(0)\vert^2}.$$
where:
\begin{equation}\label{chardef}
T_\alpha (r) := \int_{\mathring{D}(0,r)} \log (r/\vert z \vert) \, \alpha^\ast \omega_{\rm{FS}}.
\end{equation}

\end{proposition}

When $\alpha$ is holomorphic on $\mathring{D}(0, R)$, a straightforward application of Green's formula shows the the characteristic function $T_\alpha$ admits the following expression:
$$ T_\alpha(r) = (1/2) \int_0^1 \log (1+ \vert \alpha(r e^{2\pi it})\vert^2) \, dt 
-(1/2) \log (1 + \vert \alpha(0) \vert^2).$$
Consequently, applied to a holomorphic function, Proposition \ref{exboundchar} takes the following form:
\begin{corollary} 
Let $\alpha: \mathring{D}(0, R) \ra \C$ be a non-constant holomorphic function for some $R$ in $(0, +\infty]$, and let $e$ be the ramification index of $\alpha$ at $0$. For every $r$ in $(0, R),$ the following estimate holds:
\begin{equation}\label{exboundcharhol}
\mathrm{Ex}(\alpha : (\overline D(0,r), 0) \ra \C) \leq \int_0^1 \log (1+ \vert \alpha(r e^{2\pi it})\vert^2) \, dt -  e \log r - \log|\alpha^{(e)}(0)/e!|.
\end{equation}
\end{corollary}

The estimate \eqref{exboundcharhol} immediately implies the following one, which is  also a simple consequence of Proposition \ref{prop:D-to-C}:
\begin{equation}\label{exboundcharholbis}
\mathrm{Ex}(\alpha : (\overline D(0,r), 0) \ra \C) \leq 2 \int_0^1 \log^+ \vert \alpha(r e^{2\pi it})\vert \, dt + \log 2 -  e \log r - \log|\alpha^{(e)}(0)/e!|.
\end{equation}

\begin{example} Let $P \in \C[X]$ be a polynomial of degree $d \geq 1$. We may write it:
$$P(X) = a_d X^d + \dots + a_e X^e + a_0,$$
withe $1 \leq e \leq d,$ and $a_d$ and $a_e$ in $\C^\times.$ From the definition \eqref{equation:definition-overflowBis} of the overflow or from its expression in Proposition \ref{prop:D-to-C}, one easily obtains:
$$\mathrm{Ex}(P:(\overline{D}(0, r) \ra \C) )= (d-e) \log r - \log \vert a_e/a_d \vert + o(1), \quad \mbox{when $r \lra +\infty$}.$$

This expression may be compared with the upper bound  \eqref{exboundcharholbis}, which applied to $\alpha = P$ takes the following form when $r$ is large enough: 
$$\mathrm{Ex}(P:(\overline{D}(0, r) \ra \C) )\leq  (2d-e) \log r - \log \vert a_e/a_d^2 \vert + \log 2. $$
\end{example}

\part{Formal-analytic arithmetic surfaces, pseudoconcavity, and finiteness theorems}

 \chapter{Formal-analytic arithmetic surfaces and arithmetic intersection numbers}\label{ChapterFA}
 
 In this chapter, we introduce the arithmetic counterparts of the fibered (germs of) analytic surfaces studied in Chapter \ref{CNBfibered} in a geometric context. These are the \emph{smooth formal-analytic surfaces over $\Spec\OK$}, with $K$  some number field, already considered in \cite[Chapter 10]{Bost2020}, and defined by gluing at the Archimedean places some ``formal pointed curve over $\Spec \OK$" --- a typical instance of which is $\Spf \OK[[T]]$ ---  with connected compact Riemann surfaces with boundary.
 
Specifically
this chapter is devoted to diverse foundational results and to some basic examples concerning general   smooth formal-analytic surfaces. 

Notably we define the $\OK$-algebra $\cO(\Vfa)$ and the field $\cM(\Vfa)$ of regular and  meromorphic functions on some smooth formal-analytic surfaces over $\Spec\OK$. We introduce some rudiments of arithmetic intersection theory on these formal-analytic surfaces. The equilibrium potentials associated to the pointed compact Riemann surfaces with boundary  defining $\Vfa$ turn out to play a central role in this rudimentary intersection theory, which allows us  to define the  \emph{pseudoconvexity} and the \emph{pseudoconcavity} of a smooth formal-analytic surface $\Vfa$.

We  also discuss some simple examples of smooth formal-analytic surfaces over $\Spec \Z$, notably the smooth formal-analytic surfaces $\Vfa(\overline{D}(0,1), \psi)$, which are defined by gluing $\Spf \Z[[T]]$ to the closed disk $\overline{D}(0,1)$ by means of some formal series $\psi$ in the group:
$$G_{\mathrm{for}}(\R) := \R^\ast X + X^2 \R[[X]]$$
of ``formal diffeomorphism of $\R$ fixing $0$."
 These formal-analytic arithmetic surfaces  turn out to play an important role in applications since these are precisely the smooth formal-analytic surfaces over $\Spec \Z$ such that the underlying compact Riemann surface is simply connected.
 
  In particular we prove that the algebra of regular functions on $\Vfa(\overline{D}(0,1), \psi)$, namely:
 $$\cO(\Vfa(\overline{D}(0,1), \psi)):= \left\{ \widehat{\alpha} \in \Z[[T]] \mid \widehat{\alpha} \circ \psi^{-1} \mbox{ has a radius of convergence $>1$} \right\}$$  is ``very large" when $\Vfa(\overline{D}(0,1), \psi)$ is pseudoconvex --- namely when $\vert \psi'(0) \vert >1$ ---  and consists only of the constant functions with values in $\Z$ when $\Vfa(\overline{D}(0,1), \psi)$ is pseudoconcave --- namely when  $\vert \psi'(0) \vert <1$  --- and $\psi$ is generic.\footnote{In Chapter \ref{PseudoconII}, we will prove that this algebra is always a $\Z$-algebra of finite type when $\vert \psi'(0) \vert <1$.}
 
 Our proof of the ``generic smallness" of $\cO(\Vfa(\overline{D}(0,1), \psi))$ in the pseudoconcave case relies on some curious measure theoretic arguments involving the subgroup:
 $$D(\R) := X + X^2 \R[[X]],$$
 of $G_{\mathrm{for}}(\R)$. 
 Equipped with the topology induced by the natural topology of Fr\'echet space  on  $\R[[X]]$, defined by the simple convergence of coefficients, $D(\R)$ is a Polish group. Since it is not locally compact, it does not admit any non-trivial invariant Borel measure. However it contains:
 $$D(\Z):= X + X^2 \Z[[X]]$$
 as a closed subgroup, and remarkably the quotient $D(\R)/D(\Z)$ is compact and admits a canonical probability measure $\bar \mu$ that is invariant under the left action of $D(\R)$. 
 
 The canonical  measure $\bar \mu$ on $D(\R)/D(\Z)$ turns out to be very useful when investigating the generic properties of the formal-analytic surfaces $\Vfa(\overline{D}(0,1), \psi)$. Actually the isomorphism class of the  formal-analytic surface $\Vfa(\overline{D}(0,1), \psi)$ depends only on the class of $\psi$ in:
 $$G_{\mathrm{for}}(\R) / D(\Z)  \simeq (\R^\ast \ltimes D(\R))/ D(\Z),$$
 and, for a given value of $\vert \psi'(0)\vert$, this isomorphism class is given by some  element of the quotient $D(\R)/D(\Z)$.

 \medskip
 
 In this chapter, we denote by $K$ a number field, and by $\OK$ its ring of integers.

\section{Definitions}\label{BasicFA}

\subsection{Smooth formal-analytic surfaces}\label{subsubsection:definition-fa}
As in \cite[10.6.1]{Bost2020}, a \emph{smooth formal-analytic arithmetic surface over $\Spec \OK$}, or shortly \emph{over $\OK$,} is defined as a pair:
$$\Vfa:=(\Vf, (V_{\sigma}, O_\sigma, \iota_{\sigma})_{\sigma: K\hra\C})$$
where:
\begin{itemize}
\item $\Vf$ 
is a Noetherian affine formal scheme over $\OK$:
$$\pi : \Vf=\mathrm{Spf} B\lra \Spec\OK$$
such that the restriction of $\pi$ to the reduced scheme of definition $\vert \Vf \vert$ of $\Vf$ is an isomorphism:
\begin{equation}\label{pirestricted}
\pi_{\mid \vert \Vf \vert}: \vert \Vf \vert \lrasim \Spec \OK,
\end{equation}
and such that the topological algebra $B=\Gamma(\Vf, \mathcal O_{\Vf})$ is formally smooth over $\OK$, and $\pi$ has one-dimensional fibers; 
\item for every complex embedding $\sigma$ of $K$, $V_{\sigma}$ is a connected compact  Riemann surface with nonempty boundary,\footnote{in the sense of \ref{RSB}; it would be more properly denoted by $(V_\sigma, V_\sigma^+)$.} $O_{\sigma}$ is a point in the interior $\mathring V_{\sigma}$ of $V_{\sigma}$, and 
$$\iota_{\sigma} : \Vf_{\sigma}\lrasim \widehat V_{\sigma, O_{\sigma}}$$
is an isomorphism between the smooth formal complex curve $\Vf_{\sigma}:=\Vf\otimes_{\OK, \sigma}\C$, defined by the topological $\C$-algebra $B \widehat{\otimes}_{\OK, \sigma} \C$, and the formal completion $\widehat V_{\sigma, O_{\sigma}}$ of $V_{\sigma}$ at the point~$O_{\sigma}$. 
\end{itemize}

These data are moreover assumed to be compatible with complex conjugation; see \emph{loc. cit.}.  In particular, the compact Riemann surface with boundary:
$$V_{\C} :=\coprod_{\sigma : K\hra\C} V_{\sigma}$$
is endowed with a real structure, that is  with an antiholomorphic involution --- which we shall refer to as complex conjugation -- which exchanges $(V_{\sigma}, O_\sigma)$ and $(V_{\overline\sigma}, O_{\overline{\sigma}})$. 
The isomorphism:
$$(\iota_{\sigma})_{\sigma : K\hra\C} :\coprod_{\sigma : K\hra\C}\Vf_\sigma\lrasim 
\coprod_{\sigma : K\hra\C} \widehat{V}_{\sigma, O_{\sigma}}$$
is compatible with the canonical real structure on:
$$\Vf_{\C}:=\Vf\otimes_{\Z}\C \simeq\coprod_{\sigma : K\hra\C}\Vf_\sigma$$ 
and the real structure on: 
$$\coprod_{\sigma : K\hra\C} \widehat{V}_{\sigma, O_{\sigma}}$$
deduced from the real structure on $V_\C$.

The isomorphism inverse of \eqref{pirestricted} defines a section of $\pi,$ that will be denoted:
$$P : \Spec \OK \lra \Vf.$$
Endowed with $P,$ $\Vf$ is a pointed smooth formal curve over $\Spec \OK$ in the sense of \cite[10.4.1]{Bost2020}.

For every field embedding $\sigma: K \hra \C,$ the support\footnote{namely, the underlying topological space.} of the complex formal curve $\Vf_\sigma$ (resp. $\widehat{V}_{\sigma, O_{\sigma}}$) is the set $\{P_\sigma\}$, (resp. the set $\{O_\sigma\}$), and the isomorphism $\iota_\sigma$ sends $P_\sigma$ to $O_\sigma$. It will be notationally convenient to identify these two complex points, and from now on, we will write $P_\sigma$ instead of $O_\sigma$.

If $\mathcal I$ denotes the largest ideal of definition of $B$, we define the \emph{normal line bundle} of  $\Vf$ as:
$$N_{P}\Vf:=P^{*}(\mathcal I/\mathcal I^{2}).$$
It is an invertible sheaf over $\OK$, and there exists a (non-canonical) isomorphism of topological $\OK$-algebras:
$$B\lrasim \widehat{\mathrm{Sym}}_{\OK}(N_{\cP}\Vf)^{\vee};$$
see \cite[Section 10.4]{Bost2020} for more details.

\medskip

The smooth formal-analytic surfaces over $\Spec \OK$ will often be refered to 
as \emph{smooth formal-analytic arithmetic surfaces} when the base scheme $\Spec \OK$ is clear from the context, or, for short, as \emph{smooth \fa arithmetic surfaces}.


\subsection{Vector bundles and spaces of sections} 

\subsubsection{}\label{VectbundleDef} As discussed in \cite[10.6.1]{Bost2020}, there is a natural notion of a
 \emph{vector bundle}: 
 $$\widetilde E:=(\widehat E, (E_{\sigma}, \phi_{\sigma})_{\sigma : K\hra\C}),$$
 and of a \emph{Hermitian vector bundle}:
 \begin{equation}\label{defHermiVectVfa}
 \widetilde{\overline E}:=(\widehat E, (E_{\sigma}, \phi_{\sigma}, \Vert.\Vert_{\sigma})_{\sigma : K\hra\C})
 \end{equation}
  over a smooth \fa arithmetic surface $\Vfa$ as above.

These are defined by  a vector bundle $\widehat E$ over the formal scheme $\Vf$, a complex analytic vector bundle $E_{\sigma}$ (resp. a Hermitian vector bundle $(E_{\sigma}, \Vert.\Vert_{\sigma})$)  over the Riemann surface with boundary\footnote{See  \ref{RSB} for the definitions of complex analytic and Hermitian vector bundles over Riemann surfaces with boundary. } $V_{\sigma}$, and by some gluing data $(\phi_{\sigma})_{\sigma : K\hra\C}$ that consists of isomorphisms:
$$\phi_{\sigma} : \widehat E_{\sigma}:=\widehat E\otimes_{\OK, \sigma}\C\lrasim \iota_{\sigma}^{*}\big(E_{\sigma \mid \widehat{V}_{\sigma, P_{\sigma}}}\big)$$
of vector bundles over the complex formal curves $\cV_\sigma.$
These data are assumed to be compatible with complex conjugation.

The usual tensor operations --- such as the direct sum, the tensor product, or the exterior powers --- make sense for Hermitian vector bundles over smooth \fa arithmetic surfaces, and satisfy formal properties similar to the properties of Hermitian vector bundles over reduced analytic spaces or over reduced schemes of finite type over $\Spec \Z$.

In particular, to any smooth \fa arithmetic surface $\Vfa$ as above, we may attach the Picard group $\overline{\Pic} (\Vfa)$, defined as the set of isomorphism classes of Hermitian line bundles over $\Vfa$, endowed with the commutative group law defined by the tensor product. We may also introduce its subgroup $\overline{\Pic}_{\cC^{\bD}} (\Vfa)$ defined by those Hermitian line bundles whose metrics are restriction to $V_\C$ of metric with $\cC^{\bD}$ regularity on:
$$V^+_\C := \coprod_{\sigma: K \hra \C} V_\sigma^+.$$

\subsubsection{}\label{directimageproH} Consider a formal-analytic arithmetic surface $\Vfa$ and a Hermitian vector bundle $\widetilde{\overline E}$ on $\Vfa$ as above. We assume that  the Riemann surface $V_{\C}$, is endowed with
a positive smooth volume forme $\mu$
invariant under complex conjugation, and, for every field embedding $\sigma: K \hra \C$, we let:
$$\mu_{\sigma}:=\mu_{|V_{\sigma}}.$$

To  $\Vfa$, $\widetilde{\overline E}$, and $\mu$ is associated a pro-Hermitian vector bundle over $\Spec\OK$, in the sense of \cite[Chapter 5]{Bost2020}:
\begin{equation}\label{GammaL2}
\Gamma_{L^{2}}(\Vfa, \mu ; \widetilde{\overline E}):=\Big(\Gamma(\Vf, \widehat E), (\Gamma_{L^{2}}(V_{\sigma}, \mu_{\sigma} ; E_{\sigma}, \Vert.\Vert_{\sigma}), \widehat\eta_{\sigma})_{\sigma:K\hra\C}\Big),
\end{equation}
see \cite[10.6.2]{Bost2020}. Here $\Gamma_{L^{2}}(V_{\sigma}, \mu_{\sigma} ; E_{\sigma}, \Vert.\Vert_{\sigma})$ denotes the Hilbert space of holomorphic sections $s$ of $E_{\sigma}$ over $\mathring{V}_{\sigma}$ such that 
\begin{equation}\label{L2def}
\Vert s\Vert_{\sigma, L^{2}}^{2}:=\int_{V_{\sigma}}\Vert s(x)\Vert^{2}_{\sigma} d\mu_{\sigma}(x)
\end{equation}
is finite, and $\widehat \eta_{\sigma}$ is the ``Taylor series expansion'':
$$\widehat \eta_{\sigma} : \Gamma_{L^{2}}(V_{\sigma}, \mu_{\sigma} ; E_{\sigma}, \Vert.\Vert_{\sigma})\lra 
\Gamma(V_{\sigma, P_{\sigma}}, E_{\sigma})
\simeq  \Gamma(\Vf_{\sigma, P_{\sigma}}, \widehat E_{\sigma})
\simeq \Gamma(\Vf, \widehat E)\widehat{\otimes}_{\OK, \sigma}\C$$
that sends a $L^{2}$ holomorphic section $s$ of $E_{\sigma}$ over $\mathring{V}_{\sigma}$ to its jet of infinite order  at $P_{\sigma}$.

In this monograph, we shall use the notation:
\begin{equation}\label{newpistar}
\pi^{L^2}_{(\Vfa, \mu)\ast} \widetilde{\overline E} 
:=\Gamma_{L^{2}}(\Vfa, \mu ; \widetilde{\overline E})
\end{equation}
for the pro-Hermitian vector bundle defined by \eqref{GammaL2}. When it is $\theta$-finite, as defined in  \cite[7.7.2]{Bost2020},  we shall denote its $\hot$-invariant by:
\begin{equation}\label{hotpistar}
\hot(\Vfa, \mu; \widetilde{\overline E}) := \hot\Big(\pi^{L^2}_{(\Vfa, \mu)\ast} \widetilde{\overline E}\Big) \quad (\in \R_+).
\end{equation}
This non-negative real number may be seen as  an arithmetic analogue of the dimension of the space of analytic (resp. regular) sections of an analytic vector bundle (resp.  of a vector bundle) over a germ of analytic surface (resp. over a formal scheme) fibered over a smooth projective curve, as considered in Chapter \ref{CNBfibered}. 

The new notation \eqref{newpistar} is intended to avoid any  confusion with the $\OK$-module:
$$\Gamma(\Vfa, \widetilde E):=\Big\{(\widehat s, (s_{\sigma})_{\sigma : K\hra\C})\in \Gamma(\Vf, \widehat E)\times\prod_{\sigma : K\hra\C}\Gamma(V_{\sigma}^+, E_{\sigma})\,\big|\, \forall \sigma : K\hra\C,\,\widehat\eta_{\sigma}(s_{\sigma})=\widehat s {\otimes}_\sigma
1\Big\}$$
attached to a vector bundle $\widetilde E$ on $\Vfa$; here $\widehat\eta_{\sigma}$ denotes again the Taylor series expansion at $P_{\sigma}$:
$$\widehat \eta_{\sigma} : \Gamma (V_{\sigma}^+,  E_{\sigma})\lra 
 \Gamma(\Vf, \widehat E)\widehat{\otimes}_{\OK, \sigma}\C,$$
 on the space of sections of $E_\sigma$ that are analytic up to the boundary of $V_\sigma$.

The projection maps:
$$\Gamma(\Vfa, \widetilde E)\lra \Gamma(\Vf, \widehat E)$$
and: 
$$\Gamma(\Vfa, \widetilde E)\lra\Gamma(V_{\sigma}^+, E_{\sigma})$$
are injective, and the elements of $\Gamma(\Vfa, \widetilde E)$ may be described as those sections of the formal vector bundle $\widehat E$ on $\Vf$ that extend to sections of the vector bundles $E_{\sigma}$, analytic up to the boundary of the Riemann surfaces $V_{\sigma}$. Observe that this condition of analyticity on $V^+_{\sigma}$ is stronger than the analyticity over $\mathring{V}_{\sigma}$ together with the finiteness of \eqref{L2def} which enters in the definition of $\Gamma_{L^{2}}(\Vfa, \mu ; \widetilde{\overline E}).$

\subsubsection{}\label{proHtfc} The construction of the  direct image pro-Hermitian vector bundle $\pi^{L^2}_{(\Vfa, \mu)\ast} \widetilde{\overline E}$ associated to a Hermitian vector bundle $\widetilde{\overline E}$ over $\Vfa$ admits the following minor generalization which will be useful in the next chapters.

We may define a \emph{Hermitian torsion free coherent sheaf} $ \widetilde{\overline E}$ over $\Vfa$  as a pair \eqref{defHermiVectVfa}, where the  
$(E_{\sigma}, \Vert.\Vert_{\sigma})$) are still Hermitian vector bundles over the Riemann surfaces with boundary $V_{\sigma}$, but where now $\Eh$ is a torsion free coherent sheaf on the formal scheme $\Vf$. 

Recall that, for every  torsion free coherent sheaf $\Eh$ on $\Vf$, the bidual coherent sheaf $\Eh^{\vee \vee}$ is a vector bundle over $\Vf$, the biduality morphism:
$$\iota: \Eh \lra \Eh^{\vee \vee}$$
is an injective morphism of $\cO_{\Vf}$-modules, and its cokernel $\coker \iota$ is a coherent $\cO_\Vf$-module supported by a finite set of closed points in $\vert \Vf \vert \simeq \Spec \OK$. 

 In other words, $\coker \iota$ is the $\cO_{\Vf}$-module associated to a \emph{finite} $\cO(\Vf)$-module $C$. In particular, for every embedding $\sigma: K \hra \C$, the vectors bundles $\Eh^{\vee \vee}_\sigma$ and $\Eh_\sigma$ on the formal curve $\Vf_\sigma$ may be identified.

Associated to the vector bundle $\Eh^{\vee \vee}$ over $\Vf$, we may define the Hermitian vector bundle over~$\Vfa$:
 \begin{equation}\label{defHermiVectVfaBid}
 {\widetilde{\overline E}}^{\vee \vee}:=(\Eh^{\vee \vee}, (E_{\sigma}, \phi_{\sigma}, \Vert.\Vert_{\sigma})_{\sigma : K\hra\C})
 \end{equation}
 
 The short exact sequence of topological $\OK$-modules:
 \begin{equation}\label{EbidC}
0 \lra \Gamma(\Vf, \Eh) \stackrel{\iota}{\lra} \Gamma (\Vf, \Eh^{\vee \vee}) \lra C \lra 0
\end{equation}
shows that  $\Gamma(\Vf, \Eh)$, like $\Gamma (\Vf, \Eh^{\vee \vee})$, is an object of the category $\CTC_{\OK}$, with the notation of \cite[Chapter 4]{Bost2020}. Since the direct image of $\widetilde{\overline E}^{\vee \vee}$, namely:
$$
\pi^{L^2}_{(\Vfa, \mu)\ast} \widetilde{\overline E}^{\vee \vee} :=\Big(\Gamma(\Vf, \widehat E), (\Gamma_{L^{2}}(V_{\sigma}, \mu_{\sigma} ; E_{\sigma}, \Vert.\Vert_{\sigma}), \widehat\eta_{\sigma})_{\sigma:K\hra\C}\Big),
$$
is a pro-Hermitian vector bundle over $\Spec \OK,$ this shows  that the direct image of $\widetilde{\overline E}$, defined as:
$$
\pi^{L^2}_{(\Vfa, \mu)\ast} \widetilde{\overline E}:=\Big(\Gamma(\Vf, \widehat E^{\vee \vee}), (\Gamma_{L^{2}}(V_{\sigma}, \mu_{\sigma} ; E_{\sigma}, \Vert.\Vert_{\sigma}), \widehat\eta_{\sigma})_{\sigma:K\hra\C}\Big),
$$
 is also a pro-Hermitian vector bundle over $\Spec \OK$. 
 
 Actually the short exact sequence \eqref{EbidC} show that the properties of the pro-Hermitian vector bundles $\pi^{L^2}_{(\Vfa, \mu)\ast} \widetilde{\overline E}^{\vee \vee}$ and $\pi^{L^2}_{(\Vfa, \mu)\ast} \widetilde{\overline E}$ are closely related.  For instance, we have:
 
 \begin{proposition}\label{CompPistarEbid} For every
 Hermitian torsion free coherent sheaf $ \widetilde{\overline E}$ over $\Vfa$, the pro-Hermi\-tian vector bundles $\pi^{L^2}_{(\Vfa, \mu)\ast} \widetilde{\overline E}^{\vee \vee}$ is $\theta$-finite if and only if $\pi^{L^2}_{(\Vfa, \mu)\ast} \widetilde{\overline E}$ is $\theta$-finite. When this holds, using the notation \eqref{hotpistar}, their $\theta$-invariants satisfy the following estimates:
 \begin{equation}\label{comphotbid}
\hot(\Vfa, \mu; \widetilde{\overline E}^{\vee \vee}) - \log \vert C \vert \leq  
\hot(\Vfa, \mu; \widetilde{\overline E}) \leq 
\hot(\Vfa, \mu; \widetilde{\overline E}^{\vee \vee}). 
\end{equation}
 \end{proposition}

Proposition \ref{CompPistarEbid} is a direct consequence of the following result of independent interest concerning pro-Hermitian vector bundles over $\Spec \OK$.

\begin{proposition}\label{ProFiniteIndex} Let $\Fbh := (\Fh, (F^\hilb_\sigma, \Vert. \Vert_{\sigma}, i_\sigma)_{\sigma: K \hra \C})$ be a pro-Hermitian vector bundle over $\Spec \OK$.  If $\Fh'$ is an open $\OK$-submodule of finite index in $\Fh,$ then: 
\begin{equation}\label{defFprime}
\Fbh' := (\Fh', (F^\hilb_\sigma, \Vert. \Vert_{\sigma}, i_\sigma)_{\sigma: K \hra \C})
\end{equation}
is a pro-Hermitian vector bundle over $\Spec \OK$. Moreover $\Fbh'$ is $\theta$-finite if and only if $\Fbh$ is, and when this holds, the following estimates are satisfied:
$$\hot(\Fbh') \leq \hot(\Fbh) \leq \hot(\Fbh') + \log \vert \Fh/\Fh' \vert.$$
\end{proposition}

Observe that the right-hand side of \eqref{defFprime} indeed defines a pro-Hermitian vector bundle since, for every embedding $\sigma: K \hra \C$, $\Fh'_\sigma: = \Fh' \widehat{\otimes}_{\OK, \sigma}\C$ may be identified to $\Fh_\sigma: = \Fh \widehat{\otimes}_{\OK, \sigma}\C$ by the completed base change of the inclusion map $\Fh' \hra \Fh.$    

 Proposition \ref{ProFiniteIndex} is a simple consequence of its special case when $\Fbh$ is a Hermitian vector bundle, established in \cite[Corollary 2.3.4]{Bost2020}, and of the characterizations of strongly summable and $\theta$-finite pro-Hermitian vector bundles in \cite[Section 7.7]{Bost2020}.

\subsection{The algebra $\cO(\Vfa)$ and the field $\cM(\Vfa)$}\label{defcoVM}

We may apply the construction of the space of sections $\Gamma(\Vfa, \widetilde{E})$ to the  vector bundle $\mathcal O_{\Vfa}$ over $\Vfa$ defined by:
$$\mathcal O_{\Vfa}:=(\mathcal O_{\Vf}, (\mathcal O^{\mathrm{an}}_{V_{\sigma}}, \phi_{\sigma})_{\sigma : K\hra\C}),$$
where $\phi_{\sigma}$ is the tautological isomorphism:
$$\mathcal O_{\Vf}\widehat{\otimes}_{\OK, \sigma}\C\lrasim \iota_{\sigma}^{*}\mathcal O_{\widehat V_{\sigma, P_{\sigma}}}.$$
This defines an $\OK$-algebra:
$$\mathcal O(\Vfa):=\Gamma(\Vfa, \mathcal O_{\Vfa}) \subseteq 
\cO(\Vf) \times \prod_{\sigma: K \hra \C} \cO^{\an}(V_\sigma^+),$$
the algebra of \emph{regular functions on $\Vfa$,} 
whose elements may be described as  those formal functions on $\Vf$ that extend to  functions analytic up to the boundary on the Riemann surfaces $V_{\sigma}$.
It is readily seen that the canonical morphism:
$$\OK\lra \mathcal O(\Vfa), \quad  a \longmapsto \left(\pi^\ast a, (\sigma(a))_{\sigma: K \hra \C}\right)$$
is injective, and that $\mathcal O(\Vfa)$ is an integrally closed domain.

In the same spirit, we may consider the field $\cM(\Vf)$ of  formal meromorphic functions on $\Vf$, as defined in \cite{HironakaMatsumura68},\footnote{The field  $\cM(\Vf)$ coincides with the fraction field of the domain $\cO(\Vf) := \Gamma(\vert \Vf \vert, \cO_{\Vf})$. For instance, when $\OK =\Z$, the formal scheme $\Vf$ is isomorphic to $\Spf \Z[[T]]$ and $\cM(\Vf)$ to $\mathrm{Frac}\, \Z[[T]]$.} and the fields $\cM(V_\sigma^+)$ of meromorphic functions  up to the boundary on the Riemann surfaces $V_\sigma$, and define the field of \emph{ meromorphic functions on $\Vfa$} as follows:
$$\cM(\Vfa):=\Big\{(\hat f, (f_{\sigma})_{\sigma : K\hra\C})\in \cM(\Vf)\times\prod_{\sigma : K\hra\C}\cM(V_{\sigma}^+)\,\big|\, \forall \sigma : K\hra\C,\,\widehat\eta_{\sigma}(f_{\sigma})=\hat f {\otimes}_\sigma1\Big\}.$$
Here $\widehat{\eta}_\sigma$ denotes the Laurent expansion at $P_\sigma$:
$$\widehat{\eta}_\sigma: \cM(V_\sigma^+) \lra \mathrm{Frac}\, \cO^\an_{V_\sigma\!, P_\sigma} \hra  \mathrm{Frac}\, \cO_{\widehat V_\sigma\!, P_\sigma},$$
and $\hat f \otimes_\sigma \!1$ is an element of the $\C$-algebra $\cM(\Vf) \otimes_{\OK\!,\sigma} \C$, which is canonically embedded in $\mathrm{Frac}\, (\mathcal O_{\Vf}\widehat{\otimes}_{\OK\!, \sigma}\C)$, which in turn is isomorphic to 
$\mathrm{Frac}\, \cO_{\widehat V_\sigma\!, P_\sigma}$
by $\iota_\sigma^\ast$.  If $z$ denotes a local analytic coordinate on some neighborhood of $P_\sigma$ in $\mathring{V}_\sigma$ that vanishes at $P_\sigma,$ the field $\mathrm{Frac}\,  \cO^{\mathrm{an}}_{V_\sigma, P_\sigma}$ (resp.  $\mathrm{Frac} \, \cO_{\Vh_\sigma, P_\sigma}$) may be identified with the field of Laurent series $\C\{z\}[z^{-1}]$ (resp. $\C[[z]][z^{-1}]$).

In brief, an element of the field $\cM(\Vfa)$ is a formal  meromorphic function on $\Vf$ that extends to  meromorphic functions  up to the boundary on the Riemann surfaces $V_\sigma^+$. 

The $\OK$-algebra $\cO(\Vfa)$ naturally embeds into $\cM(\Vfa)$, which is therefore an extension of the fraction field $\mathrm{Frac}\, \cO(\Vfa)$, hence of the number field $K$.

\section[Arakelov divisors and intersection numbers on formal-analytic surfaces]{Arakelov divisors and intersection numbers on \fa arithmetic surfaces}\label{ArIntFa}

In this section, we still denote by $\Vfa:=(\Vf, (V_{\sigma}, O_\sigma, \iota_{\sigma})_{\sigma: K\hra\C})$ a smooth \fa arithmetic surface as defined in \ref{subsubsection:definition-fa}. 

\subsection{The group  $\Zb^1_c(\Vfa)$} 
We define an \emph{Arakelov divisor with compact support} on $\Vfa$ as a pair $(D,g)$ where:
\begin{itemize}
\item 
 $D$ is a $1$-dimensional algebraic cycle on $\vert \Vf \vert$;
 \item $g$ is 
a Green function (with $\cC^{\bD}$ regularity) for the divisor $D_\C$ on the Riemann surface $V^+$
that is invariant under complex conjugation, and whose support satisfies:
\begin{equation}\label{suppgV}
\supp g \subseteq V_\C := \coprod_{\sigma: K \hra \C} V_\sigma.
\end{equation}
\end{itemize}

The first condition on $D$ simply means that $D$ is of the form:
$D = n P$
for some $n \in \Z.$ 

For every field embedding $\sigma: K \hra \C$, we shall use the notation:
$$g_\sigma := g_{\mid V_\sigma^+}.$$ The support condition  \eqref{suppgV} on $g$ is equivalent to the vanishing of  the each of these functions $g_\sigma$ on $V^+_\sigma \setminus \mathring{V}_\sigma$. 
In particular the restriction  $g_{\sigma \mid V_\C}$ of $g_\sigma$ to the the compact submanifold $V_\sigma$ satisfies the ``Dirichlet boundary condition":
$$g_{\mid \partial V_\sigma} =0.$$

The Arakelov divisors with compact supports on $\Vfa$ define an additive group which will be denoted by $\Zb^1_c(\Vfa).$ There is a canonical morphism of commutative groups:
\begin{equation}\label{Z1Pic}
\Zb^1_c(\Vfa) \lra \overline{\Pic}_{\cC^{\bD}} (\Vfa),
\end{equation}
that maps the Arakelov divisor $(D,g)$ in $\Zb^1_c(\Vfa)$ to the isomorphism class of the Hermitian line bundle over $\Vfa$:
$$\widetilde{\cOb}(D,g) := (\cO_{\Vf} (D), (\cO^{\an}_{V^+_\sigma}(D_\sigma), \Vert.\Vert_{g_\sigma}, \phi_\sigma)_{\sigma: K \hra \C}),$$
where the isomorphism $\phi_\sigma$ is the tautological isomorphism induced by $\iota_\sigma$.

Let us emphasize that the properties of the morphism \eqref{Z1Pic}  differ strikingly from those of the analogue morphism:
$$
\Zb^1_c(X) \lra \overline{\Pic}_{\cC^{\bD}} (X),
$$
relating Arakelov divisors and Hermitian line bundles on a regular projective arithmetic surface $X$.
For instance, when $\OK = \Z$, the morphism \eqref{Z1Pic} is easily seen to be injective. Moreover its image is expected to be a ``very small" subgroup of   $\overline{\Pic}_{\cC^{\bD}} (\Vfa)$ when $\Vfa$ belongs to the class of pseudoconcave \fa arithmetic surfaces, which are investigated in the next chapters.

\subsection{The Arakelov divisor $(P, g_{\Vfa_\C})$ and the line bundle $\Nb_P \Vfa$}\label{PgVNb} 

The divisor
$$P_\C := \sum_{\sigma: K \hra \C} P_\sigma$$
in $V^+_\C$ admits a distinguished Green function ${g}_{\Vfa_\C}$ that satisfies the support condition \eqref{suppgV}, namely the Green function defined by the equilibrium potentials for the points $P_\sigma$ in the connected compact Riemann surfaces with boundary $V_\sigma$:
$$g_{\Vfa_\C \mid V_\sigma^+} := g_{V_\sigma, P_\sigma},$$
introduced in \ref{CompactBoundEqu} above.

The Green function $g_{\Vfa_\C}$ is invariant under complex conjugation, and the pair $(P, g_{\Vfa_\C})$ defines an element of $\Zb^1_c(\Vfa)$, canonically attached to $\Vfa$.

To this Arakelov divisor is attached the Hermitian line bundle $\widetilde{\cOb}(P,g_{\Vfa_\C})$ on $\Vfa$. Its restriction to $P = \vert \Vf \vert$ may be identified with the Hermitian line bundle:
\begin{equation}\label{defNPVfa}
\overline{N}_P \Vfa := \left(N_P \Vf, (\Vert.\Vert^{\mathrm{cap}}_{V_\sigma, P_\sigma})_{\sigma: K \hra \C}\right)
\end{equation}
defined by the normal bundle $N_P \Vf$ of $P$ in $\Vf$ equipped with the capacitary metrics $\Vert.\Vert^{\mathrm{cap}}_{V_\sigma, P_\sigma}$ on the complex lines:
$$ (N_P \Vf)_\sigma \stackrel{T_{\iota_\sigma}}{\lrasim} N_{P_\sigma} V_\sigma = T_{P_\sigma} V_\sigma,$$
introduced in \ref{CapMetricDef} and \ref{CbDvaria} above.

The group $\Zb^1_c(\Vfa)$ admits a simple description in terms of the Arakelov divisor $(P,g_{\Vfa_\C})$. 

Let us define $\cC^{\bD}(V_\C)_{\R, \Dir}$ as the subspace of $\cC^{\bD}(V^+_\C)$ consisting in the $\cC^{\bD}$-functions $\phi:V^+_\C \ra \R$ that are invariant under complex conjugation acting on $V^+_\C$ and satisfy the following ``Dirichlet boundary condition:"
$$\supp \phi \subseteq V_\C.$$ 

The following proposition is then a straightforward consequence of the definitions:

\begin{proposition} The morphism of commutative groups:
$$\Z \oplus \cC^{\bD}(V_\C)_{\R, \Dir} \lra \Zb^1_c(\Vfa),
\quad (n, \phi) \longmapsto n(P, g_{\Vfa_\C}) + (0, \phi)$$
is an isomorphism.
\end{proposition}

\subsection{Arithmetic intersection theory on $\Vfa$}\label{ArIntVfa}

Let $\Lb$ be a Hermitian line bundle on $\Vfa$. Its restriction to $\vert \Vf \vert = P$ is a Hermitian line bundle on a one-dimensional scheme proper over $\Spec \Z$, and admits a well defined arithmetic degree:
$$\dega (\Lb \vert P ) = \height_{\Lb} (\vert \Vf \vert)  := \dega P^\ast \Lb.$$

More generally, a divisor $D$ in $\Vfa$ supported by $\vert \Vf \vert$ may be written $D = n P,$ with $P \in \Z,$ and we have:
$$\dega (\Lb \vert D ) := n\, \dega P^\ast \Lb.$$

For every Hermitian line bundle $\overline L$ on $\Vfa$, defined by a metric of regularity $\mathcal C^{\bD}$, and any Arakelov divisor in $\Zb^1_c(\Vfa)$, we may define their arithmetic intersection number by the same formula as in Subsections \ref{ArIntCinfty} and \ref{intCbD}:
\begin{equation}\label{intVfa}
\overline L \cdot(D, g):=\dega (\Lb \vert D ) + \int_{V^+_{\C}}g\, c_{1}(\overline L_\C)  \quad \in \R.
\end{equation}
The product $g\, c_{1}(\overline L_\C)$ is a product of a function in $C^{\bD}(V^+_\C)$ and of a measure in $\bM^{\cp}(V^+_\C)$, and is supported by the compact submanifold $V_\C$. The integral in the right-hand side of \eqref{intVfa} is therefore well-defined, as discussed in \ref{CbDIntegral}, and could also be written: 
$$\int_{V_{\C}}g\, c_{1}(\overline L_\C).$$

We may specialize the definition \eqref{intVfa} of arithmetic intersection numbers to Hermitian line bundles whose isomorphism class belong to the image of the morphism \eqref{Z1Pic}. Thus we define, for any two Arakelov divisors $(D,g)$ and $(D',g')$ in $\Zb^1_c(\Vfa)$:
\begin{equation}\label{intVfabis}
(D', g') \cdot(D, g):= \widetilde{\cOb}(D', g') \cdot (D,g) = \dega (\widetilde{\cOb}(D', g') \vert D ) + \int_{V_{\C}}g\, \omega(g').
\end{equation}


Observe that, for every $(D,g)\in  \Zb^1_c(\Vfa)$ and any $\phi \in  \cC^{\bD}(V_\C)_{\R, \Dir},$ we have:
\begin{align}\label{phiDg}
(0, \phi) \cdot (D,g)  & = \dega (\cOb (0, \phi) \mid D )  + \int_{V_\C} g \, \frac{i}{\pi} \partial\overline{\partial} \phi \notag\\
& = \int_{V_\C} \phi  \, \delta_{D_\C} + \pi^{-1} \int_{V_\C} g \, i\partial \overline{\partial} \phi.
\end{align}
In particular, for every $\phi_1$ and $\phi_2$ in $\cC^{\bD}(V_\C)_{\R, \Dir}$, we have:
\begin{equation}\label{phiphi}
(0, \phi_1) \cdot (0, \phi_2) = - \pi^{-1} \langle \phi_1 , \phi_2 \rangle_\Dir. 
\end{equation}

The fact that the  measure $\omega(g_{\Vfa_\C})$ is supported by $\partial V_\C$ implies that  arithmetic intersection numbers involving the Arakelov divisor $(P, g_{\Vfa_\C})$ admit simple expressions:

\begin{proposition} For every $n \in \Z$ and every Arakelov divisor of the form $(nP, g)$ in $\Zb^1_c(\Vfa)$, we have:
\begin{equation}\label{intnPg}
(P, g_{\Vfa_\C}) \cdot (nP, g) = n \, \dega P^\ast \Nb_P \Vfa. 
\end{equation}
In particular:
\begin{equation}\label{clear}
(P, g_{\Vfa_\C}) \cdot (P, g_{\Vfa_\C}) =  \dega P^\ast \Nb_P \Vfa, 
\end{equation}
and, for every $\phi \in \cC^{\bD}(V_\C)_{\R, \Dir},$ we have:
\begin{equation}\label{clearbis}
(0, \phi) \cdot (P, g_{\Vfa_\C}) =  0. 
\end{equation}
\end{proposition}

\begin{proof} According to the definition \eqref{intVfabis} of the arithmetic intersection number, we have:
$$ (P, g_{\Vfa_\C}) \cdot (nP, g) = n \,\dega P^\ast \widetilde{\cOb}(P,g_{\Vfa_\C}) + \int_{V_\C} g \, \omega(g_{\Vfa_\C}).$$
As observed in \ref{PgVNb}, the restriction of $\widetilde{\cOb}(P,g_{\Vfa_\C})$ is canonically isomorphic to $\overline{N}_P \Vfa$, and therefore:
$$\dega P^\ast \widetilde{\cOb}(P,g_{\Vfa_\C}) = \dega P^\ast \overline{N}_P \Vfa.$$
Moreover the product $g \, \omega(g_{\Vfa_\C})$ is zero. This establishes \eqref{intnPg}.
The relations \eqref{clear} and \eqref{clearbis} are special cases of \eqref{intnPg}.
\end{proof}

\begin{corollary}
For every $n_1$ and $n_2$ in $\Z$ and $\phi_1$ and $\phi_2$ in $\cC^{\bD}(V_\C)_{\R, \Dir}$, we have:
\begin{equation}\label{intVfaexplicit}
(n_1 P, n_1\, g_{\Vfa_\C} + \phi_1) \cdot (n_2 P, n_2 \, g_{\Vfa_\C}+ \phi_2 ) =  n_1 n_2 \,\dega P^\ast \Nb_P \Vfa - \pi^{-1} \langle \phi_1 , \phi_2 \rangle_\Dir. 
\end{equation}
\end{corollary}

In particular, the arithmetic intersection pairing on  $\Zb^1_c(\Vfa)$ defined by \eqref{intVfabis} is symmetric. Moreover, for every $\phi$ in $\cC^{\bD}(V_\C)_{\R, \Dir}$, we have: 
\begin{equation}\label{intVfaexplicitpart}
(P,  g_{\Vfa_\C} + \phi) \cdot (P, \, g_{\Vfa_\C}+ \phi) =  \dega P^\ast \Nb_P \Vfa - \pi^{-1} \langle \phi , \phi \rangle_\Dir. 
\end{equation}
This shows that the self-intersection of an Arakelov divisor  in $\Zb^1_c(\Vfa)$ of the form $(P, h)$ satisfies:
\begin{equation}\label{selPh}
(P, h) \cdot (P,h) \geq (P, g_{\Vfa_\C}) \cdot (P, g_{\Vfa_\C}),
\end{equation}
and that equality holds in \eqref{selPh} if and only if $h = g_{\Vfa_\C}$.

\section{Pseudoconcavity and pseudoconvexity}\label{pseudobis}

The next two propositions establish an analogue, concerning \fa arithmetic surfaces, of the results about $\mathbf{CNB}$ divisors on analytic surfaces fibered over a projective curve in Proposition \ref{Denough} and Corollary \ref{cor:CNB}.

\begin{proposition}\label{nefVt}
Let $h$ be a Green function $h$ for $P_\C$ in $V^+_\C$ such that $(P, h)$ belongs to $\Zb^1_c(\Vfa)$.

The following two conditions are equivalent:
\begin{enumerate}[(i)]
\item for every $(D,g) \in \Zb^1_c(\Vfa)$:
\begin{equation*}
D \geq 0 \quad \mbox{and} \quad g \geq 0 \,\, \Longrightarrow (P,h) \cdot (D, g) \geq 0;
\end{equation*}
\item $\omega(h)_{\mid \mathring{V}_\C}  \geq 0 \mbox{ and } 
 \dega (\widetilde{\cOb}(P, h) \vert P ) \geq 0.$ 
\end{enumerate}
\end{proposition}

\begin{proof}
Let us assume that $(i)$ holds. Then, for every non-negative function $\psi \in \cC^{\bD}(V_\C)_{\R, \Dir},$ we have:
$$ 0 \leq (P, h) \cdot (0, \psi) =  \int_{V_\C} \psi \, \omega(h).$$
This is readily seen to imply that, over $\mathring{V}_\C$, the measure $\omega(h)$ is non-negative. 

Moreover, for any non-negative Green function $g$ for $P_\C$ in $V_\C^+$ such that $(P,g)$ belongs to $\Zb^1_c(\Vfa)$, we have:
\begin{equation}\label{Pghgeq0}
0 \leq (P,h) \cdot (P,g) = \dega (\widetilde{\cOb}(P, h) \vert P ) + \int_{V_\C} g \, \omega(h). 
\end{equation}
Clearly there exists such a non-negative Green function $g,$ and the construction in \ref{Approx} of $\mathcal C^{\bD}$ functions approximating a Green function  with $\mathcal C^{\bD}$  regularity establishes the existence of a sequence $(g_n)$ in $\cC^{\bD}(V_\C)_{\R, \Dir}$ such that:
$$0 \leq g_n \leq g,$$
and:
$$\lim_{n \ra + \infty} \int_{V_\C} (g-g_n) \, \omega(h) =0.$$
The inequality \eqref{Pghgeq0}, applied to $g-g_n$ instead of $g$, implies:
$$0 \leq \dega (\widetilde{\cOb}(P, h) \vert P ) + \int_{V_\C} (g-g_n) \, \omega(h),$$
and by letting $n$ go to infinity, we obtain:
$$\dega (\widetilde{\cOb}(P, h) \vert P )  \geq 0.$$

Conversely, assume that (ii) holds, and let $(D,g)$ be an element of $\Zb^1_c(\Vfa)$ such that $D \geq 0$ and $g \geq 0.$ Then $D = nP$ for some $n \in \N$, and:
$$(P,h) \cdot (D,g) = (P,h) \cdot (nP, g) = n\, \dega (\widetilde{\cOb}(P, h) \vert P ) + \int_{V_\C} g \, \omega(h)$$
is clearly non-negative.
\end{proof}

Let us denote :
$$\mathcal{SH}(V_\C)_{\R, \Dir} := \left\{ \phi \in \cC^{\bD}(V_\C)_{\R, \Dir} \mid i \partial \overline{\partial} \phi_{\mid \mathring{V}_\C}  \geq 0 \right\}.$$
The functions in $\mathcal{SH}(V_\C)_{\R, \Dir}$ are continuous on $V_\C$, subharmonic on $\mathring{V}_\C$, and vanish on $\partial V_\C.$ Consequently, they are non-positive on $V_\C$.

\begin{proposition}\label{CNBarith} The following two conditions are equivalent:
\begin{enumerate}[(i)]
\item There exists a Green function $h$ for $P_\C$ in $V^+_\C$ such that $(P, h)$ belongs to $\Zb^1_c(\Vfa)$ and such that, for every $(D,g) \in \Zb^1_c(\Vfa)$:
\begin{equation}\label{nefVfa}
D \geq 0 \quad \mbox{and} \quad g \geq 0 \,\, \Longrightarrow (P,h) \cdot (D, g) \geq 0.
\end{equation}

\item $\dega P^\ast\Nb_P \Vfa \geq 0.$
\end{enumerate}

When these conditions are satisfied, the set of Green functions $h$ as in (i) is $g_{\Vfa_\C} + \cC$, where:
$$\cC :=  \Big\{ \phi \in \mathcal{SH}(V_\C)_{\R, \Dir} \mid  \sum_{\sigma:K \hra \C} \phi(P_\sigma) + \dega P^\ast\Nb_P \Vfa \geq 0 \Big\}.$$ 
In particular, (i) holds with $h = g_{\Vfa_C}$, and every Green function $h$ for $P_\C$ as in (i) satisfies: $h \leq g_{\Vfa_\C}.$
\end{proposition} 

Since $g_{\Vfa_C}$ is non-negative,  the conditions (i) and (ii) are also equivalent to the strengthened variant of (i) where $h$ is required to be non-negative. 

\begin{proof} 
The Green functions $h$ for $P_\C$ such that $(P, h)$ belongs to $\Zb^1_c(\Vfa)$ are precisely of the form:
$$h = g_{\Vfa_C} + \phi$$
with $\phi$ in $\cC^{\bD}(V_\C)_{\R, \Dir}$. When this holds, we have:
$$ \omega(h)_{\mid \mathring{V}_\C}  \geq 0 \Longleftrightarrow \phi \in \mathcal{SH}(V_\C)_{\R, \Dir}, $$
and:
$$
\dega (\widetilde{\cOb}(P, h) \vert P )  
= \int_{V_\C} \phi \, \delta_{P_\C} + \dega P^\ast\Nb_P \Vfa   = \sum_{\sigma: K \hra \C} \phi(P_\sigma) + \dega P^\ast\Nb_P \Vfa.
$$
The proposition therefore follows from Proposition \ref{nefVt} and from the non-positivity of the functions in $\mathcal{SH}(V_\C)_{\R, \Dir}.$
\end{proof}

\begin{definition} We shall say that the smooth \fa arithmetic surface $\Vfa$ is \emph{pseudoconcave} (resp. \emph{pseudoconvex}) when:
\begin{equation}\label{pseudocdef}
\dega \Nb_P \Vfa  > 0 \quad \mbox{( resp. $\dega \Nb_P \Vfa  < 0$)}.
\end{equation}
\end{definition}

Proposition \ref{CNBarith} and the inequality \eqref{selPh} show that $\Vfa$ is pseudoconcave if and only if there exists a non-negative Green function $h$ for $P_\C$ in $V^+_\C$ such that $(P, h)$ belongs to  $\Zb^1_c(\Vfa)$ and satisfies the condition of ``numerical effectivity"    \eqref{nefVfa} and the ``bigness" condition:
$$(P, h)\cdot (P,h) > 0.$$
Moreover, when $\Vfa $ is pseudoconcave, these conditions are satisfied by $h = g_{\Vfa_\C}.$

\section[The arithmetic surfaces $\Vfa(\overline{D}(0;1), \psi)$]{The arithmetic surfaces $\Vfa(\overline{D}(0;1), \psi)$}\label{Vfadpsi}

In this section, we discuss some simple examples of smooth \fa arithmetic surfaces $\Vfa$ and of the associated algebra $\cO(\Vfa)$. Our aim is to demonstrate, on some simple but significant examples, the crucial role played the numerical conditions \eqref{pseudocdef} on the ``size" of the ring $\cO(\Vfa)$.

For simplicity, we restrict ourselves to \fa arithmetic surfaces  over $\Z$, and leave it to the interested reader to extend the results of this sections when $\Z$ is replaced by the ring $\cOK$ of integers of an arbitrary number field $K$.

Indeed, when $\OK = \Z,$ the definitions in \ref{subsubsection:definition-fa} and \ref{defcoVM} above become especially simple. Namely, a smooth \fa arithmetic surface 
$\Vfa = (\widehat{\cV}, (V, O, i))$ over $\Z$ is the data of  a formal scheme $\widehat{\cV}$ over $\Spec \Z$, isomorphic to $\mathrm{Spf\, } \Z[[X]]$, of a connected compact Riemann surface with non-empty boundary $V$ equipped with a real structure, of a real point $O$ in $\Vcirc$, and of an isomorphism of  complex formal curves, compatible with the real structures:
$$i: \widehat{\cV}_\C \lrasim \widehat{V}_O.$$

The section $P: \Spec \Z \ra \mathrm{Spf\, } \Z[[X]]$ is the inverse of the tautological isomorphism $\big\vert \mathrm{Spf\, } \Z[[X]] \big\vert \simeq \Spec \Z$. 

Moreover the elements of the ring $\cO(\Vfa)$ (resp. of the field $\cM(\Vfa)$) may be described as pairs $\alpha:=(\hat{\alpha}, \alpha^{\an})$ where $\hat{\alpha}$ and  $\alpha^{\an}$ belong respectively to $\cO(\widehat{\cV})$ and $\cO^{\an}(V^+)$ (resp. to $\cM(\widehat{\cV}) \simeq \mathrm{Frac}\, \Z[[X]]$ and $\cM(V^+)$) and satisfy the following gluing  condition:
\begin{equation*}
i^\ast \alpha^{\an}_{\mid \widehat{V}_O} = \hat{\alpha}_C.
\end{equation*}

\subsection{The \fa arithmetic surfaces $\Vfa (D, \psi)$} 

\subsubsection{} Let $D$ be a compact domain with $\cC^\infty$ boundary\footnote{that is, a compact $\cC^\infty$ submanifold with boundary, of dimension 2.}  in $\C$, that is invariant under complex conjugation and contains $0$ in its interior, and let $\psi$ be a formal series in $\R[[X]]$ such that:
\begin{equation}\label{proprpsi}
\psi(0) =0 \quad \mbox{and}  \quad \psi'(0) \neq 0.
\end{equation}

To $(D, \psi)$,  we may attach the smooth \fa arithmetic surface $\Vfa = (\widehat{\cV}, (V, P, i))$ over $\Z$  defined as follows:
$$\widehat{\cV} := \mathrm{Spf\,} \Z[[X]], \quad V := D, \quad O =0,$$
and:\footnote{In other words, for any $f \in \C[[z]] \simeq \cO_{\hat{D}_0},$$ i^\ast f := f\circ \psi.$}
$$i := \psi : \mathrm{Spf\,} \C[[X]] \lrasim \mathrm{Spf\,} \C[[z]] \simeq \widehat{D}_0.$$
We shall denote it by $\Vfa (D, \psi)$.

By definition, the ring $\cO(\Vfa (D, \psi))$ may be identified with the subring of formal series $\hat{\alpha}$ in $\Z[[X]]$ such that $\hat{\alpha} \circ \psi^{-1}$ is the Taylor expansion in $0$ of an analytic function on some open neighborhood of $D$.  

Similarly, the field $\cM(\Vfa (D, \psi))$ may be identified with the field of Laurent  series $\widehat{\alpha}$ in $\Q((X))$ that belong to its subfield $\mathrm{Frac}\,  \Z[[T]]$, such that $\widehat{\alpha} \circ \psi^{-1}$ is the Taylor expansion in $0$ of an meromorphic function on some open neighborhood of $D$.  

\medskip


\subsubsection{} For instance, for any $r \in \Rpa$, we may define the \emph{Borel \fa arithmetic surface of radius $r$}, denoted by $\widetilde{\cB}(r),$  by the above construction applied to $D:=\overline{D}(0; r)$, the closed disk of radius $r$,  and to the ``identity" series $\psi = X$:
$$\widetilde{\cB}(r) := \Vfa (\overline{D}(0; r), X).$$

The  equilibrium potential for $P_\C = 0$ in $\overline{D}(0; r)$ is the function: $$g_{\overline{D}(0; r), 0}  = (z \mapsto \log^+ \vert z /r \vert),$$ and therefore the associated capacitary metric $\Vert.\Vert_{\overline{D}(0; r), 0}^{\mathrm{cap}}$ satisfies:
$$\Vert \partial /\partial z \Vert_{\overline{D}(0; r), 0}^{\mathrm{cap}} = r^{-1}.$$
Consequently the Hermitian line bundle $\Nb_{P} \widetilde{\cB}(r)$ over $\Spec \Z$ may be identified with:
$$(\Z\;  \partial /\partial X, \Vert \partial /\partial X \Vert = r^{-1}),$$
and therefore satisfies:
\begin{equation}\label{degaBorel}
\dega N_{P} \widetilde{\cB}(r) = \log r.
\end{equation}

The  ``size" of the ring $\cO(\widetilde{\cB}(r))$ clearly depends on the sign of \eqref{degaBorel}. Indeed when $$\dega N_{P} \widetilde{\cB}(r) \geq 0,$$ or equivalently when $r\geq 1$, it is easily seen to be reduced to the ring of polynomials:
$$\cO(\widetilde{\cB}(r)) = \Z[X].$$
In contrast, when $\dega N_{P} \widetilde{\cB}(r) < 0$, that is when $r < 1$, the ring  $\cO(\widetilde{\cB}(r))$  is ``very large": it contains all formal series in $\Z[[X]]$ with bounded coefficients and therefore has the cardinality of the continuum.

\medskip
 
 \subsubsection{}\label{defVfapsi} More generally, to
 any formal series $\psi$ in $\R[[X]]$ satisfying the conditions \eqref{proprpsi}, we may attach the \fa arithmetic surfaces $\Vfa(\overline{D}(0;1), \psi)$. 
When $\psi = X/r$, it is isomorphic to the Borel \fa arithmetic surface $\widetilde{\cB}(r)$.

 The Hermitian line bundle $N_{\cP}\Vfa(\overline{D}(0;1), \psi)$ may be identified with $(\Z \, \partial/\partial X, \Vert.\Vert_\psi)$ where the metric $\Vert.\Vert_\psi$ satisfies:
$$ \Vert \psi'(0)^{-1} \partial/\partial X \Vert_\psi = 1.$$
Consequently:
\begin{equation}\label{degaVpsi}
\dega N_{P}\Vfa(\overline{D}(0;1), \psi) = \log \vert \psi'(0)\vert^{-1}.
\end{equation}

According to the uniformization theorem, every simply connected pointed compact Riemann surface with boundary $(V^+, P)$ is isomorphic to $(\overline{D}(0,1)^+, 0)$. This implies that a smooth \fa arithmetic surface $\Vfa$ over $\Spec \Z$  is isomorphic to a \fa surface $\Vfa(\overline{D}(0;1), \psi)$ for some formal series $\psi$ in the group:
$$G_{\mathrm{for}}(\R) := \R^\ast X + X^2 \R[[X]]$$
if and only if the underlying compact Riemann with boundary is simply connected.

\subsection{The arithmetic surfaces $\Vfa(\overline{D}(0;1), \psi)$: the pseudoconvex case}\label{Vfpsiconvex}

In this subsection, we want to prove that, when the arithmetic degree \eqref{degaVpsi} is negative,  the ring of regular functions $\cO(\Vfa(\overline{D}(0;1), \psi))$ is large. This will extend the discussion above concerning the ring $\cO(\widetilde{\cB}(r))$ associated to the Borel \fa arithmetic surface $\widetilde{\cB}(r)$ when $r <1$.

 To achieve this, we will rely on the following 
elementary proposition:

\begin{proposition}\label{Grelem}
Let $\psi $ be a formal series in $\R[[X]]$ such that:
$$\psi(0)= 0 \quad \mbox{and} \quad \lambda := \psi'(0) \neq 0.$$
For every integer $e \geq 1,$ there exists a formal series in $\Z[[X]]$:
$$\widehat{\alpha} = X^e + \sum_{n \geq e+1} \alpha_n X^n$$
such that the coefficients of the formal series in $\R[[T]]$:
$$\widehat{\alpha} \circ \psi^{-1} (T) =: \lambda^{-e} \;T^e + \sum_{n \geq e+1} a_n T^n$$
satisfy the estimates:
\begin{equation}\label{anestimates}
  \vert a_n \vert \leq \vert \lambda \vert^{-n}/2  \quad \mbox{for every $n \geq e+1$.}
\end{equation}
\end{proposition}

\begin{proof}
For every $n \geq e+1$, the $n$-th coefficients $a_n$ of the formal series $\widehat\alpha \circ \psi^{-1}$ may be written:
\begin{equation*}
a_n = \lambda^{-n} \alpha_n + P_n(\alpha_{e+1}, \dots, \alpha_{n-1}),
\end{equation*}
where $P_n$ denotes some  polynomial with real coefficients in $n-e -1$ variables, depending of $e$ and $\psi$; in particular $P_{e+1}$ is a constant depending only on $e$ and $\psi$. We may construct the coefficients $(\alpha_n)_{n \geq e+1}$ inductively by choosing $\alpha_n$ as an integer such that:
$$\vert \alpha_n +  \lambda^n \, P_n (a_{e+1}, \dots, a_{n-1}, \psi_1, \dots, \psi_n) \vert \leq 1/2.$$
This immediately implies the estimates \eqref{anestimates}
\end{proof}

We will use the following straightforward consequence of the estimates  \eqref{anestimates}:

\begin{corollary}\label{CorGrelem} The radius of convergence of the series $\widehat{\alpha} \circ \psi^{-1} (T)$ is at least $\vert \lambda \vert$. When $\vert \lambda \vert > 1,$ the series $\widehat{\alpha} \circ \psi^{-1}$ defines a complex analytic function on some open neighborhood of $\overline{D}(0;1)$; moreover:
\begin{equation}\label{ineqe}
\max_{z \in \overline{D}(0;1)} \vert \widehat{\alpha} \circ \psi^{-1}(z) \vert \leq \vert \lambda \vert^{-e} + \sum_{n \geq e+1} \vert \lambda\vert^{-n}/2 = \frac{1 - \vert \lambda \vert^{-1}/2}{1 - \vert \lambda \vert^{-1}} \, \vert \lambda \vert^{-e}.
\end{equation}
 \end{corollary}

Corollary \ref{CorGrelem} shows that, when $\vert \lambda\vert >1$, the series $\hat{\alpha}$ defines a non-constant  element $\alpha := (\widehat{\alpha}, \alpha^{\an})$  in the ring $\cO(\Vfa(\overline{D}(0;1), \psi))$.  If $e$ is chosen large enough, the image 
$$\alpha^{\an} (\overline{D}(0;1)) := \widehat{\alpha} \circ \psi^{-1}  (\overline{D}(0;1))$$
is a compact subset of the open disc $\mathring{D}(0;1)$.

The previous construction therefore establishes the following proposition:

\begin{proposition}\label{Grelem+unif}  Let $\Vfa:= (\widehat{\cV}, (V, P, i))$ be a smooth \fa arithmetic surface over $\Z$ such that $V$ is simply connected. If it satisfies the condition:
\begin{equation}\label{arpseudoconvex}
\dega \Nb_{P}\Vfa <0,
\end{equation}
then the algebra $\cO(\Vfa)$ contains a non-constant element $\alpha := (\widehat{\alpha}, \alpha^{\an})$ such that $\widehat{\alpha}(P) = 0$ and $\alpha^{\an}(V)$ is contained in $\mathring{D}(0;1)$.
\end{proposition}

Observe that the conclusion of Proposition \ref{Grelem+unif} implies that the ring $\cO(\Vfa)$ is large. Indeed, for $\alpha$ as in the conclusion of Proposition \ref{Grelem+unif}, for every $r \in (0,1)$ such that: 
$$\alpha^{\an} (V) \subseteq \overline{D}(0;r),$$
we may define an injective morphism of rings:
\begin{equation*}
\cO(\widetilde{\cB}(r)) \lra \cO(\Vfa), \quad \beta := (\widehat{\beta}, \beta^\an) \longmapsto \beta \circ \alpha := (\widehat{\beta} \circ \widehat{\alpha}, \beta^\an \circ \alpha^\an).
\end{equation*}

It is likely that Proposition \ref{Grelem+unif} still holds for  every smooth \fa arithmetic surface over $\Spec \OK$ --- where $K$ denotes an arbitrary number field --- that satisfies the pseudoconcavity condition \eqref{arpseudoconvex}. This would constitute an arithmetic analogue of the theorem of Grauert concerning pseudoconvex neighborhoods of projective curves discussed in \ref{GrauertContract} above.

\subsection{The arithmetic surfaces $\Vfa(\overline{D}(0,1), \psi)$: the pseudoconcave case}\label{pseudoconcave case}

In this subsection and in the next section, we are going to show that, if $\lambda \in \R$ satisfies:
$$0 < \vert \lambda \vert < 1,$$
 then for most choices of the formal series $\psi$ in $\lambda X + X^2 \R[[X]],$ the algebra $\cO(\Vfa (\overline{D}(0,1), \psi)$ is as small as possible, namely is reduced to  $\Z$. 
 

To formulate our results, we need to introduce some notation.

For any commutative ring $A,$ we may consider the subset:
$$D(A) := X + X^2 A[[X]]$$
of the algebra of formal series $A[[X]]$. Endowed with the composition $\circ$ of formal series with vanishing constant terms, it defines a group $(D(A), \circ)$. For every   $n \in \N$, we may consider its normal subgroup:
$$H_n(A) := X + X^{n+2} A[[X]].$$
The quotient group:
$$D_n(A) := D(A)/H_n(A)$$
may be identified with the group $X + X^2 A[[X]]/ (X^{n+2})$ of formal series truncated at order $n+1,$ endowed with the composition. The groups $D_n(A)$ are nilpotent, and $D(A)$ may be identified with the limit:
$$D(A) \lrasim \varprojlim_n D_n(A)$$
of the following diagram of surjective morphisms of groups:
$$D_0(A)=\{ e\} \longleftarrow D_1(A) \longleftarrow \cdots \longleftarrow
D_n(A) \longleftarrow D_{n +1}(A) \longleftarrow \cdots.$$

For every $n \in \N,$ $D_n(\R)$ is a nilpotent Lie group of dimension $n$.
We shall denote by $\mu_n$ the measure on $D_n(\R)$ image of the Lebesgue measure by the diffeomorphism:
$$\iota_n : \R^n \lrasim D_n(\R), \quad (a_2, \dots, a_{n+1}) \longmapsto X + \sum_{i =2}^{n +1} a_i X^i  \mod X^{n+2}.$$
It is easily seen to be a left and right Haar measure on the Lie group $D_n(\R)$.

For every $n \in \N,$ the discrete subgroup $D_n(\Z)$ of $D_n(\R)$ is cocompact, and the map 
$$[0, 1)^{n} \lra D_n(\R) /D_n(\Z), \quad (a_2, \dots, a_{n+1}) \longmapsto [X +\sum_{i =2}^{n+1} a_i X^i]$$
is bijective. This shows that the measure $\bar \mu_n$ on $D_n(\R)/D_n(\Z)$ deduced from $\mu_n$ satisfies:
$$\bar \mu_n(D_n(\R)/D_n(\Z)) =1.$$ The measure $\bar \mu_n$ is the unique probability measure on the compact space $D_n(\R)/D_n(\Z)$ invariant under the left action of $D_n(\R)$.

The group:
$$D(\R) \simeq \varprojlim_n D_n(\R),$$
as the projective limit of a  countable projective system of Lie groups, is a Polish topological group. This topology coincide with the topology induces by the natural topology of Fr\'echet space on $\R[[X]] \simeq \R^\N.$ The topological group $D(\R)$ contains:
$$D(\Z) \simeq \varprojlim_n D_n(\Z)$$
as a closed subgroup, and the quotient space:
$$D(\R) /D(\Z) \simeq \varprojlim_n D_n(\R)/D_n(\Z)$$
is compact. We shall denote by $\bar \mu$ the unique probability measure on the compact space $D(\R) /D(\Z)$ whose direct image by the projection:
$$\bar p_n: D(\R) /D(\Z) \lra  D_n(\R)/D_n(\Z)$$
coincides with $\bar\mu_n$ for every $n \geq 1$. The measure $\bar\mu$ is easily seen to be the unique probability measure on $D(\R) /D(\Z)$ that is invariant under the left action of $D(\R)$. This notably implies that the mesure $\bar \mu(V)$ of any non-empty open subset $V$ of $D(\R) /D(\Z)$ is positive. 

We shall say that a Borel subset $B$ of $D(\R)$, invariant under the right action of $D(\Z)$ \emph{has full measure in $D(\R)$} when:
$$\bar\mu(q(B))= 1,$$
where $q: D(\R) \ra D(\R)/D(\Z)$ denotes the quotient map.

For every $\lambda \in \R^\ast,$  we shall denote:
$$[\lambda] := \lambda X \quad (\in \R[[X]]).$$
The map $(g \mapsto [\lambda] \circ g)$ 
establishes a bijection from $D(\R)$ onto 
$\lambda X + X^2 \R[[X]].$

  Our result concerning the generic triviality of the algebra $\cO(\Vfa (\overline{D}(0,1), \psi))$ in the pseudoconcave case reads as follows:

\begin{theorem}\label{psipseudoconcave}
For every $\lambda \in \R$ such that $\vert \lambda \vert >1,$ the set:
$$\cT_\lambda := \left\{ g \in D(\R) \mid \cO(\Vfa (\overline{D}(0,1), [\lambda] \circ g) ) = \Z \right\}$$
is a $G_\delta$ subset, invariant under the right action of $D(\Z),$ dense and of full measure in $D(\R)$.
\end{theorem}

In other words, when $\vert \lambda \vert > 1,$ the algebra $ \cO(\Vfa (\overline{D}(0,1), \psi )$ contains only constant functions for ``almost all" choice of $\psi$ with $\psi'(0) = \lambda$, both in the sense of measure theory and of Baire category.

The proof of Theorem  \ref{psipseudoconcave} is the object of the next section. The techniques used in this proof play no role in the next chapters, and this section could be skipped at first reading.

\section{Proof of the generic triviality of $\cO(\Vfa(\overline{D}(0,1), \psi))$ in the pseudo-concave case}

\subsection{} In order to prove Theorem \ref{psipseudoconcave}, we need to introduce some further notation.

For every commutative ring $A$, we define a left action of the group $D(A)$ on the algebra $A[[X]]$ by:
\begin{equation}\label{defaction}
 g. \phi := \phi \circ g^{-1},
 \end{equation}
for $g \in D(A)$ and $\phi \in A[[X]]$.
For every $e \in \N_{>0}$ and $a \in A\setminus\{0\}$, we also define:
$$O(e,a):= a X^e + X^{e+1} A[[X]].$$
The sets $O(e,a)$ are invariant under the action of $D(A),$ and we  have:
\begin{equation}\label{partOea}
XA[[X]]\setminus\{0\} = \coprod_{e\in \N_{>0}, a \in A\setminus\{0\}} O(e,a).
\end{equation}

Observe that, when $k$ is a field of characteristic zero, the action of $D(k)$ on  $X k[[X]]\setminus\{0\}$ is free, and   
 the sets $O(e, a)$ for $(e,a) \in \N_{>0} \times k^\times$ are precisely the orbits of this action; see \ref{proofkeyF} (b) below when $k = \R$.

 For every $\rho \in \R_+,$ we define a pseudo-norm $\Vert . \Vert_\rho$ on $\R[[X]]$ by the formula:
\begin{equation*}
\big\Vert \sum_{i \in \N} a_i X^i \big\Vert_\rho := \sup_{i \in \N} \vert a_i \vert \rho^i  \quad ( \in [0, +\infty]).
\end{equation*}
Moreover, for every $R \in \R_+$, we let:
\begin{equation*}
B_\rho (R) := \left\{ \phi \in \R[[X]] \mid \Vert \phi \Vert_\rho \leq R    \right\}.
\end{equation*}
Observe that a formal series $\phi \in \R[[X]]$ defines a function analytic on some open neighborhood of $\overline{D}(0, \vert \lambda \vert^{-1})$ if and only if there exists $\rho$ in $(\vert \lambda \vert^{-1}, +\infty )$ such that:
$$ \Vert \phi \Vert_\rho < + \infty.$$

For $e$ in $\N_{>0},$ $a$ in $\Z\setminus \{0\},$ and $\rho$ and $R$ in $\R_+,$ let us consider the following subset of $D(\R)$:
\begin{equation*}
F(e, a, \rho, R) := \left\{ g \in D(\R) \mid g. O(e,a)(\Z) \cap B_\rho (R) \neq \emptyset
\right\}.
\end{equation*}

Our previous observation concerning the analyticity of elements of $\R[[X]]$ and the decomposition \eqref{partOea} when $A = \Z$ imply the following description of the complement of $\cT_\lambda$:
\begin{align*}
D(\R) \setminus \cT_\lambda & = \bigcup_{\rho > \vert \lambda \vert^{-1}} \left\{  g \in D(\R) \mid \exists\, \alpha \in g (X \Z[[X]] \setminus\{0\}), \Vert \alpha \Vert_\rho < +\infty 
\right\} \\
& = \bigcup_{\rho > \vert \lambda \vert^{-1}, R > 0} \left\{  g \in D(\R) \mid g (X \Z[[X]] \setminus\{0\}) \cap B_\rho(R) \neq \emptyset
\right\} \\
& = \bigcup_{(e, a, \rho, R) \in \cE (\lambda)} F(e, a, \rho, R), 
\end{align*}
where:
$$\cE (\lambda) := \N_{>0} \times  (\Z\setminus \{0\}) \times (\vert \lambda \vert^{-1}, +\infty) \times \Rpa.$$

In turn, this immediately shows that $D(\R) \setminus \cT_\lambda$ may be written as a countable union:
\begin{equation}\label{countdecomp}
D(\R) \setminus \cT_\lambda =  \bigcup_{(e, a, \rho, R) \in \cE' (\lambda)} F(e, a, \rho, R),
\end{equation}
where:
$$\cE' (\lambda) := \N_{>0} \times  (\Z\setminus \{0\}) \times \left\{\vert \lambda \vert^{-1} + 1/i; i \in \N_{>0} \right\}\times \N.$$

This decomposition of $D(\R) \setminus \cT_\lambda$ shows that Theorem \ref{psipseudoconcave} is a consequence of the following proposition and of Baire's theorem:

\begin{proposition}\label{keyF} (a) For every $(e, a, \rho, R)$ in $\N_{>0} \times  (\Z\setminus \{0\}) \times \R_+^{2}$, 
$F(e, a, \rho, R)$ is a closed subset of $D(\R)$, invariant under the right action of $D(\Z).$

(b) If moreover $\rho >1,$ then:
$$\bar\mu(q(F(e, a, \rho, R))) =0,$$
and $F(e, a, \rho, R)$ has an empty interior in $D(\R)$.
\end{proposition}

\subsection{Proof of Proposition \ref{keyF}}\label{proofkeyF} 

\subsubsection{}\emph{Proof of (a).}  The invariance of $F(e, a, \rho, R)$ under the right action of $D(\Z)$ follows from its definition.

To prove that $F(e, a, \rho, R)$  is closed  in $D(\R),$ observe that the left action of $D(\R)$ on $\R[[X]]$, defined above by \eqref{defaction}, is continuous and that, for every $(\rho, R) \in \Rpa \times \R_+,$ $B_\rho(R)$ is a compact subset of the Fr\'echet space $\R[[X]]$.

Let us consider a sequence $(g_n)$ in $F(e, a, \rho, R)$ that admits a limit $g$ in $D(\R)$. There exists a sequence $(\phi_n)$ in $O(e, a)(\Z)$ such that the formal series 
$g_n . \phi_n := \phi_n \circ g_n^{-1}$
belong to $B_\rho(R)$. After possibly passing to a subsequence, we may assume that $(\phi_n \circ g_n^{-1})$ admits a limit $l$ in $B_\rho(R)$. Then
$(\phi_n) = ((\phi_n \circ g_n^{-1}) \circ g_n)$
converges to 
$\phi:= l \circ g$ in 
$\R[[X]]$. Since $O(e,a)(\Z)$ is closed in $\R[[X]],$ this limit $\phi$ belongs to $O(e,a)(\Z)$. Finally, 
$$g.\phi := \phi \circ g^{-1} =  l,$$
belongs to $g. O(e,a)(\Z) \cap B_\rho (R)$, and therefore $g$ belongs to $F(e, a, \rho, R)$.
\medskip

\subsubsection{}\emph{Proof of (b).} Observe that, for every $(e,a)$ in $\N_{>0} \times  (\R\setminus \{0\}),$ the map:
$$I_{e,a} : D(\R): = X + X^2 \R[[X]] \lra O(e,a) (\R), \quad \psi \longmapsto a \psi^e$$
is bijective. This already implies that the action by composition of $D(\R)$ on $O(e,a) (\R)$ is free and transitive. 

By truncation, for every $n\in \N,$ the action of $D(\R)$  on $O(e,a) (\R)$ defines an action of $D_n(\R)$ 
on $$O(e,a)_n(\R) := \left(a X^e + X^{e+1} \R[[X]]\right) \! /(X^{n + e +1}).$$ Moreover, by truncation, for every $n \in \N,$ the map $I_{e,a}$ induces a bijection:
$$I^n_{e,a}: D_n(\R) \lrasim O(e,a)_n(\R),$$
which is clearly compatible with the action of $D_n(\R)$ by composition on $O(e,a)_n(\R)$, which is also free and transitive.

We shall denote by $\mu_n^{e,a}$ the measure on $O(e,a)_n(\R)$ image of the Lebesgue measure on $\R^n$ by the diffeomorphism:
$$\iota_n^{e,a} : \R^n \lrasim O(e,a)_n(\R), \quad (a_{e+1}, \dots, a_{e+n}) \longmapsto aX^e + \sum_{i= e+1}^{e+n} a_i X^i \mod X^{n+e +1}.$$

\begin{lemma}\label{bulletphi}
For 
 every $(e,a)$ in $\N_{>0} \times  (\R\setminus \{0\})$ and every $\phi \in O(e,a)_n(\R)$, the image of the  Haar measure $\mu_n$ on $D_n(\R)$ by the diffeomorphism
$$\bullet . \phi : D_n(\R) \lrasim  O(e,a)_n(\R), \quad g \longmapsto g. \phi$$ is the measure $(e \vert a \vert)^{-n} \mu_n^{e,a}$.
\end{lemma}
 \begin{proof} Since $D_n(\R)$ is unimodular, this is equivalent to the fact that the image of $\mu_n$ by the diffeomorphism:
 $$\bullet^{-1} . \phi : D_n(\R) \lrasim  O(e,a)_n(\R), \quad g \longmapsto \phi \circ g$$ is the measure $(e \vert a \vert)^{-n} \mu_n^{e,a}$. In turn, this is equivalent to the fact that the direct image by the diffeomorphism:
 $$J(e,a,n,\phi) := (\iota_n^{e,a})^{-1} \circ (\bullet^{-1} . \phi) \circ \iota_n : \R^n \lra \R^n$$
 of the $n$-dimensional Lebesgue measure is $(e \vert a \vert)^{-n}$ times the $n$-dimensional Lebesgue measure. To establish this, we shall show that the Jacobian of  $J(e,a,n,\phi)$, namely $\det D J(e,a,n,\phi)$,  is constant on $\R^n$ and satisfies:
 \begin{equation}
\det D J(e,a,n,\phi) = (ea)^n.
\end{equation}
According to the change-of-variables formula for the Lebesgue measure, this will complete the proof. 

To compute the differential of the diffeomorphism $\bullet^{-1} . \phi$, we may identify $D_n(\R)$ and $O(e,a)_n(\R)$ to $X + X^2 \R[X]_{< n}$ and to $a X^e + X^{e+1} \R[X]_{< n}$ respectively, were $\R[X]_{< n}$ denotes the $\R$-vector space of polynomial of degree $<n$. Then the differential of $\bullet^{-1} . \phi$ at some point $g$ of  $D_n(\R)$ maps an element $\delta g$ in $X^2 \R[X]_{< n}$ to:
$$D(\bullet^{-1} . \phi)(g) \delta g = (\phi' \circ g)  \, \delta g \in \left (X^{e+1} \R[[X]]\right) /(X^{n +e +1}) \simeq X^{e+1} \R[X]_{< n}.$$

This shows that, after identifying $\R^n$ and $\R[X]_{< n}$ by means of the standard basis: $$(1, X, \dots, X^{n-1}),$$ the differential  $D J(e,a,n,\phi)$ at the point $\iota_n^{-1}(g)$ is given by the mutiplication:
$$\R[[X]]_{< n} \lra \R[[X]]_{< n}, \quad  P \longmapsto \lfloor X^{-e +1}  (\phi' \circ g)  P\rfloor,  $$ where $\lfloor X^{-e +1}  (\phi' \circ g)  P\rfloor$ denotes the unique representative in $\R[[X]]_{< n}$ of the class 
$\mod X^n$ of $\R[[X]]_{< n}$. 
Since $ X^{-e +1}  (\phi' \circ g)$ is an element of 
$a e + X \R[[X]],$
the determinant of this map is $(ea)^n$.
 \end{proof}

For every $(e,a)$ in $\N_{>0} \times  (\Z\setminus \{0\}),$ we may introduce the following subset of $O(e,a)(\Z)$:
\begin{equation*}
\Delta(e,a) := \Big\{ a X^e + \sum_{i \geq e+1} a_i X^i  \in \Z[[X]] \mid  \forall i \in \N_{\geq e +1}, 0 \leq a_i < e \vert a \vert
\Big\}.
\end{equation*}
For every $n \in \N$, we shall also denote by $\Delta(e,a)_n$ the image of $\Delta(e,a)$ in $O(e,a)_n(\R)$ by the truncation morphism. In other words:
\begin{equation*}
\Delta(e,a)_n := \Big\{ \Big[a X^e + \sum_{i= e+1}^{e+n} a_i X^i\Big]  \in \Z[[X]] /(X^{n+e+1}) \mid  \forall i \in \{ e +1, \dots, e+n\}, 0 \leq a_i < e\vert a \vert
\Big\}.
\end{equation*}

\begin{lemma}\label{DeltaO} With the previous notation, we have:
 \begin{equation*}
D(\Z). \Delta(e,a) := \bigcup_{\gamma \in D(\Z)} \gamma. \Delta(a,e) = O(e,a).
\end{equation*}
\end{lemma}
\begin{proof} For every element:
$$\phi := a X^e + \sum_{i \geq e +1} a_i X^i$$ in $O(e,a)(\Z)$ and  for every integer $k \geq 2$ and every $b \in \Z,$ the image of $\phi$ under the action of the element
$X + b X^k$ of $D(\Z)$  is:
\begin{align*}
 (X+ b X^k) . \phi & := \phi \circ (X + b X^k)^{\circ (-1)} \\
 & = \phi \circ (X - b X^k) \mod X^{e + k} \\
 & = a X^e + \sum_{e+1 \leq i \leq e +k -2} a_i X^i + (a_{e+k-1} - ae b) X^{e +k -1} \mod X^{e+k}.
\end{align*}

This identity allows one to produce an element $\gamma$ in $D(\Z)$ such that $\phi$ is contained in $\gamma. \Delta(a,e)$ by defining $\gamma^{-1}$ as the limit of elements of the form 
$$(X + b_k X^k) \circ \dots \circ (X + b_2 X^2),$$ with $b_2, \dots, b_k$ in $\Z$, when $k \geq 2$ goes to infinity. 
\end{proof}

According to Lemma \ref{DeltaO}, we have:
\begin{equation}\label{FFtilde}
 q(F(e, a, \rho, R)) = q(\widetilde{F}(e, a, \rho, R)),
\end{equation}
where:
$$\widetilde{F}(e, a, \rho, R)) := \left\{ g \in D(\R) \mid g. \Delta(e,a)(\Z) \cap B_\rho (R) \neq \emptyset
\right\}.$$

For every $n \in \N$, we shall denote by: 
$$p_n: D(\R) \lra D_n(\R):= D_n(\R) /H_n(\R)$$
 the quotient (or truncation) morphism. It fits into the following commutative diagram:
 \[
\xymatrix{
 D(\R) \ar[d]^{p_n} \ar[r]^{q\quad} & D(\R) /D(\Z) \ar[d]^{\bar p_n} \\
D_n(\R)  \ar[r]^{q_n\quad} & D_n(\R) /D_n(\Z).
}
\]
Therefore the following inequalities hold:
\begin{multline}\label{mubarmun}
\bar \mu (q(\widetilde{F}(e, a, \rho, R)))  \leq \bar \mu (\bar p_n^{-1} (\bar p_n \circ q(\widetilde{F}(e, a, \rho, R)))) = 
\bar \mu_n (\bar p_n \circ q(\widetilde{F}(e, a, \rho, R)))
\\  = \bar \mu_n (q_n \circ p_n (\widetilde{F}(e, a, \rho, R)))  \leq \mu_n( p_n(\widetilde{F}(e, a, \rho, R))).
\end{multline}

Moreover the image of $O(e,a)(\R) \cap B_\rho(R)$ in $O(e,a)_n(\R)$ by the truncation map is contained in:
$$\iota_n^{e,a}\Big(\prod_{i= e+1}^{e+n} [- R \rho^{-i}, R \rho^{-i}]\Big).$$
Consequently we have:
\begin{equation*}
p_n(\widetilde{F}(e, a, \rho, R)) \subseteq \bigcup_{\phi \in \Delta(e,a)_n(\Z)} \Big\{ g \in D_n(\R) \mid g.\phi \in \iota_n^{e,a}(\prod_{i= e+1}^{e+n} [- R \rho^{-i}, R \rho^{-i}]) \Big\},
\end{equation*}
and therefore:
\begin{equation*}
\mu_n( p_n(\widetilde{F}(e, a, \rho, R))) \leq \sum_{\phi \in \Delta(e,a)_n(\Z)} \mu_n\Big(
\big\{ g \in D_n(\R) \mid g.\phi \in \iota_n^{e,a}(\prod_{i= e+1}^{e+n} [- R \rho^{-i}, R \rho^{-i}]) \Big\}
\Big). 
\end{equation*}
Using Lemma \ref{bulletphi}, we finally obtain: 
\begin{equation}\label{munpn}
\begin{split}
\mu_n( p_n(\widetilde{F}(e, a, \rho, R))) & \leq \big\vert \Delta(e,a)_n(\Z) \big\vert \, \,  (e \vert a \vert)^{-n} \, 
\prod_{i= e+1}^{e+n} (2 R \rho^{-i}) \\
& \quad \quad = (2R)^n \rho^{-\sum_{i= e+1}^{e+n} i} = (2R)^n \rho^{-(n+ 2e +2) (n-1)/2}.
\end{split}
\end{equation}

If $\rho > 1,$ the right-hand side of \eqref{munpn} goes to $0$ when $n$ goes to infinity.  Together with \eqref{FFtilde} and \eqref{mubarmun}, this establishes the vanishing of $\bar\mu(q(F(e, a, \rho, R)))$ when $\rho >1$. In turn this vanishing implies that the interior $\mu(q(F(e, a, \rho, R)))$ in $D(\R)/D(\Z)$ is empty, as observed above in \ref{pseudoconcave case}.

\chapter{Maps from formal-analytic arithmetic surfaces to arithmetic schemes}

In this chapter, we define  morphisms from a smooth \fa surface $\Vfa$ to an arithmetic scheme\footnote{that is, a separated scheme of finite type over $\Spec \Z$.} $X$ over $\Spec \OK$, for $K$ some number field, and we discuss some basic constructions involving them. 

When $X$ is a normal arithmetic surface, we investigate the relations of the rudimentary arithmetic intersection theory on $\Vfa$ introduced in Section \ref{ArIntFa}  with the more classical arithmetic intersection theory on quasi-projective arithmetic surfaces, as presented in Chapters  \ref{chapterL21} and \ref{chapterCbD}.
These relations involve the Archimedean overflow $\mathrm{Ex}(\alpha : (V, P)\ra N)$ studied in Chapter \ref{OverflowArchimede} and its counterpart ``at finite places"   $\Ex (\widehat \alpha: \Vf \ra X)$.  

Finally we extend these results to the situation where the morphisms from $\Vfa$ to $X$ are replaced by suitably defined ``meromorphic  maps."

 \medskip
 
We denote by $K$ a number field, and by $\OK$ its ring of integers.
 
\section[Morphisms  to $\OK$-schemes]{Morphisms from formal-analytic arithmetic surfaces to $\OK$-schemes}

\subsection{Definitions and basic properties}\label{morphfasurf}

\subsubsection{}\label{defmorVfaSch}
 If  $\Vfa: = (\Vf, (V_{\sigma}, O_\sigma, \iota_{\sigma})_{\sigma: K\hra\C})$ denotes a smooth \fa arithmetic surface over $\Spec \OK$ as defined in \ref{subsubsection:definition-fa}, and if $X$ is a separated scheme of finite type over $\OK$, we may define a \emph{morphism}
$$f : \Vfa\lra X$$
\emph{over} $\Spec \OK$ as a pair:
$$f:=(\widehat f, (f_{\sigma})_{\sigma : K\hra \C})$$
where: 
$$\widehat f : \Vf\lra X$$
is a morphism of formal schemes over $\OK$ and, for every complex embedding $\sigma$ of $K$, 
$$f_{\sigma} : V_{\sigma}\lra X_{\sigma}(\C)$$
is a complex analytic map.  

These data are moreover assumed to be compatible with the gluing data $(\iota_{\sigma})_{\sigma : K\hra\C}$. Namely, for every  embedding $\sigma: K \hra \C,$ the morphisms from the smooth formal curve $\Vf_\sigma$ to the complex scheme $X_\sigma$ deduced from $\widehat f$ (by the base change $\sigma: \OK \ra \C$) and from $f_\sigma$ (by considering its ``formal germ" at $P_\sigma$) are required  to coincide; see \cite[10.6.3]{Bost2020}.

As a consequence, these data are compatible with complex conjugation.

We shall use the notation:
$$f^\an : V^+_\C \lra X(\C)$$
for the complex analytic map from: 
$$V^+_{\C} :=\coprod_{\sigma : K\hra\C} V^+_{\sigma}$$
to:
$$X(\C) = \coprod_{\sigma : K\hra\C} X_{\sigma}(\C)$$
defined by the maps $f_\sigma$.

If $E$ (resp. $\Eb:= (E, \Vert.\Vert)$)  is a vector bundle over $X$ (resp. a Hermitian vector bundle over $X$, supposed to have a reduced generic fiber $X_K$), we may form its pull-back $f^\ast E$ (resp. $f^\ast \Eb$), defined by the ``formal" pull-back $\widehat{f}^\ast E$ on $\Vf$ and the ``complex analytic" pull-back $f_\sigma^\ast E_\sigma$ (resp.   $f_\sigma^\ast \Eb_\sigma$) on $V_\sigma^+$, and some canonical gluing data; see \cite[10.6.3 ]{Bost2020}.

\subsubsection{}\label{Zariskiclosures}

To every morphism $f : \Vfa\lra X$ as in \ref{defmorVfaSch}, we may attach the Zariski closures of the images of the morphisms of ringed spaces:
$$\widehat f : \Vf\lra X, \quad \widehat f_{K} : \Vf_{K}\lra  X_{K},\quad \widehat f_{\sigma} : \Vf_{\sigma}\lra X_{\sigma}, \quad \mbox{and} \quad  f_{\sigma} : V_{\sigma}^+\lra X_{\sigma},$$
which we shall denote by $\overline{\mathrm{im}\widehat f}$, $\overline{\mathrm{im}\widehat f_{K}}$, $\overline{\mathrm{im}\widehat f_{\sigma}}$ and $\overline{\mathrm{im}f_{\sigma}}$. Each of these is defined as the smallest closed subscheme of the range of the morphism through which the morphism factors; see  \cite[10.6.4]{Bost2020}. These schemes are integral. Moreover, $\overline{\mathrm{im}\widehat f}$ is flat over $\OK$, $\overline{\mathrm{im}\widehat f_{K}}$ is geometrically irreducible over $K$, and the following relations hold:
$$\overline{\mathrm{im}\widehat f_{K}}=\big(\overline{\mathrm{im}\widehat f}\big)_{K} \quad 
\mbox{and} 
\quad \overline{\mathrm{im} f_{\sigma}}=\overline{\mathrm{im}\widehat f_{\sigma}}=\big(\overline{\mathrm{im}\widehat f}\big)_{\sigma}.$$

We say that $f$ is a \emph{constant} morphism when the following equivalent conditions are satisfied:
$$\dim \overline{\mathrm{im}\widehat f}=1, \quad
\dim \overline{\mathrm{im}\widehat f_{K}}=0, \quad \mbox{or} \quad
\dim \overline{\mathrm{im}f_{\sigma}}=0.$$

This holds precisely when $f$ factors through the morphism:
$$\pi_{\Vfa} : \Vfa\lra \Spec\OK,$$
or equivalently, when $f$ may written as $f=Q\circ\pi_{\Vfa}$ for some $\OK$-point $Q$ of $X$. In concrete terms, this holds if and only if the morphism: 
$$\widehat f_{K} : \Vfa_{K}\simeq \mathrm{Spf} K[[T]]\lra X_{K}$$
is constant in the obvious sense.

The $\OK$-morphism from $\Vfa$ to the affine line $\mathbb A^{1}_{\OK}$ may be identified with the elements of the $\OK$-algebra $\mathcal O(\Vfa)$, and the morphism: 
$$f : \Vfa\lra \mathbb A^{1}_{\OK}$$
defined by an element $f$ of $\mathcal O(\Vfa)$ is constant if and only if $f$ belongs to the subring $\OK$ of $\mathcal O(\Vfa).$

Finally, we say that the image of an $\OK$-morphism:
$$f : \Vfa\lra X$$
is \emph{algebraic} when the following equivalent conditions hold:
$$\dim \overline{\mathrm{im}\widehat f}\leq 2, \quad \dim \overline{\mathrm{im}\widehat f_{K}}\leq 1, \quad 
\mbox{or} \quad
\dim \overline{\mathrm{im}f_{\sigma}}\leq 1.$$

\subsection{Morphisms and sections of vector bundles}\label{subsection:vb}

\subsubsection{}\label{subsubsection:vb} As in  \ref{defmorVfaSch}, let us consider a scheme $X$ separated and of finite type over $\OK$. Let us also assume that $X$ is reduced and that the structure morphism $\pi_{X} : X\ra\Spec\OK$ is proper and flat. 

For every Hermitian vector bundle $\overline E=(E, \Vert.\Vert)$ over $X$ --- defined by a vector bundle $E$ on $X$ and a continuous metric $\Vert.\Vert$ on the complex analytic vector bundle $E^{\mathrm{an}}_{\C}$ on: $$X(\C)=\coprod_{\sigma : K\hra \C}X_{\sigma}(\C)$$ --- the direct image $\pi_{X*}E$ is a vector bundle over $\Spec\OK$. Moreover, for every  embedding $\sigma: K \hra \C$, the complex vector space:
$$(\pi_{X*}E)_{\sigma}\simeq \Gamma(X_{\sigma}, E_{\sigma})$$
may be endowed with the norm $\Vert.\Vert_{\infty, \sigma}$ defined by:
$$\Vert s\Vert_{\infty, \sigma}:=\sup_{x\in X_{\sigma}(\C)}\Vert s(x)\Vert.$$

In general,  the norm $\Vert.\Vert_{\infty, \sigma}$ is not a Hermitian norm. However we may consider the \emph{John norm} $\Vert.\Vert_{J, \sigma}$ on $\Gamma(X_{\sigma}, E_{\sigma})$ attached to $\Vert.\Vert_{\infty, \sigma}$, namely, the smallest Hermitian norm such that 
$$\Vert.\Vert_{\infty, \sigma}\leq \Vert.\Vert_{J, \sigma};$$
see \cite[Appendix F]{Bost2020}.  This construction defines a Hermitian vector bundle over $\Spec\OK$:
$$\pi^J_{X\ast}\widetilde{\overline E} :=(\pi_{X*}E, (\Vert.\Vert_{J, \sigma})_{\sigma : K\hra\C}).$$

The following proposition is a straightforward consequence of the definitions:

\begin{proposition}[compare to \protect{\cite[10.8.1]{Bost2020}}] \label{proposition:pull-back-norm}
Let $\Vfa$ be a smooth \fa arithmetic surface over $\OK$, let  $\mu$ be a $\cC^\infty$ positive volume form on $V_{\C}$, invariant under complex conjugation, and let:
$$f : \Vfa\lra X$$
be a morphism over $\OK$.  


For every Hermitian vector bundle $\overline E$ over $X$,
 the pull-back of sections of $E$ by $f$ defines a morphism:
$$\phi_{E} : \pi^J_{X\ast}\widetilde{\overline E} \lra \pi^{L^{2}}_{(\Vfa, \mu) \ast} f^{*}\overline E$$
of (pro-)Hermitian vector bundles  over $\Spec\OK$. If moreover:
$$\int_{V_{\sigma}}\mu\leq 1$$
for every  embedding $\sigma: K \hra \C$, then the 
 Archimedean norms of $\phi_E$ are bounded above by $1$.
\end{proposition}

\subsubsection{} If $\alpha : X\ra Y$ is morphism between two separated $\OK$-schemes of finite type, and if 
$$f:=(\widehat f, (f_{\sigma})_{\sigma : K\hra\C}) : \Vfa\lra X$$
is a morphism over $\OK$ from a smooth \fa arithmetic surface $\Vfa$ to $X$, then we may form the composition $\alpha\circ f$ from $\Vfa$ to $Y$:
$$\alpha\circ f := (\alpha\circ \widehat f, (\alpha_{\sigma}\circ f_{\sigma})_{\sigma : K\hra\C}) : \Vfa\lra Y.$$
This construction satisfies obvious compatibilities with the pull-back of vector bundle, and with the pull-back of sections of vector bundles discussed in \ref{subsubsection:vb}.

\section[Morphisms to arithmetic surfaces and overflow]{Morphisms to arithmetic surfaces, arithmetic intersection numbers, and overflow}\label{MorFaArSurf}

In this section, we consider 
a smooth \fa arithmetic surface $\Vfa$ over $\OK$ as defined in \ref{subsubsection:definition-fa} above, an integral normal arithmetic surface:
$$\pi_X: X \lra \Spec \OK,$$
and a morphism over $\Spec \OK$:
$$\alpha:= (\widehat \alpha, (\alpha_\sigma)_{\sigma:K \hra \C}) : \Vfa\lra X,$$
as defined in \ref{morphfasurf}. 

We shall denote by:
$$P: \Spec \OK \lra \Vf$$
the canonical section of the structural morphism of $\Vf$:
$$\pi_{\Vf} : \Vf \lra \Spec \OK,$$
and by:
$$Q := \widehat \alpha \circ P: \Spec \OK \lra X$$
the $\OK$-point of $X$ defined as the image of $P$ by $\widehat \alpha$.  The existence of this $\OK$-point guarantees that $X_K$ is geometrically irreducible over $K$, and that the fibers of $\pi_{X}$ are geometrically connected when moreover $X$ is projective.

We shall also assume that $\alpha$ is non-constant, and we shall denote by $e(\alpha)$ the ramification index of the morphism of smooth (formal) curves over $K$:
$$\widehat{\alpha}_K : \Vf_K \lra X_K.$$ 

\subsection{The map $\alpha_\ast: \Zb^1_c(\Vfa) \ra \Zb^1_c(X)$} 

To any Arakelov divisor $(D,g) = (D, (g_\sigma)_{\sigma: K \hra \C})$ in $\Zb^1_c(\Vfa),$ we may attach its \emph{direct image by the morphism} $\alpha: \Vfa \ra X$, defined as:
$$\alpha_\ast (D, g) := (\walpha_\ast D, (\alpha_{\sigma\ast} g_\sigma)_{\sigma: K \hra \C}).$$
It is straightforward that it is an Arakelov divisor in $\Zb_c^1(X)$, and that this construction attaches a morphism of $\Z$-modules:
$$\alpha_\ast : \Zb^1_c(\Vfa) \lra \Zb^1_c(X)$$ to any non-constant morphism $\alpha: \Vfa \ra X$ over $\Spec \OK$. 

Observe that the ``arithmetic part" of this construction --- namely the map $\walpha_\ast$ that  maps a divisor supported by $\vert \Vf \vert$ to its direct image in $X$ --- is nothing more than the map that sends a multiple $n P$ of the section $P$, for some $n \in \Z$, to the multiple $n Q$ of the $\OK$-point
$Q := \alpha^\an (P)$
of $X$ image of $P$ by $\alpha^\an$. 

This construction satisfies the following compatibility with the direct image functoriality for the Arakelov divisors on quasi-projective arithmetic surfaces:

\begin{proposition}\label{prop:directVfaclass} For every $\OK$-morphisms $f : X \ra X',$
from $X$ to an integral normal arithmetic surface $X'$ over $\Spec \OK,$ and every Arakelov divisor $(D,g)$ in $\Zb^1_c(\Vfa)$, the following equality holds in $\Zb^1_c(X')$:
\begin{equation}\label{directVfaclass}
(f \circ \alpha)_\ast (D,g) = f_\ast \alpha_\ast (D, g).
\end{equation}
\end{proposition}

The arithmetic intersection theory on the \fa arithmetic surface $\Vfa$ introduced in Subsection \ref{ArIntVfa} and the ``classical" arithmetic intersection theory on the normal  quasi-projective arithmetic surface $X$  developed in Chapters \ref{chapterL21} and \ref{chapterCbD} are related by the following projection formula:

\begin{proposition}\label{proposition:projection}
Let $\overline L: = (L, (\Vert. \Vert_\sigma)_{\sigma: K \hra \C})$ be a Hermitian line bundle on $X$, defined by Hermitian metrics $\Vert. \Vert_\sigma)$ of regularity $\mathcal C^{\bD}$. For every Arakelov divisor $(D,g)$ in $\Zb_c^1(\Vfa),$ the following equality of arithmetic intersection numbers holds: 
\begin{equation}\label{projfaproj}
\alpha^{*}\Lb \cdot (P, g) =\Lb \cdot \alpha_{*}(P, g).
\end{equation}
\end{proposition}

Propositions \ref{prop:directVfaclass} and \ref{proposition:projection} are straightforward consequences of the definitions and of the basic properties of direct images of (Green) functions.  

\subsection{The invariant $\Ex (\widehat \alpha: \Vf \ra X)$}

\subsubsection{}\label{XregulardefR}  In this paragraph, we assume that $X$ is regular. Then $Q:= \widehat \alpha (P)$ is an effective Cartier divisor in $X$, and its inverse image by $\widehat \alpha$ defines an effective (Cartier) divisor in $\Vf$:
\begin{equation*}
\widehat \alpha^\ast (Q) = \widehat \alpha^\ast (\widehat \alpha (P)).
\end{equation*}
The divisor $P$ appears with multiplicity $e(\alpha)$ in this divisor, and we may write:
\begin{equation}\label{defR}
\widehat\alpha^{*}(Q)=e(\alpha) P+R,
\end{equation}
where $R$ is an effective Cartier divisor in $\Vf$ that intersects $P$ properly. 

In particular, the intersection $R\cdot P$ is a well-defined effective $0$-cycle with support on the arithmetic curve $P$; namely:
\begin{equation}\label{defnx}
R\cdot P:=\sum_{x\in P_{0}}n_{x} x,
\end{equation}
where $P_{0}$ denotes the set of closed points of the scheme $P$ and $n_{x}$  the length of the $\mathcal O_{\Vf, x}$-module $\mathcal O_{R\cap P}$. The multiplicity $n_{x}$ is positive if and only if the closed point $x$ belongs to the support $|R|$ of $R$. 

The $0$-cycle $R\cdot P$ has a well-defined arithmetic degree:
\begin{equation}\label{defRP}
\dega R\cdot P =\sum_{x\in P_{0}}n_{x}\log |\kappa(x)|=\log|\mathcal O(R\cap P)|,
\end{equation}
where $\kappa(x)$ denotes the residue field of $x$ and $\mathcal O(R\cap P)$ is the ring of regular functions on the $0$-dimensional subscheme $R\cap P$ of $P$. The  arithmetic degree $\dega R\cdot P$ is non-negative and vanishes if and only if $R=0$.
%
%
\begin{example}\label{example:arithmetic-excess}
Assume that $\Vf=\mathrm{Spf\,}\OK[[T]]$ and $X= \mathbb  A^{1}_{\OK}$. Then $\widehat \alpha$ identifies with a formal series in $\OK[[T]]$, which may be written:
$$\widehat\alpha=a_0 + \sum_{i\geq e}a_{i}T^{i}$$
with $e \in \N_{>0}$ and  $a_{e}\neq 0.$ Then $Q$ is the divisor $(X-a_0 = 0)$ in $\mathbb  A^{1}_{\OK}$, the ramification degree of $\widehat \alpha_K$ is $e$, and $R$ is the divisor of: 
$$T^{-e}(\widehat\alpha - a_0) =\sum_{k\geq 0}a_{k+e} T^{k}$$ in $\mathrm{Spf\,}\OK[[T]]$.
The cycle $R\cdot P$ coincides with the divisor of $a_{e}$ in $P=\Spec\OK$, and therefore:
$$\dega R\cdot P=\log |N_{K/\Q} \, a_{e}|.$$
The $0$-cycle $R\cdot P$, or equivalently the arithmetic degree $\dega R\cdot P$, vanishes if and only if $a_e$ is a unit in~$\OK$.
\end{example}

\subsubsection{}\label{defarithmeticoverflow} When the arithmetic surface $X$ is no longer assumed to be regular, but only normal, the previous construction still makes sense with the following modifications.

The divisor $Q$ in $X$ may not be a Cartier divisor, but it is always $\Q$-Cartier; see \ref{BasicAS}. If $N$ denotes a positive integer such that $NQ$ is a Cartier divisor in $X,$ then we may define $\walpha^\ast(Q)$ as the effective $\Q$-divisor:
\begin{equation*}
\widehat \alpha^\ast (Q) := \frac{1}{N} \widehat \alpha^\ast (N Q).
\end{equation*}
Then we may define $R$ as an effective $\Q$-divisor in $\Vf$ by \eqref{defR}, and $R\cdot P$ as an effective $0$-cycle with $\Q$-coefficients supported by $P$. We may finally define $\dega R \cdot P$ by \eqref{defnx} and by the first equality in~\eqref{defRP}.

\begin{definition} With the above notation, we define:
\begin{equation}\label{defExar}
\Ex \big(\widehat \alpha: \Vf \ra X\big) := \dega R\cdot P. 
\end{equation}
\end{definition}

In other words, we have:
\begin{equation}\label{arithmeticoverflow}
\Ex \big(\widehat \alpha: \Vf \ra X\big)= \dega \big(\widehat \alpha^\ast (\widehat\alpha (P)) - e(\alpha) P\big) \cdot P.
\end{equation}

The invariant  $\Ex (\widehat \alpha: \Vf \ra X)$ constitutes an arithmetic counterpart of the Archimedean overflow invariant $\Ex(\alpha: V \ra X)$ introduced in Chapter \ref{OverflowArchimede}; compare for instance  \eqref{arithmeticoverflow} and the expression \eqref{equation:definition-overflowstar} 
for $\Ex(\alpha: V \ra X)$
in terms of $\ast$-product. 

The invariant  $\Ex (\widehat \alpha: \Vf \ra X)$ is a non-negative real number in $\Q^\ast_+ \log \Z_{>0}$. However, there are instances where we may guarantee that the invariant $\mathrm{Ex}(\widehat \alpha : \Vf\ra X)$ is of the form:
$$\mathrm{Ex}(\widehat \alpha : \Vf\ra X)=\log n$$ 
for some positive integer $n$ even when $X$ is not assumed to be regular. This is an immediate consequence of the following statement.

\begin{proposition}\label{proposition:integral-excess}
Assume that $\widehat\alpha$ is quasi-finite, namely, that the fibers of $\widehat\alpha$ over closed points of $X$ are finite. Let $D$ be an effective Weil divisor on $X$ and let $N$ be a positive integer such that $ND$ is Cartier. 
Then there exists a positive integer $n$ such that 
$$N^{-1} \dega \widehat\alpha^*(ND)\cdot P=\log n.$$
In particular, if $N^{-1} \dega \widehat\alpha^*(ND)\cdot P\neq 0,$ then $N^{-1} \dega \widehat\alpha^*(ND)\cdot P\geq \log 2.$
\end{proposition}

To prove Proposition \ref{proposition:integral-excess}, it is enough to show that the effective $\Q$-divisor $N^{-1}\widehat\alpha^*(ND)$ is actually a divisor, namely, that it has integral coefficients. In turn, this is equivalent to showing that if $E$ is an irreducible Cartier divisor on $\Vf$, that appears as a component of the effective Cartier divisor $\widehat\alpha^*(ND)$, then the multiplicity of $E$ in $\widehat\alpha^*(ND)$ is a multiple of $N$. If $E =P,$ the result is clear, so we may assume that $E$ is not $P$.

In order to simplify the argument we will only prove that the multiplicity of $E$ in $\widehat\alpha^*(NE)$ is larger than $N$. This still proves that the intersection number $N^{-1} \dega \widehat\alpha^*(ND)\cdot P$ is at least equal to $\log 2 $ when it is nonzero, which is the only part of the statement that we will use in this monograph -- see \ref{subsection:universal} below. We leave the derivation of the more precise statement to the interested reader.

\begin{proof}
Let $\widehat Z$ be the preimage of $D$ in $\Vf$ and let $D'$ be the associated effective Cartier divisor in $\Vf$ obtained by discarding the lower-dimensional components and the embedded points of $\widehat Z$. By construction, the schematic image of $D'$ is contained in $D$ and $E$ is a component of $D'$. The schematic image of $E$ in $X$ is purely one-dimensional in $X$ since $\widehat\alpha$ is quasi-finite. In particular, the schematic image of $NE$ is contained in $ND$, so that $NE$ is a component of $\widehat\alpha^*(ND)$.
%
%
%
%
\end{proof}

We shall call $\Ex (\widehat \alpha: \Vf \ra X)$ the \emph{overflow} or \emph{excess} of the morphism $\widehat \alpha: \Vf \ra X$. It clearly depends only on the morphism from $\Vf$ to the formal completion $\widehat{X}_Q$ of $X$ along its section $Q= \widehat \alpha(P)$ defined by $\widehat{\alpha}$, and accordingly we will also denote it by $\Ex (\widehat \alpha: \Vf \ra \widehat{X}_Q).$

\subsection{The self-intersections of $\alpha_\ast (P, g)$ and $\alpha_\ast( P, g_{\Vfa_\C})$}\subsubsection{}
\label{defNPg} Let $g :=(g_\sigma)_{\sigma: K \hra \C}$ be a family of Green functions for the points $P_\sigma$ in the Riemann surfaces $V^+_\sigma$ such that $(P, g)$ is  
an Arakelov divisor in  $\Zb^1_c(\Vfa)$. In other words, for every field embedding $\sigma: K \hra \C,$ $g_\sigma$ is a Green function with $\cC^{\bD}$ regularity that vanishes on $V^+_\sigma \setminus \mathring{V}_\sigma$, and  the family $(g_\sigma)_{\sigma: K \hra \C}$ is invariant under complex conjugation.

To each of the Green functions $g_\sigma$ is associated a capacitary metric $\Vert.\Vert_{g_\sigma}^{\mathrm{cap}}$ on the complex line :
$$T_{P_\sigma} V_\sigma \lrasim (N_P \Vf)_\sigma,$$
 as defined in \ref{CapMetricDef}, and we may consider the Hermitian line bundle $$ \Nb_{P, g} \Vf := (N_P \Vf,  (\Vert.\Vert_{g_\sigma}^{\mathrm{cap}})_{\sigma: K \hra \C})$$ over $P \simeq \vert \Vf \vert.$
 
 The following proposition expresses the self-intersection of the direct image $\alpha_\ast (P, g)$ in $\Zb^1_c(X)$ in terms of the Arakelov degree of the metrized normal bundle $ \Nb_{P, g} \Vf$ and of the invariants $\Ex(\walpha: \Vf \ra X)$ and $\Ex(\alpha_\sigma, g_\sigma)$ attached to the formal and Archimedean components $\walpha$ and $\alpha_\sigma$ of the morphism~$\alpha$.

\begin{proposition}\label{proposition:intersection-ov}
With the notation above, the following equality holds:
\begin{equation}\label{equation:on-V}
\alpha_\ast (P, g) \cdot \alpha_\ast (P,g)=e(\alpha) \,\widehat{\deg}\,P^*\overline N_{P, g}\Vf+\mathrm{Ex}\big(\widehat \alpha : \Vf\ra X\big)+\sum_{\sigma: K \hra \C}\mathrm{Ex}(\alpha_\sigma, g_\sigma).
\end{equation}
\end{proposition}

Proposition \ref{proposition:intersection-ov} will be a consequence of the following lemma: 

\begin{lemma}\label{degaQast} Let $Q:= \walpha(P)$ be the $\OK$-point of $X$ image of $P$ by $\walpha$, and let $N$ be a positive integer such the divisor $NQ$ in $X$ is Cartier. Then the following equality holds:
\begin{multline}\label{QastOQ}
N^{-1} \dega Q^\ast \cOb(NQ, N \alpha^\an_\ast g) = e(\alpha) \,\widehat{\deg}\,P^*\overline N_{P, g}\Vf+\mathrm{Ex}\big(\widehat \alpha : \Vf\ra X\big) \\ +\sum_{\sigma: K \hra \C} \int_{V_\sigma}g_\sigma\,\delta_{\alpha^{\mathrm{an}}_{\sigma *}(Q_\sigma)-eP_\sigma}.
\end{multline}
\end{lemma}

\begin{proof}[Proof of Lemma \ref{degaQast}] With the notation introduced in \ref{defarithmeticoverflow}, we have the equality of Cartier divisors in $\Vf$:
\begin{equation}\label{QPR}
\walpha^\ast (NQ) = e(\alpha)N P + NR,
\end{equation}
 and therefore an isomorphism of line bundles over $\Vf$:
\begin{equation}\label{Iiso}
I : \walpha^\ast \cO_X(NQ) \lrasim \cO_{\Vf} (e(\alpha) N P) \otimes \cO_{\Vf} (NR).
\end{equation}
Recall that $NR$ is an effective Cartier divisor in $\Vf$ such that $\vert NR \vert \cap \vert \Vf \vert$ is a finite set of closed points. In particular $P \cap NR$ is a well-defined effective divisor in $P$. 

By restriction to the section $P,$ the isomorphism $I$ this defines an isomorphism of line bundles over $\Spec \OK$:
\begin{multline}\label{IPiso}
I_{\mid P} : Q^\ast \cO_X(NQ) = P^\ast \walpha^\ast \cO_{\Vf} (NQ) \lrasim P^\ast \cO_\Vf \big(e(\alpha) NP\big) \otimes P^\ast \cO_\Vf (NR) \\ = P^\ast N_P \Vf^{\otimes e(\alpha) N} \otimes P^\ast \cO_P(P \cap NR).
\end{multline}
For every field embedding $\sigma: K \hra \C,$ the isomorphism $I_{\mid P}$ induces an isomorphism of complex lines:
\begin{equation*}
I_{\mid P, \sigma}: \big(T_{Q_\sigma} X_\sigma\big)^{\otimes N} \simeq \cO_{X_\sigma} (N Q_\sigma)_{\mid Q_\sigma} \lrasim (T_{Q_\sigma} X_\sigma)^{ \otimes e(\alpha) N}.
\end{equation*}
Its norm, with respect to the Hermitian metrics $\Vert.\Vert_{\alpha^{\mathrm{an}}_{\sigma, *}g_\sigma}^{\otimes  N}$ and $(\Vert.\Vert^{\mathrm{cap}}_{g_\sigma})^{\otimes e(\alpha) N}$ is readily seen to be:
$$\Vert I_{\mid P, \sigma} \Vert = \exp\Big(-N \int_{V_\sigma}g_\sigma\, \delta_{\alpha^{\mathrm{an}}_{\sigma, *}(Q_\sigma)-eP_\sigma}\Big).$$

This implies the following relation between  Arakelov degrees:
\begin{align*}
\dega Q^\ast \cOb(NQ, N \alpha^\an_\ast g) &= \dega (P^*\overline N_{P, g}\Vf)^{\otimes e(\alpha) N} + \dega (NR) \cdot P  - \sum_{\sigma: K \hra \C} \log \Vert I_{\mid P, \sigma} \Vert \\
& = \dega (P^*\overline N_{P, g}\Vf)^{\otimes e(\alpha) N} + \dega (NR) \cdot P  + N \sum_{\sigma: K \hra \C} \int_{V_\sigma}g_\sigma\delta_{\alpha^{\mathrm{an}}_{\sigma, *}(Q_\sigma)-eP_\sigma}.
\end{align*}
Using the additivity of the map:
$\dega: \overline{\Pic} (\Spec \OK) \ra \R$
and dividing by $N$, this becomes the equality \eqref{QastOQ}.
\end{proof}

\begin{proof}[Proof of Proposition \ref{proposition:intersection-ov}] According to the definition \eqref{defArIntBis} of the Arakelov intersection pairing on $X$, we have:
\begin{align*}
\alpha_\ast (P, g) \cdot \alpha_\ast (P,g) &= N^{-1} (NQ, N \alpha_\ast^\an g) \cdot (Q, \alpha_\ast^\an g)) \\
& = N^{-1} \dega Q^\ast \cOb(NQ, N \alpha^\an_\ast g) + N^{-1}\int_{X(\C)} 
\alpha^\an_\ast g \,\,  \omega (N \alpha^\an_\ast g) \\
& = N^{-1} \dega Q^\ast \cOb(NQ, N \alpha^\an_\ast g) +  \sum_{\sigma: K \hra \C} \int_{X_\sigma(\C)} 
\alpha_{\sigma \ast} g_\sigma  \, \alpha_{\sigma \ast}\omega ( g_\sigma).
\end{align*}

Together with the equality \eqref{QastOQ} and the definition \eqref{equation:definition-ex} of the invariants $\mathrm{Ex}(\alpha_\sigma, g_\sigma)$, this establishes~\eqref{equation:on-V}.
\end{proof}

\subsubsection{} 
When $g = g_{\Vfa_\C},$ that is when the Green functions $g_\sigma$ are the equilibrium potentials 
 $g_{V_\sigma, P_\sigma}$ associated to the points $P_\sigma$ of the compact Riemann surface with boundary $V_\sigma$,  then the Hermitian line bundle $ \Nb_{P, g} \Vf$ is the Hermitian line bundle $\overline{N}_P \Vfa$ introduced in \eqref{defNPVfa}, and the invariant $\mathrm{Ex}(\alpha_\sigma, g_\sigma)$ is the overflow $\Ex (\alpha_\sigma:(V_\sigma, P_\sigma) \ra X_\sigma)$. Consequently, applied to $g=g_{\Vfa_\C},$ Proposition \ref{proposition:intersection-ov} becomes the following result, which will play a key role in the remainder of this memoir.
 
 \begin{corollary}\label{cor: KeySelfInt} For every non-constant morphism over $\Spec \OK$,
 $$\alpha:= (\widehat \alpha, (\alpha_\sigma)_{\sigma:K \hra \C}) : \Vfa\lra X,$$
 from a smooth \fa arithmetic surface $\Vfa$ to an integral normal arithmetic surface $X$ over $\Spec \OK$, the following equality holds:
 \begin{multline}\label{equation:on-V-equ}
\alpha_\ast (P, g_{\Vfa_\C}) \cdot \alpha_\ast (P,g_{\Vfa_\C})=e(\alpha) \,\widehat{\deg}\,P^*\overline N_{P}\Vfa+\mathrm{Ex}\big(\widehat \alpha : \Vf\ra X\big)\\ +\sum_{\sigma: K \hra \C}\mathrm{Ex}\big(\alpha_\sigma : (V_\sigma, P_\sigma)\ra X_\sigma)\big).
\end{multline}
 
 \end{corollary}

In turn, the non-negativity of the overflow invariants implies: 

\begin{corollary}\label{cor:inequality-ov} With the notation above, we have:
\begin{equation}\label{equation:inequality-ov}
\alpha_\ast (P, g_{\Vfa_\C}) \cdot \alpha_\ast (P,g_{\Vfa_\C}) \geq e(\alpha) \, \widehat{\deg}\,P^*\overline N_{P}\Vfa.
\end{equation}
\end{corollary}

\subsection{Self-intersection of $\alpha_\ast( P, g_{\Vfa_\C})$ and  Green function for the diagonal}\label{selfintdiagonal}

\subsubsection{}
We may combine Corollary \ref{cor: KeySelfInt} 
and the expression in Theorem \ref{proposition:explicit} for the archime\-dean overflow in terms of  a Green function of the diagonal. To formulate the resulting expression for the self-intersection of $\alpha_\ast (P, g_{\Vfa_\C})$, we need to introduce some notation.

For every embedding $\sigma: K \hra \C,$ we assume that the connected Riemann surface $X_\sigma(\C)$ is endowed with a real $2$-form $\beta_\sigma$ of class $\cC^\infty$, and that:
$$g_{X_\sigma}: X_\sigma(\C) \times X_\sigma(\C) \lra (-\infty, +\infty]$$
is a Green function for the diagonal of $X_\sigma(\C)$ associated to $\beta_\sigma,$ as defined in \ref{Greendiagdef}.  We assume that the family $(g_{X_\sigma})_{\sigma: K \hra \C},$ and therefore the family $(\beta_\sigma)_{\sigma: K \hra \C}$ is invariant under complex conjugation.  

Recall that we may attach   the capacitary metric $\Vert.\Vert_{g_\sigma}^{\mathrm{cap}}$ on the complex tangent bundle $T_{X_\sigma}$
to the Green function $g_\sigma$;  see \ref{Greendiagdef}. The family of Hermitian metrics $(\Vert.\Vert_{g_\sigma}^{\mathrm{cap}})_{\sigma: K \hra \C}$ is clearly invariant under complex conjugation.

We shall also assume that the image of \emph{the $\OK$-point $Q := \walpha(P)$ of $X$ lies in the regular locus of} $X$. Then the normal bundle $N_Q X$ is a well-defined line bundle over $Q$, and equipped with the capacitary metrics $\Vert .\Vert_{g_\sigma, Q_\sigma}^{\mathrm{cap}}$ on the complex lines:
$$(N_Q X)_\sigma \simeq T_{X_\sigma, Q_\sigma},$$
it defines a Hermitian line bundle over $Q$:
$$\Nb_Q^{\mathrm{cap}} X := \big(N_Q X, (\Vert .\Vert_{g_\sigma, Q_\sigma}^{\mathrm{cap}})_{\sigma: K \hra \C}\big).$$

\begin{proposition} With the notation above, the following formula holds:
\begin{multline}\label{selfNQ}
\alpha_\ast (P, g_{\Vfa_\C}) \cdot \alpha_\ast (P,g_{\Vfa_\C}) = \dega Q^\ast\Nb_Q^{\mathrm{cap}} X 
\\+  \sum_{\sigma: K \hra \C}\Big(2\int_{V_\sigma}g_{V_\sigma, P_\sigma} \, \alpha_\sigma^{*}\beta_\sigma-\int_{(\partial V_\sigma)^{2}}g_{X_\sigma}(\alpha_\sigma(z_{1}),\alpha_\sigma(z_{2}))\, d\mu_{V_\sigma, P_\sigma}(z_{1}) \, d\mu_{V_\sigma, P_\sigma}(z_{2})\Big).
\end{multline}
\end{proposition} 

Observe that the image of the section $Q$ actually lies in the open subscheme $X_{\rm sm}$ of $X$ where the structure morphisme $\pi_X: X \ra \Spec \OK$ is smooth. Endowed with the Hermitian metrics $\big(\Vert.\Vert^{\rm cap}_{g_\sigma}\big)_{\sigma: K \hra \C}$ on the Riemann surfaces $X_{{\rm sm}, \sigma} (\C) = X_\sigma(\C),$ the relative tangent bundle $T_{\pi_X}$ over $X_{\rm sm}$ defines a Hermitian line bundle $\Tb^{\rm cap}_{\pi_X}$ over $X_{\rm sm}$. The Hermitian line bundle $Q^\ast\Nb_Q^{\mathrm{cap}} X$ may be identified with $Q^\ast \Tb^{\rm cap}_{\pi_X}$, and its Arakelov degree coincides with the height of $Q$ with respect to $\Tb^{\rm cap}_{\pi_X}$:
$$\dega Q^\ast\Nb_Q^{\mathrm{cap}} X = \dega Q^\ast \Tb^{\rm cap}_{\pi_X} =: \height_{\Tb^{\rm cap}_{\pi_X}} (Q).$$

\begin{proof} Let us write $e:= e(\alpha)$. The $e$-jet $\walpha_K^{[e]}(P_K)$ of $\walpha_K: \Vf_K \ra X_K$ at $P_K$ defines an isomorphism of $K$-lines:
$$\walpha_K^{[e]}: \big( N_P \Vf\big)_K ^{\otimes e} \lrasim T_{Q_K} X_K.$$

For every embedding $\sigma: K \hra \C$, by means of the identification: 
$$(T \iota_\sigma)^{\otimes e} : \big(N_P \Vf)_\sigma^{\otimes e} \lrasim \big(T_{P_\sigma}V_\sigma\big)^{\otimes e},$$
the
 base change of $\walpha_K^{[e]}$ by $\sigma$ becomes the isomorphism of $\C$-lines:
$$\alpha_\sigma^{[e]}(P_\sigma):\big(N_P \Vf)_\sigma^{\otimes e} 
{\simeq} \big(T_{P_\sigma}V_\sigma\big)^{\otimes e} \lrasim T_{Q_\sigma} X_\sigma,$$
already considered, in a more general context, in Subsection \ref{alternative}. Its norm with respect to the metric $\Vert.\Vert_{V_\sigma, P_\sigma}^{\mathrm{cap} \otimes e}$ on $\big(T_{P_\sigma}V_\sigma\big)^{\otimes e}$ and $\Vert.\Vert^{\mathrm{cap}}_{g_{X_\sigma}}$ on $T_{Q_\sigma} X_\sigma$ will be denoted by $\Vert \alpha_\sigma^{[e]}(P_\sigma)\Vert^{\mathrm{cap}}_e$. 
With this notation, according to Theorem \ref{proposition:explicit}, the following equality holds for every embedding $\sigma: K \hra \C$:
\begin{multline}\label{equation:explicitsigma}
\mathrm{Ex}\big(\alpha_\sigma : (V_\sigma, P_\sigma)\ra X_\sigma\big)  = 2\int_{V_\sigma}g_{V_\sigma, P_\sigma} \, \alpha_\sigma^{*}\beta_\sigma \\ -\int_{(\partial V_\sigma)^{2}}g_{X_\sigma}(\alpha_\sigma(z_{1}),\alpha_\sigma(z_{2}))\, d\mu_{V_\sigma, P_\sigma}(z_{1}) \, d\mu_{V_\sigma, P_\sigma}(z_{2})   - \log\Vert \alpha_\sigma^{[e]}(P_\sigma)\Vert^{\mathrm{cap}}_e .
\end{multline}

Observe that the isomorphism $\walpha_K^{[e]}$ is the inverse of the restriction to $\Spec K$ of the isomorphism $I_{\mid P}$ in the proof of Lemma   \ref{degaQast} in the special case $N=1$. Therefore  $\walpha_K^{[e]}$ extends to an isomorphism of line bundles over $\Spec \OK$:
$$\walpha^{[e]}: P^\ast N_P\Vf^{\otimes e} \otimes P^\ast \cO_P(P \cap R) \lrasim Q^\ast N_Q X,$$
where $R$ is defined as in \ref{XregulardefR}. This implies the following relation between Arakelov degrees:
\begin{equation}\label{selfNQbis}
\dega Q^\ast \overline N_Q X  = e\, \dega P^\ast \overline  N_P\Vfa + \dega P\cdot R 
- \sum_{\sigma: K \hra \C}\log\Vert \alpha_\sigma^{[e]}(P_\sigma)\Vert^{\mathrm{cap}}_e. 
\end{equation}

The equality \eqref{selfNQ} follows from
the expression \eqref{equation:on-V-equ} for the self-intersection of $\alpha_\ast (P, g_{\Vfa_\C})$, combined with \eqref{equation:explicitsigma}, \eqref{selfNQbis}, and the definition \eqref{defExar} of $\Ex \big(\widehat \alpha: \Vf \ra X\big)$.
%
\end{proof}

\subsubsection{} Let $\psi$ a formal series in:  $$G_{\mathrm{for}}(\R) := \R^\ast X + X^2 \R[[X]],$$ and let:
$$\alpha : 
\Vfa(\Db(0,1), \psi) \lra \A^1_\Z$$
be a morphism from the \fa arithmetic surface $\Vfa(\Db(0,1), \psi)$ over $\Spec \Z$ attached to $\psi$ defined  in \ref{defVfapsi}. 

Equivalently, $\alpha$ is an element of the ring $\cO(\Vfa(\Db(0,1), \psi)$ of regular functions over $\Vfa(\Db(0,1), \psi)$. It is defined as a pair: $$\alpha := (\walpha, \alpha^{\an}),$$ where $\walpha$ is a morphism from $\Spf \Z[[T]]$ to $\A^1_\Z$ --- that is an element of $\Z[[T]]$ --- and $\alpha^{\an}$ is an analytic function:
$$\alpha^{\an}: \Db(0,1)^+ \lra \A^1_\Z(\C) = \C,$$
that satisfy the gluing relation:
 \begin{equation}\label{gluealpha}
 \walpha = \alpha^\an \circ \psi,
 \end{equation}
when $\walpha,$ $\alpha^\an$, and $\psi$ are seen as formal series in $\C[[X]]$. The morphism $\alpha$ is non-constant if and only if $\alpha^\an$ belongs to $\Z[[T]] \setminus \Z$, or equivalently if and only if $\alpha^{\an}$ is not a constant analytic function.

\begin{corollary}\label{selfA1} For every $\psi \in G_{\mathrm{for}}(\R)$, and every non-constant morphism:
$$\alpha := (\walpha, \alpha^{\an}) :
\Vfa := \Vfa(\Db(0,1), \psi) \lra \A^1_\Z,$$ 
the following equality holds:
\begin{equation}\label{selfintpsiA}
\alpha_\ast (P, g_{\Vfa_\C}) \cdot \alpha_\ast (P,g_{\Vfa_\C}) = 2 \int_0^1 \int_0^1 \log \left\vert \alpha^\an (e^{2 \pi i t_1}) - \alpha^\an (e^{2 \pi i t_2}) \right\vert \, dt_1 \, dt_2.
\end{equation}
\end{corollary}

\begin{proof}
The function:
$$g_{\C} : \C \times \C \lra (-\infty, +\infty], \quad (z_1, z_2) \lra \log \vert z_1 - z_2 \vert^{-1}$$
is a Green function for the diagonal of $\C= \A^1_\Z(\C)$, associated to the $2$-form $\beta = 0.$ The associated capacitary metric on $T_{\A^1_\C}$ satisfies:
$$ \Vert \partial /\partial z \Vert^{\mathrm{cap}}_{g_{\C}} = 1.$$
Moreover the equilibrium measure $\mu_{\Db(0,1), 0}$ associated to the equilibrium potential: 
$g_{\Db(0,1), 0} = (z \mapsto \log^+ \vert z \vert^{-1})$
for the origin in the unit disk, is the normalized Haar measure on:
$$\partial \Db(0,1) = \{ z \in \C \mid \vert z \vert =1 \} = U(1).$$

Using these observations, it is readily seen that, applied to $\Spec \OK = \Spec \Z$, to the \fa arithmetic surface $\Vfa = \Vfa(\Db(0,1), \psi)$ and to $X= \A^1_\Z$ equipped with the Green function $g_{\C}$, the expression \eqref{selfNQ} for the self-intersection of $\alpha_\ast (P, g_{\Vfa_\C})$ becomes \eqref{selfintpsiA}.
 \end{proof}
 
 \subsubsection{}\label{exampleP1} Corollary \ref{selfA1} admits a variant, concerning a morphism $\alpha$ from $\Vfa(\Db(0,1), \psi)$ to $\PP^1_\Z,$ obtained by using the Green function $g_{\PP^1(\C)}$ for the diagonal of $\PP^1(\C)$ introduced in \ref{example:Green} (2) in the expression \eqref{selfNQ} for the self-intersection of $\alpha_\ast (P, g_{\Vfa_\C})$. 
As already observed  in \ref{ExamplesOverflow}, the first integral in the right-hand side of \eqref{selfNQ} is the value at 1 of the characteristic function $T_\alpha$ of Nevanlinna-Ahlfors-Shimizu associated to $\alpha$. 

The following formulation of this special instance of Proposition \ref{proposition:explicit-C} emphasizes this relation to Nevalinna theory, and facilitates the comparison of our results in Section \ref{SectMeropseudocon} below with the ones in \cite{CalegariDimitrovTang21}.

Consider the following data:
\begin{itemize}
\item a formal series $\psi$ in $G_{\mathrm{for}}(\R)$;
\item a morphism of (formal) schemes: $$\walpha: \Spf \Z[[T]] \lra \PP^1_\Z;$$
\item a  meromorphic function, defined on the open disk $D(0, R)$ of radius $R \in (0, +\infty]$:
$$\alpha^{\an}: D(0, R) \lra \PP^1(\C);$$
\end{itemize}
and assume that $\walpha$ and $\alpha^\an$ are not constant  and satisfy the gluing relation \eqref{gluealpha}. 

Let us denote by:
$$\height: \PP^1(\Q) \lra \R_+$$
the usual height, defined by:
\begin{equation}\label{usualheight}
\height(x_0: x_1) := \log \left(x_0^2 + x_1^2\right)^{1/2}
\end{equation}
for every pair $(x_0, x_1)$ of integers that are prime together. 

For every $r \in (0, +\infty),$ we may consider the following \fa surface over $\Spec \Z$:
$$\Vfa_r := \Vfa(\overline{D}(0,r), \psi),$$
and the morphism:
$$\alpha_r := \big(\walpha, \alpha^\an_{\mid \overline{D}(0,r)^+}\big): \Vfa_r \lra \PP^1_\Z.$$

\begin{corollary} With the notation above, for every $r \in (0, R),$ the following equality holds:
\begin{multline}\label{selfintpsiP1}
\alpha_\ast (P, g_{\Vfa_{r,\C}}) \cdot \alpha_\ast (P,g_{\Vfa_{r,\C}}) = 2 \, \height (\alpha(0)) + 2 \,  T_{\alpha^\an} (r) 
\\ - \int_0^1 \int_0^1 g_{\PP^1(\C)}\big(\alpha^\an (r \,e^{2 \pi i t_1}), \alpha^\an (r\, e^{2 \pi i t_2}) \big) \, dt_1 \, dt_2.
\end{multline}
\end{corollary}

In the right-hand side of \eqref{selfintpsiP1}, we denote by $\alpha(0)$ the point $\walpha(0) = \alpha^\an(0)$ of $\PP^1(\Q),$ by $T_{\alpha^\an}$ the characteristic function of ${\alpha^\an}$ as defined in \eqref{chardef}, and by  $g_{\PP^1(\C)}$ the Green function for the diagonal of $\PP^1(\C)$ defined by  \eqref{gP1def}. 

To derive \eqref{selfintpsiP1} from the expression \eqref{selfNQ}  applied to the morphism $\alpha:=\alpha_r$ from $\Vfa := \Vfa_r$ to $X: = \PP^1_\Z$ and to the Green function $g := g_{\PP^1(\C)},$ observe that the Hermitian line bundle $\Tb^{\rm cap}_{\pi_X}$ is then isomorphic to the second tensor power of the Hermitian line bundle $\cOb_{\PP^1}(1)$ over $\PP^1_\Z$ that defines the usual height~\eqref{usualheight}, and that the equilibrium measure $\mu_{\Db(0,1), r}$ is the rotation invariant probability measure on $\partial \Db(0,r).$

Observe that, since $g_{\PP^1(\C)}$ is non-negative, \eqref{selfintpsiP1} implies the following upper bound:
\begin{equation}\label{selfintpsiP1ineq}
\alpha_\ast (P, g_{\Vfa_{r,\C}}) \cdot \alpha_\ast (P,g_{\Vfa_{r,\C}}) \leq  2 \, \height (\alpha(0)) + 2 \,  T_{\alpha^\an} (r). 
\end{equation}

\section{Meromorphic maps from \fa arithmetic surfaces to proper arithmetic schemes}

The results of the previous sections may be extended to the situation where, instead of a morphism $f: \Vfa \ra X$ from a smooth \fa surface $\Vfa$ to a quasi-projective scheme or an arithmetic surface 
$X$ over $\OK,$ we consider a meromorphic map:
\begin{equation}\label{meroVfaX}
f : \Vfa \dashrightarrow X
\end{equation}
from $\Vfa$ to some projective scheme over $\cOK$.

For lack of suitable available references concerning rational maps or  morphisms defined on general Noetherian formal schemes, we will follow a pedestrian approach to define the meromorphic maps 
\eqref{meroVfaX}, by taking advantage of the two-dimensional nature of the formal scheme $\Vf$ underlying the \fa surface $\Vfa$.

In this section, we denote by: 
$$\Vfa:=\big(\Vf, (V_{\sigma}, P_\sigma, \iota_{\sigma})_{\sigma: K\hra\C}\big)$$
a smooth \fa arithmetic surface over $\cOK$.

\subsection{Regular modifications of $\Vf$ and meromorphic maps from $\Vf$ to projective $\OK$-schemes}

\subsubsection{} We define a \emph{regular modification} of $\Vf$ as a morphism of formal schemes: 
$$\nu: \Vf' \lra \Vf$$
that may be written as a composition:
$$\nu := \nu_1 \circ \dots \circ \nu_n : \Vf' = \Vf_n \stackrel{\nu_n}{\lra} \Vf_{n-1} \stackrel{\nu_{n-1}}{\lra}  \dots  \stackrel{\nu_{2}}{\lra}\Vf_1 \stackrel{\nu_1}{\lra} \Vf_{0} =\Vf$$
where $\nu_i$ is the blowing up of a finite set of closed points of the definition scheme $\vert \Vf_i \vert$ of $\Vf_i$.

By construction, $\Vf'$ is a ``regular formal arithmetic surface", whose reduced scheme of definition $\vert \Vf' \vert$ is purely one-dimensional and proper over $\Spec \OK$. 

The scheme $\vert \Vf' \vert$ admits a unique irreducible component that is finite over $\Spec \OK$. We shall denote it by $\vert \Vf' \vert_{\mathrm{hor}}$; the restriction of $\nu$ defines an isomorphism:
$$\nu_{ \mid \vert \Vf' \vert_{\mathrm{hor}}} : \vert \Vf' \vert_{\mathrm{hor}} \lrasim \vert \Vf \vert.$$

The other components of $\vert \Vf \vert$ are (isomorphic to) projective lines over some finite fields. We shall denote the union of these vertical components by $\vert \Vf' \vert_{\mathrm{vert}}$.

We shall denote by $\Ind(\nu^{-1})$ the indeterminacy locus of $\nu^{-1}$, namely the finite set of closed points of $\vert \Vf \vert$ in the image of the centers of the morphisms $\nu_i$:
$$\Ind(\nu^{-1}) := \bigcup_{1 \leq i \leq n} \nu_1 \circ \dots \circ \nu_{i-1} (C_i),$$
where $C_i$ denotes the center of the blowing up $\nu_i$. 

The inverse image by $\nu$ of the closed subscheme $\Ind(\nu^{-1})$ of $\vert \Vf \vert$, hence of $\Vf$, is an effective divisor on $\Vf'$, of support $\vert \Vf' \vert_{\mathrm{vert}}$.   

\subsubsection{} A regular modification $\nu: \Vf' \ra \Vf$ as above is an adic morphism of formal schemes. For every coherent $\cO_{\Vf'}$-module $\cF$, the direct image $\nu_\ast \cF$ is a coherent $\cO_{\Vf}$-module. When $\cF$ is locally free, its direct image $\nu_\ast \cF$ is not always locally free. However it is torsion free, its bidual $(\nu_\ast \cF)^{\vee \vee}$ is a locally free $\cO_\Vf$-module, the tautological morphism:
$$\tau: \nu_\ast \cF  \lra (\nu_\ast \cF)^{\vee \vee}$$
is injective, and its cokernel $\coker \tau$ is an $\cO_{\Vf}$-module of finite length, supported by $\Ind(\nu^{-1}).$

By applying these observations to a line bundle $L'$ over $\Vf',$ one obtains:
\begin{proposition}\label{lineprime} For every regular modification $\nu: \Vf' \ra \Vf$ as above and every line bundle $L'$ over $\Vf'$, there exists a line bundle $L$ over $\Vf$, a divisor $W$ in $\Vf'$ supported by $\vert \Vf' \vert_{\mathrm{vert}}$, and an isomorphism of lines bundles over $\Vf'$:
$$\iota : L' \lrasim \nu^\ast L \otimes \cO_{\Vf'}(W).$$
When this holds, $L$ is canonically isomorphic to $(\nu_\ast L')^{\vee \vee}$.
\end{proposition} 

\subsubsection{} We may introduce the \emph{category of regular modifications} of $\Vf$ by defining a morphism for a regular modification $\nu_1: \Vf_1 \ra \Vf$ to a regular modification $\nu_2: \Vf_2 \ra \Vf$ as a morphism of formal schemes $\phi: \Vf_1 \ra \Vf_2$ such that $\nu_2 \circ \phi = \nu_1$.

Actually, if such a morphism exists, it is unique and may be written as a composition of blowing-ups of closed points. Moreover, for any pair of regular modifications:
$$\nu_i : \Vf_i \lra \Vf, \quad i= 1, 2,$$
there exists a regular modification $\nu: \Vf' \ra \Vf$ and two morphisms $\phi_1$ and $\phi_2$ from $\nu$ to $\nu_1$ and $\nu_2$ respectively. This notably implies that the category of regular modifications is equivalent to a directed set.

\subsubsection{} Let $X$ be a projective scheme over $\Spec \OK$. We define a \emph{meromorphic map}:
$$\widehat{f}: \Vf \dashrightarrow X$$
over $\Spec \OK$
as a morphism $\widehat{f}_K: \widehat{\mathcal V}_K \ra X_K$ of (formal) schemes over $K$
such that there exists a regular modification $\nu: \Vf' \ra \Vf$
and a morphism ${\widehat{f}}': \widehat{\mathcal V}' \ra X$ of (formal) schemes over $\Spec \OK$:
such that:
$${\widehat{f}}'_K = \widehat{f}_K \circ \nu_K.$$
When this holds, we will say that the regular modification $\nu$ is \emph{adapted} to $\widehat{f}$.

The set of  meromorphic maps from $\Vf$ to $X$ over $\Spec \OK$ may be identified with the limit:
$$\lim_{\nu: \Vf' \ra \Vf} \mathrm{Mor}_{\OK}( \Vf', X)$$ 
of the sets $\mathrm{Mor}_{\OK}( \Vf', X)$ of morphisms over $\Spec \OK$, taken over the directed set of formal modifications $\nu: \Vf' \ra \Vf$.

The elements of the field $\cM(\Vf)$ of formal  meromorphic functions  on $\Vf$, as defined in  \cite{HironakaMatsumura68}, may be identified with the  meromorphic maps:
$$\widehat{f}: \Vf \dashrightarrow \PP^1_{\OK}$$
over $\Spec \OK$
that are distinct from the ``constant" morphism $\infty$. This follows from the desingularization \emph{\`a la} M. Noether of divisors in the regular formal scheme $\Vf$; see for instance \cite{Abhyankar56} and \cite[pages 38-44]{Shafarevich66}. 

More generally, consider a  morphism $\widehat{f}_K: \widehat{\mathcal V}_K \ra X_K$ of (formal) schemes over $K$, an affine open neighborhood $U$ of $\widehat{f}_K(P_K)$ in $X_K$, and an embedding of $K$-schemes:
$$(x_1, \dots, x_n) : U \hlra \A^N_K.$$ 
Then $\widehat{f}_K$ defines a  meromorphic map $\widehat{f}: \Vf \dashrightarrow X$ if and only if the elements $x_1 \circ \widehat{f}_K, \dots, x_N \circ \widehat{f}_K$ of $\cO(\Vf_K)$ are elements of the subfield $\cM(\Vf) := \mathrm{Frac}\,  \cO(\Vf)$ of $\cM(\Vf_K) := \mathrm{Frac}\, \cO(\Vf_K).$


\subsection{Meromorphic maps from $\Vfa$ to projective $\OK$-schemes} 

\subsubsection{} A \emph{meromorphic map}:
$${f}: \Vfa \dashrightarrow X$$
\emph{over $\Spec \OK$} from the smooth \fa surface $\Vfa$ to a projective $\OK$-scheme $X$ is defined by a  generalization of the definition of morphisms in  
\ref{morphfasurf}, namely  as a 
pair:
$$f:=(\widehat f, (f_{\sigma})_{\sigma : K\hra \C})$$
where: 
$$\widehat f : \Vf \dashrightarrow X$$
is a  meromorphic map over $\OK$ and where, for every complex embedding $\sigma$ of $K$, 
$$f_{\sigma} : V_{\sigma}\lra X_{\sigma}(\C)$$
is a complex analytic map. Moreover, 
for every  embedding $\sigma: K \hra \C,$ the morphisms from the smooth formal curve $\Vf_\sigma$ to the complex scheme $X_\sigma$ deduced from $\widehat f_K$ (by the base change $\sigma$) and from $f_\sigma$ (by considering its ``formal germ" at $P_\sigma$) are required  to coincide.

The discussion in \ref{Zariskiclosures} of the relations between the diverse ``Zariski closures of the image" associated to a morphism $f:\Vfa \ra X$ easily extends to the  meromorphic maps $f:\Vfa \dashrightarrow X$ so defined. In particular, it makes sense to say that  $f:\Vfa \dashrightarrow X$ is constant, or that its image is algebraic. 

\subsubsection{}\label{pullingbackprime} The construction of the pro-Hermitian vector bundle $\pi^{L^2}_{(\Vfa, \mu)\ast} \widetilde{\overline E}$ over $\Spec \OK$ associated to a Hermitian vector bundle  $\widetilde{\overline E}$ on $\Vfa$ discussed in \ref{directimageproH}, and the construction of the morphism of (pro-)Hermitian vector bundles :
$$\phi_{E} : \pi^J_{X\ast}{\overline E} \lra \pi^{L^{2}}_{(\Vfa, \mu) \ast} f^{*}\overline E$$
associated to a morphism $f$ from $\Vfa$ to some projective $\OK$-scheme $X$ equipped with a Hermitian vector bundle $\Eb$ that we discussed in \ref{subsubsection:vb} admit the following generalizations involving regular modifications. 

Consider a regular modification $\nu: \Vf' \ra \Vf$ as above. The pair:
$$\Vfa' := (\Vf', (V_{\sigma}, O_\sigma, \iota_{\sigma})_{\sigma: K\hra\C})$$
defines a generalized smooth formal-analytic surface. There is an obvious notion of a vector bundle:
$$\widetilde E:=(\widehat E, (E_{\sigma}, \phi_{\sigma})_{\sigma : K\hra\C}),$$
 and of a Hermitian vector bundle:
 $$\widetilde{\overline E}:=(\widehat E, (E_{\sigma}, \phi_{\sigma}, \Vert.\Vert_{\sigma})_{\sigma : K\hra\C})$$ 
 over $\Vf'$, defined as in \ref{VectbundleDef}, where $\widehat E$ is now a vector bundle over $\Vf'$.

 Let $\mu$ be a positive volume form on $V_\C$, invariant under complex conjugation, as in \ref{directimageproH}. Then to a  Hermitian vector bundle $\widetilde{\overline E}$ over $\Vfa'$, we may attach:
 \begin{equation}\label{piastprime}\pi^{L^2}_{(\Vfa', \mu)\ast} \widetilde{\overline E} :=
 \Big(\Gamma(\Vf', \widehat E), (\Gamma_{L^{2}}(V_{\sigma}, \mu_{\sigma} ; E_{\sigma}, \Vert.\Vert_{\sigma}), \widehat\eta_{\sigma})_{\sigma:K\hra\C}\Big),
 \end{equation}
 where $\widehat\eta_{\sigma}$ is defined as in \ref{directimageproH}, with obvious minor modifications. 
 An extension of the proof in \cite[10.6.2]{Bost2020} shows that it is a pro-Hermitian vector bundle over $\Spec \OK$, as in the case where $\Vfa'$ is $\Vfa$ considered in \ref{directimageproH} above. 
 
The construction \eqref{piastprime} of the direct image  on $\Spec \OK$ of some Hermitian vector bundle over $\Vfa'$ is related as follows to the construction of the direct image of Hermitian torsion free coherent sheaves and vector bundles  over $\Vfa$.

Consider a Hermitian vector bundle:
$$\Eb':= \big(\widehat E', (E_{\sigma}, \phi_{\sigma}, \Vert.\Vert_{\sigma})_{\sigma : K\hra\C}\big)$$ 
over $\Vfa'$. The direct image $\nu_\ast \Eh'$ of the underlying vector bundle $\Eh'$ over $\Vf$ by the modification $\nu$ is a torsion free coherent sheaf over $\Vf$. Moreover its ``restriction"  $(\nu_\ast \Eh')_K$ to $\Vf'_K \simeq \Vf_K$ may be identified with $\Eh'_K$. We  therefore define a Hermitian torsion free coherent sheaf over  $\Vfa$ by letting:
$$\nu_\ast \Eb':= \big(\nu_\ast\widehat E', (E_{\sigma}, \phi_{\sigma}, \Vert.\Vert_{\sigma})_{\sigma : K\hra\C}\big).$$ 

 The identification of coherent sheaves of $\cO_{\Vf}$-modules:
 $$\pi_{\Vf\ast} \nu_\ast \Eh' \simeq \pi_{\Vf'\ast} \Eh'$$
 extends to an identification:
 \begin{equation}\label{piVVprime}
\pi^{L^2}_{(\Vfa, \mu)\ast}(\nu_\ast \Eb')
\simeq
\pi^{L^2}_{(\Vfa', \mu)\ast} \Eb'.
\end{equation}
As explained in \ref{proHtfc}, the left-hand side of \eqref{piVVprime} is a pro-Hermitian vector bundle over $\Spec \OK$, and therefore $\pi^{L^2}_{(\Vfa', \mu)\ast} \Eb'$ also is. Moreover, as shown by Proposition \ref{CompPistarEbid}, the properties of the pro-Hermitian vector bundle $\pi^{L^2}_{(\Vfa', \mu)\ast} \Eb'$ --- for instance its $\theta$-finiteness --- are closely related to the properties of $\pi^{L^2}_{(\Vfa, \mu)\ast} \Eb$, where:
$$\Eb := (\nu_\ast \Eb')^{\vee \vee}$$
is the Hermitian vector bundle over $\Vfa$ deduced from $\nu_\ast \Eb'$ by biduality.

 \subsubsection{} Once the direct image \eqref{piastprime} has been introduced, Proposition \ref{proposition:pull-back-norm} immediately extends to the situation where   $f$ is a  meromorphic map:
 $$f: = (\widehat{f}, f^\an) : \Vfa \dashrightarrow X$$
 from $X$ to a reduced flat projective $\OK$-scheme $X$ equipped with some Hermitian vector bundle $\Eb$.
 
 Indeed, if $\nu: \Vf' \ra \Vf$ denotes a regular modification of $\Vf$ adapted to the  meromorphic map:
 $$\widehat{f}: \Vf \dashrightarrow X$$
 and if:
 $$f' := (\widehat{f}', f^\an): \Vfa' \lra X$$
 denotes the associated $\OK$-morphism, then we may consider the Hermitian vector bundle $f'^\ast \Eb$ over $\Vfa'$ and its direct image $\pi^{L^2}_{(\Vfa', \mu)\ast} f'^\ast \Eb$ over $\Spec \OK$. It is readily seen to be ``independent" of the choice of the adapted modification $\nu$. Moreover the pull-back of sections of $E$ by $f'$ defines a morphism:
 $$\phi_{E} : \pi^J_{X\ast}\widetilde{\overline E} \lra \pi^{L^{2}}_{(\Vfa', \mu) \ast} f'^{*}\overline E$$
of (pro-)Hermitian vector bundles  over $\Spec\OK$, and the last assertion of Proposition \ref{proposition:pull-back-norm} still holds.

\subsection{Meromorphic maps from $\Vfa$ to projective arithmetic surfaces and arithmetic intersections numbers}subsection

Assume that $X$ is an integral normal projective arithmetic surface over $\Spec \OK$, and consider a non-constant  meromorphic map:
$$\alpha: (\walpha, (\alpha_\sigma)_{\sigma: K \hra \C}): \Vfa \dashrightarrow X.$$
As before, we shall denote by $\alpha^\an$ the  complex analytic map from $V_\C$ to $X(\C)$ defined by the maps~$\alpha_\sigma$, and by $e(\alpha)$ the ramification index of the morphism $\walpha_K: \Vf_K \ra X_K$. 

In this subsection, we discuss the extension to this framework of the constructions and results of Section \ref{MorFaArSurf}. 

\subsubsection{} Let us choose a regular modification $\nu: \Vf' \ra \Vf$ adapted to $\walpha$, and let us denote by:
$$\walpha': \Vf' \lra X$$
the morphism of (formal) schemes over $\Spec \OK$ that defines $\walpha$.

Let $(D,g)$ be an Arakelov divisor in $\Zb^1_c(\Vfa)$. The inverse image $\nu^\ast D$ is a divisor in $\Vf$ supported by $\vert \Vf' \vert,$ and its direct image by the proper map $\walpha': \vert \Vf' \vert \ra X$ defines a divisor:
$$\Gamma_{\walpha \ast} D := \walpha'_\ast \nu^\ast D$$
in $Z^1(X)$. It is readily seen to be independent of the choice of the adapted regular modification $\nu$, and we may define:
$$\Gamma_{\alpha\ast} (D,g) := (\Gamma_{\walpha \ast} D, \alpha^\an_\ast g) = (\Gamma_{\walpha \ast} D, (\alpha_{\sigma \ast} g_\sigma)_{\sigma: K \hra \C}).$$
It is an Arakelov divisor in $\Zb^1(X)$, and this construction defines a morphism of $\Z$-modules:
$$\Gamma_{\alpha \ast}: \Zb^1_c(\Vfa) \lra \Zb^1(X),$$
which is compatible with the positivity of Arakelov divisors. 

If $f: X \lra X'$ is a dominant, or equivalently surjective, morphism of integral normal projective arithmetic surfaces over $\Spec \OK$, then:
$$f \circ \alpha := (f \circ \walpha, (f_\sigma \circ \alpha_\sigma)_{\sigma: K \hra \C}): \Vfa \dashrightarrow X'$$
is a non-constant meromorphic map, and for every Arakelov divisor $(D,g)$ in $\Zb^1_c(\Vfa)$, the following equality holds in $\Zb^1(X')$:
\begin{equation}\label{falphaast}
(f\circ \alpha)_\ast (D,g) = f_\ast \alpha_\ast (D,g).
\end{equation}
The proof is straighforward.

\subsubsection{} We shall denote by $P'$ the section of the structure morphism:
$$\pi_{\Vf'} : \Vf' \lra \Spec \OK$$
defined as the proper transform of $P$. It is the inverse of the isomorphism:
$$\pi_{\Vf' \mid \vert \Vf' \vert_{\mathrm{hor}}}: \vert \Vf' \vert_{\mathrm{hor}}  \lrasim \Spec \OK.$$

The divisor $\nu^\ast P$ in $\Vf'$ may be written:
\begin{equation}\label{PVprime}
\nu^\ast P = P' + V',
\end{equation}
where $V'$ is an effective divisor supported by $\vert \Vf' \vert_{\mathrm{vert}}$.

As before, we denote by $Q$ the section $\walpha' \circ P$ of the structure morphism $\pi_X : X \ra \Spec \OK$. The inverse image $\walpha'^\ast (Q)$ is a $\Q$-divisor in $\Vf$, and may be written: 
\begin{equation}\label{PRprime}
\walpha'^\ast (Q) = e(\alpha) P' + R',
\end{equation}
where $R'$ is an effective $\Q$-divisor in $\Vf'$ such that $\vert R' \vert \cap P' = \vert R' \vert \cap \vert \Vf' \vert_{\mathrm{hor}}$ is a finite set of closed points, and therefore $\vert R' \vert \cap \vert \nu^\ast P \vert$ is proper over $\Spec \OK$.

The direct image 
$\pi_{\vert \Vf \vert \ast} (R' \cdot \nu^\ast P)$ by the structure morphism $\pi_{\vert \Vf \vert}: \vert \Vf \vert \ra \Spec \OK$
of the $0$-dimensional intersection cycle  of $R'$ and $\nu^\ast P$, which is defined up to vertical linear equivalence, is a well defined $0$-dimensional $\Q$-cycle in $\Spec \OK$. It is readily seen to be independent of the choice of the adapted regular modification  $\nu$, and consequently its Arakelov degree:
\begin{equation}\label{RprimeP}
\dega R' \cdot \nu^\ast P
\end{equation}
also is.  

Observe  that the Arakelov degree \eqref{RprimeP} is non-negative. Actually, for any effective $\Q$-divisor $D$ in $\Vf'$ that does not contain $P'$, the Arakelov degree:
$$\dega D \cdot \nu^\ast P$$
is well-defined and non-negative. This non-negativity indeed follows from the fact that, for every component $W$ of $\vert \Vf' \vert_{\mathrm{hor}}$, the  intersection number $\dega W \cdot \nu^\ast P$ vanishes.

The direct image $\walpha_\ast' V'$ is a vertical effective divisor in $X$, and its inverse image $\walpha'^\ast \walpha_\ast' V'$ is an effective $\Q$-divisor in $\Vf'$. Here again $\vert \walpha'^\ast \walpha_\ast' V \vert \cap P'$ is  a finite set of closed points, the direct image $\pi_{\vert \Vf \vert \ast} (\walpha'^\ast \walpha_\ast' V' \cdot \nu^\ast P')$  is a well-defined 0-dimensional $\Q$-cycle in $\Spec \OK$ which does not depend of the choice of $\nu,$ and its Arakelov degree:
$$\dega \walpha'^\ast \walpha_\ast' V' \cdot \nu^\ast P'$$
also is non-negative.

\begin{definition} With the above notation, we define:
\begin{equation}\label{defExarprime}
\Ex (\walpha: \Vf \dashrightarrow X) := \dega (R' + \walpha'^\ast \walpha'_\ast V')\cdot  \nu^\ast P . 
\end{equation}
\end{definition}

One easily checks that the invariant $\Ex (\walpha: \Vf \dashrightarrow X)$ also admits the following expressions:
\begin{align}
\Ex (\walpha: \Vf \dashrightarrow X) & =  \dega \big(\walpha'^\ast \walpha'_\ast (\nu^\ast (P)\big) - e(\alpha)\, P')\cdot \nu^\ast P \label{Excmerbis}\\
& = \dega \big(\walpha'^\ast \walpha'_\ast (\nu^\ast P) - e(\alpha) \, \nu^\ast P\big)\cdot \nu^\ast P,  \label{Excmerter}
\end{align}
which extend the equality \eqref{arithmeticoverflow} to the  situation, and are formally similar to  \eqref{equation:definition-overflowstar}.

Observe that \eqref{Excmerbis} makes clear that    $\Ex (\walpha: \Vf \dashrightarrow X)$ is non-negative,  and \eqref{Excmerter} that it is independent of the choice of the regular modification $\nu$ adapted to $\walpha$.

\subsubsection{} As in \ref{defNPg}, let us consider a family $g :=(g_\sigma)_{\sigma: K \hra \C}$ of Green functions for the points $P_\sigma$ in the Riemann surfaces $V^+_\sigma$ such that $(P, g)$ is  
an Arakelov divisor in  $\Zb^1_c(\Vfa)$, and the associated Hermitian line bundle:
 $$ \Nb_{P, g} \Vf := (N_P \Vf,  (\Vert.\Vert_{g_\sigma}^{\mathrm{cap}})_{\sigma: K \hra \C})$$
 defined by the associated capacitary metrics $\Vert.\Vert_{g_\sigma}^{\mathrm{cap}}$.
 
 The following proposition is the generalization of Proposition \ref{proposition:intersection-ov} to  meromorphic maps. 
 \begin{proposition}\label{proposition:intersection-ovMero}
With the notation above, the following equality holds:
 \begin{equation}\label{equation:on-Vmero}
\Gamma_{\alpha\ast} (P, g) \cdot \Gamma_{\alpha\ast} (P,g)=e(\alpha) \,\widehat{\deg}\,P^*\overline N_{P, g}\Vf+\mathrm{Ex}\big(\widehat \alpha : \Vf\dashrightarrow  X\big)+\sum_{\sigma: K \hra \C}\mathrm{Ex}(\alpha_\sigma, g_\sigma).
\end{equation}
\end{proposition}

When the Green functions $g_\sigma$ are the equilibrium potentials  $g_{V_\sigma, P_\sigma}$, the equality  \eqref{equation:on-Vmero} becomes the following generalization of \eqref{equation:on-V-equ}:
\begin{multline}\label{equation:on-V-equMero}
\Gamma_{\alpha\ast}(P, g_{\Vfa_\C}) \cdot \Gamma_{\alpha\ast}(P, g_{\Vfa_\C}) =e(\alpha) \,\widehat{\deg}\,P^*\overline N_{P}\Vfa+\mathrm{Ex}\big(\widehat \alpha : \Vf\dashrightarrow X\big)\\ +\sum_{\sigma: K \hra \C}\mathrm{Ex}\big(\alpha_\sigma : (V_\sigma, P_\sigma)\ra X_\sigma)\big).
\end{multline}

\begin{proof}[Proof of Proposition \ref{proposition:intersection-ovMero}] We may choose a positive integer $N$ such that the divisor $NQ$ and $N \walpha'_\ast V'$ are Cartier. Then we may define a Hermitian line bundle $\Lb$ over $X$ by:
\begin{equation}
\Lb:= \cOb(N \Gamma_{\alpha*} P, N \alpha^\an_\ast g) = \cOb(N Q + N \walpha'_\ast V', N \alpha^\an_\ast g). 
\end{equation}
We want to show that $N^{-2} \Lb \cdot \Lb$ equals the right-hand side of \eqref{equation:on-Vmero}.

 In the present context, the relations \eqref{QPR}, \eqref{Iiso}, and \eqref{IPiso} in the proof of Lemma \ref{degaQast} are replaced by:
\begin{equation*}
\walpha'^\ast (NQ) = e(\alpha)N P' + NR',
\end{equation*}
\begin{equation*}
I : \walpha'^\ast \cO_X(NQ) \lrasim \cO_{\Vf'} (e(\alpha) N P') \otimes \cO_{\Vf'} (NR'),
\end{equation*}
and:
\begin{equation*}
I_{\mid P'} : Q^\ast \cO_X(NQ) \lrasim P'^\ast \cO_{\Vf'} \big(e(\alpha) NP'\big) \otimes P'^\ast \cO_\Vf (NR'), 
\end{equation*}
and the conclusion \eqref{QastOQ} of Lemma \ref{degaQast} becomes:
\begin{equation*}
N^{-1}\,  \dega Q^\ast \cOb(NQ, N \alpha^\an_\ast g) = e(\alpha) \,\widehat{\deg}\,P'^*\overline N_{P', g}\Vf'+\dega P' \cdot R'  +\sum_{\sigma: K \hra \C} \int_{V_\sigma}g_\sigma\,\delta_{\alpha^{\mathrm{an}}_{\sigma *}(Q_\sigma)-eP_\sigma}.
\end{equation*}
Therefore:
\begin{multline}\label{QastLb}
N^{-1}\, \dega Q^\ast \Lb  = e(\alpha) \,\widehat{\deg}\,P'^*\overline N_{P', g}\Vf'+\dega P' \cdot R' + \dega Q \cdot \walpha'_\ast V' \\ +\sum_{\sigma: K \hra \C} \int_{V_\sigma}g_\sigma\,\delta_{\alpha^{\mathrm{an}}_{\sigma *}(Q_\sigma)-eP_\sigma}.
\end{multline}

Moreover, we have:
\begin{align}\label{LbLb}
N^{-2} \, \Lb \cdot \Lb   &=  N^{-1} \, \Lb \cdot (Q + \walpha'_\ast V', \alpha^\an_\ast g)  \notag \\
& = N^{-1} \, \dega Q^\ast \Lb + N^{-1} \, \dega L\cdot \walpha'_\ast V' + \int_{X(\C)} \alpha^\an_\ast g \, \omega( \alpha^\an_\ast g).
\end{align}

From \eqref{QastLb} and \eqref{LbLb}, we get:
\begin{multline}\label{LbLbbis}
N^{-2} \, \Lb \cdot \Lb = e(\alpha) \,\widehat{\deg}\,P'^*\overline N_{P', g}\Vf' +\dega P' \cdot R' + \dega Q \cdot \walpha'_\ast V' + \dega (Q + \walpha'_\ast V') \cdot \walpha'_\ast V' \\ + \sum_{\sigma: K \hra \C}\mathrm{Ex}(\alpha_\sigma, g_\sigma).
\end{multline}

Using the isomorphism:
$$P^\ast N_P \Vf \lrasim P'^\ast \nu^\ast \cO_{\Vf}(P) \lrasim P'^\ast \cO_{\Vf}(P' + V') \lrasim 
P'^\ast N_P' \Vf' \otimes \cO(P' \cdot V'),$$
we get the following equality between arithmetic degrees:
\begin{equation}\label{degaNNprime}
\dega P'^\ast \Nb_{P',g} \Vf' =  \dega P^\ast \Nb_{P,g} \Vf - \dega P' \cdot V'.
\end{equation}

Using \eqref{LbLbbis} and \eqref{degaNNprime}, we see that the validity of  \eqref{equation:on-Vmero}
follows from the following equality between intersection numbers:
\begin{multline}
\dega (R' + \walpha'^\ast \walpha'_\ast V')\cdot  \nu^\ast P =
- e(\alpha) \dega P' \cdot V' + \dega P' \cdot R' + \dega Q \cdot \walpha'_\ast V'\\ + \dega (Q + \walpha'_\ast V') \cdot \walpha'_\ast V'.
\end{multline}
This is a straighforward consequence of  \eqref{PVprime} and~\eqref{PRprime} and of the projection formula, which implies the following relations:
$$ \dega Q \cdot \widehat{\alpha}'_\ast V' = \dega \widehat{\alpha}'_\ast  P' \cdot \widehat{\alpha}'_\ast V'  =\dega P' \cdot \widehat{\alpha}'^\ast  \widehat{\alpha}'_\ast V',$$
and 
\begin{equation*}
\dega (Q + \widehat{\alpha}'_\ast V') \cdot \widehat{\alpha}'_\ast V' = \dega 
(\widehat{\alpha}'^\ast Q + \widehat{\alpha}'^\ast \widehat{\alpha}'_\ast V') \cdot  V'.
\qedhere
\end{equation*}
\end{proof}

\subsubsection{} Let $U$ be  an integral normal arithmetic surface over $\Spec \OK$, possibly non-projective.  There exists an open imbedding  of $U$ into some  integral normal projective arithmetic surface $X$  over $\Spec \OK$. One might define a  meromorphic map:
$$\alpha := (\walpha, \alpha^\an): \Vfa \dashrightarrow U$$
as a  meromorphic map:
$$\alpha' := (\walpha', \alpha'^\an): \Vfa \dashrightarrow X$$
such that the morphism $\alpha'_K: \Vf_K \ra X_K$ does not factor through $(X\setminus U)_K$. This definition may be seen to be independent of the choice of the compactification $X$ of $U$. However the self-intersection  $\Gamma_{\alpha'\ast}(P, g_{\Vfa_\C}) \cdot \Gamma_{\alpha'\ast}(P, g_{\Vfa_\C})$ may in general  depend of this choice. 

This dependence forbids us to define the self-intersection  $\Gamma_{\alpha\ast}(P, g_{\Vfa_\C}) \cdot \Gamma_{\alpha\ast}(P, g_{\Vfa_\C})$  when the range $U$ of $\alpha$ is  not projective, \eg  affine,  and the compactification $X$ is not specified.

\chapter[Pseudoconcave formal-analytic arithmetic surfaces I]{Pseudoconcave formal-analytic arithmetic surfaces I:  degree bounds, algebraicity,  and the field~\protect{$\cM(\widetilde{\mathcal{V}})$}}\label{Chapter8}

This chapter is devoted to the properties of pseudoconcave arithmetic surfaces and of the morphisms from those to arithmetic schemes. These properties constitute the main  results of this memoir, and are established by transposing in the arithmetic setting the arguments already introduced in a geometric framework in Part 1, by using the tools developed in Part~2. 

Notably we derive a bound on the degree of a morphism between arithmetic surfaces which is an arithmetic analogue of the geometric bounds in Propositions \ref{prop:degree-bound-geom} and 
\ref{degboundmero}, and where the role of the auxiliary complex analytic surface $\cV$ is played by a  \fa arithmetic surface $\Vfa$. 

We improve on  the algebraicity results concerning smooth morphisms from pseudoconcave \fa surfaces to arithmetic schemes established in \cite[Chapter 10]{Bost2020} by proving in Theorem \ref{theoremMero} and Corollary \ref{degMKf}
some finiteness results concerning the field $\cM(\Vfa)$ of meromorphic functions on a pseudoconcave \fa arithmetic surface $\Vfa$, which play the role of Theorems \ref{McV}  and \ref{cvXfin}. We conclude this chapter by discussing how these finiteness results contain as a special case 
the arithmetic holonomicity theorem of \cite{CalegariDimitrovTang21}.

In the next chapter, we shall complement these results by  some arithmetic analogues of  the finiteness results concerning 
the universal meromorphic map $\phi: \cV \dashrightarrow \cV^{\mathrm{alg}}$ and
the algebra $\cA$ introduced in Subsection \ref{universalcV}, and of the Lefschetz-Nori theorems on fundamental groups   of algebraic surfaces established in Proposition \ref{easyNori} and Theorem \ref{LefschetzNoriCNB}.

\medskip

In this chapter, we denote by $K$ a number field, and by $\OK$ its ring of integers.

\section{The invariant $D(\alpha)$. Degree bounds on morphisms between arithmetic surfaces}

\subsection{The invariant $D(\alpha)$}

\subsubsection{Definitions}

Recall from Section \ref{pseudobis} that a smooth \fa arithmetic surface $\Vfa$ over $\Spec \OK$ is called pseudoconcave when the metrized normal bundle $\Nb_P \Vfa$ satisfies the following positivity condition:
$$\dega\Nb_P \Vfa >0.$$

\begin{definition}\label{definition:D(a)} Let $\Vfa$ be a pseudoconcave \fa arithmetic surface over $\Spec \OK$. 

For every non-constant morphism over $\Spec \OK$:
$$\alpha:= (\walpha, (\alpha_\sigma)_{\sigma: K \hra \C}): \Vfa \lra X$$
from $\Vfa$ to some integral normal  surface $X$ over $\OK,$ we let:
\begin{equation}\label{Dalpha}
D\big(\alpha: \Vfa \ra X\big) := \frac{ \alpha_\ast \big(P, g_{\Vfa_\C}\big) \cdot \alpha_\ast \big(P, g_{\Vfa_\C}\big)}{\dega\Nb_P \Vfa}
\end{equation}

Similarly, for every non-constant meromorphic map over $\Spec \OK$:
$$\alpha:= (\walpha, (\alpha_\sigma)_{\sigma: K \hra \C}): \Vfa \dashrightarrow X$$
from $\Vfa$ to some integral normal projective surface $X$ over $\OK,$ we let:
\begin{equation}\label{Dalphabis}
D\big(\alpha: \Vfa \dashrightarrow X\big) := \frac{ \Gamma_{\alpha \ast} \big(P, g_{\Vfa_\C}\big) \cdot \Gamma_{\alpha \ast} \big(P, g_{\Vfa_\C}\big)}{\dega\Nb_P \Vfa}
\end{equation}
\end{definition}

\subsubsection{} According to Corollary \ref{cor: KeySelfInt} and \eqref{equation:on-V-equ}, and to Proposition
\ref{proposition:intersection-ovMero} and
\eqref{equation:on-V-equMero}, the invariants  $D\big(\alpha: \Vfa \ra X\big)$ and $D\big(\alpha: \Vfa \dashrightarrow X\big)$ admit the following expressions:
\begin{equation}\label{Dalphater}
D\big(\alpha: \Vfa \ra X\big) = e(\alpha) + \big(\dega\Nb_P \Vfa\big)^{-1} \Big( \Ex (\walpha: \Vfa \ra X) + \sum_{\sigma: K \hra \C} \Ex (\alpha_\sigma: (V_\sigma, P_\sigma) \ra  X_\sigma) \Big)
\end{equation}
and:
\begin{equation}\label{Dalphamerobis}
D\big(\alpha: \Vfa \dashrightarrow X\big) = e(\alpha) + \big(\dega\Nb_P \Vfa\big)^{-1} \Big( \Ex (\walpha: \Vfa \dashrightarrow X) + \sum_{\sigma: K \hra \C} \Ex (\alpha_\sigma: (V_\sigma, P_\sigma) \ra  X_\sigma) \Big),
\end{equation}
where, as usual, $e(\alpha)$ denotes the ramification index of $\alpha_K: \Vf_K \ra X_K$.

Together with the non-negativity of the overflow invariants $\Ex (\walpha: \Vfa \ra X)$, $\Ex (\walpha: \Vfa \dashrightarrow X)$, and $\Ex (\alpha_\sigma: (V_\sigma, P_\sigma)\ra X_\sigma)$, the relations \eqref{Dalphater} and  \eqref{Dalphamerobis} imply the following lower bounds:
\begin{equation}
D\big(\alpha: \Vfa \ra X\big) \geq  e(\alpha) \qaq D\big(\alpha: \Vfa \dashrightarrow X\big) \geq  e(\alpha).
\end{equation}

\subsubsection{Examples}\label{ExamplesD} Thanks to the expression of the Archimedean overflow in Theorem \ref{proposition:explicit} and to its applications in Subsections \ref{ExamplesOverflow} and \ref{selfintdiagonal},
 one gets explicit formulae for the invariant $D(\alpha: \Vfa \ra X)$   when $\Vfa$ is the \fa surface $\Vfa(\overline{D}(0,r), \psi)$ over $\Spec \Z$ associated to a formal series in $G_{\mathrm{for}}(\R):= \R^\ast T + T^2 \R[[T]]$, and when  $X$ is $\A^1_\Z$ or $\PP^1_\Z$.  

Consider for instance the situation in paragraph \ref{exampleP1}. Namely suppose that we are given a series $\psi$ in $G_{\mathrm{for}}(\R)$, a morphism of formal schemes:
$$\walpha: \Spf \Z[[T]] \lra \PP^1_\Z,$$
and
a  meromorphic function, defined on the open disk $D(0, R)$ of radius $R \in (0, +\infty]$:
$$\alpha^{\an}: D(0, R) \lra \PP^1(\C).$$
We also assume  that $\walpha$ and $\alpha^\an$ are not constant  and satisfy the relation:
$$ \walpha = \alpha^\an \circ \psi.$$
For every $r \in (0, R),$ we  introduce the smooth
\fa surface over $\Spec \Z$:
$$\Vfa_r := \Vfa(\overline{D}(0,r), \psi),$$
and the morphism:
$$\alpha_r := \big(\walpha, \alpha^\an_{\mid \overline{D}(0,r)^+}\big): \Vfa_r \lra \PP^1_\Z.$$

Then we have:
\begin{equation}
\label{degaNVr}
\dega \Nb_P \Vfa_r = \log ( r/\vert \psi'(0) \vert),
\end{equation}
and consequently $\Vfa_r$ is pseudoconcave if and only if:
$$r > \vert \psi'(0) \vert.$$

The invariant $D\big(\alpha_r: \Vfa_r \ra \PP^1_\Z\big)$ is defined for every $r \in (\vert \psi'(0) \vert, R)$. According to \eqref{selfintpsiP1} 
and \eqref{degaNVr}, it satisfies:
\begin{multline}\label{DalphapsiP1}
D\big(\alpha_r: \Vfa_r \ra \PP^1_\Z\big) = 2\, \big(  \log ( r/\vert \psi'(0) \vert)\big)^{-1} \big( \height (\alpha(0)) + T_{\alpha^\an} (r) \big) 
\\ - \big(  \log ( r/\vert \psi'(0) \vert)\big)^{-1} \int_{[0,1]^2}  g_{\PP^1(\C)}\big(\alpha^\an (r \,e^{2 \pi i t_1}), \alpha^\an (r\, e^{2 \pi i t_2}) \big) \, dt_1 \, dt_2.
\end{multline}
See \eqref{usualheight} and \eqref{chardef} for the definitions of the height $\height (\alpha(0))$ and of the Nevanlinna characteristic function $T_{\alpha^\an}.$ Observe also that \eqref{DalphapsiP1} implies the following upper-bound:
\begin{equation}\label{DalphapsiP1bisInt}
D\big(\alpha_r: \Vfa_r \ra \PP^1_\Z\big) \leq 2\, \big(  \log ( r/\vert \psi'(0) \vert)\big)^{-1} \big( \height (\alpha(0)) + T_{\alpha^\an} (r) \big).
\end{equation}

When $\walpha$ is a morphism from $\Spf \Z[[T]]$ to $\A^1_\Z$ (equivalently when  $\walpha$ is defined by a formal series in $\Z[[T]]$) and $\alpha^{\an}$ takes its values in $\C$, then using  \eqref{selfintpsiA} we also get the following equality:
\begin{equation}\label{DA1Int}
D\big(\alpha_r: \Vfa_r \ra \A^1_\Z\big) = 2\, \big(  \log ( r/\vert \psi'(0) \vert)\big)^{-1}
\int_{[0, 1]^2}  \log \left\vert \alpha^\an (r e^{2 \pi i t_1}) - \alpha^\an (r e^{2 \pi i t_2}) \right\vert \, dt_1 \, dt_2.
\end{equation}

\subsection{Pseudoconcave \fa surfaces and degree bounds on maps between arithmetic surfaces}

We may now establish the arithmetic counterparts of Propositions \ref{prop:degree-bound-geom} and \ref{degboundmero} and of the degree bound \eqref{degCNB} and \eqref{eq:degboundmero} that constitute the central result of this memoir.  

\begin{theorem}\label{theorem:main}
Let $\Vfa$ be a pseudoconcave formal-analytic arithmetic surface over $\Spec\OK$, and let $U$ and $V$ be two integral normal arithmetic surfaces over $\Spec \OK$. Consider a commutative diagram of morphisms over $\Spec \OK$: 
\[
\xymatrix{
& V\ar[d]^{f}\\
\Vfa\ar[r]^{\alpha}\ar[ur]^{\beta} & U.
}
\]

If $\alpha$ is non-constant, then $\beta$ also is non-constant, $f$ is dominant and generically finite, and its degree $\deg f$  satisfies the upper bound:
\begin{equation}\label{equation:degree-bound}
\deg f \leq \frac{D(\alpha: \Vfa \ra U)}{D(\beta: \Vfa \ra V)}.
\end{equation}
\end{theorem}

Recall that by definition $\deg f$ is the degree of the finite extension of fields:
$$f^\ast:  \kappa(X) = K(X_K) \, \hlra \, \kappa(Y) = K(Y_K).$$
Observe also that \eqref{equation:degree-bound} immediately implies the following upper bound on $\deg f$:
\begin{equation}\label{equation:degree-boundBis}
\deg f \leq D(\alpha: \Vfa \ra U) / e(\beta) \leq D(\alpha: \Vfa \ra U).
\end{equation}

\begin{proof} When $\alpha$ is non-constant, then $\beta$ also is clearly non-constant, and $f_K$ is non-constant. This implies that $f$ is dominant and generically finite.

If we let:
$$\Ab := \alpha_\ast (P, g_{\Vfa_\C})
\qaq
\Bb := \beta_\ast (P, g_{\Vfa_\C}),$$
then we have:
$$\Ab = f_\ast \Bb$$
by Proposition \ref{prop:directVfaclass}. Moreover, according to the definition and to the positivity of $D(\alpha: \Vfa \ra U)$, we have:
\begin{align*}
\Ab \cdot \Ab & =  \alpha_\ast (P, g_{\Vfa_\C}) \cdot  \alpha_\ast (P, g_{\Vfa_\C}) \\
& = D(\alpha: \Vfa \ra U) \; \dega \Nb_P \Vfa > 0.
\end{align*}
Similarly we have:
\begin{align*}
\Bb \cdot \Bb & =  \beta_\ast (P, g_{\Vfa_\C}) \cdot  \beta_\ast (P, g_{\Vfa_\C}) \\
& = D(\beta: \Vfa \ra V) \; \dega \Nb_P \Vfa > 0.
\end{align*}

Therefore the upper bound \eqref{equation:degree-bound} follows from the upper bound \eqref{equation:inequality-degreeAr}:
$$\deg f \leq \frac{\Ab \cdot \Ab}{\Bb \cdot \Bb}
$$ established in Theorem~\ref{theorem:main-Arakelov}.
\end{proof}

Theorem \ref{theorem:main} admits the following variant concerning meromorphic maps from pseudoconcave \fa  surfaces to projective arithmetic surfaces:

\begin{theorem}\label{theorem:mainMero} Let $\Vfa$ be a pseudoconcave formal-analytic arithmetic surface over $\Spec\OK$, and let $X$ and $Y$ be two integral normal projective arithmetic surfaces over $\Spec \OK$. Consider a commutative diagram of morphisms over $\Spec \OK$:
\begin{equation}\label{diagram:setupprimebis}
\begin{gathered}
\xymatrix{
& Y\ar@{-->}[d]^{f}\\
\cV\ar@{-->}[r]^{\alpha}\ar@{-->}[ur]^{\beta} & X,
}
\end{gathered}
\end{equation}
where $X$ and $Y$ are integral projective arithmetic surfaces over $\Spec \OK$, with $X$ normal, where $\alpha$ and $\beta$ are meromorphic maps over $\Spec \OK$, and where $f$ is a  rational map over of $\Spec\OK$. 

If $\alpha$ is non-constant, then $\beta$ also is non-constant, $f$ is dominant and generically finite, and its degree $\deg f$  satisfies the upper bound:
\begin{equation}\label{equation:degree-boundMero}
\deg f \leq D(\alpha: \Vfa \dashrightarrow X)/e(\beta).
\end{equation}
\end{theorem}

Recall that the data of a rational map $f:Y\dashrightarrow X$ over $\Spec \OK$ is equivalent to the data of a $K$-morphism of curves: $f_K: Y_K \ra X_K$ and that the degree of $f$ coincides with the degree $[K(Y_K): f_K^\ast K(X_K)]$ of this morphism. 
The commutativity of the diagram \eqref{diagram:setupprimebis} may be defined by the commutativity of the following diagram of (formal) schemes over $K$:
\[
\xymatrix{
& Y_K\ar[d]^{f_K}\\
\Vf_K\ar[r]^{\walpha_K}\ar[ur]^{\widehat{\beta}_K} & X_K.
}
\]

\begin{proof} After possibly replacing $Y$ by the normalization of the closure of the graph of $f$, we may assume that $f$ is a not only a rational map, but an actual morphism from $Y$ to $X$. Indeed this reduction does not modify $f_K$ and its degree, nor the  meromorphic character of $\widehat{\beta}$, nor the ramification index $e(\beta)$ of $\beta_K: \Vf_K \ra Y_K$.\footnote{However it may modify $D(\beta: \Vfa \dashrightarrow V)$.}

Then we may define:
$$\Ab := \Gamma_{\alpha\ast} (P, g_{\Vfa_\C})
\qaq
\Bb := \Gamma_{\beta\ast} (P, g_{\Vfa_\C}).$$
We still have:
$$\Ab = f_\ast \Bb$$
by \eqref{falphaast}. Moreover:
$$\Ab \cdot \Ab  = D(\alpha: \Vfa \dashrightarrow X) \; \dega \Nb_P \Vfa$$
and
$$\Bb \cdot \Bb    = D(\beta: \Vfa \dashrightarrow Y) \; \dega \Nb_P \Vfa \geq e(\beta) \, \dega \Nb_P \Vfa.$$

Therefore the upper bound \eqref{equation:degree-boundMero} again follows from \eqref{equation:inequality-degreeAr}.
\end{proof}

\section[The algebraicity of maps from pseudoconcave f.-a.\! arithmetic surfaces]{Algebraicity of maps from pseudoconcave f.-a.\! arithmetic surfaces to arithmetic schemes}

\subsection{Upper bounds on $h^{0}_{\theta}(\Vfa, \Mb^{\otimes D})$ and algebraicity}

\subsubsection{} The following algebraicity theorem is an arithmetic analogue of the algebraicity results  established in a geometric framework in Proposition \ref{algebraicityPseudoconcaveAnalytic} and Proposition \ref{algebraicityPseudoconcaveAnalyticD}. 

\begin{theorem}\label{pseudoconcavefaalgebraic} Let $\Vfa$ be a pseudoconcave smooth \fa arithmetic surface. For every quasi-projective (resp. projective) $\OK$-scheme and any morphism
$\alpha: \Vfa \lra X$
(resp. any meromorphic map:
$\alpha: \Vfa \dashrightarrow X$)
over $\Spec \OK,$ the image of $\alpha$ is algebraic.
\end{theorem}

When $\alpha$ is a morphism to a quasi-projective $\OK$-scheme, that is the main result in \cite[Chapter 10]{Bost2020}; see \emph{loc. cit.} 
Theorem 10.8.1. In this section, we explain how the proof in \cite{Bost2020}  may be extended to cover the case of a meromorphic map $\alpha$ from $\Vfa$ to a projective $\OK$-scheme. We also derive an alternative proof of the degree bound in \eqref{equation:degree-boundBis}:
$$\deg f \leq D(\alpha: \Vfa \ra U)$$
 from the results of \cite{Bost2020}, analogue to the alternative proof of  the geometric degree bound \eqref{equation:degree-bound-geomBis} given in Subsection~\ref{pseudoconcaveEasy}.

\subsubsection{} As in the geometric situations studied in Subsections \ref{SectionsAlgebraicity} and \ref{CNB}, the algebraicity theorem \ref{pseudoconcavefaalgebraic} is a consequence of a finiteness result concerning the spaces of sections attached  some Hermitian line bundle $\Mb$ over a pseudoconcave \fa arithmetic surface $\Vfa$ and to its tensor powers $M^{\otimes D}$, and of the asymptotic behavior of the ``dimensions" $\hot\big( \Vfa, \mu; \Mb^{\otimes D}\big)$ of these spaces.

Our main tool will be the following theorem, which is  basically established in \cite[Sections 10.5 and 10.7]{Bost2020}, and constitutes an arithmetic counterpart of
Propositions \ref{finitepseudoconcave} and \ref{prop:boundpseudoconcaveD}.

\begin{theorem}\label{theorem:HS-estimates}  Let $\Vfa: = \big(\Vf, (V_{\sigma}, P_\sigma, \iota_{\sigma})_{\sigma: K\hra\C}\big)$ be a pseudoconcave smooth \fa arithmetic surface, and let  $\mu$ be a $\mathcal C^{\infty}$ positive volume form on $V_{\C}$ invariant under complex conjugation.

(1)  For  every Hermitian vector bundle $\Eb$ over $\Vfa$, the pro-Hermitian vector bundle $\pi^{L^2}_{(\Vfa, \mu)\ast} {\overline E}$ is $\theta$-finite, and we may therefore define:
$$h^{0}_{\theta, L^{2}}(\Vfa, \mu; \Eb):= \hot \big( \pi^{L^2}_{(\Vfa, \mu)\ast}{\overline E} \big).$$ 

(2)  For every Hermitian line bundle $\Mb$ on $\Vfa$, when $D\in\mathbb N$ goes to infinity, we have:
\begin{equation}\label{equation:upper-bound}
h^{0}_{\theta, L^{2}}(\Vfa, \mu; \Mb^{\otimes D})=O(D^{2}).
\end{equation}
More precisely, when $\widehat\deg\,P^{*}\Mb<0$, we have:
\begin{equation}\label{equation:negativebis}
\lim_{D\ra+\infty}h^{0}_{\theta, L^{2}}(\Vfa, \mu; \Mb^{\otimes D})=0,
\end{equation}
and  in general:
\begin{equation}\label{equation:positive}
\limsup_{D\ra+\infty}D^{-2} \, h^{0}_{\theta, L^{2}}(\Vfa, \mu; \Mb^{\otimes D})\leq \frac{1}{2}\frac{\big(\Mb \cdot (P, g_{\Vfa_\C})\big)^{2}}{\widehat\deg\, \overline N_{P}\Vfa}.
\end{equation}
\end{theorem}

In the right-hand side of \eqref{equation:positive}, $P$ denotes the canonical section of the structure morphism $\pi_{\Vf}: \Vf \ra \Spec \OK$, $(P, g_{\Vfa_\C})$ is the Arakelov divisor on $\Vfa$ defined   in Subsection \ref{PgVNb} by the equilibrium potentials $(g_{V_\sigma, P_\sigma})_{\sigma:K \hra \C}$ introduced in \ref{CompactBoundEqu}, and $\Nb_P \Vfa$ the metrized normal bundle of $P$ in $\Vfa$ defined by  \eqref{defNPVfa}. The arithmetic intersection number $\Mb \cdot (P, g_{\Vfa})$ has been defined in Subsection \ref{ArIntVfa}; explicitly, we have:
$$\overline M\cdot (P, g_{\Vfa}):=\dega P^{*}\overline M + \int_{V_{\C}}g_{\Vfa}\, c_{1}(\overline M)=\dega P^{*}\overline M + \sum_{\sigma : K\hra\C}\int_{V_{\sigma}}g_{V_{\sigma}, P_{\sigma}}\, c_{1}(\Mb_{\sigma}).$$

Theorem \ref{theorem:HS-estimates} follows from Theorem 10.7.1 in \cite{Bost2020} and from its proof. Indeed, the limit estimates \eqref{equation:negativebis} and \eqref{equation:positive} follow from the proof of Theorem 10.7.1 in \cite[10.7.4 pp. 286 -- 287]{Bost2020}, and from the bound on the constant $C_{\eta}$ in Lemma 10.7.2 provided by the Schwarz lemma on a compact Riemann surface with boundary in \cite[10.5.5]{Bost2020}.

To handle the case of a meromorphic map $\alpha: \Vfa \dashrightarrow X$ in Theorem \ref{pseudoconcavefaalgebraic} only, we will use the following partial generalization\footnote{The proof of  Theorem \ref{ODtwomodif} will rely on \cite[Theorem 10.7.1]{Bost2020}, and not on the more precise version, including the estimate \eqref{equation:positive}, stated in  Theorem  \ref{theorem:HS-estimates}.} of Theorem   \ref{theorem:HS-estimates}.

\begin{theorem}\label{ODtwomodif} Let $\Vfa: = \big(\Vf, (V_{\sigma}, P_\sigma, \iota_{\sigma})_{\sigma: K\hra\C}\big)$ be a pseudoconcave smooth \fa arithmetic surface, let  $\mu$ be a $\mathcal C^{\infty}$ positive volume form on $V_{\C}$ invariant under complex conjugation, and let $$\nu: \Vf' \lra \Vf$$ be a regular modification. 

 (1)  For  every Hermitian vector bundle $\Eb'$ over $\Vfa' := \big(\Vf', (V_{\sigma}, P_\sigma, \iota_{\sigma})_{\sigma: K\hra\C}\big)$, the pro-Hermitian vector bundle $\pi^{L^2}_{(\Vfa', \mu)\ast}{\overline E}'$ is $\theta$-finite, and we may therefore define:
$$h^{0}_{\theta, L^{2}}(\Vfa', \mu; \Eb'):= \hot \big( \pi^{L^2}_{(\Vfa', \mu)\ast} {\overline E}' \big).$$

(2)  For every Hermitian line bundle $\Mb'$ on $\Vfa'$, when $D\in\mathbb N$ goes to infinity, we have:
\begin{equation}\label{equation:upper-boundprime}
h^{0}_{\theta, L^{2}}(\Vfa', \mu; \Mb'^{\otimes D})=O(D^{2}).
\end{equation}

\end{theorem}

\begin{proof}

 Let $\Eb':= (\widehat E', (E'_{\sigma}, \phi_{\sigma}, \Vert.\Vert_{\sigma})_{\sigma : K\hra\C})
$ be a Hermitian vector bundle over $\Vfa'$. We may define the Hermitian torsion free coherent sheaf over $\Vfa$:
$$\nu_\ast \Eb' :=  \big(\nu_\ast\widehat E', (E'_{\sigma}, \phi_{\sigma}, \Vert.\Vert_{\sigma})_{\sigma : K\hra\C}\big),$$
and its bidual:
$$\Eb := (\nu_\ast \Eb')^{\vee \vee}:=  \big((\nu_\ast\widehat E)^{\vee \vee}, (E'_{\sigma}, \phi_{\sigma}, \Vert.\Vert_{\sigma})_{\sigma : K\hra\C}\big).$$
As already observed in \ref{pullingbackprime}, the direct image $\pi^{L^2}_{(\Vfa', \mu)\ast}{\overline E'}$ may be identified with 
$\pi^{L^2}_{(\Vfa, \mu)\ast} \nu_\ast \Eb'$. According to Proposition \ref{CompPistarEbid}, the $\theta$-finiteness of $\pi^{L^2}_{(\Vfa, \mu)\ast} \nu_\ast \Eb'$ follows from the one of 
$\pi^{L^2}_{(\Vfa, \mu)\ast}  \Eb$, itself established in \cite[Theorem 10.7.1]{Bost2020}, and stated in Theorem \ref{theorem:HS-estimates}, (1), above. Moreover the following inequality holds:
\begin{equation}\label{hotEEprime}
h^{0}_{\theta, L^{2}}(\Vfa', \mu; \Eb')
=
h^{0}_{\theta, L^{2}}(\Vfa, \mu; \nu_\ast \Eb')
\leq
h^{0}_{\theta, L^{2}}(\Vfa, \mu; \Eb).
\end{equation}

Let $ \Mb'$ be a Hermitian line bundle over $\Vfa'$. According to Proposition \ref{lineprime}, it may be written $\nu^\ast  \Mb (W)$ for some Hermitian line bundle $ \Mb$ over $\Vfa$ and some divisor $W$ in $\Vfa'$ supported by the union $\vert \Vf' \vert_{\mathrm{vert}}$ of the vertical components of $\vert \Vf \vert$. For every integer $D$, $ \Mb'^{\otimes D}$ is isomorphic to $\nu^\ast  \Mb^{\otimes D} (DW)$, and the bidual of $\nu_\ast  \Mb'^{\otimes D}$ is isomorphic to $ \Mb^{\otimes D}$. 

Applied to $\Eb' =  \Mb'^{\otimes D},$ the inequality \eqref{hotEEprime} reads:
\begin{equation}\label{hotLLprime}
h^{0}_{\theta, L^{2}}(\Vfa', \mu;  \Mb'^{\otimes D})
\leq
h^{0}_{\theta, L^{2}}(\Vfa, \mu;  \Mb^{\otimes D}).
\end{equation}
As proved in \cite[Theorem 10.7.1]{Bost2020}, and stated in Theorem \ref{theorem:HS-estimates}, (2), when $D\in \N$ goes to infinity, we have:
$$h^{0}_{\theta, L^{2}}(\Vfa, \mu;  \Mb^{\otimes D})= O(D^2).$$
Together with \eqref{hotLLprime}, this establishes \eqref{equation:upper-boundprime}.
\end{proof}

\subsubsection{}\label{algAr} 
Let us complete the proof of Theorem \ref{pseudoconcavefaalgebraic} by establishing the algebraicity of the image of a meromorphic map:
$$\alpha:= (\walpha, (\alpha_\sigma)_{\sigma: K \hra \C}):   \Vfa \dashrightarrow X$$
over $\Spec \OK$, when $\Vfa$ and $X$ are respectively a pseudoconcave smooth \fa surface and a projective scheme over $\Spec \OK$. The following argument is a minor variant of the proof of \cite[Theorem 10.8.1]{Bost2020}, but we include it for the sake of completeness.

With the above notation, we want to prove that the Zariski closure $\overline{\im \walpha}$ of  the image of $\walpha$ satisfies:
$$\dim \overline{\im \walpha} \leq 2.$$

To achieve this, we may assume that $X$ coincides with $\overline{\mathrm{im}\widehat \alpha},$ and therefore is an integral projective flat scheme over $\Spec \OK$.  

Moreover we may chose a Hermitian line bundle $\Lb$ over $X$ such that, if we let, for every integer~$D$:\footnote{See paragraph \ref{subsubsection:vb} for the definition of the Hermitian vector bundle $\pi^J_{X*}\overline L^{\otimes D}$.} 
$$h^{0}_{\theta, J}(X, \overline L^{\otimes D}):=h^{0}_{\theta}(\pi^J_{X*}\overline L^{\otimes D}),$$
then the  following condition is satisfied:
\begin{equation}\label{equation:lower-bound}
\liminf_{D\ra+\infty}D^{-\dim X}h^{0}_{\theta, J}(X, \overline L^{\otimes D}) >0.
\end{equation}
See \cite[10.3]{Bost2020} for an elementary construction of a Hermitian line bundle $\Lb$ satisfying \eqref{equation:lower-bound}. Actually,  as shown in \cite[Theorem 10.3.2]{Bost2020}, any Hermitian line bundle $\overline L$ over $X$ that is arithmetically ample in the sense of Zhang will do.

We may choose a regular modification:  
$$\nu: \Vf' \lra \Vf$$
adapted to the meromorphic map $\walpha$, and denote by:
$$\walpha': \Vf' \lra X$$
the morphism of (formal) schemes defining $\walpha$.

Finally we may choose a positive smooth volume form $\mu$ on $V_{\C}$, invariant under complex conjugation, such that, for every complex embedding $\sigma$ of $K$, we have $\mu(V_{\sigma})\leq 1.$

As observed in \ref{subsubsection:vb}, Proposition \ref{proposition:pull-back-norm}, and in \ref{pullingbackprime}, pulling back sections of $\overline L^{\otimes D}$ on $X$ along the morphism $\alpha':= \big(\walpha', (\alpha_\sigma)_{\sigma: H \hra \C}\big)$ defines a morphism:
$$\eta_{D} : \pi^J_{X*}\overline L^{\otimes D}\lra \pi^{L^{2}}_{(\Vfa', \nu) *} \alpha'^{*}\overline L^{\otimes D}$$
of (pro-)Hermitian vector bundles over $\Spec\OK$ with Archimedean norms bounded above by $1$. Moreover, since $X$ is the Zariski closure of the image of $\walpha'$,   the morphisms $\eta_{D}$ are injective, and consequently the following estimates hold:
\begin{equation}\label{equation:estimate-inj}
h^{0}_{\theta, J}(X, \overline L^{\otimes D})\leq h^{0}_{\theta, L^{2}}(\Vfa, \mu ; \alpha'^{*}\overline L^{\otimes D}).
\end{equation}
The asymptotic bound \eqref{equation:upper-boundprime} applied to $\overline M=\alpha'^{*}\overline L^{\otimes D}$ shows that the right-hand side of \eqref{equation:estimate-inj} is $O(D^{2})$ when $D$ goes to $+\infty$. Together with \eqref{equation:lower-bound}, this implies:
$$\dim \overline{\mathrm{im}\,  \walpha'}=\dim X\leq 2,$$
and finishes the proof.

\subsection{An alternative proof of the estimate $\deg f \leq D(\alpha: \Vfa \ra U)$}

As already mentioned, it is possible to establish the inequality 
$$\deg f \leq D(\alpha: \Vfa \ra U)$$
in Theorem \ref{theorem:main} by a proof similar to the one in Subsection~\ref{pseudoconcaveEasy}.
The reader may compare the argument below with the argument in \cite[Sections 4 and 5]{arithmeticbertini}.

\subsubsection{Nef and big Hermitian line bundles on projective arithmetic surfaces} 
Recall that a Hermitian line bundle $\overline L=(L, \Vert.\Vert)$, defined by a $\mathcal C^{\bD}$ metric $\Vert.\Vert$, over an integral normal projective arithmetic surface $X$ is \emph{nef} when, for every effective\footnote{An Arakelov divisor $(Z,g)$ is effective when the divisor $Z$ is effective and the Green function is non-negative.} Arakelov divisor $(Z, g)$ on $X$, 
we have:
$$\overline L.(Z, g)\geq 0.$$
A Hermitian line bundle $\overline L$ is nef if and only if the two following conditions hold:
\begin{enumerate}[(i)]
\item for every closed integral one-dimensional subscheme $C$ of $X$:
$$h_{\overline L}(C):=\widehat\deg\, \overline L_{|C}\geq 0;$$
\item the measure $c_{1}(\overline L)$ on $X(\C)$ is semi-positive.
\end{enumerate}

A nef Hermitian line bundle is \emph{big} if, additionally:
$$\overline L \cdot \overline L>0.$$

Using that any divisor on $X$ is $\Q$-Cartier, we may define similarly nef,\footnote{An effective Arakelov divisor $(Z, g)$ is nef if and only if $\omega(g) \geq 0$ and $\dega (\cOb(Z,g) \mid C) \geq 0$ for every component $C$ of $Z$; see \cite[Proposition 6.9]{Bost99}.} and nef and big, Arakelov $\Q$-divisors on $X$. The following proposition is a straightforward consequence of 
the basic properties of the Arakelov intersection pairing.

\begin{proposition}\label{pull-backnefbig}. Let $f: X' \ra X$ a dominant (hence surjective) morphism between two integral normal projective arithmetic surfaces. If a Hermitian line bundle $\Lb$ on $X$ is nef (resp. nef and big), then its pull-back $f^\ast \Lb$ is a nef (resp. nef and big) Hermitian line bundle on $X'$.
\end{proposition}

The following ``arithmetic Hilbert-Samuel formula" is a direct consequence of \cite[Theorem 1.4]{Zhang95} by the arguments in the proof of \cite[Theorem 10.3.2]{Bost2020}.

\begin{theorem}\label{theorem:HSA}
Let $\overline L$ be a big and nef Hermitian line bundle over an integral, normal projective arithmetic surface $X$ over $\OK$. When the integer $D$ goes to $+\infty,$ we have:
\begin{equation}\label{HS}
h^{0}_{\theta, J}(X, \overline L^{\otimes D})= 
\Lb \cdot \Lb  \; D^2/2+o(D^{2}).
\end{equation}
\end{theorem}

\subsubsection{} Part (1) in the following proposition is an arithmetic analogue of Corollary \ref{corollary:ineq-deg}.
Its proof will be a variation on the algebraicity proof in \ref{algAr} above, where instead of the crude estimates 
\eqref{equation:upper-bound} and \eqref{equation:lower-bound}, we shall use \eqref{equation:positive} and \eqref{HS}.

\begin{proposition}\label{proposition:ineq-bignef}
Let $\Vfa$ be a pseudoconcave smooth formal-analytic arithmetic surface  over $\Spec \OK$ and let $$\alpha : \Vfa\ra X$$ be a nonconstant $\OK$-morphism from $\Vfa$ to an integral normal projective arithmetic surface $X$ over $\Spec \OK$.

(1)  For every Hermitian line bundle $\Lb$ over $X$ that is nef and big, the following inequality holds: 
\begin{equation}\label{equation:bound-self-int}
\Lb \cdot \Lb\leq \frac{\big (\alpha^{*}\Lb \cdot (P, g_{\Vfa_\C})\big)^{2}}{\widehat\deg\, \overline N_{P}\Vfa}.
\end{equation}

(2)  The Arakelov divisor on $X$:
$$\alpha_{*}\big(P, g_{\Vfa_\C}\big):=\big(\widehat\alpha_{*}(P), \alpha^{\mathrm{an}}_{*}g_{\Vfa_\C}\big)$$
 is nef and big.
\end{proposition}

\begin{proof}[Proof of Proposition \ref{proposition:ineq-bignef}]
(1) We choose a $\mathcal C^{\infty}$ positive volume form $\mu$ on $V_{\C}$, invariant under complex conjugation, such that, for every complex embedding $\sigma$ of $K$, $\mu(V_{\sigma})\leq 1.$ 

As in paragraph \ref{algAr},  
for every nonnegative integer $D$, pulling back sections of $L^{\otimes D}$ on $X$ by $\alpha$ defines a morphism of (pro-)Hermitian vector bundles over $\Spec\OK$ with Archimedean norms bounded above by $1$:
$$\eta_{D} : \pi^J_{X*}\overline L^{\otimes D}\lra \pi^{L^{2}}_{(\Vfa, \nu) *} \alpha^{*}\overline L^{\otimes D}.$$
Since $\alpha$ is not constant, the Zariski closure $\overline{\mathrm{im}\walpha}$ of the image of $\walpha$ in $X$ is an arithmetic surface, and therefore is $X$ itself. As a consequence, the morphisms $\eta_{D}$ are injective, so that the following inequality holds:
$$h^{0}_{\theta, J}(X, \Lb^{\otimes D})\leq h^{0}_{\theta, L^{2}}(\Vfa, \mu ; \alpha^{*}\Lb^{\otimes D}).$$

As a consequence:
\begin{equation}\label{equation:ineq-h0}
\limsup_{D\ra+\infty} D^{-2} h^{0}_{\theta, J}(X, \Lb^{\otimes D})\leq \limsup_{D\ra+\infty} D^{-2}h^{0}_{\theta, L^{2}}(\Vfa, \mu ; \alpha^{*}\Lb^{\otimes D}).
\end{equation}
The inequality \eqref{equation:bound-self-int} now  follows from the arithmetic Hilbert-Samuel formula 
\eqref{HS} in  Theorem \ref{theorem:HSA} and from the upper bound on the right-hand side of \eqref{equation:ineq-h0} provided by Theorem \ref{theorem:HS-estimates}.

(2) The Arakelov divisor $\alpha_{*}(P, g_{\Vfa_\C})$ is  effective. Moreover, for every embedding $\sigma: K \hra \C,$ the measure:
$$\omega(\alpha^\an_\ast g_{\Vfa_\C})_{\mid X_\sigma(\C)} = \omega(\alpha_{\sigma \ast} g_{V_\sigma, P_\sigma}) = \alpha_{\sigma \ast}\mu_{V_\sigma, P_\sigma}$$
is positive. 

To prove that $\alpha_{*}(P, g_{\Vfa_\C})$ is nef,  it is enough to prove that the Arakelov degree along $\widehat\alpha(P)$ of the Hermitian $\Q$-line bundle attached to $\alpha_{*}(P, g_{\Vfa_\C})$ is nonnegative, since  $\widehat\alpha(P)$ is the unique component of the support of the cycle underlying $\alpha_{*}(P, g_{\Vfa})$. This follows from Lemma \ref{degaQast}, which shows that this Arakelov degree is $\geq e(\alpha) \,\widehat{\deg}\,P^*\overline N_{P, g}\Vf$, and therefore positive. 

Corollary \ref{cor:inequality-ov} shows that the self-intersection of $\alpha_{*}(P, g_{\Vfa_\C})$ is positive, and  therefore establishes that this Arakelov divsior  is big. 
\end{proof}

\subsubsection{} In this paragraph, we return to the notation of Theorem \ref{theorem:main} and we derive the bound:
\begin{equation*}
\deg f \leq D(\alpha: \Vfa \ra U)
\end{equation*}
from Proposition \ref{proposition:ineq-bignef}.

To achieve this, we may replace $U$ (resp. $V$) by the normalization of the closure of $U$ (resp. of the graph of $f$) in some ambient projective spaces, and consequently assume that $U$ and $V$ are projective over $\Spec \OK$.

Let $\Lb$ be a nef and big Hermitian line bundle on $X$. The pull-back $f^{*}\Lb$ of $\Lb$ on $Y$ is nef and big, and we may apply Proposition \ref{proposition:ineq-bignef} (1) to the morphism $\beta: \Vfa \ra Y$ and to the Hermitian line bundle $f^{*}\Lb$ on $Y$. This establishes the estimate:
\begin{equation}\label{equation:first-ineq}
f^{*}\Lb \cdot f^{*}\Lb\leq \frac{\big(\beta^{*}f^{*}\Lb \cdot (P, g_{\Vfa})\big)^{2}}{\widehat\deg\, \overline N_{P}\Vfa}.
\end{equation}
Using the equality $\alpha=f\circ\beta$ and the projection formula of Proposition \ref{proposition:projection}, \eqref{equation:first-ineq} becomes:
\begin{equation}\label{equation:second-ineq}
(\deg f)\; \Lb \cdot \Lb\leq \frac{\big(\Lb \cdot\alpha_{*}(P, g_{\Vfa})\big)^{2}}{\widehat\deg\, \overline N_{P}\Vfa}.
\end{equation}

By homogeneity, \eqref{equation:second-ineq} holds when $\Lb$ is an arbitrary nef and big $\Q$-Hermitian line bundle on $X$. In particular we may apply it to the $\Q$-Hermitian line bundle $\big(\mathcal O(\alpha(P)), \Vert.\Vert_{\alpha_{*}^{\mathrm{an}}g_{\Vfa}}\big)$ associated to the Arakelov divisor $\alpha_{*}(P, g_{\Vfa})$ on $X$, which is nef and big by  Proposition \ref{proposition:ineq-bignef} (2). In this situation, 
$$\Lb \cdot \Lb = \Lb\cdot \alpha_{*}(P, g_{\Vfa}) = \alpha_{*}(P, g_{\Vfa}) \cdot \alpha_{*}(P, g_{\Vfa}),$$ and  \eqref{equation:second-ineq} becomes:
\begin{equation*}\deg f\leq \frac{\alpha_{*}(P, g_{\Vfa}) \cdot \alpha_{*}(P, g_{\Vfa})}{\widehat\deg\, \overline N_{P\Vfa}} =: D(\alpha: \Vfa \ra U).  
\end{equation*}

\section[Meromorphic functions on a pseudoconcave f.-a.\! arithmetic surface]{The field of  meromorphic functions on a pseudoconcave f.-a.\! arithmetic surface}\label{SectMeropseudocon}
\subsection{A finiteness theorem}

A simple consequence of the algebraicity assertion in Theorem \ref{pseudoconcavefaalgebraic} and of the degree bound in  Theorem \ref{theorem:mainMero} is the following finiteness property concerning the field of  meromorphic functions on a pseudoconcave \fa arithmetic surface.\footnote{Conversely the fact that  the transcendence degree over $K$ of this field is at most one is readily seen to imply Theorem \ref{pseudoconcavefaalgebraic}.} 

\begin{theorem}\label{theoremMero} Let $\Vfa$ be a pseudoconcave smooth \fa arithmetic surface.  The following alternative holds: either $\cM(\Vfa) = K,$ or $\cM(\Vfa)$ is an extension of finite type and of transcendence degree one of $K$. 

Moreover, for every integral, normal, projective arithmetic surface $X$ over $\Spec \OK$, and every non-constant meromorphic map:
$$\alpha:= (\walpha, (\alpha_\sigma)_{\sigma: K \hra \C}) : \Vfa \dashrightarrow X$$
over $\Spec \OK$, then the degree of the field extension:
$$\alpha^\ast:  K(X_K) \hlra \cM(\Vfa), \quad \phi \longmapsto \big(\walpha^\ast \phi, (\alpha_\sigma^\ast \phi_\sigma)_{\sigma: K \hra \C}\big)$$
is finite and satisfies:
\begin{equation}\label{MvfKalpha}
[\cM(\Vfa): \alpha^\ast K(X_K) ] \leq D(\alpha: \Vfa \dashrightarrow X).
\end{equation}
 \end{theorem}
 
 In the special case when $X = \PP^1_\OK,$ the second part of Theorem \ref{theoremMero} reads as follows:
 
   \begin{corollary}\label{degMKf}  With the notation of Theorem \ref{theoremMero}, for every $f \in \cM(\Vfa) \setminus K,$ $\cM(\Vfa)$ is a finite extension of the purely transcendental extension $K(f)$ of $K$, and its degree satisfies the following upper bound:
  \begin{equation}\label{MvfKf}
[\cM(\Vfa): K(f)] \leq D(f: \Vfa \dashrightarrow \PP^1_\OK). 
\end{equation}
\end{corollary}

\begin{proof}[Proof of Theorem \ref{theoremMero}] (1) We first prove the second part of the theorem. To achieve this, it is enough to show that, for any element $g$ of $\cM(\Vfa)$, the subfield $\alpha^\ast K(X_K) (g)$ of $\cM(\Vfa)$  is a finite extension of $\alpha^\ast K(X_K)$ and satisfies:
\begin{equation}\label{KXg}
[\alpha^\ast K(X_K) (g) : \alpha^\ast K(X_K) ] \leq D(\alpha: \Vfa \dashrightarrow X).
\end{equation}

To prove this, we identify $g$ to a meromorphic map:
$$g := \big(\widehat{g}, (g_\sigma)_{\sigma: K \hra \C}\big): \Vfa \dashrightarrow \PP^1_\OK,$$
and we consider the meromorphic map defined as the product of $\alpha$ and $g$:
$$\beta := \big((\walpha, \widehat{g}), (\alpha_\sigma, g_\sigma)_{\sigma: K \hra \C}\big): \Vfa \dashrightarrow X \times_{\OK} \PP^1_\OK.$$

According to the algebraicity theorem \ref{pseudoconcavefaalgebraic}, $\beta$ factors through a closed integral arithmetic surface $Y$ in $X \times_{\OK} \PP^1_\OK$, and the projection :
$$\pr_1: Y \lra X$$
is dominant, hence surjective and generically finite, since $\alpha$ is non-constant. By construction, $K(Y_K)$ may be identified, by means of $\beta^\ast$, with $\alpha^\ast K(X_K) (g)$, which is therefore a finite extension of $\alpha^\ast K(X_K) \simeq K(X_K)$.

Moreover the meromorphic map $\beta$ is readily seen to lift to the normalization $\nu: \widetilde{Y} \ra Y$ of $Y$.  This defines a meromorphic map:
$$\widetilde{\beta}: \Vfa \dashrightarrow \widetilde{Y}$$
which fits into a commutative diagram:
\begin{equation}\label{VfaYtildeX}
\begin{gathered}
\xymatrix{
& \widetilde{Y}\ar[d]^{\pr_1 \circ \nu}\\
\cV\ar@{-->}[r]^{\alpha}\ar@{-->}[ur]^{\widetilde{\beta}} & X.
}
\end{gathered}
\end{equation}

By construction, we have:
$$K(\widetilde{Y}_K) = K(Y_K) \simeq \alpha^\ast K(X_K) (g),$$
and:
$$\deg \pr_1 \circ \nu = [\alpha^\ast K(X_K) (g) : K(X_K)].$$
Therefore the estimate \eqref{KXg} follows from Theorem \ref{theorem:mainMero} applied to the diagram \eqref{VfaYtildeX}.

(2) To complete the proof, observe that an element $f$ of $\cM(\Vfa)\setminus K$ is defined by an element $\widehat{f}_K$ of $\mathrm{Frac}\, \cO(\Vf_K) \simeq K((T))$ not in $K$, and therefore is transcendental over $K$. As already observed in Corollary \ref{degMKf}, part (1) of this proof applied to $f$ seen as a meromorphic map from $\Vfa$ to $\PP^1_\OK$ establishes that $\cM(\Vfa)$ is a finite extension of $K(f)$.
\end{proof}
 
 \begin{corollary} With the notation of Theorem \ref{theoremMero}, if the map $\walpha$ defines an isomorphism of formal schemes\footnote{We denote by $\widehat{X}_Q$ the formal completion of $X$ along $Q$.}:
 $$\walpha: \Vf \lrasim \widehat{X}_Q,$$
 where $Q := \walpha(P)$, and if the maps $(\alpha_\sigma)_{\sigma: K \hra \C}$ restrict to injections:
 $$\alpha_{\sigma \mid \mathring{V}_\sigma}:  \mathring{V}_\sigma \, \hlra \, X_\sigma (\C),$$
 then $\alpha^\ast$ establishes an isomorphism of fields:
 $$\alpha^\ast:  K(X_K) \lrasim \cM(\Vfa).$$
 \end{corollary}
 
 A  weaker result  follows from \cite[Theorem 7.9]{BostChambert-Loir07}, which implies that $\alpha^\ast K(X_K)$ contains every element $f := (\widehat{f}, (f_\sigma)_{\sigma: K \hra \C})$ of $\cM(\Vfa)$ such that $\widehat{f}$ belongs to $\cO(\Vf)$.
 
 \begin{proof} The assumptions on $\widehat{\alpha}$ and $(\alpha_\sigma)_{\sigma: K \hra \C}$ imply that   $\alpha$ is a morphism from $\Vfa$ to $X$, that the ramification index $e(\alpha)$ is 1, and that the overflow invariants $\Ex (\walpha: \Vfa \ra X)$ and $\Ex (\alpha_\sigma: (V_\sigma, P_\sigma)$ vanish. Consequently:
 $$ D(\alpha: \Vfa \ra X) \leq 1,$$
 and therefore \eqref{MvfKalpha} becomes:  
 \begin{equation*}
 [\cM(\Vfa): \alpha^\ast K(X_K) ] =1.
 \qedhere
 \end{equation*}
 \end{proof}

\subsection[The arithmetic holonomicity theorem]{Application: the arithmetic holonomicity theorem of \cite{CalegariDimitrovTang21}}\label{AppCDT}

\subsubsection{}
Let us recall the statement of the arithmetic holonomicity theorem, which is one of the main ingredients of the proof of the unbounded denominators conjecture by Calegari-Dimitrov-Tang \cite{CalegariDimitrovTang21}. We follow the notation of \cite[section 2]{CalegariDimitrovTang21} with only minor modifications.

Let $U$ be a connected open subset of $\C$ containing $0$, and let $x\in\Q[[T]]$ be a formal power series such that:
\begin{equation} \label{x01}
x(0)=0\quad \mathrm{and}\quad x'(0)=1.
\end{equation}
We identify $x$ with the morphism
$$x : \mathrm{Spf}\,\C[[T]]\lra\widehat U_{0}$$
that it defines.

Consider the subset $\mathcal H(U, x, \Z)$ of the space $\Q[[X]]$ whose elements are those power series $f\in\Q[[X]]$ satisfying the following two conditions:
\begin{enumerate}[(i)]
\item the formal power series $f\circ x\in \Q[[T]]$ lies in $\Z[[T]]$;
\item the formal power series $f$ is holonomic on $U$, namely, there exists a nonzero linear differential operator $L$ in $\C(X)[d/dx]$, without singularities on $U$, such that:
$$L(f)=0.$$
\end{enumerate}
The set $\mathcal H(U, x, \Z)$ is readily seen to be a $\Z$-subalgebra of $\Q[[X]]$.

We assume that the uniformization radius of the pointed Riemann surface $(U, 0)$ is strictly greater than $1$ or, equivalently, that there exists a holomorphic map
$$\phi : \overline{D(0, 1)}^+\lra U$$
such that
$$\phi(0)=0$$
and
\begin{equation}\label{equation:condition-positive}
|\phi'(0)|>1.
\end{equation}

The following statement is \cite[Theorem 2.0.2]{CalegariDimitrovTang21}.
\begin{theorem}\label{theorem:CDT-main}
Let $p\in\Q(X)\setminus\Q$ be a nonconstant rational function without poles in $U$ which defines an element of $\mathcal H(U, x, \Z)$ -- namely:
$$p\circ x\in\Z[[T]].$$
Then the $\Q(p)$-algebra $\mathcal H(U, x, \Z)\otimes_{\Z[p]}\Q(p)$ is a finite-dimensional $\Q(p)$-vector space. Moreover:
\begin{equation}\label{equation:estimate-holonomy}
\dim_{\Q(p)}\mathcal H(U, x, \Z)\otimes_{\Z[p]}\Q(p)\leq e\frac{\int_0^1\log^+|p\circ\phi(e^{2i\pi t})|dt}{\log|\phi'(0)|}.
\end{equation}
\end{theorem}

Observe that the $\Z$-algebra $\mathcal H(U, x, \Z)$ does indeed contain the ring $\Z[p]$ of polynomials in $p$, and therefore is a $\Z[p]$-algebra. Its ``base change" $\mathcal H(U, x, \Z)\otimes_{\Z[p]}\Q(p)$ is therefore a well-defined $\Q(p)$-algebra.

In the right-hand side of \eqref{equation:estimate-holonomy}, $e$ is the real number $\sum_{n=0}^{+\infty}n!^{-1}$.

\subsubsection{} Keeping the same notation as above, assume for simplicity that $U$ is invariant under complex conjugation and that $\phi$ is ``real", namely:
$$\phi(\overline z)=\overline{\phi(z)}$$
for all $z\in\overline{D(0, 1)}.$
These conditions are actually satisfied in the application of Theorem \ref{theorem:CDT-main} to the proof of the unbounded denominators conjecture, and will allow us to derive Theorem \ref{theorem:CDT-main} from Corollary \ref{degMKf} with $K=\Q$ applied to the \fa arithmetic surfaces $\Vfa(\overline{D}(0,1), \psi)$ introduced in Chapter \ref{ChapterFA}. 

When these conditions are not satisfied,  Theorem \ref{theorem:CDT-main} may be deduced from Corollary \ref{degMKf} with $K$ an arbitrary imaginary quadratic field by a variant of the discussion below which we leave to the reader.

Consider the universal cover:
$$\nu : (\widetilde U, O)\lra (U, 0)$$
of the pointed Riemann surface $(U, 0)$. The map $\nu$ defines an isomorphism between the formal germs:
$$\widehat \nu : \widehat{\widetilde U}_{O}\lrasim \widehat U_{0}.$$
Moreover, the map $\phi$ factors through $\nu$; namely there exists a unique holomorphic map:
$$\widetilde\phi : \overline{D(0,1)}^{+}\lra\widetilde U$$
such that:
$$\widetilde \phi(0)=O \qaq \nu\circ\widetilde\phi=\phi.$$
The formal germs at $0$ of $\phi$ and $\widetilde\phi$ define isomorphisms of complex formal curves:
$$\widehat\phi : \widehat \C_{0}=\mathrm{Spf}\,\C[[z]]\lrasim\widehat U_{0}
\qaq
\widehat{\widetilde\phi} : \mathrm{Spf}\,\C[[z]]\lrasim\widehat{\widetilde U}_{O}.$$

The data of $U, x, \phi, \nu, \widetilde\phi$ and their formal germs fit into the following commutative diagram of locally ringed spaces:

\begin{equation}\label{BigDiagram}
\begin{gathered}
\xymatrix{
& \widehat U_0=\mathrm{Spf}\,\C[[X]]\ar@{^{(}->}'[d][dd] & \mathrm{Spf}\,\C[[T]]\ar[l]^-{\sim}_-x\ar@{^{(}->}[r] & \mathrm{Spf}\,\Z[[T]]\\
\mathrm{Spf}\,\C[[z]]\ar[rr]^-{\sim}_-{\widehat{\widetilde\phi}}\ar@{^{(}->}[dd]\ar[ur]^{\sim}_{\widehat\phi} & & \widehat{\widetilde U}_O\ar[ul]^{\widehat\nu}_{\sim}\ar@{^{(}->}[dd] & \\
& U & & \\
\overline{D(0,1)}^+\ar[ur]^\phi\ar[rr]^{\widetilde\phi} & & \widetilde U\ar[ul]^\nu &
}
\end{gathered}
\end{equation}
 
Recall that a smooth \fa arithmetic surface 
$\Vfa=(\Vf, (V, P, i))$
over $\Z$ is defined by ``gluing'' the formal scheme $\Vf$ over $\Z$ and the Riemann surface with boundary $V$ according to the following diagram of locally ringed spaces:
\[
\xymatrix{
\widehat V_{P} & \widehat{\mathcal V}_{\C}\ar[l]^{\iota}_{\sim}\ar@{^{(}->}[r] & \Vf.
}
\]
An element of the algebra $\mathcal O(\Vfa)$, or equivalently a morphism from $\Vfa$ to $\mathbb A^{1}_{\Z}$, is defined by a pair of arrows $(\widehat f, f^{\mathrm{an}})$ that fit into the commutative diagram below:
\[
\xymatrix{
\widehat V_{P}\ar@{^{(}->}[d] & \widehat{\mathcal V}_{\C}\ar[l]^{\iota}_{\sim}\ar@{^{(}->}[r] & \Vf\ar[d]^{\widehat f}\\
V^{+}\ar[rr]^{f^{\mathrm{an}}} & & \mathbb A^{1}_{\Z}.
}
\]

From the diagram \eqref{BigDiagram}, we 
 may extract the following two diagrams:
\begin{equation}\label{equation:with-disk}
\begin{gathered}
\xymatrix{
\mathrm{Spf}\,\C[[z]]\ar@{^{(}->}[d] & \mathrm{Spf}\,\C[[T]]\ar[l]^{\widehat\phi^{-1}\circ x}_{\sim}\ar[r] & \mathrm{Spf}\,\Z[[T]]\\
\overline{D(0,1)}^{+} & &
}
\end{gathered}
\end{equation}
and 
\begin{equation}\label{equation:without-disk}
\begin{gathered}
\xymatrix{
\widehat{\widetilde U}_{O}\ar@{^{(}->}[d] & \mathrm{Spf}\,\C[[T]]\ar[l]^{\widehat\nu^{-1}\circ x}_{\sim}\ar[r] & \mathrm{Spf}\,\Z[[T]]\\
\widetilde U & &.
}
\end{gathered}
\end{equation}

The diagram \eqref{equation:with-disk} defines the smooth \fa arithmetic surface over $\Z$:
$$\Vfa(\phi, x) := \Vfa(\overline{D(0,1)}, \widehat\phi^{-1}\circ x)=(\mathrm{Spf}\,\Z[[T]], (\overline{D(0,1)}, \widehat\phi^{-1}\circ x))$$
and the diagram \eqref{equation:without-disk} defines a smooth \fa arithmetic surface over $\Z$ in a generalized sense\footnote{Indeed, the Riemann surface $\widetilde U$ is not compact.}:
$$\Vfa(U, x):=(\mathrm{Spf}\,\Z[[T]], (\widetilde U, \widehat\nu^{-1}\circ x)).$$

We may consider the $\Z$-algebras of regular functions on the \fa arithmetic surfaces $\Vfa(\phi, x)$ and $\Vfa(U, x)$:
$$\mathcal O(\Vfa(\phi, x)):=\{(\widehat \alpha, \beta)\in\Z[[T]]\times\mathcal O^{\mathrm{an}}(\overline{D(0,1)}^{+})\,|\,\widehat\alpha=\widehat\beta\circ\widehat\phi^{-1}\circ x\}$$
and
$$\mathcal O(\Vfa(U, x)):=\{(\widehat \alpha, \gamma)\in\Z[[T]]\times\mathcal O^{\mathrm{an}}(\widetilde U)\,|\,\widehat\alpha=\widehat\gamma\circ\widehat\nu^{-1}\circ x\}$$
where $\widehat \beta$ (resp. $\widehat\gamma$) denotes the Taylor expansion of the holomorphic function $\beta$ (resp. $\gamma$) at the point $0$ (resp. $O$) of $\C$ (resp. $\widetilde U$).

We may define injective morphisms of $\Z$-algebras:
$$\mathcal H(U, x, \Z)\stackrel{i_{1}}{\lra}\mathcal O(\Vfa(U, x))\stackrel{i_{2}}{\lra}\mathcal O(\Vfa(\phi, x))$$
defined as follows:
\begin{enumerate}
\item If $f$ is an element of $\Q[[X]]$ that belongs to $\mathcal H(U, x, \Z)$, we let
$$i_{1}(f):=(f\circ x, \gamma)$$
where $\gamma$ is the holomorphic function on the universal cover $\widetilde U$ of $U$ defined by the analytic continuation of the germ of the holomorphic function on $U$ at $0$ -- or equivalently on $\widetilde U$ at $0$ -- defined by $f$. The analyticity of $f$ and the existence of its analytic continuation $\gamma$ on $\widetilde U$ follows from  the holonomicity condition in the definition of $\mathcal H(U, x, \Z)$.

The condition $f\circ x=\widehat\gamma\circ\widehat\nu^{-1}$ is clearly satisfied.
\item For every $(\widehat\alpha, \gamma)$ in $\mathcal O(\Vfa(U, x))$, we let:
$$i_2(\widehat\alpha, \gamma):=(\widehat\alpha, \gamma\circ\widetilde\phi).$$
Since $(\gamma\circ\widetilde\phi)^\vee\circ\widehat\phi^{-1}=\widehat\gamma\circ\widehat\nu^{-1},$ $i_2(\widehat\alpha, \gamma)$ is indeed an element of $\mathcal O(\Vfa(\phi, x))$.
\end{enumerate}

Observe that the injective morphism
$$i_2\circ i_1 : \mathcal H(U, x, \Z)\lra\mathcal O(\Vfa(\phi, x))$$
maps a series $f$ in $\mathcal H(U, x, \Z)$ to the pair $(\widehat\alpha, \gamma)$ in $\mathcal O(\Vfa(\phi, x))$ defined by $\widehat\alpha=f\circ x$, and by the analytic continuation $\gamma$ of $f\circ\phi$ to $\overline{D(0,1)}^+$, which is a priori defined as the germ of a holomorphic function at the origin.

Furthermore, according to \eqref{x01}, the differential at $0$ of $\widehat\phi^{-1}\circ x$ maps the generator $\partial/\partial T$ of the normal bundle of $P:=\Spec\Z$ in $\mathrm{Spf}\,\Z[[T]]$ to the vector $\phi'(0)^{-1}\partial/\partial z$ in $T_0\overline{D(0,1)}^+$. Its capacitary norm is:
$$\Vert\phi'(0)^{-1}\partial/\partial z\Vert^{\mathrm{cap}}_{\overline{D(0,1)}, 0}=|\phi'(0)|^{-1}.$$
As a consequence, we have:
\begin{equation}\label{equation:explicit-degree-N}
\widehat\deg \overline N_P\Vfa(\phi, x)=\log |\phi'(0)|,
\end{equation}
so that \eqref{equation:condition-positive} --- namely, the assumption $|\phi'(0)|>1$, itself a consequence of our assumption on the uniformization radius of $(U, 0)$ --- is equivalent to the pseudoconcavity of the \fa arithmetic surface $\Vfa(\phi, x)$.

Applied to $\Vfa=\Vfa(\phi, x)$ and to some $f$ in $\cO(\Vfa(\phi,x))$, Corollary \ref{degMKf}   takes the following form, when we take into account the expression \eqref{DA1Int} for the invariants $D\big(\alpha_r: \Vfa_r \ra \A^1_\Z\big)$:

\begin{corollary}\label{corollary:bound-degree-CDT}
Let $\alpha=(\widehat\alpha, \beta)$ be a non-constant element of $\mathcal O(\Vfa(\phi, x)).$  The field $\cM(\Vfa(\phi, x))$ is finite extension of $\Q(\alpha)$, and its degree satisfies:
$$[\cM(\Vfa(\phi, x)) : \Q(\alpha) ] \leq 
({\log|\phi'(0)|})^{-1}
{\int_{[0,1]^{2}}\log\big|\beta(e^{2i\pi t_{1}})-\beta(e^{2i\pi t_{2}})\big|\, dt_{1}dt_{2}}$$
\end{corollary}

Using the injective morphism of $\Z$-algebras: 
$$i_2\circ i_1 : \mathcal H(U, x, \Z)\hlra \mathcal O(\Vfa(\phi, x))$$
defined above,  and the associated injective morphisms of fields:
$$\mathrm{Frac}\, \mathcal H(U, x, \Z) \, \hlra \, \mathrm{Frac}\, \mathcal O(\Vfa(\phi, x)) 
\, \hlra \, \cM(\Vfa(  \phi, x)),$$ this implies the following improved form of \cite[Theorem 2.0.2]{CalegariDimitrovTang21}:

\begin{corollary}
Let $f$ be a non-constant element of $\mathcal H(U, x, \Z).$ Then the $\Q(f)$-algebra $$\mathcal H(U, x, \Z)\otimes_{\Q[f]}\Q(f)$$ is finite-dimensional; in particular, it may be identified with the field $\mathrm{Frac}\, \mathcal H(U, x, \Z)$. The dimension of this $\Q(\alpha)$-algebra is bounded above by:
\begin{equation}\label{VarCDT}
\big({\log|\phi'(0)|}\big)^{-1} 
{\int_{[0,1]^{2}}\log \big|(f\circ\phi)(e^{2i\pi t_{1}})-(f\circ\phi)(e^{2i\pi t_{2}})\big| \, dt_{1}dt_{2}}. 
\end{equation}
\end{corollary}
As mentioned in the discussion of examples in Subsections \ref{ExamplesOverflow} and  \ref{ExamplesD}, the expression \eqref{VarCDT} admits the following upper bound in terms of the characteristic function $T_{f \circ \phi}$ \emph{\`a la} Nevanlin\-na-Ahlfors-Shimizu of $f \circ \phi$:
$$\big({\log|\phi'(0)|}\big)^{-1} 
{\int_{[0,1]^{2}}\log\big|(f\circ\phi)(e^{2i\pi t_{1}})-(f\circ\phi)(e^{2i\pi t_{2}})\big| \,dt_{1}dt_{2}} 
\leq  
\big({\log|\phi'(0)|}\big)^{-1}\,  T_{f \circ \phi} (1).$$


\chapter[Pseudoconcave formal-analytic arithmetic surfaces II]{Pseudoconcave formal-analytic arithmetic surfaces II:  the algebra $\cO(\Vfa)$; fundamental groups of arithmetic surfaces}\label{PseudoconII}

In this chapter, we establish some further properties of   pseudoconcave formal-analytic arithmetic surfaces and of their maps to quasi-projective arithmetic surfaces. These are obtained by combining the theorems  in Chapter \ref{Chapter8} with various classical techniques of arithmetic geometry. 

In Section \ref{FinOVfa}, we establish some 
 arithmetic analogues of  the finiteness results concerning 
the universal meromorphic map $\phi: \cV \dashrightarrow \cV^{\mathrm{alg}}$, attached to a germ of complex analytic surface $\cV$ fibered over a complex projective curve, and
the algebra $\cA$ introduced in Subsection \ref{universalcV}. Notably we show that the $\Z$-algebra $\mathcal O(\Vfa)$ attached to a pseudoconcave formal-analytic arithmetic surfaces $\Vfa$ is finitely generated.

Section \ref{ArLefNor} is devoted to some arithmetic analogues of  Lefschetz-Nori theorems on \'etale fundamental groups of complex algebraic surfaces presented in Part 1 as  Proposition \ref{easyNori} and Theorem \ref{LefschetzNoriCNB}.
In Section \ref{FinEtpi1}, these arithmetic  Lefschetz-Nori theorems 
are used to construct quasi-projective arithmetic surfaces that admit a finite \'etale fundamental group. Notably we  show that this is the case of certain classical  integral models of modular curves.

\section[Finite generation of $\mathcal O(\Vfa)$]{A finiteness result for the algebra $\mathcal O(\Vfa)$ and a structure theorem for morphisms to affine arithmetic surfaces}\label{FinOVfa}

\subsection{Universal morphisms to affine arithmetic schemes and finiteness}

\subsubsection{The universal meromorphic map $\phi: \Vfa \dashrightarrow \cC$} Let $\Vfa$ be a pseudoconcave smooth \fa arithmetic surface over $\Spec \OK$ such that the field $\cM(\Vfa)$ of meromorphic functions over $\Vfa$ is not reduced to $K$.

According to Theorem \ref{theoremMero}, the field $\cM(\Vfa)$ is a finitely generated field extension of $K$ of transcendence degree one. Moreover if the divisor of some element $f \in \cM(\Vfa)$ does not contain $\vert \Vf \vert = \im P$ with a negative multiplicity, then  the restriction of $f$ to $\vert \Vf \vert$, or equivalently its pull-back by $P$, defines an element of $K$. The partially defined map:
\begin{equation}\label{placeM}
\cM(\Vfa) \dashrightarrow K, \quad f \longmapsto P^\ast f
\end{equation}
defines a place of  $\cM(\Vfa)$ over $K$.

As a consequence, there exist a smooth projective geometrically irreducible curve $C$ over $K$, endowed with a $K$-rational point $O$, and an isomorphism of $K$-algebras:
$$\iota : K(C) \lrasim \cM(\Vfa)$$
whose composition with  \eqref{placeM} is the place of $K(C)$ over $K$ defined by $O$. Moreover the pair $(C, \iota)$ is unique up to unique isomorphism, and the point $O$ in $C(K)$ is uniquely determined by this condition.

We may consider a projective normal arithmetic surface:
\begin{equation}\label{projmodC}
\cC \lra \Spec \OK
\end{equation}
which is  a model of $C$ --- so that the $K$-scheme $\cC_K$ is isomorphic to $C$ --- and the section $O$ of \eqref{projmodC} that extends the rational point $O$ in $C(K) \simeq \cC(K)$.

The isomorphism $\iota$ defines a canonical meromorphic map:
$$\phi:= (\widehat{\phi}, (\phi_\sigma)_{\sigma \hra \C}): \Vfa \dashrightarrow \cC$$
such that the composition $\widehat{\phi} \circ P$ coincides with $O$.

The map $\phi$ is easily checked to be a ``universal meromorphic map"  from $\Vfa$ to a projective $\OK$-scheme and to play the role of the universal meromorphic map $\phi: \cV \dashrightarrow \cV^{\mathrm{alg}}$ studied in Subsection~\ref{universalcV} in the geometric case. 

The construction in Subsection \ref{univram}, suitably adapted to the arithmetic setting, is expected to produce examples of pseudoconcave smooth \fa arithmetic surfaces $\Vfa$ such that $\cM(\Vfa) \neq K$ for which the map $\phi$ may be ramified along $P$. We leave the details to the interested reader. 

\subsubsection{The main finiteness result} Observe that the fraction field $\mathrm{Frac}\, \cO(\Vfa)$ of the domain $\cO(\Vfa)$ embeds in $\cM(\Vfa)$, and, like $\cM(\Vfa)$, is an extension of $K$ of finite type, of transcendence degree at most one. According to Theorem \ref{theoremMero}, if $\phi$ is an element of $\cO(\Vfa) \setminus \OK$, then  the field $\mathrm{Frac}\, \cO(\Vfa)$ is an algebraic extension of $K(\phi)$, and its degree satisfies the upper bound:
\begin{equation}
[\mathrm{Frac}\, \cO(\Vfa) : K(\phi) ] \leq D(\phi: \Vfa \rightarrow \A^1_{\OK}).
\end{equation}
Note that the inequality above is enough in itself to prove the results of \cite{CalegariDimitrovTang21}.

In Subsection~\ref{universalcV}, we have introduced the $\C[\mathring{C}]$-algebra $\cA$, which is a geometric counterpart of the $\OK$-algebra $\cO(\Vfa)$ and we have shown  that, in the pseudoconcave case, it is a  finitely generated $\C[\mathring{C}]$-algebra, as a consequence  of a finiteness result of Zariski \cite{Zariski54} concerning the algebra of regular functions on quasi-projective varieties over a field. 

In this section, we establish an arithmetic analogue of this finiteness result.

\begin{theorem}\label{finOVfa}
For every pseudoconcave smooth \fa arithmetic surface $\Vfa$ over $\Spec \OK$, the $\OK$-algebra $\cO(\Vfa)$ is finitely generated.
\end{theorem}
 
\subsubsection{The universal map to an affine arithmetic surface} According to Theorem \ref{finOVfa},  the scheme:
 $$\Vfa^{\mathrm{aff}} := \Spec \cO(\Vfa)$$
 is an affine  scheme of finite type over $\Spec \OK$. 
  When $\cO(\Vfa) \neq \OK,$ it is an integral normal affine arithmetic surface over $\Spec \OK$. 
 Moreover the identification 
 $$\cO(\Vfa) \simeq \Gamma(\Vfa^{\mathrm{aff}}, \cO_{\Vfa^{\mathrm{aff}}})$$
 defines a tautological morphism over $\Spec \OK$:
 $$\alpha: \Vfa \lra \Vfa^{\mathrm{aff}},$$
 and this morphism $\alpha$  is easily seen to be universal among the $\OK$-morphisms from $\Vfa$ to affine arithmetic $\OK$-schemes. 
 
 We will actually prove Theorem \ref{finOVfa} by first constructing the universal morphism $\alpha: \Vfa \lra \Vfa^{\mathrm{aff}}$ to affine arithmetic $\OK$-schemes -- this is the content of Theorem \ref{theorem:final-surface} -- and then using it to prove the relevant finiteness result on $\cO(\Vfa)$ by proving that $\alpha$ induces an isomorphism
 $$\Gamma(\Vfa^{\mathrm{aff}}, \cO_{\Vfa^{\mathrm{aff}}})\lrasim \cO(\Vfa).$$

\subsection{Preliminary results}

We gather a few general results of independent interest.

\subsubsection{Monotonicity and discreteness properties of non-Archimedean overflows}

\begin{proposition}\label{proposition:variation-ov}
Consider a commutative diagram of morphisms over $\OK$:
\[
\xymatrix{
& S_2\ar[d]^{p}\\
\Vfa\ar[r]^{\gamma_1}\ar[ur]^{\gamma_2} &S_1
}
\]
where $S_1$ and $S_2$ are integral, normal, affine arithmetic surfaces and $p$ is birational. Then 
\begin{equation}\label{equation:ineq-na-ov}
\mathrm{Ex}(\widehat\gamma_1 : \Vf\ra X_1)\geq \mathrm{Ex}(\widehat\gamma_2 : \Vf\ra X_2),
\end{equation}
and equality holds in \eqref{equation:ineq-na-ov} if and only if there exists a neighborhood $U$ of the $\OK$-point $\gamma_1(P)$ in $X_1$ such that 
$$p_{|p^{-1}(U)} : p^{-1}(U)\ra U$$
is an isomorphism. Furthermore, if equality does not hold in \eqref{equation:ineq-na-ov} and $\widehat\gamma_2$ is quasi-finite, then:
\begin{equation}\label{equation:ineq-na-ov-strict}
\mathrm{Ex}(\widehat\gamma_1 : \Vf\ra X_1)\geq \mathrm{Ex}(\widehat\gamma_2 : \Vf\ra X_2)+\log(2).
\end{equation}
\end{proposition}

\begin{proof}
Since $p$ is birational, $p$ defines an isomorphism between the generic fibers of $S_1$ and $S_2$ over $\OK$. In particular, the ramification indices $e(\gamma_1)$ and $e(\gamma_2)$ coincide. We denote this integer by $e:=e(\gamma_1)=e(\gamma_2)$. By definition, for $i=1$ and $i=2$, we have:
$$\mathrm{Ex}(\widehat\gamma_i : \Vfa\lra S_i)=\widehat\deg(P.(\widehat\gamma_i^*(\widehat\gamma_i(P))-eP)).$$
We have:
$$\widehat\gamma_1^*(\widehat\gamma_1(P))=(\widehat\gamma_2^*\circ p^*)(\widehat\gamma_1(P))=\widehat\gamma_2^*(\widehat\gamma_2(P))+\widehat\gamma_2^*(p^*(\widehat\gamma_1(P))-\widehat\gamma_2(P)).$$
As a consequence, we may write:
\begin{equation}\label{equation:compare-ov}
\mathrm{Ex}(\widehat\gamma_1)=\mathrm{Ex}(\widehat\gamma_2)+\widehat\deg(P.\widehat\gamma_2^*(p^*(\widehat\gamma_1(P))-\widehat\gamma_2(P)).
\end{equation}
Since the divisor $p^*(\widehat\gamma_1(P))-\widehat\gamma_2(P)$ is clearly effective, this proves \eqref{equation:ineq-na-ov}. The inequality \eqref{equation:ineq-na-ov-strict} follows from Proposition \ref{proposition:integral-excess}.

Assume that equality holds in \eqref{equation:ineq-na-ov} and let $\overline{S}_1$ be a normal, projective compactification of $S_1$. Let $S'_2$ be a projective compactification of $S_2$, and let $\overline S_2$ be the normalization of the closure of the graph of $p$ in $\overline S_1\times S'_2.$ Let $\overline p$ denote the projection of $\overline S_2$ onto $\overline S_1$. We obtain a commutative diagram:
\[
\xymatrix{
& S_2\ar[d]^{p}\ar@{^(->}[r]^{j_2} & \overline S_2\ar[d]^{\overline p}\\
\Vfa\ar[r]^{\gamma_1}\ar[ur]^{\gamma_2} &S_1\ar@{^(->}[r]^{j_2} & \overline S_1
}
\]
in which $j_1$ and $j_2$ are open immersions. 

We argue by contradiction and assume that $p$ -- hence $\overline p$ -- is not an isomorphism above a neighborhood of $\gamma_1(P)$ in $S_2$. By Zariski's main theorem, there exists a nonzero effective Weil divisor $\overline D$ on $\overline S_2$ that intersects $\gamma_2(P)$  and such that $\overline p(\overline D)$ is a  closed point in $\gamma_1(P).$  

Let $D$ be the intersection of $\overline D$ with $S_2$. This is a nonempty effective Weil divisor that intersects $\gamma_2(P)$. By construction, the effective divisor the divisor $p^*(\widehat\gamma_1(P))-\widehat\gamma_2(P)$ contains $D$, namely, 
$$p^*(\widehat\gamma_1(P))-\widehat\gamma_2(P)-D$$
is effective. Applying \eqref{equation:compare-ov}, we find:
$$\mathrm{Ex}(\widehat\gamma_1)\geq \mathrm{Ex}(\widehat\gamma_2)+\widehat\deg(P.\widehat\gamma_2^*D)>\mathrm{Ex}(\widehat\gamma_2),$$
which finishes the proof.
\end{proof}

\subsubsection{Blow-ups and contractions}

The following statement is standard, at least for morphisms between regular arithmetic surfaces. We offer a short proof for the sake of completeness.

\begin{proposition}\label{proposition:contraction-blowup}
Let $\widehat\alpha : \Vf\ra X$ be a non-constant morphism from $\Vf$ to a smooth, quasi-projective scheme over $\Spec\OK$. Let $x$ be a closed point of $X$ lying over a finite prime $\p$ of $\OK$. If the preimage of $x$ in $\Vf$ is not finite, then it is equal to the fiber of $\Vf$ above $\p$, and there exists a unique lift $\widehat\alpha'$ of $\widehat\alpha$ to the blow-up
$$\pi : X'\lra X$$ 
of $X$ at $x$.
\end{proposition}

\begin{proof}
We first compute the preimage of $x$ in $\Vf$. This may be done by a local computation as follows: let $\p$ be the nonzero prime ideal of $\OK$ which is the image of $X$, and let $R$ be the completion of $\OK$ at $\p$. Consider a morphism
$$\mathrm{Spf}R[[X]]\lra \mathrm{Spf} R[[Y]]$$
given by a power series 
$$f=\sum_{i\geq 0}a_i X^i\in R[[X]].$$
Let $\varpi$ be a uniformizer of $R$. This is a local model for $\widehat\alpha$ near any point in the preimage of $x$. The closed point of $\mathrm{Spf} R[[Y]]$ is defined by the ideal $(Y, \varpi)$. Its preimage in $\Vf$ is the closed subscheme defined by the ideal $(f, \varpi)$ in $R[[X]].$ This closed subscheme is not finite if and only if all the $a_i$ are divisible by $\varpi$, in which case
$$(f, \varpi)=\varpi.$$
In particular, the preimage of $x$ is the divisor defined by $\varpi$, which is Cartier.

The remaining part of the proposition now follows from the universal property of the blow-up, see e.g. \cite[Tag 085U]{stacks-project}, whose proof applies in the setting of formal schemes.
\end{proof}

Keeping the notation of the previous proposition, a standard computation shows that $\widehat\alpha'$ factors through the smooth locus of the structure map $X'\ra \Spec\OK$.

Let $E$ be the irreducible Cartier divisor contracted to $x$ -- namely, the preimage of $\p$ in $\Vf$. If $D$ is a Cartier divisor on $X$ containing $x$ with multiplicity $1$, we may write 
$$\widehat\alpha^* D=\widetilde D+e E$$
for some positive integer $e$, where $\widetilde D$ is an effective Cartier divisor which does not admit $E$ as a component -- namely, $\widetilde D$ is the strict transform of $D$. With the notation of the proof, $\varpi^e$ is the largest power of $\varpi$ that divides all the $a_i$. It is immediate to check that if $D'$ is a Cartier divisor on $X'$ containing the image of $\p$ in $\Spec\OK\simeq P$ with multiplicity $1$ we have:
$$\widehat\alpha'^* D'=\widetilde D'+(e-1) E,$$
where $\widetilde D'$ is effective and does not contain $E$ as a component.

As a consequence of the previous paragraph, we may apply Proposition \ref{proposition:contraction-blowup} repeatedly to obtain the following result:

\begin{proposition}\label{proposition:making-qf}
Let $\widehat\alpha : \Vf\ra X$ be a nonconstant morphism from $\Vf$ to a smooth, quasi-projective scheme over $\Spec\OK$. Then there exists a morphism $\pi : X'\ra X$ which is a composition of blow-ups of closed points, such that $\widehat\alpha$ lifts to a quasi-finite morphism $\widehat\alpha' : \Vfa\ra X'$.
\end{proposition}

\subsubsection{A finiteness result for exceptional curves}

The following finiteness theorem is a special case of a result on $A$-mod-affine $A$-schemes that will appear in the forthcoming work \cite{infgeom2} and holds in arbitrary dimension. We provide a proof of the special case we are interested in. 

\begin{theorem}\label{theorem:exc-finite}
Let $\Vfa$  be a pseudoconcave formal-analytic arithmetic surface over $\OK$ and let 
$$\alpha=(\widehat \alpha, (\alpha_{\sigma}^{\mathrm{an}})_{\sigma: K\hra\C}) : \Vfa\lra X$$
be a nonconstant morphism from $\Vfa$ to an integral, normal, quasi-projective arithmetic surface $X$. Then there exist only finitely many irreducible Weil divisors $D$ on $X$ such that both $\widehat \alpha^{-1}(D)$ and the $(\alpha_{\sigma}^{\mathrm{an}})^{-1}(D_{\sigma}^{\mathrm{an}})$ are empty -- i.e., such that $\alpha$ factors through $X\setminus D$.
\end{theorem}

\begin{proof}
Let $D$ be an irreducible Weil divisor $D$ on $X$ such that both $\widehat \alpha^{-1}(D)$ and the various $(\alpha_{\sigma}^{\mathrm{an}})^{-1}(D_{\sigma}^{\mathrm{an}})$ are empty as $\sigma$ runs through the complex embeddings of $K$.

Let $P$ be the $\OK$-point of $\Vf$ corresponding to $|\Vf|,$ let $g_{\Vfa_\C}$ denote the equilibrium potential on the disjoint union of the Riemann surfaces $V_\sigma$, and let $\overline L$ be the $\Q$-Hermitian line bundle on $X$ defined by the Arakelov divisor: 
$$\alpha_{*}(P, g_{\Vfa_\C}):=(\widehat\alpha_{*}(P), \alpha^{\mathrm{an}}_{*}g_{\Vfa_\C}).$$
Proposition \ref{proposition:ineq-bignef} shows that $\overline L$ is nef and big. Additionally, as $\alpha^{\mathrm{an}}_{*}g_{\Vfa_\C}$ vanishes outside the images of the Riemann surfaces $V_\sigma$, the Arakelov degree of the restriction of $\overline L$ to $D$ is zero. 

By the arithmetic Hilbert-Samuel formula of \cite[Theorem 1.4]{Zhang95}, we may find a nonzero section $\sigma$ of $\overline L$ such that $\sigma$ has norm strictly smaller than $1$  at every complex point of $X$. In particular, since the Arakelov degree of the restriction of $\overline L$ to $D$ is zero, $\sigma$ vanishes identically on $D$. This proves that $D$ is contained in the divisor of $\sigma$. 

The set of Weil divisors considered in the statement of the theorem is contained in those Weil divisors that are components of the divisor of $\sigma$, which proves the result.
\end{proof}

\begin{corollary}\label{corollary:smallest-affine}
Let 
$$\alpha=(\widehat \alpha, (\alpha_{\sigma}^{\mathrm{an}})_{\sigma: K\hra\C}) : \Vfa\lra X$$
be a nonconstant morphism from $\Vfa$ to an integral, normal, affine arithmetic surface $X$. Let $D_{\mathrm{ex}}$ be the reunion of those irreducible divisors $D$ on $X$ such that $\alpha$ factors through $X\setminus D$. Then $D_{\mathrm{ex}}$ is the support of a Cartier divisor on $X$, and the open subset $X\setminus D_{\mathrm{ex}}$ is the smallest open affine subset $U$ of $X$ such that $\alpha$ factors through $U$.
\end{corollary}

\begin{proof}
Theorem \ref{theorem:exc-finite} shows that $D_{\mathrm{ex}}$ is a finite union of irreducible Weil divisors on $X$. As a consequence of the general results on arithmetic surfaces recalled in \ref{subsection:recollection-AS}, $D_{\mathrm{ex}}$ is $\Q$-Cartier, so it is the support of a Cartier divisor and $X\setminus D_{\mathrm{ex}}$ is affine.

Let $U$ be an open affine subset of $X$ such that $\alpha$ factors through $U$. As $X$ is normal, the complement of $U$ in $X$ is a Weil divisor $D$ on $X$ which is necessarily contained in $D_{\mathrm{ex}}$, so that $U$ contains $X\setminus D_{\mathrm{ex}}.$
\end{proof}

\subsection{A universal morphism to an affine arithmetic surface}\label{subsection:universal}

As before, let $\Vfa$ be a pseudoconcave smooth f.-a. arithmetic surface over $\Spec\OK$. We will construct a universal morphism from $\Vfa$ to an affine arithmetic $\OK$-scheme.

\subsubsection{Existence of a quasi-finite morphism}

\begin{proposition}\label{proposition:exists-qf}
Assume that the algebra $\mathcal O(\Vfa)$ contains a non-constant element. Then there exists a morphism 
$$\beta : \Vfa\lra Y$$
to an affine regular arithmetic surface $Y$ that is quasi-finite, namely, the preimage of any closed point of $Y$ is finite in $\Vf$.
\end{proposition}

\begin{proof}
The elements of $\mathcal O(\Vfa)$ correspond bijectively to the morphisms from $\Vfa$ to $\mathbb A^1_\OK.$ Consider a non-constant morphism
$$\gamma : \Vfa\lra\mathbb A^1_\OK.$$

Applying Proposition \ref{proposition:making-qf}, we may find a quasi-projective, regular arithmetic surface $Y'$ and a morphism
$$\pi : Y'\lra\mathbb A^1_\OK$$
that is a composition of blow-ups of closed points, such that $\gamma$ lifts to a quasi-finite morphism 
$$\gamma' : \Vfa\lra Y'.$$

Let $D$ be the reunion of those components of the fibers of the structure morphism $Y'\ra \Spec\OK$ that do not intersect $\gamma'(P)$. It is readily checked, since $\gamma'(P)$ lies in the smooth locus of $Y'\ra \Spec\OK$, that $D$ meets every component of every fiber of $Y'\ra\Spec\OK$. As a consequence of the discussion in \ref{subsection:recollection-AS}, $Y'\setminus D$ is affine and, by construction, $\gamma'$ factors through a morphism $\beta : \Vfa\ra Y:=Y'\setminus D$ which satisfies the conditions of the proposition.
\end{proof}

\subsubsection{}

Consider the morphism $\beta : \Vfa\ra Y$ of Proposition \ref{proposition:exists-qf}. As in Definition~\ref{definition:D(a)}, we introduce:
$$D(\beta):= \frac{ \beta_\ast \big(P, g_{\Vfa_\C}\big) \cdot \beta_\ast \big(P, g_{\Vfa_\C}\big)}{\dega\Nb_P \Vfa}$$
so that:
$$D(\beta)=e(\beta)+\frac{1}{\widehat\deg\,\overline N_{P}\Vfa}\Big(\mathrm{Ex}(\widehat\beta : \Vf\lra \widehat Y_{\beta(P)})+\sum_{\sigma : K\hra \C}\mathrm{Ex}(\beta_{\sigma} : (V_{\sigma}, P_{\sigma})\lra Y_{\sigma})\Big).$$
Consider a commutative diagram of morphisms over $\OK$:
\begin{equation}\label{equation:diagram-lift}
\begin{gathered}
\xymatrix{
& Y'\ar[d]^{f}\\
\Vfa\ar[r]^{\beta}\ar[ur]^{\beta'} & Y
}
\end{gathered}
\end{equation}
where $Y'$ is an integral, normal, affine arithmetic surface over $\OK$. By Theorem \ref{theorem:main}, we have:
\begin{equation}\label{equation:basic-degree}
\deg(f)\leq D(\phi).
\end{equation}

Since $\beta$ is quasi-finite, $\beta'$ is quasi-finite. In particular, Proposition \ref{proposition:excess-nonnegative} guarantees that we may choose $\beta' : \Vfa\ra Y'$ in such a way that the pair $(\deg f, -\mathrm{Ex}(\widehat\beta' : \Vf\ra \widehat Y_{\beta'(P)}))$ is maximal for the lexicographic order.

%
%

\begin{proposition}\label{proposition:universal-birational}
Consider a commutative diagram:
\[
\xymatrix{
& Z\ar[d]^{g}\\
\Vfa\ar[r]^{\beta'}\ar[ur]^{\gamma} & Y'
}
\]
in which $Z$ is an affine, integral arithmetic surface over $\OK$. Then there exists an open subset $U$ of $Y'$ such that $\beta'$ factors through $U$ and 
$$g_{|g^{-1}(U)} : g^{-1}(U)\lra U$$
is an isomorphism.
\end{proposition}

\begin{proof}

Since $\Vfa$ is smooth, $\gamma$ lifts to a morphism from $\Vfa$ to the normalization of $Z$. As a consequence, we may assume that $Z$ is normal.

By definition of $Y'$, the degree of $g\circ f$ is bounded above by the degree of $f$, so that the degree of $g$ is $1$, i.e., $g$ is a birational morphism. Proposition \ref{proposition:variation-ov} together with the choice of $\beta'$ shows that $g$ is an isomorphism above a neighborhood of $\beta'(P)$ in $Y'$. 

Let $U$ be the smallest open subset of $Y'$ on which $g^{-1}$ is defined. Then $g$ contains $\beta'(P)$. Furthermore, let $\sigma$ be a complex embedding of $K$. Then, as $g$ has degree $1$, $g_\sigma : Z_\sigma\ra Y'_\sigma$ is an open immersion whose image contains $\beta'^{\mathrm{an}}_\sigma(V_\sigma)=g_\sigma^{\mathrm{an}}(\gamma^{\mathrm{an}}_\sigma(V_\sigma)).$ This proves that $U_\sigma\subset Y'_\sigma$ contains $\beta_\sigma^{\mathrm{an}}(V_\sigma)$. As a consequence, $\beta'$ factors through $U\subset Y'$. Since the birational map $g^{-1}$ is defined on $U$, the morphism 
$$g_{|g^{-1}(U)} : g^{-1}(U)\lra U$$
is an isomorphism.
\end{proof}

\subsubsection{} We want to improve on Proposition \ref{proposition:universal-birational} by using the finiteness statement of Theorem \ref{theorem:exc-finite}.


Let $X$ be the largest open affine subscheme of $Y'$ such that $\beta'$ factors through $X$. The existence of $X$ is guaranteed by Corollary \ref{corollary:smallest-affine}.

We denote by $\alpha$  the natural morphism:
$$\alpha : \Vfa\lra X.$$

\begin{proposition}\label{proposition:final-surface}
For any morphism over $\OK$ 
$$\gamma : \Vfa\lra Z$$
 from $\Vfa$ to an integral, affine arithmetic surface $Z$ and any commutative diagram of morphisms over $\OK$
\[
\xymatrix{
& Z\ar[d]^{p}\\
\Vfa\ar[r]^{\alpha}\ar[ur]^{\gamma} & X,
}
\]
$p$ is an isomorphism.
\end{proposition}

\begin{proof}
Proposition \ref{proposition:universal-birational} shows that there exists an open subset $U$ of $X$ such that $\alpha$ factors through $U$ and 
$$p_{|p^{-1}(U)}\lra U$$
is an isomorphism. Since the complement of any effective Weil divisor in $X$ is affine, any nonzero Weil divisor in $X$ meets $U$.

Consider $\mathcal O(X)$ s a sub-algebra of $\mathcal O(Z)$ via $p^* : \mathcal O(X)\ra \mathcal O(Z)$. Let $f$ be an element of $\mathcal O(Z)$. Then $f$ lies in $Frac(\mathcal O(Z))=Frac(\mathcal O(Z))$. Furthermore, $f$ defines a regular function on $U\subset X$. Since $X$ is normal and every nonzero Weil divisor in $X$ meets $U$, this implies that $f$ defines a regular function on $X$. This proves that $p^* : \mathcal O(X)\ra \mathcal O(Z)$ is an isomorphism. Since $Z$ is affine, this proves that $p$ is an isomorphism.

\end{proof}

We may finally prove that $\alpha : \Vfa\ra X$ is universal for morphisms to arithmetic schemes.

\begin{theorem}\label{theorem:final-surface}
Let $\gamma : \Vfa\ra Z$ be a morphism over $\OK$ to an affine arithmetic scheme $Y$. There exists a unique morphism 
$$f : X \lra Z$$
such that the diagram
\[
\xymatrix{
& X\ar[d]^{f}\\
\Vfa\ar[r]^{\gamma}\ar[ur]^{\alpha} & Z
}
\]
commutes.
\end{theorem}

\begin{proof}
The unicity of $f$ is clear. As $X$ is integral, any $f$ as in the statement factors through the the reduced subscheme of $Z$, so we may assume that $Z$ is reduced. Finally, we may replace $Z$ with an irreducible component of $Z$ containing the image of $\gamma$ and assume that $Z$ is integral.

Consider the morphism
$$\delta=\alpha\times\gamma : \Vfa\lra X\times_{\OK} Z.$$
By Theorem \ref{pseudoconcavefaalgebraic}, the closed subset $\mathrm{im}(\delta)$ has dimension $2$. Let $Z'$ be the normalization of $\mathrm{im}(\delta)$. It is an integral, normal, affine arithmetic surface over $\OK$, and there is a commutative diagram:
\[
\xymatrix{
& Z'\ar[d]^{p}\\
\Vfa\ar[r]^{\alpha}\ar[ur]^{\delta} & X
}
\]
where $p$ is the composition $Z'\ra X\times_{\OK}Z\ra X.$ By Proposition \ref{proposition:final-surface}, $p$ is an isomorphism. Let $f : X\ra Y$ be the composition:
$$X\stackrel{p^{-1}}{\lra} Z'\stackrel{q}{\lra} Z$$
where $q$ is the composition $Z'\ra X\times_{\OK}Z\ra Z.$

Then 
$$f\circ \alpha=q\circ p^{-1}\circ\alpha=q\circ\delta=\gamma.$$ 
This finishes the proof.
\end{proof}

Clearly, the Theorem above characterizes the morphism 
$$\alpha : \Vfa\lra X$$
uniquely as it represents the functor sending an arithmetic scheme $Z$ to the morphisms from $\Vfa$ to $Z$. We will denote $X$ by $\Vfa^{\mathrm{aff}}.$

\subsection{Application to the algebra $\mathcal O(\Vfa)$}

The following result immediately implies Theorem~\ref{finOVfa}.

\begin{theorem}\label{theorem:identify-Y}
The morphism 
$$\alpha : \Vfa\lra \Vfa^{\mathrm{aff}}$$
induces an isomorphism of $\OK$-algebras:
$$\alpha^* : \mathcal O(\Vfa^{\mathrm{aff}})\lra \mathcal O(\Vfa).$$
In particular, the $\OK$-algebra $\mathcal O(\Vfa)$ is finitely generated and there is a canonical isomorphism:
$$\Vfa^{\mathrm{aff}}\simeq \Spec\mathcal O(\Vfa).$$
\end{theorem}

\begin{proof}
The morphism $\alpha$ induces an injective morphism:
$$\alpha^* : \mathcal O(\Vfa^{\mathrm{aff}})\lra \mathcal O(\Vfa)$$
that sends $f$ to $f\circ\alpha.$

An element $\gamma$ of $\mathcal O(\Vfa)$ may be identified with a morphism over $\OK$, also denoted by $\gamma$ : 
$$\gamma : \Vfa\lra \mathbb A^{1}_{\OK}.$$

Theorem \ref{theorem:final-surface} shows that there exists a morphism $f : X\ra \mathbb A^1_\OK$ such that the diagram

\[
\xymatrix{
& X\ar[d]^{f}\\
\Vfa\ar[r]^{\gamma}\ar[ur]^{\alpha} & \mathbb A^{1}_{\OK}
}
\]
commutes. When we identify $f$ with an element of $\mathcal O(\Vfa^{\mathrm{aff}})$, $\alpha^*(f)$ is the composition 
$$\Vfa\stackrel{\alpha}{\lra} X\stackrel{f}{\lra}\mathbb A^1_\OK,$$
namely, $\alpha^*(f)=\gamma.$ This proves that $\gamma$ lies in the image of $\alpha^{*}$, and shows that $\alpha^{*}$ is an isomorphism. The rest of the statement is clear.
\end{proof}

\subsubsection{A quasi-projective variant}

We offer without proof the following variant of Theorem \ref{theorem:final-surface} and Theorem \ref{theorem:identify-Y} which does not use affineness assumptions. It follows from the same arguments.

\begin{proposition}\label{proposition:final-surface-qp}
Assume that there exists a nonconstant morphism $\alpha_0 : \Vfa\ra X_0$ to an integral, quasi-projective arithmetic surface $X_0$ over $\OK$. There exists a commutative diagram
\[
\xymatrix{
& X\ar[d]^{f}\\
\Vfa\ar[r]^{\alpha_0}\ar[ur]^{\alpha} & X_0
}
\]
in which $X$ is a normal, integral, quasi-projective arithmetic surface over $\OK$ and the pair: $$(\deg f, -\mathrm{Ex}(\widehat\alpha))$$ is maximal for the lexicographic order. Given such a diagram, let $X$ be the complement in $X'$ of those finitely many irreducible Weil divisors $D$ in $X'$ such that $\alpha'$ factors through $X\setminus D.$ Let $\alpha : \Vfa\ra X$ be the morphism induced by $\alpha'$.

Let $\gamma : \Vfa\ra Z$ be a morphism over $\OK$ to a quasi-projective arithmetic surface $Y$. Then there exists a finite, closed subset $F$ of $X\setminus \alpha(P)$ such that, letting $U$ be the complement of $F$ in $X$,  there exists a morphism 
$$f : U \lra Y$$
such that the diagram
\[
\xymatrix{
& U\ar[d]^{f}\\
\Vfa\ar[r]^{\gamma}\ar[ur]^{\alpha} & Z
}
\]
commutes. The rational map $f : X\dashrightarrow Z$ is unique. 
\end{proposition}

Without affineness assumptions, we cannot get rid of the choice of the open subset $U$ of $X$: indeed, we may always take $Z$ to be the blow-up of $X$ along a finite subscheme of $X\setminus \alpha(P).$

\subsubsection{}

There are cases where the algebra $\mathcal O(\Vfa)$ is empty even though $\Vfa$ comes from algebraic geometry, in some sense. In particular, there might exist non-constant morphisms from $\Vfa$ to arithmetic surfaces over $\OK$ even though $\mathcal O(\Vfa)$ is reduced to $\OK$. 

Indeed, let $X$ be a projective, integral, normal arithmetic surface over $\OK$. Let $P$ be an $\OK$-point of $X$. Assume that the structure map $X\ra\Spec\OK$ is smooth along $P$. For every complex embedding $\sigma$ of $\OK$, let $V_\sigma$ be a Riemann surface with boundary embedded in $X_\sigma$ and containing $P_\sigma$ in its interior. Assume that the $V_\sigma$ are invariant under complex conjugation. Letting $\Vf$ be the formal completion of $X$ along $P$ defines a formal-analytic arithmetic surface $\Vfa=(\Vf, (V_\sigma, \iota_\sigma)_{\sigma : K\ra\C})$ by letting the $\iota_\sigma$ be the natural isomorphisms of complex formal schemes. 

Assume that $\Vfa$ is pseudoconcave -- this is the situation considered in \cite{Bost99}, and we refer to \cite[section 7]{Bost99} for explicit examples. By Theorem \ref{theorem:exc-finite}, there exist only finitely many closed, irreducible Weil divisors $Z$ on $X$ such that the intersection of $Z$ with $P$ and of $Z_\sigma$ with the $V_\sigma$ are all empty. After enlarging the embedded Riemann surfaces with boundary $V_\sigma$, we may assume that any such Weil divisor $Z$ is vertical, namely, it does not intersect the generic fiber $X_K$.

By construction, there is a non-constant morphism $i : \Vfa\ra X$ defined by the immersion $\Vf\ra X$ and $V_\sigma\ra X_\sigma.$ 

\begin{proposition}
The $\OK$-algebra $\mathcal O(\Vfa)$ is reduced to $\OK$. In other words, any morphism 
$$\Vfa\ra Z$$
over $\OK$ from $\Vfa$ to an affine scheme $Z$ over $\OK$ is constant.
\end{proposition}

As both the Archimedean and formal components of $i$ are immersions, the setup of the Proposition is the one considered in \cite{Bost99} and \cite{BostChambert-Loir07}. The argument below is very close to the arguments in those papers.

\begin{proof}
Let $\beta$ be an element in $\mathcal O(\Vfa)$. We consider $\beta$ as a morphism:
$$\gamma : \Vfa\lra\mathbb A^1_\OK.$$

The morphism $i : \Vfa\ra X$ satisfies:
$$\mathrm{Ex}(\widehat i)=0$$
and, for any complex embedding $\sigma$ of $K$:
$$\mathrm{Ex}(i_\sigma^{\mathrm{an}})=0.$$
As a consequence of Theorem \ref{theorem:main}, any commutative diagram  
\[
\xymatrix{
& X'\ar[d]^{f}\\
\Vfa\ar[r]^{\alpha}\ar[ur]^{\alpha'} & X
}
\]
satisfies $\deg f=1.$ Finally, let $D$ be the reunion of those vertical Weil divisors of $X$ that do not meet $i(P).$ Then $i : \Vfa\ra X\setminus D$ satisfies the conclusion of Proposition \ref{proposition:final-surface-qp}: there exists an open subset $U$ of $X\setminus D$ that is the complement of finitely many closed points, and a morphism $f : U\ra \mathbb A^1_\OK$ such that $\alpha$ factors through $U$ and the diagram 
\[
\xymatrix{
& U\ar[d]^{f}\\
\Vfa\ar[r]^{\beta}\ar[ur]^{\alpha} & \mathbb A^1_\OK
}
\]
commutes.

By construction, the generic fiber $U_K$ of $U$ coincides with $X_K$, so that it is proper over $K$. As a consequence, the morphism 
$$f_K : X_K\lra \mathbb A^1_K$$
is constant. Since $X$ is irreducible, this implies that $f : X\ra\mathbb A^1_\OK$ is constant, so that $\beta$ is constant.
\end{proof}

\section{Arithmetic Lefschetz-Nori theorems on \'etale fundamental groups}\label{ArLefNor}
\subsection{}
Let $\Vfa$ be a smooth formal-analytic arithmetic surface over $\OK$, and consider a morphism over $\OK$:
$$\alpha : \Vfa\lra X$$
to a normal, integral, quasi-projective arithmetic surface $X$ over $\OK$. 
Let $\sigma : K\hra\C$ be a complex embedding of $K$. The image $P$ of $|\Vf|\simeq\Spec\OK$ in $X$ is an $\OK$-point $Q$ of $X$. The composition 
$$\Spec\C\stackrel{\sigma^*}{\lra}\Spec K\lra\Spec\OK$$
defines a geometric point $Q_\sigma$ of $X$, and we may consider the \'etale fundamental group
$$\pi^{\mathrm{\acute{e}t}}_1(X, Q_\sigma)$$
of $X$ with respect to the geometric point $Q_\sigma$.
Consider the fundamental group
$$\pi^{\mathrm{\acute{e}t}}_1(\Spec\OK, \sigma).$$
The morphism 
$$\widehat\alpha : \Vf\lra X$$
defines a morphism:
\begin{equation}\label{equation:morphism-formal}
\widehat\alpha^\sigma_{*} : \pi^{\mathrm{\acute{e}t}}_1(\Spec\OK, \sigma)\lra \pi^{\mathrm{\acute{e}t}}_1(X, Q_\sigma).
\end{equation}
Similarly, consider the fundamental group
$$\pi^{\mathrm{\acute{e}t}}_1(V_\sigma, P_\sigma),$$
namely the profinite completion of the topological fundamental group of the complex manifold $V_\sigma$. The holomorphic map
$$\alpha_\sigma^{\mathrm{an}} : V_\sigma\lra X^{\mathrm{an}}_\sigma$$
sends the point $P_\sigma$ of $V_\sigma$ to the complex point $\alpha^{\mathrm{an}}_\sigma(P_\sigma)$ of $X_\sigma$, which we will aslo denote by $Q_\sigma$. We obtain a morphism
$$\alpha_{\sigma, *}^{\mathrm{an}} : \pi^{\mathrm{\acute{e}t}}_1(V_\sigma, P_\sigma)\lra \pi^{\mathrm{\acute{e}t}}_1(X_\sigma, Q_\sigma)$$
obtained by composing the morphism
$$\pi^{\mathrm{\acute{e}t}}_1(V_\sigma, P_\sigma)\lra \pi^{\mathrm{\acute{e}t}}_1(X_\sigma^{\mathrm{an}},Q_\sigma)$$
with the natural morphism 
$$\pi^{\mathrm{\acute{e}t}}_1(X_\sigma^{\mathrm{an}}, Q_\sigma)\lra \pi^{\mathrm{\acute{e}t}}_1(X_\sigma, Q_\sigma),$$
which is actually an isomorphism by the GAGA theorem. In particular, we obtain a morphism:
\begin{equation}\label{equation:morphism-an}
\alpha_{\sigma, *} : \pi^{\mathrm{\acute{e}t}}_1(V_\sigma, P_\sigma)\lra \pi^{\mathrm{\acute{e}t}}_1(X, Q_\sigma).
\end{equation}
Fix a complex embedding $\sigma_0 : K\hra\C$. For any other complex embedding $\sigma : K\hra\C$, choose an automorphism $\phi_\sigma$ of $\C$ such that 
$$\sigma = \phi_\sigma\circ\sigma_0.$$
The induced morphism
$$\phi_\sigma^* : \Spec\C\lra \Spec\C$$
satisfies by construction:
$$\phi_\sigma^*\circ\eta_{\sigma_0}=\eta_\sigma$$
and as a consequence it induces an isomorphism of \'etale fundamental groups
$$f_\sigma : \pi^{\mathrm{\acute{e}t}}_1(X, Q_\sigma)\lrasim \pi^{\mathrm{\acute{e}t}}_1(X, Q_{\sigma_0}).$$
When $\sigma=\sigma_0$, we choose $\phi_{\sigma_0}=\mathrm{Id}_\C$ so that $f_{\sigma_0}$ is the identity.
We denote by $\alpha_{*}$ the morphism of groups:
$$\alpha_{*} : \pi^{\mathrm{\acute{e}t}}_1(\Spec\OK, \sigma_0)\ast\Big(\bigast_{\sigma : K\hra\C}\pi^{\mathrm{\acute{e}t}}_1(V_\sigma, P_\sigma)\Big)\lra \pi^{\mathrm{\acute{e}t}}_1(X, Q_{\sigma_0}),$$
where $\ast$ denotes the free product, with components $\widehat\alpha_{\sigma_0, *}$ and $f_\sigma\circ\alpha_{\sigma, *}$.
Note that the morphism $\alpha_*$ depends on the choice of $\sigma_0$ and on the choices of the automorphisms $\phi_\sigma$ of $\C$ for $\sigma\neq\sigma_0$ -- the dependence on the automorphisms $\phi_\sigma$ only relies on the restriction of $\phi_\sigma$ to the algebraic closure $\overline\Q$ of $\Q$ in $\C$.
\subsection{}
We may state the main result of this subsection and some of its consequences.
\begin{theorem}\label{theorem:fundamental}
Assume that $\Vfa$ is pseudoconcave and define, with the notation of Theorem \ref{theorem:main}:
$$D(\alpha):=e(\alpha)+\frac{1}{\widehat\deg\,\overline N_{P}\Vfa}\Big(\mathrm{Ex}(\widehat\alpha : \Vf\lra \widehat X_{\alpha(P)})+\sum_{\sigma : K\hra \C}\mathrm{Ex}(\alpha_{\sigma} : (V_{\sigma}, P_{\sigma})\lra X_{\sigma})\Big).$$
Then the index of the closure of the subgroup 
$$\alpha_{*}\Big(\pi^{\mathrm{\acute{e}t}}_1(\Spec\OK, \sigma_0)\ast\Big(\bigast_{\sigma : K\hra\C}\pi^{\mathrm{\acute{e}t}}_1(V_\sigma, P_\sigma)\Big)\Big)$$
in $\pi^{\mathrm{\acute{e}t}}_1(X, Q_{\sigma_0})$ is finite and bounded above by $D(\alpha)$.
\end{theorem}
In other words, the closed subgroup generated by the image of $\pi^{\mathrm{\acute{e}t}}_1(\Spec\OK, \sigma_0)$ and the groups $\pi^{\mathrm{\acute{e}t}}_1(V_\sigma, P_\sigma)$ in $\pi^{\mathrm{\acute{e}t}}_1(X, Q_{\sigma_0})$ has index bounded above by $D(\alpha)$.
As an immediate corollary of Theorem \ref{theorem:fundamental}, we find, in the setting of Theorem \ref{theorem:fundamental}:
\begin{corollary}
Assume that the Riemann surfaces with boundary $V_\sigma$ are all simply connected. Then the index of $\alpha_{*}\Big(\pi^{\mathrm{\acute{e}t}}_1(\Spec\OK, \sigma_0)\Big)$ in $\pi^{\mathrm{\acute{e}t}}_1(X, Q_{\sigma_0})$ is bounded above by $D(\alpha)$.
\end{corollary}
Finally, we recover the following extension of the main result of \cite{Bost99}, in the setting of Theorem \ref{theorem:fundamental} again:
\begin{corollary}
Assume that $\alpha$ is an immersion\footnote{Namely, both $\widehat\alpha$ and the holomorphic maps $\alpha^{\mathrm{an}}_\sigma$ are immersions.}. Then the subgroup 
$$\alpha_{*}\Big(\pi^{\mathrm{\acute{e}t}}_1(\Spec\OK, \sigma_0)\ast\Big(\bigast_{\sigma : K\hra\C}\pi^{\mathrm{\acute{e}t}}_1(V_\sigma, P_\sigma)\Big)\Big)$$
is dense in $\pi^{\mathrm{\acute{e}t}}_1(X, Q_{\sigma_0})$.
\end{corollary}
\begin{proof}
This follows from Theorem \ref{theorem:fundamental}, the vanishing of the Archimedean overflow proved in Proposition \ref{proposition:excess-nonnegative} and the obvious vanishing of the term $\mathrm{Ex}(\widehat\alpha : \Vf\lra \widehat X_{\alpha(P)})$ when $\widehat\alpha$ is an immersion.
\end{proof}
\subsection{}
Before proving Theorem \ref{theorem:fundamental}, we prove some preliminary results. We keep the notation above.
Let $Y$ be a normal, integral, quasi-projective arithmetic surface over $\OK$, and let $\pi : Y\ra X$ be a finite \'etale cover over $\OK$. Let $R_{\sigma_0}$ be a complex point of $Y$ mapping to $Q_{\sigma_0}$ by $f$. Let $\sigma : K\hra\C$ be a complex embedding. Define $R_{\sigma}$ to be the complex point $\phi_\sigma\circ\eta_{Y, \sigma_0}$ of $Y$. We denote also by $R_\sigma$ the corresponding complex point of $Y_\sigma$.
\begin{lemma}\label{lemma:various-emb}
Assume that the image of the morphism $f_\sigma\circ\alpha_{\sigma, *} : \pi^{\mathrm{\acute{e}t}}_1(V_\sigma, P_\sigma)\ra \pi^{\mathrm{\acute{e}t}}_1(X, Q_{\sigma_0})$ is contained in $\pi_*(\pi^{\mathrm{\acute{e}t}}_1(Y, R_{\sigma_0})).$ Then the image of the morphism $\alpha_{\sigma, *}^{\mathrm{an}} : \pi^{\mathrm{\acute{e}t}}_1(V_\sigma, P_\sigma)\ra \pi^{\mathrm{\acute{e}t}}_1(X_\sigma, Q_\sigma)$ is contained in $\pi_{\sigma,*}(\pi^{\mathrm{\acute{e}t}}_1(Y_\sigma, R_\sigma).$
\end{lemma}
\begin{proof}
 We have a commutative diagram of profinite groups:
\[
\xymatrix{
& \pi^{\mathrm{\acute{e}t}}_1(Y_\sigma, R_\sigma)\ar[r]\ar[d]^{\pi_{\sigma, *}} & \pi^{\mathrm{\acute{e}t}}_1(Y, R_{\sigma})\ar[r]^{g_\sigma}\ar[d]^{\pi_{\sigma, *}} & \pi^{\mathrm{\acute{e}t}}_1(Y, R_{\sigma_0}) \ar[d]^{\pi_*}\\
\pi^{\mathrm{\acute{e}t}}_1(V_\sigma, P_\sigma) \ar[r]^{\alpha^{\mathrm{an}}_{\sigma, *}} &  \pi^{\mathrm{\acute{e}t}}_1(X_\sigma, Q_\sigma) \ar[r] & \pi^{\mathrm{\acute{e}t}}_1(X, \eta_\sigma)\ar[r]^{f_\sigma} & \pi^{\mathrm{\acute{e}t}}_1(X, Q_{\sigma_0})
}
\]
where $g_\sigma$ is induced by $\phi_\sigma$ just as $f_\sigma$ is and the unlabeled arrows are the natural ones. Note that both $g_\sigma$ and $f_\sigma$ are isomorphisms. As the two leftmost vertical maps are injective and their images have finite index, both equal to the degree of the covering $\pi$, the left-hand square in the diagram above is cartesian.
Since the image of the morphism $f_\sigma\circ\alpha_{\sigma, *} : \pi^{\mathrm{\acute{e}t}}_1(V_\sigma, P_\sigma)\ra \pi^{\mathrm{\acute{e}t}}_1(X, Q_{\sigma_0})$ is contained in the subgroup $\pi_*(\pi^{\mathrm{\acute{e}t}}_1(Y, R_{\sigma_0})),$ the image of the morphism $\alpha_{\sigma, *}^{\mathrm{an}} : \pi^{\mathrm{\acute{e}t}}_1(V_\sigma, P_\sigma)\ra \pi^{\mathrm{\acute{e}t}}_1(X_\sigma, Q_\sigma)$ is contained in $\pi_{\sigma,*}(\pi^{\mathrm{\acute{e}t}}_1(Y_\sigma, R_\sigma).$
\end{proof}
\begin{lemma}\label{lemma:lifting-criterion}
The following statements hold:
\begin{enumerate}
\item Assume that the image of the morphism $\widehat\alpha^{\sigma_0}_{*} : \pi^{\mathrm{\acute{e}t}}_1(\Spec\OK, \sigma_0)\ra \pi^{\mathrm{\acute{e}t}}_1(X, Q_{\sigma_0})$ is contained in $\pi_*(\pi^{\mathrm{\acute{e}t}}_1(Y, R_{\sigma_0})).$ Then: 
$$\widehat\alpha : \Vf\lra X$$
lifts to a morphism $\widehat\beta : \Vf\ra Y$ such that $\pi\circ\widehat\beta=\widehat\alpha.$
\item Let $\sigma : K\hra\C$ be a complex embedding. Assume that the image of the morphism $f_\sigma\circ\alpha_{\sigma, *} : \pi^{\mathrm{\acute{e}t}}_1(V_\sigma, P_\sigma)\ra \pi^{\mathrm{\acute{e}t}}_1(X, Q_{\sigma_0})$ is contained in $\pi_*(\pi^{\mathrm{\acute{e}t}}_1(Y, R_{\sigma_0})).$ Then 
$$\alpha_{\sigma}^{\mathrm{an}} : V_\sigma\lra X_\sigma^{\mathrm{an}}$$
lifts to a morphism $\beta_{\sigma}^{\mathrm{an}} : V_\sigma\ra Y_\sigma^{\mathrm{an}}$ such that $\pi_\sigma^{\mathrm{an}}\circ\beta_{\sigma}^{\mathrm{an}}=\alpha_{\sigma}^{\mathrm{an}}.$
\end{enumerate}
Furthermore, if $(1)$ and $(2)$ both hold, then we may choose the lifts $\widehat\beta$ and $\beta_\sigma^{\mathrm{an}}$ above in such a way that, for any $\sigma : K\hra\C$, the point $\beta_\sigma^{\mathrm{an}}(P_\sigma)$ is the complex point of $Y_\sigma$ defined as the base change by $\sigma$ of the $K$-point
$$\Spec K\lra \Spec\OK\simeq|\Vf|\stackrel{\widehat\beta}{\lra} Y$$
of $Y$.
\end{lemma}
\begin{proof}
With the assumption of $(1)$, the general theory of the \'etale fundamental group shows that the composition
$$\Spec\OK=|\Vf|\hlra\Vf\stackrel{\alpha}{\lra} X$$
lifts to a morphism $|\widehat\beta| : \Spec\OK\ra Y$ such that the composition
$$\Spec\C\stackrel{\eta_{\sigma_0}}{\lra}\Spec\OK\stackrel{|\widehat\beta|}{\lra} Y$$
is the complex point $R_{\sigma_0}.$ In turn, since $\Vf$ is smooth over $\Spec\OK$, this implies that $\widehat\alpha$ lifts to a morphism $\widehat\beta : \Vf\ra Y$ such that 
the composition
$$\Spec\C\stackrel{\eta_{\sigma_0}}{\lra}\Vfa\stackrel{\widehat\beta}{\lra} Y$$
is the complex point $R_{\sigma}.$
If the assumption of $(2)$ holds, Lemma \ref{lemma:various-emb} shows the image of the morphism $\alpha_{\sigma, *}^{\mathrm{an}} : \pi^{\mathrm{\acute{e}t}}_1(V_\sigma, P_\sigma)\ra \pi^{\mathrm{\acute{e}t}}_1(X_\sigma, Q_\sigma)$ is contained in $\pi_{\sigma,*}(\pi^{\mathrm{\acute{e}t}}_1(Y_\sigma, R_\sigma).$
As in $(i)$, this implies that $\alpha_\sigma^{\mathrm{an}}$ lifts to a morphism
$$\beta_\sigma^{\mathrm{an}} : V_\sigma\lra Y_\sigma^{\mathrm{an}}$$
with $\beta_\sigma^{\mathrm{an}}(P_\sigma)=R_\sigma.$ 
To finish the proof, it remains to check that $R_\sigma$ is indeed the complex point of $Y_\sigma$ defined as the base change by $\sigma$ of the $K$-point
$$\Spec K\lra \Spec\OK\simeq|\Vf|\stackrel{\widehat\beta}{\lra} Y.$$
This is clear by the properties of the liftings $\widehat\beta$ and $\beta_\sigma^{\mathrm{an}}$ described above.
\end{proof}
\begin{lemma}\label{lemma:glue-lifts}
Assume that the morphism 
$$\widehat\alpha : \Vf\lra X$$
lifts to a morphism $\widehat\beta : \Vf\lra Y$ such that $f\circ\beta=\alpha$, and assume that, for all complex embeddings $\sigma : K\hra\C$, the holomorphic map $\alpha_\sigma^{\mathrm{an}} : V_\sigma\lra X_\sigma^{\mathrm{an}}$ lifts to a morphism $\beta_\sigma^{\mathrm{an}} : V_\sigma\lra Y_\sigma^{\mathrm{an}}$ such that $f_\sigma^{\mathrm{an}}\circ\beta_\sigma^{\mathrm{an}}=\alpha^{\mathrm{an}}$ and $\beta_\sigma^{\mathrm{an}}(P_\sigma)$ is the complex point of $Y_\sigma$ defined as the base change by $\sigma$ of the $K-point$ 
$$\Spec K\lra \Spec\OK\simeq|\Vf|\stackrel{\widehat\beta}{\lra} Y$$
of $Y$.
Then 
$$\beta :=(\widehat \beta, (\beta_\sigma^{\mathrm{an}})_{\sigma : K\hra\C}))$$
defines a morphism $\beta : \Vfa\lra Y$ such that $f\circ\beta=\alpha.$
\end{lemma}
\begin{proof}
To prove the lemma, it suffices to show that $(\widehat \beta, (\beta_\sigma^{\mathrm{an}})_{\sigma : K\hra\C}))$ does define a morphism from the formal-analytic arithmetic surface $\Vfa$ to $Y$, as, if it does, the equality $f\circ\beta=\alpha$ follows from the similar componentwise equalities.
Write 
$$\Vfa=(\Vf, (V_\sigma, \iota_\sigma)_{\sigma : K\hra\C}))$$
as in \ref{subsubsection:definition-fa}. We want to prove that the datum $(\widehat \beta, (\beta_\sigma^{\mathrm{an}})_{\sigma : K\hra\C}))$ is compatible with the gluing data $(\iota_\sigma)_{\sigma : K\hra\C}$.   This follows immediately from the unicity of the liftings of the holomorphic maps $\alpha_\sigma^{\mathrm{an}}$ with prescribed value at $P_\sigma$.
\end{proof}
\subsection{} We may finish the proof of Theorem \ref{theorem:fundamental}.
\begin{proof}[Proof of Theorem \ref{theorem:fundamental}]
Let $G$ be the closure of the subgroup $$\alpha_\ast \Big(\pi^{\mathrm{\acute{e}t}}_1(\Spec\OK, \sigma_0)\ast \bigast_{\sigma : K\hra\C}\pi^{\mathrm{\acute{e}t}}_1(V_\sigma, P_\sigma)\Big)$$ in $\pi^{\mathrm{\acute{e}t}}_1(X, Q_{\sigma_0})$. 
We argue by contradiction and assume that the index of $G$ -- be it finite or infinite -- is strictly larger than $D(\alpha)$. As a consequence, since the group $\pi^{\mathrm{\acute{e}t}}_1(X, Q_{\sigma_0})$ is profinite, we may find a closed finite-index subgroup $H$ of $\pi^{\mathrm{\acute{e}t}}_1(X, Q_{\sigma_0})$ containing $G$ such that the index of $H$ is $d>D(\alpha).$
The group $H$ correspomds to a finite \'etale cover of pointed arithmetic surfaces:
$$\pi : (Y, R_{\sigma_0})\lra (X, Q_{\sigma_0})$$
that satisfies the following property: the image of the morphism of profinite groups
$$\pi_* : \pi^{\mathrm{\acute{e}t}}_1(Y, R_{\sigma_0})\lra \pi^{\mathrm{\acute{e}t}}_1(X, Q_{\sigma_0})$$
equals $H$. In particular, the degree of $\pi$ is $d>D(\alpha)$, and the image of $\pi_*$ contains $G$.
Lemma \ref{lemma:lifting-criterion} and Lemma \ref{lemma:glue-lifts} guarantee that the morphism $\alpha : \Vfa\lra X$ lifts to a morphism $\beta : \Vfa\lra Y$ such that $\pi\circ\beta=\alpha$. As the degree $d$ of $\pi$ is strictly larger than $D(\alpha)$, this contradicts Theorem \ref{theorem:main}.
\end{proof}

\section[Arithmetic surfaces with finite  \'etale fundamental group and modular curves]{Applications: arithmetic surfaces with finite  \'etale fundamental group and integral models of  modular curves}\label{FinEtpi1} In this final section, we present  applications  of Theorem  \ref{theorem:fundamental} to the construction of arithmetic surfaces that admit a \emph{finite \'etale fundamental group}, in the spirit of the paper \cite{Ihara94} of Ihara   and of \cite{Bost99} --- where arithmetic surfaces with trivial  \'etale fundamental group were constructed --- and of the applications in \cite{CalegariDimitrovTang21} of the ``arithmetic holonomicity theorem" to modular curves.

We have not attempted  to formulate our results in a definitive form of maximal generality. Instead we have tried to spell out some simple qualitative consequences of our results on \'etale fundamental groups of arithmetic surfaces, in the specific case where the ``base arithmetic curve" $\Spec \OK$ is $\Spec \Z$.

\subsection{} When $\Spec \OK = \Spec \Z$, the statement of Theorem \ref{theorem:fundamental} becomes especially simple. Indeed  there is a unique field embedding from $K = \Q$ into $\C$ (!!) and the \'etale fundamental group $\pi^{\mathrm{\acute{e}t}}_1(\Spec\Z, \sigma_0)$ is trivial according to Minkowski's Theorem. Accordingly Theorem \ref{theorem:fundamental}   takes the following form:

\begin{corollary}\label{finiteindexZ} Let $\Vfa:= (\Vf, (V, O), \iota)$ be a smooth formal-analytic arithmetic surface over $\Spec \Z$, and let: 
$$\alpha := (\widehat{\alpha}, \alpha^\an): \Vfa\lra X$$
be a morphism over $\Spec \Z$ from  $\Vfa$
to a normal integral arithmetic surface $X$ such that $\widehat{\alpha}_\Q$, or equivalently $\alpha^\an$, is non-constant. 
Let:
$$\alpha_\ast : \pi^{\mathrm{\acute{e}t}}_1(V,O) \lra  \pi^{\mathrm{\acute{e}t}}_1(X, \alpha^\an(O))$$
be the continuous morphism of profinite groups defined by the composite map:
$$V \stackrel{\alpha^{\an}}{\lra} X_\C \lra X.$$

If $\Vfa$ is pseudoconcave, then $\alpha_\ast (\pi^{\mathrm{\acute{e}t}}_1(V,O))$ is a subgroup of finite index in $\pi^{\mathrm{\acute{e}t}}_1(X, \alpha_\C(O))$. Moreover:
\begin{equation}\label{eq:finiteindexZ}
[\pi^{\mathrm{\acute{e}t}}_1(X, \alpha^\an(O))) : \alpha_\ast (\pi^{\mathrm{\acute{e}t}}_1(V,O)) ] \leq D(\alpha: \Vfa \ra X).
\end{equation}
\end{corollary}  

In turn, specialized to the situation where the compact Riemann surface with boundary $V$ is simply connected, or equivalently when $\Vfa$ is an instance of the \fa arithmetic surfaces  $\Vfa( \Db(0;1), \psi)$ introduced in Section \ref{Vfadpsi},  Corollary \ref{finiteindexZ} becomes:

\begin{corollary}\label{finiteZ} Let $\psi$ be a formal series in:
$$G_{\mathrm{for}}(\R) := \R^\ast X + X^2 \R[[X]],$$
and let: 
$$\alpha := (\widehat{\alpha}, \alpha^\an): \Vfa( \Db(0;1), \psi)\lra X$$
be a morphism over $\Spec \Z$ from  $\Vfa$
to a normal integral arithmetic surface $X$ such that $\widehat{\alpha}_\Q$, or equivalently $\alpha^\an$, is non-constant. 

If $\Vfa( \Db(0;1), \psi)$ is pseudoconcave, or equivalently if $\vert \psi'(0)\vert < 1,$ then the \'etale fundamental group $\pi^{\mathrm{\acute{e}t}}_1(X, \ast)$ is finite. Moreover its cardinality satisfies:
\begin{equation}\label{eq:finiteZ}
\big\vert \pi^{\mathrm{\acute{e}t}}_1(X, \ast)\big\vert \leq D(\alpha: \Vfa \ra X).
\end{equation}
\end{corollary}

In \eqref{eq:finiteZ}, $\ast$ denotes an arbitrary geometric point of $X$. This same notation will be used in Propositions \ref{finiteZbis} and  \ref{notuniv} below.

\subsection{} From Corollary \ref{finiteZ}, we may derive a finiteness criterion for the \'etale fundamental group $\pi^{\mathrm{\acute{e}t}}_1(X, \ast)$ of an arithmetic surface  $X$ equipped with a section of the structure map $X \ra \Spec \Z$ that makes no reference to formal-analytic arithmetic surfaces.

\begin{proposition}\label{finiteZbis} Let $X$ be a normal integral arithmetic surface and let $Q \in X(\Z)$ be a section of the structure morphism $X \ra \Spec \Z$ the image of which lies in the regular locus of $X$.

Let moreover:
\begin{equation}\label{aDb}
a: \Db(0;1)^+ \lra X(\C)
\end{equation}
be a non-constant map, analytic up to the boundary and compatible with complex conjugation, such that:
$$a(0) = Q_\C.$$

Let $e$ denote the ramification index of $a$ at $0$, let:
$$a^{[e]}(0): T_{\Db(0;1), 0}^{\otimes e} \lrasim T_{X(\C), Q_\C}$$
be the $e$-th jet of $a$ at $0,$ and let $\Vert.\Vert_a$ be the Hermitian norm on $T_{X(\C), Q_\C} \simeq (N_Q X)_\C$ defined by:
\begin{equation}\label{VertaDef}
\Vert a^{[e]}(0)((\partial/\partial z)^{\otimes e})\Vert_a = 1.
\end{equation}

If the Hermitian line bundle over $\Spec \Z$:
$$\Nb^a_Q X := (N_Q X, \Vert.\Vert_a)$$
satisfies the positivity condition:
\begin{equation}\label{degaNaQ}
\dega \Nb^a_Q X  > 0,
\end{equation}
then the \'etale fundamental group $\pi^{\mathrm{\acute{e}t}}_1(X, \ast)$ is finite, and its cardinality satisfies:
\begin{equation}\label{eq:finiteZbis}
\big\vert \pi^{\mathrm{\acute{e}t}}_1(X, \ast)\big\vert \leq  e + e\,  \frac{\Ex\big(a: \Db(0;1) \ra X(\C)\big)}{\dega \Nb^a_Q X}.
\end{equation}
\end{proposition}

The proof of Proposition \ref{finiteZbis} is more transparent in the ``unramified case" where $e=1.$  As only this special case of Proposition \ref{finiteZbis} will be used in the following subsections, we first present its simpler proof when $e=1$.

\begin{proof}[Proof of Proposition \ref{finiteZbis} when $e =1$.] Let us consider the formal completion $\Xh_Q$ of $X$ along $Q$. Since the image of $Q$ lies in the regular locus of $X,$ the formal scheme over $\Spec \Z$:
$$\Xh_Q \lra \Spec \Z,$$
equipped with the section $Q$, defines a pointed smooth formal curve over $\Spec \Z$, in the sense \cite[10.4]{Bost2020}. 

Since $e =1,$ the formal germ of $a$ at $0$ defines an isomorphism of smooth formal complex curves:
$$\widehat{a}_0 : \widehat{\Db(0;1)}_0 \lrasim \widehat{X_\C}_{,Q_\C} \simeq (\Xh_Q)_\C.$$
Therefore the triple:
$$\Vfa := (\Xh_Q, \Db(0;1), \widehat{a}_0^{-1})$$
defines a smooth \fa arithmetic surface over $\Spec \Z$. 

If we define:
$$\widehat{\alpha}: \Xh_Q \lra X$$
as  the ``inclusion" morphism, and:
$$\alpha^{\an}:= a : \Db(0;1)^+ \lra X(\C),$$
then the pair:
$$\alpha:= (\widehat{\alpha}, \alpha^{\an}): \Vfa \lra X$$
defines a non-constant morphism with ramification index $e(\alpha) =1$.
Moreover the Hermitian line bundle $\Nb_P \Vfa$ is canonically isomorphic to $\Nb^a_Q X$ and the excess $\Ex(\widehat{\alpha}: \Vf \ra X)$ clearly vanishes.

Consequently the positivity condition \eqref{degaNaQ} is satisfied if and only if $\Vfa$ is pseudoconcave, and when this holds:
\begin{align*}
D(\alpha: \Vfa \ra X)  &  = e(\alpha) + \big(\dega\Nb_P \Vfa\big)^{-1} \Big( \Ex (\walpha: \Vfa \ra X) + \Ex (\alpha^\an: (\Db(0;1), 0)) \ra  X(\C)) \Big) \\
& =
1 + \frac{\Ex (a: \Db(0;1) \ra X(\C)\big)}{\dega \Nb^a_Q X}.
\end{align*}
Therefore the finiteness of $\pi^{\mathrm{\acute{e}t}}_1(X, \ast)$ when \eqref{degaNaQ} holds and the bound \eqref{eq:finiteZbis} on its cardinality follows from Corollary \ref{finiteZ} and from the bound \eqref{eq:finiteZ}.
\end{proof}

\begin{proof}[Proof of Proposition  \ref{finiteZbis}]  We now establish Proposition \ref{finiteZbis} when the ramification index $e$ is an arbitrary positive integer.

The smooth formal curve $\Xh_Q$ over $\Spec \Z$ is isomorphic to $\Spf \Z[[T]]$ (see for instance \cite[10.4.2]{Bost2020}), and we may choose an isomorphism of formal schemes:
$$j: \Xh_Q \lrasim \Spf \Z[[T]].$$
The data of the isomorphism $j$ is indeed equivalent to the one of $t := j^\ast T,$ which is 
a non-zero element of $\Gamma(\Xh_Q, \cO_{\Xh_Q})$ such that
$\div t =Q.$ 

Restricted to the first order neighborhood of $Q$ in $X,$ the formal function $t$ defines a trivialization of the normal bundle of $Q$ in $X$:
$$N_Q X \simeq \Z \, (\partial /\partial t).$$

The morphism of smooth formal complex curve:
\begin{equation*}
j_\C \circ \widehat{a}_0 : \Spf \C[[z]] \simeq \widehat{\Db(0;1)}_0  \stackrel{\widehat{a}_0}{\lra} (\Xh_Q)_\C \stackrel{j_\C}{\lrasim} \Spf \C[[T]]. 
\end{equation*} 
is defined by the series:
$$\phi:= (j_\C \circ \widehat{a}_0)^\ast T \in \R[[z]],$$
which is of the form:
$$\phi = \mu z^e + O(z^{e+1}),$$
with $\mu$ in $\R^\ast.$  

In other words, we have:
$$a^{[e]}(0)((\partial/\partial z)^{\otimes e}) = \mu (\partial /\partial t)_\C.$$
This implies the relations:
$$\Vert \mu  (\partial /\partial t)_\C \Vert_a = 1,$$
and:
\begin{equation}\label{deganNbQmu}
\dega \Nb^a_Q X = \log \Vert (\partial /\partial t)_\C \Vert_a^{-1} = \log \vert \mu \vert.
\end{equation}

We shall write:
$$\mu = \epsilon \vert \mu \vert, \quad \mbox{where $\epsilon \in \{1, -1 \}$.}$$
There exists a unique series $\rho$ in $\R[[z]]$ such that:
$$\rho(0) = 0,  \quad \rho'(0) = \vert \mu \vert^{1/e},$$
and:
$$\phi(z) = \epsilon \rho(z)^e.$$
We shall denote by $\psi$ its compositional inverse:
$$\psi := \rho^{-1} \in G_{\mathrm{for}} (\R) := \R^\ast X + X^2\, \R[[X]],$$
and consider 
the smooth \fa arithmetic surface over $\Spec \Z$:
$$\Vfa := \Vfa (\Db(0;1), \psi) = (\Spf \Z[[X]], \Db(0;1), i),$$
where $i$ is the isomorphism of 
 smooth formal complex curves:
$$i:  (\Spf \Z[[X]])_\C \simeq \Spf \C[[X]] \lrasim \widehat{\Db(0;1)}_0 \simeq \Spf \C[[z]]$$
defined by:
$$i^\ast (z) := \psi(X).$$

Using the expression \eqref{degaVpsi} for $\dega \Nb_P\Vfa,$ we get:
$$\dega \Nb_P\Vfa = \log \vert \psi'(0) \vert^{-1}= \log \vert \rho'(0) \vert = (1/e) \log \vert \mu \vert.$$
Together with \eqref{deganNbQmu}, this establishes the equality:
\begin{equation}\label{degNNe}
\dega \Nb_P\Vfa = (1/e)\,  \dega \Nb^a_Q X.
\end{equation}

We may define a morphism of formal scheme from $\Spf \Z[[X]]$ to $X$ as the composition:\footnote{The same letter $X$ denotes both the arithmetic surface $X$ and an indeterminate. Hopefully this would not introduce any confusion.}
\begin{equation*}
\widehat{\alpha}: \Spf \Z[[X]] \stackrel{X \mapsto \epsilon X^e}{\lra} \Spf \Z[[T] \stackrel{j^{-1}}{\lra} \Xh_Q \hlra X,
\end{equation*}
or equivalently by the relation:
\begin{equation}
\widehat{\alpha}^\ast t = \epsilon X^e.
\end{equation}
Its excess $\Ex(\widehat{\alpha}: \Spf \Z[[X]] \ra X)$ vanishes. Indeed, with the notation of \ref{XregulardefR}, we have $\widehat{\alpha}^\ast (Q) = e P.$

Moreover the analytic map:
$$\alpha^{\an} := a : \Db(0;1)^+ \lra X(\C)$$
is compatible with complex conjugation, and satisfies:
\begin{equation}\label{compalpha}
\alpha^\an \circ i = \widehat{\alpha}_\C : \Spf \C[[X]] \lra X(\C).
\end{equation}
Indeed the following equality holds in $\C[[X]]$:
$$(\alpha^\an \circ i)^\ast t_\C = i^\ast \phi(z) = \phi (\psi(X)) = \epsilon \rho(\psi(X))^e =\epsilon X^e = \widehat{\alpha}_\C^\ast t_\C.$$

The equality \eqref{compalpha} shows that the pair $(\widehat{\alpha}, \alpha^{\an})$ defines a morphism from the \fa arithmetic surface $\Vfa$ to the arithmetic surface $X$:
 $$\alpha:= (\widehat{\alpha}, \alpha^{\an}): \Vfa:=  \Vfa (\Db(0;1), \psi) \lra X.$$
 Clearly it is non-constant, of ramification index $e(\alpha) =e$.
 
 According to \eqref{degNNe}, $\Vfa$ is pseudoconcave if and only if the positivity condition \eqref{degaNaQ} is satisfied. Moreover when this holds, we have:
\begin{align*}
D(\alpha: \Vfa \ra X)  &  = e(\alpha) + \big(\dega\Nb_P \Vfa\big)^{-1} \Big( \Ex (\walpha: \Vfa \ra X) + \Ex (\alpha^\an: (\Db(0;1), 0)) \ra  X(\C)) \Big) \\
& =
e + e\,  \frac{\Ex (a: \Db(0;1) \ra X(\C)\big)}{\dega \Nb^a_Q X}.
\end{align*} 
Here again, the finiteness of $\pi^{\mathrm{\acute{e}t}}_1(X, \ast)$ when \eqref{degaNaQ} holds and the bound \eqref{eq:finiteZbis} on its cardinality follows from Corollary \ref{finiteZ} and from the bound \eqref{eq:finiteZ}.
 \end{proof}
 
  Proposition \ref{finiteZbis} applies notably to models of elliptic curves over $\Q$:

\begin{corollary}\label{pifiniteell} Let $E$ be an elliptic curve over $\Q$, and let $0_E \in E(\Q)$ be its zero element. If $\cE$ is an integral regular arithmetic surface that is a model of $E$, and if $0_E$ extends to an integral point $0_{\cE}$ in $\cE(\Z)$, then the \'etale fundamental group $\pi^{\mathrm{\acute{e}t}}_1(\cE, \ast)$ is finite.\footnote{Where $\ast$ denotes an arbitrary geometric point of $X$.}
 \end{corollary}
 
 In particular, $\pi^{\mathrm{\acute{e}t}}_1(\cE, \ast)$ is finite if $\cE$ is the connected N\'eron model of $E$. Conversely the finiteness of $\pi^{\mathrm{\acute{e}t}}_1(\cE, \ast)$ when $\cE$ is the connected N\'eron model of $E$ is easily seen to imply its finiteness for every model $\cE$ of $E$ that satisfies the assumptions of Corollary \ref{pifiniteell}.
 
The generalization of Corollary \ref{pifiniteell} where $E$ is replaced by an arbitrary abelian variety over $\Q$ is actually valid, and may be derived from Weil's classical results concerning abelian varieties over finite fields and the Frobenius action on their Tate module, as shown by Katz and Lang in \cite{KatzLang81}.\footnote{To derive Corollary \ref{pifiniteell} from  the results in Katz-Lang, observe that we may assume that $\cE$ is smooth over $\Spec \Z,$ and that \cite[Lemma 2]{KatzLang81} implies that $\pi^{\mathrm{\acute{e}t}}_1(\cE, \ast)$ is a quotient of $\pi^{\mathrm{\acute{e}t}}_1(E_{\Qb}, \ast)$, and therefore is abelian. Accordingly $\pi^{\mathrm{\acute{e}t}}_1(\cE, \ast)$ may be identified with the group $\mathrm{Ker}(X/S)$ investigated in \cite{KatzLang81} when $X= \cE$ and $S=\Spec \Z$, and is therefore finite by \cite[Theorem 1]{KatzLang81}.} 


We have explicitly stated Corollary \ref{pifiniteell} for the striking simplicity of its derivation from the general finiteness criterion in  Proposition \ref{finiteZbis}.

 \begin{proof} Let us denote by:
 $$\exp_{E(\C)}: \Lie E_\C \lra E(\C)$$
 the exponential map of the complex Lie group $E(\C)$, and let us choose a non-zero element $v$ of $\Lie E_\R$. For every $\lambda \in \R^\ast_+,$ we may introduce the following  \'etale complex analytic map, compatible with complex conjugation:
 $$a_\lambda: \Db(0;1)^+ \lra E(\C), \quad z \longmapsto \exp_{E(\C)}(z \lambda v).$$
 It is straighforward that, with the notation of Proposition \ref{finiteZbis},  the  norm $\Vert. \Vert_{a_\lambda}$ on $(T_{0_E}E)_\C \simeq \Lie E_C$ satisfies:
 $$\Vert \lambda v \Vert_{a_\lambda} = 1.$$
 Consequently:
 $$\dega \Nb^{a_\lambda}_{0_\cE} \cE = \log \lambda + \dega \Nb^{a_1}_{0_\cE} \cE.$$
 Therefore, if $\lambda$ is large enough, the positivity condition \refeq{degaNaQ} is satisfied by $X := \cE,$ $Q:= 0_\cE$, and $a:= a_\lambda$.
  \end{proof}
  
Using the general form of the ``arithmetic Lefschetz-Nori theorem" stated in Theorem \ref{theorem:fundamental}, which is valid over an arbitrary number field $K,$ a similar argument establishes more generally that, \emph{if $\pi: \cE \ra \Spec \OK$ is an integral regular model over $\OK$ of an elliptic curve $E$ over $K$ such that the zero elements $0_E \in E(K)$ extends to an integral point $0_\cE$ in $\cE(\OK)$, then:
$$\ker \big(\pi_\ast: \pi^{\mathrm{\acute{e}t}}_1(\cE, \ast) \lra 
\pi^{\mathrm{\acute{e}t}}_1(\Spec \OK, \pi(\ast))\big)$$
is a finite group.}\footnote{Actually abelian with two generators, since it is a quotient of $\pi^{\mathrm{\acute{e}t}}_1(E_{\Qb}) \simeq \widehat{\Z}^2$, say by \cite[Lemma 2, (2)]{KatzLang81}.} We leave the details to the interested reader.

\subsection{} If one is interested in the qualitative aspect of Proposition \ref{finiteZbis} only, namely in the finiteness of  $\pi^{\mathrm{\acute{e}t}}_1(X, \ast)$, its hypotheses may be slightly relaxed. 

For instance, one may  assume that the map $a$ introduced in \eqref{aDb} is a non-constant analytic map:
\begin{equation}\label{aDop}
a: \mathring{D}(0;1) \lra X(\C),
\end{equation} 
defined only on the open disk $\mathring{D}(0;1)$. Alternatively, one may replace the positivity condition \eqref{degaNaQ} by the weaker condition:
\begin{equation}\label{degaNaQbis}
\dega \Nb^a_Q X  \geq 0.
\end{equation}

In the first case, the finiteness of  $\pi^{\mathrm{\acute{e}t}}_1(X, \ast)$ follows from Corollary \ref{finiteZbis} applied to 
$$a_r: \Db(0;1)^+ \lra X(\C), \quad z \longmapsto a(rz),$$
where $r \in (0,1)$ is chosen close enough to $1$ so that:
$$\dega \Nb^{a_r}_Q X  > 0.$$

In the second case, it follows again from Corollary \ref{finiteZbis} applied to $a_r$, where now $r$ is chosen close enough to 1 in $(1, +\infty)$ to make $a_r$ defined on an open neighborhood of the closed disk $\Db(0;1)$.

It is actually possible to derive from Corollary \ref{finiteZbis} a variant which covers both the more general definition \eqref{aDop} of $a$, and  the limit case of \eqref{degaNaQbis} when:
$$ \dega \Nb^a_Q X =0.$$

\begin{proposition}\label{notuniv}
Let $X$ be a normal integral arithmetic surface and let $Q \in X(\Z)$ be a section of the structure morphism $X \ra \Spec \Z$ the image of which lies in the regular locus of $X$.

Let moreover:
\begin{equation}\label{aDopbis}
a: \mathring{D}(0;1) \lra X(\C)
\end{equation}
be an analytic map, compatible with complex conjugation and \'etale at $0$,\footnote{Namely, such that its differential $Da(0): T_{\mathring{D}(0;1), 0} \ra T_{X(\C), Q_\C}$ is an isomorphism.} such that:
$$a(0) = Q_\C,$$
and let $\Vert.\Vert_a$ be the Hermitian norm on $T_{X(\C), Q_\C}$ defined by:
\begin{equation}\label{adefnew}
\Vert Da(0)(\partial/\partial z) \Vert_a =1.
\end{equation}

If the Hermitian line bundle over $\Spec \Z$:
$$\Nb^a_Q X := (N_Q X, \Vert.\Vert_a)$$
satisfies: 
\begin{equation}\label{degaNaQvan}
\dega \Nb^a_Q X = 0,
\end{equation}
and if the map $a$ is \emph{not} a universal covering of $X(\C),$ then the \'etale fundamental group $\pi^{\mathrm{\acute{e}t}}_1(X, \ast)$ is finite.
 \end{proposition}

\begin{proof} According to the uniformization theorem, there exists a universal covering:
$$u: (M, 0) \lra (X(\C), Q_\C)$$
of the pointed Riemann surface $(X(\C), Q_\C)$ where $M$ is either $\PP^1(\C)$, $\C$, or $\mathring{D}(0;1)$.\footnote{The complex line $\C$ is embedded in $\PP^1(\C)$  by the usual map $(z \mapsto (1:z))$; in particular $0$ denotes the point $(1:0)$ of $\PP^1(\C)$.}

Since $\mathring{D}(0;1)$ is simply connected, there exists a unique complex analytic map between pointed Riemann surfaces:
$$\widetilde{a}: (\mathring{D}(0;1), 0) \lra (M,0)$$
such that the following diagram is commutative:
\[
\xymatrix{
\mathring{D}(0;1) \ar[r]^{\widetilde{a}}\ar[dr]_a & M\ar[d]^{u}\\
& X(\C).
}
\]
By construction $\widetilde{a}(0) =0,$ and we may consider the derivative $\widetilde{a}'(0)$ of $\widetilde{a}$ at $0$. 

When $M = \mathring{D}(0;1),$ the complex analytic map $\widetilde{a}$ is not a diffeomorphism --- since $a$ is not a universal covering of $X(\C)$ --- and according to the Schwarz lemma, the following estimates holds:
$$\vert  \widetilde{a}'(0) \vert < 1.$$
This implies the following relations between Hermitian norms on $T_{Q_\C} X(\C)$:
$$\Vert .\Vert_u = \vert  \widetilde{a}'(0) \vert \, \Vert .\Vert_a < \Vert.\Vert_a,$$
where $\Vert .\Vert_u$ is defined as in \eqref{adefnew} with $u$ instead of $a$,
and consequently:
$$\dega \Nb^u_Q X := \dega (N_Q X, \Vert.\Vert_a) > \dega \Nb^a_Q X \geq 0.$$
The finiteness of $\pi^{\mathrm{\acute{e}t}}_1(X, \ast)$ therefore follows from the first variant of Proposition \ref{finiteZbis} mentioned above, applied with $a := u$ and $e=1.$ 

When $M$ is $\PP^1(\C)$ or $\C$, then, for every $\lambda \in \Rpa,$ we may consider the map:
$$u_\lambda: \Db(0;1)^+ \lra X(\C), \quad z \longmapsto u(\lambda z).$$
The the following equality between Hermitian norms on $T_{Q_\C} X(\C)$ is satisfied:
$$\Vert.\Vert_{u_\lambda} = \lambda^{-1} \, \vert \widetilde{a}'(0) \vert \, \Vert.\Vert_a,$$
and therefore:
$$\dega \Nb^{u_\lambda}_Q X = \log (\lambda/\vert \widetilde{a}'(0) \vert) 
+ \dega \Nb^{a}_Q X.$$  
When $\lambda$ is large enough, this is positive, and therefore the finiteness of $\pi^{\mathrm{\acute{e}t}}_1(X, \ast)$ follows from Proposition  \ref{finiteZbis} applied with $a:= u_\lambda$ and $e=1.$\footnote{Actually, when $M =\PP^1(\C),$ the fundamental group $\pi^{\mathrm{\acute{e}t}}_1(X, \ast)$ is easily seen to be trivial.}
\end{proof}

In applications, it is convenient to combine Proposition \ref{notuniv} and the following result of independent interest about \'etale fundamental groups of arithmetic surfaces.
\begin{proposition}\label{AbhSaito} Let $X$ be a normal integral arithmetic surface, and let $D$ be a closed reduced subscheme of pure dimension 1 in $X$.

If $D$ is regular and contained in the regular locus of $X$ and if, for every irreducible component\footnote{or equivalently, according to the regularity of $D$, for every connected component.}
$C$ of $D$, the structure morphism:
$$C \lra \Spec \Z$$ is surjective, then the inclusion morphism:
$$\iota: X \setminus D \lra X$$
induces an isomorphism between \'etale fundamental groups:
\begin{equation}\label{TSaitopi1}\iota_\ast :  \pi^{\mathrm{\acute{e}t}}_1(X \setminus D, \ast)  \lrasim \pi^{\mathrm{\acute{e}t}}_1(X, \ast).
\end{equation}
\end{proposition}

In \eqref{TSaitopi1}, $\ast$ denotes an arbitrary geometric point of $X\setminus D.$

Proposition \ref{AbhSaito} follows from a variation on the proof of the triviality of the group:
$$\ker \left( \pi^{\mathrm{\acute{e}t}}_1(\PP^1_{\OK}\setminus\{0,1,\infty\}, \ast) \lra  \pi^{\mathrm{\acute{e}t}}_1(\Spec \OK, \ast')\right)$$ 
due to T. Saito presented in \cite[Appendix]{Ihara94};  see  also \cite{Emsalem99} for related results and references.

\begin{proof} Let us 	assume that $D$ is regular and contained in the regular locus of $X$, and that, for every irreducible component $C$ of $D$, the structure morphism $C \ra \Spec \Z$ is surjective, or equivalently, that for every prime $p$, there exists a closed point $x$ of $C$ such that the finite field $\kappa(x) := \cO_{X,x}/\fm_x$ has characteristic $p$.

Consider a finite \'etale covering:
$$f : Y \lra X\setminus D,$$
with $Y$ connected, and therefore integral since $Y$ like $X\setminus D$ is normal.

Establishing that the map \eqref{TSaitopi1} is an isomorphism amounts to proving that any such finite \'etale covering of $X \setminus D$ extends to a finite \'etale covering of $X$:
\begin{equation}\label{ftYX}
\tilde{f}: \widetilde{Y} \lra X.
\end{equation}

To achieve this, we define $\widetilde{Y}$ and $\tilde{f}$ as in \eqref{ftYX} as the normalization of $X$ in the function field $\kappa(Y)$ of $Y$. Since $Y$ is normal and $f$ is finite, over the open subscheme $ X\setminus D$ of $X$ this normalization may be identified to $f: Y \ra  X\setminus D$, and by purity we are left to show that the ramifications indices of $\tilde{f}$ over every component $C$ of $D$ are equal to one. 

This follows from the absolute Abhyankar Lemma in \cite[Chapitre XIII, Proposition 5.2]{SGA1bis}, which shows that, for any closed point $x$ of $C$ these ramification indices are not divisible by the characteristic of $\kappa(x).$ Indeed   $f$ is tamely ramified along $C$, since $C$ is horizontal and therefore $\kappa(C)$ has characteristic zero.
\end{proof}

\begin{corollary}\label{cor:notuniv} With the notation of Proposition \ref{notuniv}, for every geometric point $\ast$ of $X \setminus Q$, the \'etale fundamental group $\pi^{\mathrm{\acute{e}t}}_1(X \setminus Q, \ast)$ is isomorphic to $\pi^{\mathrm{\acute{e}t}}_1(X, \ast)$ and therefore finite.
 \end{corollary}

\subsection{} Proposition \ref{notuniv} and Corollary \ref{cor:notuniv} are tailored to applications to integral models of modular curves. 

For instance, for every integer $N \geq 3,$ we may consider the scheme $\mathcal{Y}(N)^{\mathrm{arith}}$ defined as in \cite[Section 2.5]{Katz76}\footnote{where it is denoted by$M(\Gamma(N)^{\mathrm{arith}})$.}  as representing the functor that maps a base scheme $S$ to the isomorphism classes of pairs $(\cE, \iota)$, where $\cE$ is an elliptic curve over $S$ and $\iota$ is an isomorphisms of finite flat group schemes over $S$:
$$\iota: (\mu_N \times \Z/N\Z)_S \lrasim \cE[N].$$

It is a smooth affine curve over $\Spec \Z$, with geometrically irreducible fibers, and  admits a partial compactification, defined by a similar moduli problem, where $\cE$ is now a generalized elliptic curve. Moreover the ``cusp at infinity" $\Gamma(N). \infty$ extends to a section $\infty_N$ over $\Spec \Z$ of this partial compactification, and the union of $\mathcal{Y}(N)^{\mathrm{arith}}$ and of the image of $\infty_N$ defines a quasi-projective arithmetic surface $X$ smooth over $\Spec \Z$, with geometrically irreducible fibers. 

The formal neighborhood of $\infty_N$ in $X$ is described by the ``Tate curve" over $\Z[[q^{1/N}]]$, the analytification of which defines a complex analytic map:
$$a: \mathring{D}(0;1) \lra X(\C)$$
which satisfies:
$$a(0) = \infty_{N,\C}$$
and:
$$a(w) := \mbox{isomorphism class of $(\C^\times/ w^{N \Z}, \iota_w)$}, $$
where: 
$$\iota_w: \mu_N(\C) \times Z/N\Z \lrasim  (\C^\times/ w^{N \Z})[N], \quad 
(\zeta, [b]) \longmapsto [\zeta w^b].$$


We may apply Proposition \ref{notuniv} to the arithmetic surface $X$ equipped with the section $Q := \infty_N$ and to the map $a$ just defined. 
Indeed, the fact that the ``formal parameter" $q^{1/N}$ describing the formal neghborhood of $\infty_{N}$ in $X$ coincides with the ``analytic parameter" $w$ in the open unit disk in the above analytic parametrization of $X(\C)$ shows that the vanishing condition \eqref{degaNaQvan} is satisfied.  Moreover the map $a$ is \'etale at $0$, and clearly is not a universal covering of $X(\C)$.\footnote{The inverse image of $\infty_{N,\C}$ is the point $0,$ while the inverse image by $a$ of every point of $X(\C)\setminus \{\infty_{N,\C}\}$ is infinite.} 

Therefore, according to Corollary \ref{cor:notuniv}, we obtain:

\begin{corollary}
 For every integer $N \geq 3$ and every geometric point $\ast$ of  $\mathcal{Y}(N)^{\mathrm{arith}}$, the \'etale fundamental group $\pi^{\mathrm{\acute{e}t}}_1(\mathcal{Y}(N)^{\mathrm{arith}}, \ast)$ is finite.
\end{corollary}

We are not aware of any example of an integral regular arithmetic surface $X$ whose structure map $X \ra \Spec \Z$ admits a section, and whose \'etale fundamental group 
 $\pi^{\mathrm{\acute{e}t}}_1(X, \ast)$ is infinite.

We leave it to the reader to establish generalizations\footnote{concerning the finiteness of the kernel of the map $\pi_\ast:  \pi^{\mathrm{\acute{e}t}}_1(X, \ast) \lra  \pi^{\mathrm{\acute{e}t}}_1(\Spec \OK, \pi(\ast))$.} of the results of this subsection to normal integral arithmetic surfaces $X$ over $\Spec \OK$ --- where $K$ is an arbitrary number field --- that are equipped with a section $Q \in X(\OK)$ of the structure morphism $\pi: X \ra \Spec \OK$, and to apply them to integral models of modular curves defined by classical moduli  problems over ring of integers of cyclotomic fields.

\begin{thebibliography}{BGKK07}

\bibitem[Abh56]{Abhyankar56}
S.~Abhyankar.
\newblock On the valuations centered in a local domain.
\newblock {\em Amer. J. Math.}, 78:321--348, 1956.

\bibitem[AG62]{AndreottiGrauert62}
A.~Andreotti and H.~Grauert.
\newblock Th\'{e}or\`eme de finitude pour la cohomologie des espaces complexes.
\newblock {\em Bull. Soc. Math. France}, 90:193--259, 1962.

\bibitem[{Ahl}30]{Ahlfors1930}
L.~V. {Ahlfors}.
\newblock {Beitr\"age zur Theorie der meromorphen Funktionen}.
\newblock {VII. Skand. Mathematikerkongre{\ss} 84-88}, 1930.

\bibitem[{Ahl}35]{Ahlfors1935}
L.~V. {Ahlfors}.
\newblock {\"Uber eine Methode in der Theorie der meromorphen Funktionen}.
\newblock {\em {Commentat. Phys.-Math.}}, 8(10):1--14, 1935.

\bibitem[And63]{Andreotti63}
A.~Andreotti.
\newblock Th\'eor\`emes de d\'ependance alg\'ebrique sur les espaces complexes
  pseudo-concaves.
\newblock {\em Bull. Soc. Math. France}, 91:1--38, 1963.

\bibitem[And04]{Andre04}
Y.~Andr{\'e}.
\newblock {Sur la conjecture des $p$-courbures de Grothendieck-Katz
  et un probl\`eme de Dwork}.
\newblock In {\em {A. Adolphson et al. (ed.), Geometric aspects of Dwork
  theory. Vol. I}}, pages 55--112. Walter de Gruyter, Berlin, 2004.

\bibitem[Ara71]{Arakelov71}
S.~Ju. Arakelov.
\newblock Families of algebraic curves with fixed degeneracies.
\newblock {\em Izv. Akad. Nauk SSSR Ser. Mat.}, 35:1269--1293, 1971.

\bibitem[Ara74]{Arakelov74}
S.~Ju. Arakelov.
\newblock An intersection theory for divisors on an arithmetic surface.
\newblock {\em Izv. Akad. Nauk SSSR Ser. Mat.}, 38:1179--1192, 1974.

\bibitem[Art70]{Artin70}
M.~Artin.
\newblock Algebraization of formal moduli. {II}. {E}xistence of modifications.
\newblock {\em Ann. of Math. (2)}, 91:88--135, 1970.

\bibitem[Art86]{Artin86}
M.~Artin.
\newblock Lipman's proof of resolution of singularities for surfaces.
\newblock In {\em Arithmetic geometry ({S}torrs, {C}onn., 1984)}, pages
  267--287. Springer, New York, 1986.

\bibitem[BC]{infgeom2}
J.-B. Bost and F.~Charles.
\newblock Infinite dimensional geometry of numbers {II}: mod-affine {A}-schemes
  and positivity in {A}rakelov geometry.
\newblock in preparation.

\bibitem[BCL09]{BostChambert-Loir07}
J.-B. Bost and A.~Chambert-Loir.
\newblock Analytic curves in algebraic varieties over number fields.
\newblock In {\em Algebra, arithmetic, and geometry: in honor of {Y}u. {I}.
  {M}anin. {V}ol. {I}}, volume 269 of {\em Progr. Math.}, pages 69--124, Birkh\"auser Boston Inc.,
  Boston, MA, 2009. 

\bibitem[BGKK07]{BKK07}
J.~I. Burgos~Gil, J.~Kramer, and U.~K\"{u}hn.
\newblock Cohomological arithmetic {C}how rings.
\newblock {\em J. Inst. Math. Jussieu}, 6(1):1--172, 2007.

\bibitem[BGS94]{BostGilletSoule94}
J.-B. Bost, H.~Gillet, and C.~Soul{\'e}.
\newblock Heights of projective varieties and positive {G}reen forms.
\newblock {\em J. Amer. Math. Soc.}, 7(4):903--1027, 1994.

\bibitem[Bor94]{Borel94}
E.~Borel.
\newblock Sur une application d'un th\'eor\`eme de {M}. {H}adamard.
\newblock {\em Bulletin des sciences math\'ematiques}, 18:22--25, 1894.

\bibitem[Bos99]{Bost99}
J.-B. Bost.
\newblock Potential theory and {L}efschetz theorems for arithmetic surfaces.
\newblock {\em Ann. Scient. \'{E}c. Norm. Sup.}, 32:241--312, 1999.

\bibitem[Bos01]{Bost01}
J.-B. Bost.
\newblock Algebraic leaves of algebraic foliations over number fields.
\newblock {\em Publ. Math. I.H.E.S.}, 93:161--221, 2001.

\bibitem[Bos20]{Bost2020}
J.-B. Bost.
\newblock {\em {Theta invariants of Euclidean lattices and infinite-dimensional
  Hermitian vector bundles over arithmetic curves}}, volume 334 of {\em Progress in Mathematics}.
\newblock Birkh\"auser, 2020.

\bibitem[CC85a]{ChudnovskysGroth85}
D.~V. Chudnovsky and G.~V. Chudnovsky.
\newblock Applications of {P}ad\'e approximations to the {G}rothendieck
  conjecture on linear differential equations.
\newblock In {\em Number theory (New York, 1983--84)}, pages 52--100. Springer,
  Berlin, 1985.

\bibitem[CC85b]{ChudnovskysAcad85}
D.~V. Chudnovsky and G.~V. Chudnovsky.
\newblock Pad\'e approximations and {D}iophantine geometry.
\newblock {\em Proc. Nat. Acad. Sci. U.S.A.}, 82(8):2212--2216, 1985.

\bibitem[CDT21]{CalegariDimitrovTang21}
F.~Calegari, V.~Dimitrov, and Y.~Tang.
\newblock The unbounded denominators conjecture.
\newblock {\em arXiv preprint arXiv:2109.09040}, 2021.

\bibitem[CG81]{CommichauGrauert1981}
M.~Commichau and H.~Grauert.
\newblock Das formale {P}rinzip f\"{u}r kompakte komplexe
  {U}ntermannigfaltigkeiten mit {$1$}-positivem {N}ormalenb\"{u}ndel.
\newblock In {\em Recent developments in several complex variables}, volume 100
  of {\em Ann. of Math. Stud.}, pages 101--126. Princeton Univ. Press,
  Princeton, N.J., 1981.

\bibitem[Cha22]{arithmeticbertini}
F.~Charles.
\newblock {A}rithmetic ampleness and an arithmetic {B}ertini theorem.
\newblock {\em Ann. Scient. \'{E}c. Norm. Sup.}, 54: 1541--1590, 2022.

\bibitem[CP16]{CharlesPoonen16}
F.~Charles and B.~Poonen.
\newblock Bertini irreducibility theorems over finite fields.
\newblock {\em J. Amer. Math. Soc.}, 29(1):81--94, 2016.

\bibitem[CS22]{ClausenScholze22}
D.~Clausen and P.~Scholze.
\newblock Condensed Mathematics and Complex Geometry,
\newblock{\tt https://people.mpim-} \newblock{\tt bonn.mpg.de} \newblock{\tt /scholze/Complex.pdf}

\bibitem[Del87]{Deligne85}
P.~Deligne.
\newblock Le d\'eterminant de la cohomologie.
\newblock In {\em Current trends in arithmetical algebraic geometry ({A}rcata,
  {C}alif., 1985)}, volume~67 of {\em Contemp. Math.}, pages 93--177. Amer.
  Math. Soc., Providence, RI, 1987.

\bibitem[DP22]{DutourPatnaik22}
M.~Dutour and M.~M. Patnaik.
\newblock Infinite rank hermitian lattices and loop groups.
\newblock {\em arXiv preprint arXiv:2203.08976}, 2022.

\bibitem[Ems99]{Emsalem99}
M.~Emsalem.
\newblock On reduction of covers of arithmetic surfaces.
\newblock In {\em Applications of curves over finite fields ({S}eattle, {WA},
  1997)}, volume 245 of {\em Contemp. Math.}, pages 117--132. Amer. Math. Soc.,
  Providence, RI, 1999.

\bibitem[EV92]{EsnaultViehweg92}
H.~Esnault and E.~Viehweg.
\newblock {\em Lectures on vanishing theorems}, volume~20 of {\em DMV Seminar}.
\newblock Birkh\"{a}user Verlag, Basel, 1992.

\bibitem[Fal84]{Faltings84}
G.~Faltings.
\newblock Calculus on arithmetic surfaces.
\newblock {\em Ann. of Math. (2)}, 119(2):387--424, 1984.

\bibitem[FG02]{FritzscheGrauert02}
K.~Fritzsche and H.~Grauert.
\newblock {\em From holomorphic functions to complex manifolds}, volume 213 of
  {\em Graduate Texts in Mathematics}.
\newblock Springer-Verlag, New York, 2002.

\bibitem[FL81]{FultonLazarsfeld81}
W.~Fulton and R.~Lazarsfeld.
\newblock Connectivity and its applications in algebraic geometry.
\newblock In {\em Algebraic geometry ({C}hicago, {I}ll., 1980)}, volume 862 of
  {\em Lecture Notes in Math.}, pages 26--92. Springer, Berlin-New York, 1981.

\bibitem[Ful98]{FultonIT}
W.~Fulton.
\newblock {\em {Intersection theory. 2nd ed.}}
\newblock {Ergebnisse der Mathematik und ihrer Grenzgebiete. 3. Folge. 2.
  Berlin: Springer}, 1998.

\bibitem[GP08]{GarlandPatnaik08}
H.~Garland and M.~Patnaik.
\newblock Geometry of loop {E}isenstein series.
\newblock Available at: {\tt https://sites.ualberta.ca/$\sim$patnaik/}, 2008.

\bibitem[GPR94]{GrauertPeternellRemmert94}
H.~Grauert, Th. Peternell, and R.~Remmert, editors.
\newblock {\em Several complex variables. {VII} Sheaf-theoretical methods in
  complex analysis}, volume~74 of {\em Encyclopaedia of Mathematical Sciences}.
\newblock Springer-Verlag, Berlin, 1994.

\bibitem[Gra62]{Grauert62}
H.~Grauert.
\newblock {\"U}ber {M}odifikationen und exzeptionelle analytische {M}engen.
\newblock {\em Math. Ann.}, 146:331--368, 1962.

\bibitem[Gri76]{Griffiths76}
P.~A. Griffiths.
\newblock {\em Entire holomorphic mappings in one and several complex
  variables}.
\newblock Annals of Mathematics Studies, No. 85. Princeton University Press,
  Princeton, N. J., 1976.

\bibitem[Gro68]{GrothendieckSGA2}
A.~Grothendieck.
\newblock {\em Cohomologie locale des faisceaux coh\'erents et th\'eor\`emes de
  {L}efschetz locaux et globaux}.
\newblock S{\'e}minaire de G{\'e}om{\'e}trie Alg{\'e}brique du Bois-Marie, SGA
  2, 1962. Advanced Studies in Pure Mathematics, Vol. 2. North-Holland
  Publishing Co., Amsterdam, 1968.

\bibitem[Gro71]{SGA1bis}
A.~Grothendieck.
\newblock {\em Rev\^{e}tements \'{e}tales et groupe fondamental}.
\newblock S\'{e}minaire de G\'{e}om\'{e}trie Alg\'{e}brique du Bois Marie, SGA
  1, 1960--1961. Dirig\'{e} par Alexandre Grothendieck. Augment\'{e} de deux
  expos\'{e}s de M. Raynaud.
  \newblock Lecture Notes in Mathematics, Vol. 224. Springer-Verlag, Berlin-New
  York, 1971.

\bibitem[GS90]{Gillet-Soule90int}
H.~Gillet and C.~Soul{\'e}.
\newblock Arithmetic intersection theory.
\newblock {\em Inst. Hautes \'Etudes Sci. Publ. Math.}, (72):93--174 (1991),
  1990.

\bibitem[Gun90]{Gunning90I}
R.~C. Gunning.
\newblock {\em Introduction to holomorphic functions of several variables.
  {V}ol. {I}:Function theory}.
\newblock Wadsworth \& Brooks/Cole Advanced Books \& Software, Pacific Grove,
  CA, 1990.

\bibitem[Har68]{Hartshorne68}
R.~Hartshorne.
\newblock Cohomological dimension of algebraic varieties.
\newblock {\em Ann. of Math.}, 88:403--450, 1968.

\bibitem[Har84]{Harbater84}
D.~Harbater.
\newblock Convergent arithmetic power series.
\newblock {\em Amer. J. Math.}, 106(4):801--846, 1984.

\bibitem[Har88]{Harbater88}
D.~Harbater.
\newblock Galois covers of an arithmetic surface.
\newblock {\em Am. J. Math.}, 110:849--885, 1988.

\bibitem[Hir81]{Hirschowitz1981}
A.~Hirschowitz.
\newblock On the convergence of formal equivalence between embeddings.
\newblock {\em Ann. of Math. (2)}, 113(3):501--514, 1981.

\bibitem[HM68]{HironakaMatsumura68}
H.~Hironaka and H.~Matsumura.
\newblock Formal functions and formal embeddings.
\newblock {\em J. Math. Soc. Japan}, 20:52--82, 1968.

\bibitem[Hri85]{Hriljac85}
P.~Hriljac.
\newblock Heights and {A}rakelov's intersection theory.
\newblock {\em Amer. J. Math.}, 107(1):23--38, 1985.

\bibitem[Iha94]{Ihara94}
Y.~Ihara.
\newblock Horizontal divisors on arithmetic surfaces associated with
  {B}ely\u{\i} uniformizations.
\newblock In {\em The {G}rothendieck theory of dessins d'enfants ({L}uminy,
  1993)}, volume 200 of {\em London Math. Soc. Lecture Note Ser.}, pages
  245--254. Cambridge Univ. Press, Cambridge, 1994.

\bibitem[Kat76]{Katz76}
N.~M. Katz.
\newblock {$p$}-adic interpolation of real analytic {E}isenstein series.
\newblock {\em Ann. of Math. (2)}, 104(3):459--571, 1976.

\bibitem[KL81]{KatzLang81}
N.~M. Katz and S.~Lang.
\newblock Finiteness theorems in geometric classfield theory.
\newblock {\em Enseign. Math. (2)}, 27(3-4):285--319 (1982), 1981.

\bibitem[K{\"{u}}h01]{Kuhn01}
U.~K{\"{u}}hn.
\newblock Generalized arithmetic intersection numbers.
\newblock {\em J. Reine Angew. Math.}, 534:209--236, 2001.

\bibitem[Lau71]{Laufer71}
H.~B. Laufer.
\newblock {\em Normal two-dimensional singularities}.
\newblock Annals of Mathematics Studies, No. 71. Princeton University Press,
  Princeton, N.J.; University of Tokyo Press, Tokyo, 1971.

\bibitem[Laz04]{Lazarsfeld04}
R.~Lazarsfeld.
\newblock {\em Positivity in algebraic geometry. {I}}, volume~48 of {\em
  Ergebnisse der Mathematik und ihrer Grenzgebiete. 3. Folge}.
\newblock Springer-Verlag, Berlin, 2004.

\bibitem[Lic68]{Lichtenbaum68}
S.~Lichtenbaum.
\newblock Curves over discrete valuation rings.
\newblock {\em Amer. J. Math.}, 90:380--405, 1968.

\bibitem[Lip78]{Lipman78}
J.~Lipman.
\newblock Desingularization of two-dimensional schemes.
\newblock {\em Ann. of Math. (2)}, 107(1):151--207, 1978.

\bibitem[LP20]{LemanissierPoineau2020}
Th. Lemanissier and J.~Poineau.
\newblock Espaces de {B}erkovich sur $\mathbf{Z}$ : cat\'egorie, topologie,
  cohomologie, 2020.
\newblock https://arxiv.org/abs/2010.08858.

\bibitem[MB89]{Moret-Bailly89}
L.~Moret-Bailly.
\newblock Groupes de {P}icard et probl\`emes de {S}kolem. {I}.
\newblock {\em Ann. Scient. \'{E}c. Norm. Sup.}, 22:161--179, 1989.

\bibitem[Mum61]{Mumford61}
D.~Mumford.
\newblock {The topology of normal singularities of an algebraic surface and a
  criterion for simplicity.}
\newblock {\em Publ. Math. I.H.E.S.}, 9:5--22, 1961.




\bibitem[Nor83]{Nori83}
M.~V. Nori.
\newblock Zariski's conjecture and related problems.
\newblock {\em Ann. Sci. \'{E}cole Norm. Sup. (4)}, 16(2):305--344, 1983.

\bibitem[NR98]{NapierRamachandran98}
T.~Napier and M.~Ramachandran.
\newblock The {$L^2\ \overline \partial$}-method, weak {L}efschetz
              theorems, and the topology of {K}\"{a}hler manifolds.
\newblock{\em J. Amer. Math. Soc.},  11:375--396,   1998.            

\bibitem[Poi10]{Poineau2010}
J.~Poineau.
\newblock La droite de {B}erkovich sur {$\mathbf Z$}.
\newblock {\em Ast\'{e}risque}, (334):viii+xii+284, 2010.

\bibitem[Poi13]{Poineau2013}
J.~Poineau.
\newblock Espaces de {B}erkovich sur {$\mathbf {Z}$}: \'{e}tude locale.
\newblock {\em Invent. Math.}, 194(3):535--590, 2013.

\bibitem[Sha66]{Shafarevich66}
I.~R. Shafarevich.
\newblock {\em Lectures on minimal models and birational transformations of two
  dimensional schemes}.
\newblock Tata Institute of Fundamental Research Lectures on Mathematics and
  Physics, No. 37. Tata Institute of Fundamental Research, Bombay, 1966.

\bibitem[Siu74]{Siu74bis}
Y.~T. Siu.
\newblock {\em Techniques of extension of analytic objects}.
\newblock Lecture Notes in Pure and Applied Mathematics, Vol. 8. Marcel Dekker,
  Inc., New York, 1974.

\bibitem[{Sta}]{stacks-project}
The {Stacks Project Authors}.
\newblock {\em Stacks {P}roject}.
\newblock http://stacks.math.columbia.edu.

\bibitem[Ste86]{Steinbiss86}
V.~Steinbi\ss.
\newblock Das formale {P}rinzip f\"{u}r reduzierte komplexe {R}\"{a}ume mit
  einer schwachen {P}ositivit\"{a}tseigenschaft.
\newblock {\em Math. Ann.}, 274(3):485--502, 1986.

\bibitem[Szp81]{SemSzpiro81}
L.~Szpiro, editor.
\newblock {\em S\'{e}minaire sur les {P}inceaux de {C}ourbes de {G}enre au
  {M}oins {D}eux}, volume 86 of {\em Ast\'{e}risque}.
\newblock Soci\'{e}t\'{e} Math\'{e}matique de France, Paris, 1981.

\bibitem[Tay11]{Taylor11}
M.~E. Taylor.
\newblock {\em Partial differential equations {I}. {B}asic theory}, volume 115
  of {\em Applied Mathematical Sciences}.
\newblock Springer, New York, second edition, 2011.

\bibitem[Zar54]{Zariski54}
O.~Zariski.
\newblock Interpr\'{e}tations alg\'{e}brico-g\'{e}om\'{e}triques du
  quatorzi\`eme probl\`eme de {H}ilbert.
\newblock {\em Bull. Sci. Math. (2)}, 78:155--168, 1954.

\bibitem[Zar73]{Zariski73}
O.~Zariski.
\newblock {\em Collected papers. {V}ol. {II}: {H}olomorphic functions and
  linear systems}.
  \newblock Edited by M. Artin and D. Mumford.
\newblock Mathematicians of Our Time, Vol. 6. MIT Press, Cambridge,
  Mass.-London, 1973.


\bibitem[Zha95]{Zhang95}
S.~Zhang.
\newblock Positive line bundles on arithmetic varieties.
\newblock {\em J. Amer. Math. Soc.}, 8(1):187--221, 1995.

\end{thebibliography}

\end{document}